\documentclass{memo-l}

\usepackage{amssymb}
\usepackage{verbatim}
\usepackage{graphicx}
\usepackage[cmtip,all]{xy}
\usepackage{amsmath}
\usepackage{amsfonts,amssymb}
\usepackage{times}
\usepackage{amscd}
\usepackage{epsfig}
\usepackage{float}
\usepackage{cleveref}
\usepackage{amsthm}
\usepackage{latexsym}
\usepackage{imakeidx}
\usepackage{pb-diagram}

\usepackage[active]{srcltx}

\makeindex

\newtheorem{theo}{Theorem}[chapter]
\newtheorem{thm}[theo]{Theorem}
\newtheorem{lem}[theo]{Lemma}
\newtheorem{cor}[theo]{Corollary}

\newtheorem{prop}[theo]{Proposition}
\newtheorem{example}[theo]{Example}

\newtheorem*{thmr1}{First Rigidity Theorem}
\newtheorem*{thmr2}{Second Rigidity Theorem}
\newtheorem*{thmlc}{Theorem on Local Charts for Dendritic Polynomials}
\newtheorem*{thmcd}{Theorem on Critically Defined Slices of Laminations}
\newtheorem*{lppm}{Local Pinched Polydisk Model for Dendritic Polynomials}

\theoremstyle{definition}
\newtheorem{dfn}[theo]{Definition}

\theoremstyle{remark}

\numberwithin{section}{chapter}
\numberwithin{equation}{section}

\newcommand{\pr}{\mathbb L\mathbb{P}}
\newcommand{\prnp}{\pr^{np}}
\newcommand{\fg}{\mathrm{FG}}
\newcommand{\di}{\ol{\mathrm{Di}}}
\def\R{\mathbb{R}}
\def\Z{\mathbb{Z}}
\def\Mc{\mathcal{M}}

\newcommand{\qcp}{\mathrm{QCP}}

\newcommand{\oc}{\ol{c}}

\newcommand{\oy}{\ol{y}}

\newcommand{\cp}{\mathcal C}

\newcommand{\cpd}{\mathcal{CPD}}
\newcommand{\cmd}{\mathcal{CMD}^{sim}}

\newcommand{\rc}{\mathcal{R}}
\newcommand{\zc}{\mathcal{Z}}
\newcommand{\nc}{\mathcal{N}}

\newcommand{\nin}{\notin}

\newcommand{\lp}{\mathcal{LP}}

\newcommand{\C}{\mathbb{C}}

\newcommand{\disk}{\mathbb{D}}
\newcommand{\cdisk}{\ol{\mathbb{D}}}

\newcommand{\la}{\lambda}
\newcommand{\ga}{\gamma}
\newcommand{\vp}{\varphi}
\newcommand{\ol}{\overline}
\newcommand{\ovc}{\ol{c}}

\newcommand{\sm}{\setminus}

\newcommand{\B}{\mathcal{B}}
\newcommand{\Tc}{\mathcal{T}}

\newcommand{\hPsi}{\widehat{\Psi}}
\newcommand{\bt}{\ol{t}}
\newcommand{\m}{\ol{m}}
\newcommand{\n}{\ol{n}}
\newcommand{\hell}{\hat{\ell}}

\newcommand{\qml}{\mathrm{QML}}
\newcommand{\bd}{\mathrm{Bd}}

\newcommand{\lam}{\mathcal{L}}

\newcommand{\happrox}{\widehat \approx}
\newcommand{\hlam}{\mathcal{\widehat L}}

\newcommand{\lamm}{\mathcal{L}^m}

\newcommand{\fqcp}{\mathcal{QCP}}

\newcommand{\ch}{\mathrm{CH}}

\newcommand{\si}{\sigma}
\newcommand{\ph}{\varphi}
\newcommand{\uc}{\mathbb{S}}

\newcommand{\g}{\mathfrak{g}}

\newcommand{\e}{\varepsilon}
\newcommand{\M}{\mathcal{M}}

\newcommand{\Ss}{\mathcal{S}}

\newcommand{\mD}{\mathcal{D}}
\newcommand{\Cc}{\mathcal{C}}
\newcommand{\Uc}{\mathcal{U}}

\def\L{\mathbb{L}}
\def\dc{\mathcal{D}}

\renewcommand\le{\leqslant}
\renewcommand\ge{\geqslant}
\def\0{\varnothing}

\begin{document}

\frontmatter

\title{Laminational models for some\\ spaces of polynomials of any degree}

\date{December 22, 2016; revised April 27, 2017 and May 12, 2017}

\author[A.~Blokh]{Alexander~Blokh}

\thanks{The first and the third named authors were partially
supported by NSF grant DMS--1201450}

\author[L.~Oversteegen]{Lex Oversteegen}

\author[R.~Ptacek]{Ross~Ptacek}

\author[V.~Timorin]{Vladlen~Timorin}

\thanks{The study has been funded by the Russian Academic Excellence Project '5-100'.}

\address[Alexander~Blokh and Lex~Oversteegen]
{Department of Mathematics\\ University of Alabama at Birmingham\\
Birmingham, AL 35294}

\address[Ross~Ptacek and Vladlen~Timorin]
{Faculty of Mathematics\\
Laboratory of Algebraic Geometry and its Applications\\
National Research University Higher School of Economics\\
6 Usacheva str., Moscow, Russia, 119048}

%\address[Vladlen~Timorin]
%{Independent University of Moscow\\
%Bolshoy Vla-syevskiy Pereulok 11, 119002 Moscow, Russia}

\email[Alexander~Blokh]{ablokh@math.uab.edu}
\email[Lex~Oversteegen]{overstee@uab.edu}
\email[Ross~Ptacek]{rptacek@uab.edu}
\email[Vladlen~Timorin]{vtimorin@hse.ru}

\subjclass[2010]{Primary 37F20; Secondary 37F10, 37F50}

\keywords{Complex dynamics; laminations; Mandelbrot set; Julia set}

\begin{abstract}
The so-called ``pinched disk'' model of the Mandelbrot set is due to
A.~Douady, J.~H.~Hubbard and W.~P.~Thurston. It can be described in
the language of geodesic laminations. The combinatorial model is the
quotient space of the unit disk under an equivalence relation that,
loosely speaking, ``pinches'' the disk in the plane (whence the name
of the model). The significance of the model lies in particular in
the fact that this quotient is planar and therefore can be easily
visualized. The conjecture that the Mandelbrot set is actually
homeomorphic to this model is equivalent to the celebrated MLC
conjecture stating that the Mandelbrot set is locally connected.

For parameter spaces of higher degree polynomials no combinatorial
model is known. One possible reason may be that the higher degree
analog of the MLC conjecture is known to be false. We investigate to
which extent a geodesic lamination is determined by the location of
its critical sets and when different choices of critical sets lead to
essentially the same lamination. This yields models of various
parameter spaces of laminations similar to the ``pinched disk'' model
of the Mandelbrot set.
\end{abstract}

\maketitle

\setcounter{page}{4}

\tableofcontents

\mainmatter

\chapter{Introduction}\label{s:intro}

The \textbf{parameter space} of complex degree $d$ polynomials is by
definition the space of affine conjugacy classes of these polynomials.
An important subset of the parameter space is the \emph{connectedness
locus}\index{connectedness locus} $\Mc_d$ consisting of classes of all
degree $d$ polynomials $P$, whose Julia sets $J(P)$ are connected.
General properties of the connectedness locus $\Mc_d$ of degree $d$
polynomials have been studied for quite some time. For instance, it is
known that $\Mc_d$ is a compact cellular set in the parameter space of
complex degree $d$ polynomials (this was proven in \cite{BrHu} in the
cubic case and in \cite{la89} for higher degrees, see also
\cite{bra86}; by definition, going back to Morton Brown \cite{bro60}, a
subset $X$  of a Euclidean space $\R^n$ is
\emph{cellular}\index{cellular set} if there exists a sequence $Q_n$ of
topological $n$-cells such that $Q_{n+1}\subset \mathrm{Int}(Q_n)$ and
$X=\cap Q_n$\,).

For $d=2$, the connectedness locus is the famous \emph{Mandelbrot
set}\index{Mandelbrot set} $\Mc_2$, which can be identified with the
set of complex numbers $c$ such that $0$ does not escape to infinity
under the iterations of the polynomial $P_c(z)=z^2+c$. The
identification is based on the fact that every quadratic polynomial is
affinely conjugate to $P_c$ for some $c\in\C$ as well as a classical
theorem of Fatou and Julia. The Mandelbrot set $\Mc_2$ has a
complicated self-similar structure (for instance, homeomorphic copies
of the Mandelbrot set are dense in the Mandelbrot set itself). A
crucial role in understanding its structure is played by the ``pinched
disk'' model by Adrien Douady, John Hamal Hubbard and William Thurston
\cite{dh82, hubbdoua85, thu85}. This model can be described as  a
\emph{geodesic lamination}\index{geodesic lamination} (see the index in
the back for the definitions of non-standard terms).

In this paper, we will partially generalize these results to the higher
degree case. We replace the notion of non-disjoint minors by linked or
essentially equal critical quadrilaterals  and show that in certain
cases two linked or essentially equal laminations must coincide. We
apply these results to construct models of some  spaces of laminations.

In what follows we assume basic knowledge of complex dynamics (a good
reference is John Milnor's book \cite{mil00}). Important developments
can be found in Curtis McMullen's book \cite{mcm94b}. We use standard
notation. However, we describe in detail less well known facts
concerning, e.g., combinatorial concepts (such as \emph{geodesic
laminations} developed by Thurston in \cite{thu85}, or
\emph{laminational equivalence relations}) that will serve as important
tools for us.

\section{Laminations}\label{ss:overview}

Laminations were introduced by Thurston in his paper \cite{thu85} and
have been used as a major tool in complex dynamics ever since.

We will write $\C$ for the plane of complex numbers, $\widehat\C$ for
the Riemann sphere, and $\disk=\{z\in\C\,:\, |z|<1\}$ for the open unit
disk. A \emph{laminational equivalence relation}\index{laminational
equivalence relation} is a closed equivalence relation $\sim$ on the
unit circle $\uc=\{z\in\C\,:\,|z|=1\}$, whose classes are finite sets,
such that the convex hulls of distinct classes are disjoint. A
laminational equivalence relation is ($\si_d$-)\emph{invariant} if the
map $\si_d:\uc\to\uc$, defined by $\si_d(z)= z^d$, takes classes to
classes, and the restriction of $\si_d$ to every class $\g$ can be
extended to an orientation preserving covering map $\tau$ of the circle
of some degree $k\le d$ so that $\g$ is a full preimage of $\tau(\g)$.

If a polynomial $P$ has a connected filled Julia set $K$, then, by the
Riemann mapping theorem, there exists a conformal map $\vp:\C\sm
\ol{\disk}\to \C\sm K$ so that $\vp\circ \si_d=P\circ \vp$. The image
of the radial segment $\{re^{i\theta} \mid r>1\}$ under $\vp$ is called
an \emph{external} ray of $K$ with argument $\theta$. If, in addition,
$J=\bd(K)$ is locally connected, then $\vp$ extends over $\uc$. In this
case, there is a laminational equivalence relation $\sim_P$ on $\uc$
which identifies pairs of angles if the corresponding external rays
land at the same point in $J$. The quotient $J_{\sim_P}=\uc/{\sim_P}$
is homeomorphic to $J$, and the self-mapping $f_{\sim_P}$ of
$J_{\sim_P}$ induced by $\si_d$ is topologically conjugate to
$P|_{J_P}$; the map $f_{\sim_P}$ and the set $J_{\sim_P}$ are called a
\emph{topological polynomial}\index{topological polynomial} and a
\emph{topological Julia set}\index{topological Julia set},
respectively.

Laminational equivalence relations can play a significant role even for
some polynomials whose connected Julia sets are not locally connected.
For these polynomials $\sim_P$ still can be defined, although
$P|_{J_P}$ and $f_{\sim_P}|_{J_{\sim_P}}$ are no longer conjugate.
However, they are semiconjugate by a \emph{monotone map}\index{monotone
map} (a continuous map, whose fibers are continua). A topological
polynomial and  topological Julia set can be defined for every
$\si_d$-invariant laminational equivalence relation even if it does not
correspond (in the above sense) to a complex polynomial.

With every laminational equivalence relation $\sim$, it is useful to
associate geometric objects defined below. We identify $\uc$ with
$\R/\Z$. For a pair of points $a$, $b\in\uc$, we will write $\ol{ab}$
for the \emph{chord} (a straight line segment in $\C$) connecting $a$
and $b$ (in particular, a chord is always contained in the closed unit
disk $\cdisk$). If $G$ is the convex hull $\ch(G')$ of some closed set
$G'\subset\uc$, then we write $\si_d(G)$ for the set $\ch(\si_d(G'))$.
The boundary of $G$ will be denoted by $\bd(G)$. If $A$ is a
$\sim$-class, then we call a chord $\ol{ab}$ in $\bd(\ch(A))$ a
\emph{leaf} of $\sim$. All points of $\uc$ are also called
(\emph{degenerate\emph{)} leaves}\index{leaf}. The family $\lam_\sim$
of all leaves of $\sim$ is called the \emph{$(\si_d$-$)$invariant
geodesic lamination generated by the relation $\sim$}\index{invariant
geodesic lamination!generated by a laminational equivalence}.

Let us explain the terminology. The set $\lam_\sim$ is called
\emph{invariant} for two reasons: for every non-degenerate leaf
$\ol{xy}\in \lam_\sim$ we have $\si_d(\ol{xy})\in \lam_\sim$, and, on
the other hand, there exist $d$ disjoint leaves in $\lam_\sim$ such
that their $\si_d$-images equal $\ol{xy}$. The set $\lam_\sim$ is
called \emph{geodesic} because the standard visual interpretation of
chords of $\lam_\sim$ uses \emph{geodesics} in the unit disk with
respect to the   Euclidean (or, equivalently, Poincar\'e) metric (see
Figure~\ref{f:rabbit} for an illustration). Denote by $\lam^+_\sim$ the
union of the unit circle and all the leaves in $\lam_\sim$. Then
$\lam^+_\sim$ is a subcontinuum of the closed unit disk $\cdisk$. In
general, collections of leaves with properties similar to those of
collections $\lam_\sim$ are also called \emph{invariant geodesic
laminations}. In fact, it is these collections that Thurston introduced
and studied in \cite{thu85}.

Let $\lam$ be an invariant geodesic lamination (for instance, we may
have $\lam=\lam_\sim$ for some invariant lamination $\sim$). The
closure in $\C$ of a non-empty component of $\disk\sm \lam^+$ is called
a \emph{gap}\index{gap} of $\lam$. \emph{Edges of a gap $G$} are
defined as leaves of $\lam$ on the boundary of $G$ and we call $G\cap
\uc$ the \emph{basis} of the gap  $G$. A gap is said to be \emph{finite
$($infinite$)$} if its basis is finite (infinite). Gaps of $\lam$ with
uncountable basis are called \emph{Fatou gaps}\index{gap!Fatou}.

The first application of geodesic laminations was in the quadratic case
\cite{thu85}. Let us discuss it in more detail.

\section{``Pinched disk'' model of the Mandelbrot set}

The  ``pinched disk'' model for $\Mc_2$ is constructed as follows, cf.
\cite{dou93, thu85}. We will identify $\uc$ with $\R/\Z$ by means of
the mapping taking an \textbf{angle} $\theta\in\R/\Z$ to the point
$e^{2\pi i\theta}\in\uc$. Under this identification, we have
$\si_2(\theta)=2\theta$.

If the Julia set $J(P_c)$ is locally connected, then, as was explained
above in Subsection~\ref{ss:overview}, Thurston associates to the
polynomial $P_c$ (and hence to the parameter value $c$) a laminational
equivalence relation $\sim_{P_c}$ and then the corresponding
$\si_2$-invariant geodesic lamination $\lam_{\sim_{P_c}}=\lam_c$. The
dynamics of $\si_2:\uc\to\uc$ descends to the quotient space, and the
induced dynamics $f_{\sim_{P_c}}:\uc/\sim_{P_c}\to \uc/\sim_{P_c}$ is
topologically conjugate to $P_c|_{J(P_c)}:J(P_c)\to J(P_c)$. A
$\si_2$-invariant geodesic lamination is also called a \emph{quadratic
invariant geodesic lamination}\index{invariant geodesic
lamination!quadratic}. The lamination $\lam_c$ is also called the
\emph{(quadratic invariant) geodesic lamination} of $P_c$. In what
follows when talking about quadratic invariant geodesic laminations we
often omit ``invariant'' or ``geodesic'' (indeed, we only deal with
\textbf{geodesic} laminations, and ``quadratic'' already assumes
``invariant'').

\begin{figure}
\includegraphics[height=5cm]{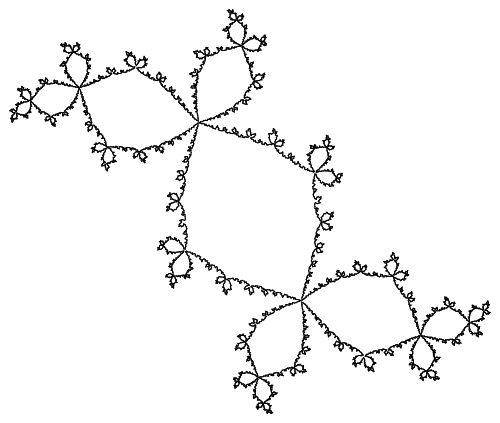}
\includegraphics[height=5cm]{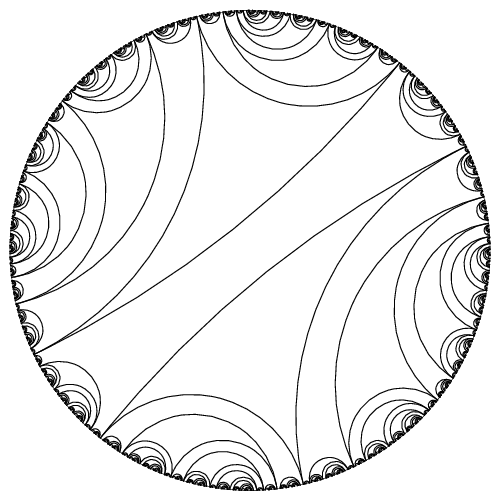}
{\caption{Left: Douady rabbit, the Julia set of the polynomial $f(z)=z^2-0.12..+0.74..i$
Right: its geodesic lamination.}\label{f:rabbit}}
\end{figure}

Thurston's geodesic laminations model the topological dynamics of
quadratic polynomials with locally connected Julia sets. So far, this
construction only provides topological models for individual quadratic
polynomials, and not even for all of them, since there are polynomials
$P_c$ such that $J(P_c)$ is connected but not locally connected;
however, we need to model the space of \textbf{all} polynomials $P_c$
with connected Julia sets. Metaphorically speaking, there are two
parallel worlds: the ``analytic'' world of complex polynomials and the
``combinatorial'' world of  geodesic laminations. Both worlds often
come close to each other: whenever we have a polynomial $P_c$ with
locally connected $J(P_c)$, then we have the corresponding invariant
geodesic lamination $\lam_c$. On the other hand, sometimes the two
worlds diverge. Still, a  model for $\Mc_2$ can be built within the
combinatorial world.

Since the space $C(\ol{\disk})$ of all subcontinua of the closed unit
disk with the Hausdorff metric is a continuum, it makes sense to
consider the closure  $\mathbb{L}_2$ of the family of all quadratic
geodesic laminations $\lam_c$ in $C(\ol{\disk})$, where $J(P_c)$ is
locally connected. Limit points of this family (called \emph{quadratic
geodesic limit laminations}\index{quadratic geodesic limit lamination})
do not immediately correspond to polynomials with connected Julia sets.
However, one can extend the correspondence between polynomials and
geodesic laminations to all polynomials. More precisely, one can
associate to each polynomial $P$ all Hausdorff limits of geodesic
laminations obtained by approximating $P$  by polynomials with locally
connected Julia sets.

The main property of the leaves of an invariant geodesic lamination is
that they are not \emph{linked}\index{leaves!linked}, that is, they do
not cross in $\disk$. Thurston gave a simple parameterization of a
quotient of $\mathbb{L}_2$. The idea is to take one particular leaf
from every quadratic limit geodesic lamination $\lam$, namely, the
leaf, called the \emph{minor}\index{minor} of $\lam$, whose endpoints
are the $\si_2$-images of the endpoints of a longest leaf of $\lam$
called a \emph{major}\index{major} of $\lam$ (it is easy to see that a
quadratic invariant geodesic lamination can have at most two longest
leaves, each of which is the rotation of the other leaf by one half of
the full angle around the center of $\disk$).

One of the main results of \cite{thu85} is that the minors of all
quadratic limit geodesic laminations are \textbf{pairwise unlinked} and
hence form a geodesic lamination called the \emph{quadratic minor
lamination}\index{quadratic minor lamination} $\mathrm{QML}$ (observe
that $\mathrm{QML}$ is not invariant). The geodesic lamination
$\mathrm{QML}$ generates a laminational equivalence relation
$\approx_\qml$ where two points $x, y$ of the unit circle are declared
to be $\approx_\qml$-equivalent if there exists a finite chain of
minors connecting $x$ and $y$ (the fact that $\sim_\qml$ indeed is a
laminational equivalence relation follows from \cite{thu85}). This
gives a conjectural model for the Mandelbrot set, in the sense that the
boundary $\bd(\Mc_2)$ of $\Mc_2$ is conjecturally homeomorphic to
$\uc/\approx_\qml$ (it is known that there exists a monotone map from
$\bd(\M_2)$ to $\uc/\approx_\qml$ and, hence,
$\uc/\approx_{\mathrm{QML}}$ is at least a monotone model of
$\bd(\M_2)$).

The leaves of $\qml$ can be described without referring to quadratic
geodesic limit laminations. To this end, let us denote by $|x-y|$, $x$,
$y\in \uc=\mathbb R/\mathbb Z$, the length of the shortest circle arc
with endpoints $x$ and $y$. Hence the length of a diameter is
$\frac12$. Denote by $\ol{ab}$ the chord with endpoints $a\in \uc$ and
$b\in \uc$. Consider the chord $\ol{ab}$ assuming that $\la=|a-b|\le
1/3$. Let $Q$ be the convex hull of the set $\si_2^{-1}(\{a,b\})$ in
the plane. \textbf{Assume that all four sides of $Q$ are unlinked with
all images $\ol{\si_2^n(a)\sigma_2^n(b)}$ ($n=0,1,2\dots$)} of
$\ol{ab}$  (this holds automatically if $\ol{ab}$ is a minor of a
quadratic invariant geodesic lamination).

The set $Q$ is called a \emph{critical
quadrilateral}\index{quadrilateral!critical} ($Q$ is a quadrilateral
that maps onto its image $\ol{ab}$ two-to-one). The set $Q$ has two
pairs of sides of equal length opposite to each other. Clearly, two
opposite sides are of lengths $\la/2\le 1/6$ and the other two are of
length $1/2-\la/2 \ge 1/3$. Denote by $\ell_1$ and $\ell_2$ the two
longer sides of $Q$ (so that the circle arcs ``behind'' $\ell_1$ and
$\ell_2$ are of length $1/2-\la/2\ge 1/3$). Then the strip  $S$, the
part of the unit disk $\disk$ located between $\ell_1$ and $\ell_2$, is
called the \emph{critical strip}\index{critical strip} (of $\ell_1$ or
$\ell_2$).

Comparing the lengths of various chords involved in the described
picture, we see that the points $a$ and $b$ do not belong to $S$;
indeed otherwise we would have had either that $\la<\la/2$ or that
$\la>1/2-\la/2>1/3$, a contradiction. In other words,
$\si_2(\ell_1)=\si_2(\ell_2)=\ol{ab}$ is disjoint from $S$ (it can be
contained in the boundary of $S$ if $a=\frac13$ and $b=\frac23$).
Similar considerations involving critical strips play an important role
in \cite{thu85} and, in particular, lead to the so-called Central Strip
Lemma (see Section \ref{ss:motiv}). This lemma yields that minors of
quadratic invariant geodesic laminations are pairwise unlinked. In the
paper \cite{chmmo15} the Central Strip Lemma is studied and extensions
of this lemma to the case of degree greater than two are obtained,
however the conclusions of these extensions are weaker than the
conclusion of the original Central Strip Lemma.

By \cite{thu85}, if, for every positive integer $n$, the chord
$\si_2^n(\ell_1)=\si_2^n(\ell_2)$ is disjoint from the interior of $S$,
then Thurston's \textbf{pullback construction} yields a quadratic
invariant geodesic lamination with the majors $\ell_1$ and $\ell_2$.
Therefore, the condition that for every positive integer $n$, the chord
$\si_2^n(\ell_1)=\si_2^n(\ell_2)$ is disjoint from the interior of $S$,
describes all chords $\ell_1$ and $\ell_2$ that are majors of quadratic
geodesic laminations. Clearly, this description does not depend on
quadratic invariant geodesic laminations.

Observe that if $a=b$, then $Q$ is a diameter of $\uc$. In this case,
$Q$ is trivially a major. However, if $a\ne b$, then the conditions
from the previous paragraph (that $\si_2^n(\ol{ab})$ is disjoint from
the interior of $S$) are non-trivial. An alternative --- and more
straightforward --- way of defining $\mathrm{QML}$ is by saying that
$\qml$ is formed by minors of all quadratic invariant geodesic
laminations, that is by chords $\si_2(\ell)$ taken for all majors
$\ell$ of all quadratic invariant geodesic laminations.

As was mentioned above, one of the main results of Thurston's from
\cite{thu85} is that $\qml$ is in fact a geodesic lamination itself (it
is not at all obvious that minors described above are pairwise
unlinked). Moreover, Thurston shows in \cite{thu85} that leaves of
$\qml$ can be broken into single (``stand alone'') leaves and finite
collections of leaves with each such collection being the boundary of a
geodesic polygon in $\disk$. One can then collapse all such leaves and
geodesic polygons to points thus defining the quotient space
$\ol{\disk}/\approx_\qml$, which serves as a combinatorial model for
$\Mc_2$, see \cite{dou93, sch09}. Denote the corresponding quotient map
by $\pi:\cdisk\to \cdisk/\approx_\qml$. Also, denote
$\ol{\disk}/\approx_\qml$ by $\Mc_2^{comb}$ reflecting the
combinatorial nature of the ``pinched disk'' model of $\Mc_2$.

The importance of results of \cite{thu85} lies, in particular, in the
fact that thanks to the minors being unlinked one can visualize
$\mathrm{QML}$ and, hence, a quotient of  the space $\mathbb{L}_2$
(distinct  quadratic invariant geodesic laminations may have the same
minor and our parameterization identifies such laminations). This in
turn allows for a visualization of $\ol{\disk}/\approx_{\mathrm{QML}}$
as the result of ``planar pinching'' of the unit disk which collapses
all the above described geodesic polygons formed by minors. More
precisely, by \cite{dave86}, there exists a homotopy $\ga: \C\times [0,
1]\to \C$ such that, for each $t\in [0, 1)$, the map $\ga_t:z\to \ga(z,
t)$ is an orientation preserving homeomorphism that shrinks every
geodesic polygon formed by minors more and more (as $t$ approaches $1$)
so that, for $t=1$, we do not have a homeomorphism, rather a
``pinching'' map $\ga_1:z\to \ga(z, 1)$ representing the quotient map
of $\ol{\disk}$ to $\ol{\disk}/\approx_\qml$.

Recall that a continuous mapping from one continuum to another
continuum is \emph{monotone}\index{monotone map} if the
\emph{fibers}\index{fiber of a map} (that is, preimages of points) are
connected. It is known \cite{sch09} that there exists a monotone map
$\pi:\M_2\to\M^{Comb}_2=\uc/QML$. The set $\Mc_2$ is locally connected
if and only if the fibers of $\pi$ are points, hence, $\pi$ gives the
desired homeomorphism between $\bd(\Mc_2)$ and $\uc/\approx_\qml$
provided that the MLC conjecture holds. In other words, the conjecture
that the boundary of $\Mc_2$ is homeomorphic to $\uc/\approx_\qml$ is
equivalent to the celebrated \textbf{MLC conjecture} claiming that the
Mandelbrot set is locally connected.

\section{Previous work}

The structure of the cubic connectedness locus $\Mc_3$ (or some parts
of it) has been studied by many authors. There are several approaches.
In some papers, higher degree connectedness loci are considered too. In
the rest of this subsection we briefly describe some relevant results.
Let us emphasize that our short overview is far from being complete.

Branner and Hubbard \cite{BrHu} initiated the study of $\Mc_3$, and
investigated the complement of  $\Mc_3$ in the full parameter space of
cubic polynomials. The complement is foliated by so-called
\emph{stretching rays}\index{stretching ray} that are in a sense
analogous to external rays of the Mandelbrot set. The combinatorics of
$\Mc_3$ is closely related to landing patterns of stretching rays.
However, we do not explore this connection here. A significant
complication is caused by the fact that there are non-landing
stretching rays. Landing properties of stretching rays in the parameter
space of real polynomials have been studied by Komori and Nakane
\cite{kona}. Of special interest is a certain subset of the complement
of $\Mc_3$ in the parameter space called the \textbf{shift locus} (see,
for example, \cite{BrHu, brhu92, dem12, dp11}).

Another approach to understanding parameter spaces of polynomials is
based on a new notion, due to Thurston, of core entropy (entropy on the
Hubbard tree of a polynomial) studied, for example, in \cite{thu14,
tio15, tio14, ds14}.

Lavaurs \cite{la89} proved that $\Mc_3$ is not locally connected.
Epstein and Yampolsky \cite{EY99} showed that the bifurcation locus in
the space of real cubic polynomials is not locally connected either.
This makes the problem of defining a combinatorial model of $\Mc_3$
that would admit visual interpretation very delicate. Buff and
Henriksen \cite{BH01} presented copies of quadratic Julia sets,
including not locally connected Julia sets, in slices of $\Mc_3$. By
McMullen \cite{mcm07}, slices of $\M_3$ contain lots of copies of
$\M_2$. In addition, Gauthier \cite{gau14} has shown that $\M_3$
contains copies of $\M_2\times \M_2$. In fact, the last two papers
contain more general results than what we mention here; we now confine
ourselves to the cubic case. Various spaces of cubic polynomials are
studied in \cite{zak99, bkm10}.

In his thesis, D. Faught \cite{Fau92} considered the slice $\mathcal A$
of $\Mc_3$ consisting of polynomials with a fixed critical point and
showed that $\mathcal A$ contains countably many homeomorphic copies of
$\Mc_2$ and is locally connected everywhere else. P. Roesch \cite{R06}
filled some gaps in Faught's arguments and generalized Faught's results
to higher degrees. Milnor \cite{M} gave a classification of hyperbolic
components in $\Mc_d$; however, this description does not involve
combinatorial tags. Schleicher \cite{sch04} constructed a geodesic
lamination modeling the space of
\emph{unicritical}\index{polynomial!unicritical} cubic polynomials,
that is, polynomials with a unique multiple critical point. We have
also heard of an unpublished old work of D. Ahmadi and M. Rees, in
which cubic geodesic laminations were studied, however we have not seen
it. Finally, a paper by J. Kiwi \cite{kiw05} studies the parameter
space of all central monic polynomials of arbitrary degree, focusing
upon the intersection of the connectedness locus and the closure of the
\emph{shift locus} (i.e. the set of all polynomials so that  \emph{all}
critical points escape). However, \cite{kiw05} does not deal with the
combinatorial structure of the connectedness locus.

\section{Overview of the method}

We now sketch the main tools developed in the present paper. The need
for them is justified by the fact that Thurston's tools used in the
construction of $\qml$ do not generalize to the cubic case. His tools
are based on the Central Strip Lemma stated in Section \ref{ss:motiv},
and include the No Wandering Triangles Theorem (also stated in Section
\ref{ss:motiv}). A straightforward extension of the Central Strip Lemma
as well as that of the No Wandering Triangles Theorem to the cubic case
fail (see a recent paper \cite{chmmo15} with possible extensions of the
Central Strip Lemma). As a consequence, cubic invariant geodesic
laminations may have wandering triangles (see \cite{bo08}). Thus, one
needs a different set of combinatorial tools. Such tools are developed
in the present paper and are based upon a principle which we call
\textbf{smart criticality}. Smart criticality works for geodesic
laminations of any degree.

Recall, that given a geodesic lamination $\lam$ in $\disk$, one defines
\emph{gaps}\index{gap} of $\lam$ as closures of components of
$\cdisk\sm \lam^+$ where $\lam^+\subset \cdisk$ is the union of all
leaves of $\lam$. The statement about  quadratic geodesic laminations
we are trying to generalize is the following: if the minors of two
quadratic geodesic laminations intersect in $\disk$, then they
coincide. However, although minors can also be defined for higher
degree laminations, they are not the right objects to consider because
they do not define geodesic laminations in a unique way. The sets that
essentially determine a given invariant geodesic lamination are in fact
its critical sets rather than their images. Thus, for the purpose of
characterizing invariant geodesic laminations we propose different
objects.

For a quadratic invariant geodesic lamination $\lam$, instead of its
non-de\-ge\-ne\-rate minor $\ol m$, we can consider the quadrilateral,
whose vertices are the four $\si_2$-preimages of the endpoints of $\ol
m$. Such a quadrilateral $Q$ is called a \emph{critical
quadrilateral}\index{quadrilateral!critical}. Note that $Q$ is not
necessarily a gap of $\lam$. Thus, $Q$ lies in some critical  gap of
$\lam$ or, if $\ol m$ is a point, coincides with the critical leaf of
$\lam$ which we see as a generalized critical quadrilateral. Similarly,
for a degree $d$ invariant geodesic lamination $\lam$, we can define
critical quadrilaterals as (possibly degenerate) quadrilaterals lying
in gaps or leaves of $\lam$ (opposite vertices of these quadrilaterals
must have the same $\si_d$-images). These critical quadrilaterals will
play the role of minors and will be used to tag higher degree geodesic
laminations.

The method of smart criticality helps to verify that, under suitable
assumptions, two linked leaves $\ell_1$, $\ell_2$ (i.e., leaves such
that $\ell_1\cap\ell_2\cap\disk\ne\0$) of \textbf{different} geodesic
laminations have linked images $\si_d^n(\ell_1)$, $\si_d^n(\ell_2)$,
for \textbf{all} $n$. One possible reason, for which $\si_d(\ell_1)$,
$\si_d(\ell_2)$ may be linked if $\ell_1, \ell_2$ are linked, is the
following: $\ell_1$ and $\ell_2$ are contained in a part of the unit
disk bounded by several circle arcs and such that these circle arcs map
forward under $\si_d$ so that the circular order among their points is
(non-strictly) preserved.

A typical reason for that phenomenon is that $\ell_1$ and $\ell_2$ are
disjoint from a \emph{full} collection of critical chords (here a
$\si_d$-\emph{critical chord}\index{chord!critical} is a chord of
$\disk$, whose endpoints map to the same point under $\si_d$, and a
\emph{full collection of critical chords}\index{full collection of
critical chords} is a collection of $d-1$ critical chords without
loops). Implementing this idea, we prove that $\si_d^n(\ell_1)$,
$\si_d^n(\ell_2)$ are linked for all $n$ by choosing, for every $n$, a
different and thus depending on $n$ full collection of critical chords
--- this is the meaning of ``smart'' as in  ``smart criticality''
above (alternatively, one could call this ``adjustable criticality'').

Smart criticality can be implemented in the following situation. Let
$\lam_1$ and $\lam_2$ be two invariant geodesic laminations. Suppose
that we can choose \emph{full}\index{full collection of critical
quadrilaterals} collections of critical quadrilaterals in $\lam_1$ and
$\lam_2$ (i.e., such collections that on the boundaries of components
of their complement the map $\si_d$ is one-to-one except perhaps for
boundary critical chords); then we say that $\lam_1$ and $\lam_2$ are
\emph{quadratically critical}\index{invariant geodesic laminations!
quadratically critical}. Critical quadrilaterals of a quadratically
critical $\si_d$-invariant geodesic lamination $\lam$ can be ordered;
if we fix that order we call the corresponding $(d-1)$-tuple of
critical quadrilaterals a \emph{quadratically critical
portrait}\index{quadratically critical portrait} of $\lam$, and $\lam$
is said to be \emph{marked}\index{invariant geodesic
laminations!marked}.

Suppose that two quadratically critical portraits $\qcp_1, \qcp_2$ are
such that equally numbered critical quadrilaterals in them either have
alternating vertices, or share a diagonal; then we say that $\qcp_1,
\qcp_2$ are \emph{linked}\index{quadratically critical portraits!linked
or essentially equal} (if at least one pair of corresponding critical
quadrilaterals with alternating vertices exists) or \emph{essentially
equal}\index{quadratically critical portraits!linked or essentially
equal} (if all pairs of corresponding quadrilaterals share a common
diagonal). Two marked invariant quadratically critical geodesic
laminations are said to be \emph{linked (essentially equal)} if their
quadratically critical portraits are linked (essentially
equal)\index{invariant geodesic laminations!linked or essentially
equal}.

In fact, being linked or essentially equal is slightly more general
than the property just stated; the precise statements can be found  in
Definition \ref{d:qclink1} and in \Cref{d:link-our-lams}. The main
result of the paper is that in a lot of cases linked or essentially
equal invariant geodesic laminations must coincide, or at least they
must share a significant common portion. This fact can be viewed as a
version of rigidity of critical data of invariant geodesic laminations.
It serves as a basis for the applications discussed in
Subsection~\ref{ss:main-app}.

To be more specific, we need to introduce a few notions. Suppose that
$\lam$ is a $\si_d$-invariant geodesic lamination. Then there are two
types of leaves of $\lam$. First, there are leaves $\ell$ of $\lam$
such that in any  neighborhood of $\ell$ there are uncountably many
leaves of $\lam$ (that is, if $\ell=\ol{ab}$, then, for every $\e>0$,
there are uncountably many leaves $\ol{xy}\in \lam$ such that $\mathrm
d(x, a)<\e$ and $\mathrm d(y, b)<\e$ where $\mathrm d(\cdot, \cdot)$ is
a distance between points on the unit circle $\uc$). The union of all
such leaves (perhaps only consisting of  the unit circle, but normally
much more significant), is itself an invariant geodesic lamination
denoted by $\lam^p$. Since  every leaf of $\lam^p$ is a limit of other
leaves of $\lam^p$ we call $\lam^p$ the \emph{perfect
part}\index{perfect part}\label{ppart} of $\lam$.

Another important part of $\lam$ is related to so-called periodic
Siegel gaps. Namely, an $n$-periodic Fatou gap $U$ of an invariant
geodesic lamination $\lam$ is said to be a periodic \emph{Siegel} gap
if $\si_d^n:\bd(U)\to \bd(U)$ is a degree one map monotonically
semiconjugate to an irrational rotation of the unit circle. It is easy
to see that edges of periodic Siegel gaps are isolated in $\lam$. The
closure of the union of the grand orbits of all periodic Siegel gaps of
an invariant geodesic lamination $\lam$ is denoted by $\lam^{Sie}$ and
is called the \emph{Siegel part}\index{Siegel part}\label{spart} of
$\lam$. It is not hard to see that the union of the perfect part and
the Siegel part of an invariant geodesic lamination is itself an
invariant geodesic lamination.

We can also consider pullbacks of periodic Fatou gaps $U$. If there is
an eventual non-periodic pullback $W$ of $U$ that maps forward by
$\si_d$ in a $k$-to-$1$ fashion with $k>1$, then $U$ is said to be of
\emph{capture} type. The terminology, due to Milnor \cite{miln93, M},
comes from the fact that in the case of complex polynomials the
periodic Fatou domain corresponding to $U$ \textbf{captures} a critical
point that belongs to the appropriate non-periodic  pullback Fatou
domain.

Our two main rigidity theorems show that the fact that two invariant
geodesic laminations are linked or essentially equal implies that the
laminations themselves are ``almost'' equal. Thus, we obtain a tool
allowing us to conclude that certain distinct geodesic laminations
cannot be linked/essentially equal. As the linkage/essential equality
of geodesic laminations is related to the mutual location of their
critical sets, out of this we choose appropriate tags of the critical
sets and draw conclusions about those tags being pairwise disjoint.
This in the end yields parameterization of the corresponding space of
geodesic laminations similar to Thurston's QML.

\begin{thmr1}\label{thrm1}
If two marked invariant quadratically critical geodesic laminations are
linked or essentially equal, then the unions of their perfect parts and
their Siegel parts are the same.
\end{thmr1}

To state the Second Rigidity Theorem we need the following definition.
Let $\sim$ be an invariant laminational equivalence relation and
$\lam_\sim$ be the geodesic lamination generated by $\sim$. Let $U$ be
an $n$-periodic Fatou gap of $\lam_\sim$ such that $\si_d^n:\bd(U)\to
\bd(U)$ has degree two. Then we call $U$ a \emph{quadratic Fatou
gap}\index{Fatou gap!quadratic}. If $U$ is quadratic, then there is a
unique edge $M$ (possibly degenerate) of $\bd(U)$ of period $n$; we
call $M$ a \emph{refixed}\index{refixed edge} edge of $U$. Let us also
denote by $M^*$ the unique edge of $U$ distinct from $M$ but with the
same image as $M$.

The convex hull of $M\cup M^*$ is said to be a \emph{legal critical
quadrilateral}\index{critical quadrilateral!legal}. There may exist a
finite gap $G$ sharing the leaf $M$ with $U$ and, accordingly, a finite
gap $G^*$ sharing the leaf $M^*$ with $U$ such that
$\si_d(G)=\si_d(G^*)$. Then \emph{in some cases} one can erase $M$ and
its entire grand orbit from $\lam_\sim$ and, possibly, replace it by a
different leaf contained in $G$ and its entire grand so that the new
geodesic lamination generates the same laminational equivalence
relation. In these cases one can insert in $U\cup G\cup G^*$ a critical
quadrilateral $Q$ with edges in $G$ and $G^*$ so that leaves from the
forward orbit of $Q$ do not cross each other. Thurston's pullback
construction implies that we can pull $Q$ back inside the grand orbit
of $U$ and add the thus constructed grand orbit of $Q$ to $\lam_\sim$.
We will call such quadrilaterals \emph{legal}\index{critical
quadrilateral!legal} too. Also, if a critical set $G$ of $\lam_\sim$ is
finite, then any critical quadrilateral inserted in $G$ and such that
sets from its forward orbit do not cross is called
\emph{legal}\index{critical quadrilateral!legal}.

Finally, suppose that an invariant geodesic lamination $\lam_\sim$ is
such that all critical sets of $\lam_{\sim_i}^p\cup
\lam_{\sim_i}^{Sie}$ are either finite sets or periodic quadratic Fatou
gaps. Then we say that $\sim$ (and $\lam_\sim$) are \emph{quadratically
almost perfect-Siegel non-capture}\index{invariant geodesic
lamination!quadratically almost perfect-Siegel non-capture}. A full
ordered collection of legal critical quadrilaterals inserted in
critical sets of a quadratically almost perfect-Siegel non-capture
geodesic laminations is said to be a \emph{legal quadratically critical
portrait}\index{quadratically critical portrait!legal} of $\sim$ if
chords from the forward orbits of these quadrilaterals are not linked
(these forward orbits are ``dynamically consistent''). If such a
portrait is chosen for $\lam$, then $\lam$ is said to be
\emph{marked}\index{invariant geodesic lamination!marked quadratically
almost perfect-Siegel non-capture}. Two marked quadratically almost
perfect-Siegel non-capture laminational equivalence relations are said
to be \emph{linked} (\emph{essentially equal}) if their legal
quadratically critical portraits are \emph{linked} (\emph{essentially
equal}).

\begin{thmr2}\label{thrm2}
If two marked invariant geodesic quadratically almost perfect-Siegel
non-capture laminations $\lam_{\sim_1}$ and $\lam_{\sim_2}$ are linked
or essentially equal, then they coincide (that is, $\sim_1=\sim_2$ and
$\lam_{\sim_1}=\lam_{\sim_2}$).
\end{thmr2}

\section{Main applications}\label{ss:main-app}

Questions concerning the existence of combinatorial models of the
connectedness loci $\Mc_d$ of degree $d$ polynomials arose soon after
Thurston's construction of a combinatorial model for $\Mc_2$ (see,
for\, example, \cite{thu85, mcm94}). The main aim of the present paper
is to generalize the ``pinched disk'' model onto some classes of
invariant geodesic laminations as well as polynomials. Inevitably, the
increase in the degree makes the problem more difficult. Thurston's
work was based on his Central Strip Lemma \cite[Lemma II.5.1]{thu85},
which implied his No Wandering Triangle Theorem and the transitivity of
the first return map of finite periodic polygons. However, the Central
Strip Lemma fails in degrees higher than two. Moreover, it is known
that in the cubic case wandering triangles exist \cite{bo04, bo08,
bco12, bco13} and that the first return map on a finite periodic
polygon is not necessarily transitive in higher degree cases
\cite{kiw02}. This shows that a new approach is necessary.

Furthermore, the connectedness locus $\Mc_3$ in the parameter space of
complex cubic polynomials is a four-dimensional set, which is known to
be non-locally connected \cite{la89}. Thus, it is hopeless to look for
a precise topological model of the boundary of $\Mc_3$ as a quotient of
a locally connected space (any quotient space of a locally connected
space is locally connected!). Yet another indication of the fact that a
new approach is needed is the fact that in the cubic case the so-called
combinatorial rigidity fails as shown by Henriksen in \cite{hen03}.

In the present paper we concentrate on the part of Thurston's work (see
\cite{thu85}) where it is shown that the family of quadratic invariant
geodesic lamination can be tagged by their minors, which, by
\cite{thu85}, are pairwise unlinked. This yields the ``pinched disk''
model QML for the Mandelbrot set. We prove similar results, which allow
us to describe various spaces of invariant geodesic laminations.

The first application can be found in Subsection~\ref{ss:loc-dendr}.
Consider the space of all polynomials with connected Julia sets such
that all their periodic points are repelling. Such polynomials exhibit
rich dynamics and have been actively studied before. In particular,
there is a nice association, due to Jan Kiwi \cite{kiwi97}, between
these polynomials and a certain class of invariant geodesic laminations
of the same degree. These invariant geodesic laminations
$\lam=\lam_\sim$ are generated by invariant laminational equivalence
relations $\sim$ that have the following property: the associated
topological Julia set $J_\sim$ is a \emph{dendrite}\index{dendrite}
(that is, a locally connected one-dimensional continuum that contains
no Jordan curves); equivalently, all gaps of $\lam_\sim$ must be
finite. Then the corresponding invariant geodesic lamination and the
corresponding invariant laminational equivalence are called
\emph{dendritic}.

Kiwi proves in \cite{kiwi97} that in this case for a given polynomial
$P$ of degree $d$ there exists an invariant laminational equivalence
relation $\sim_P$ such that the filled Julia set $J(P)$ of the
polynomial $P$ can be \textbf{monotonically} (recall that this means
that point-preimages are connected) mapped onto $J_{\sim_P}$. Moreover,
the monotone map $\psi_P:J(P)\to J_{\sim_P}$ in question semiconjugates
$P|_{J(P)}$ and the associated topological polynomial
$f_{\sim_P}:J_{\sim_P}\to J_{\sim_P}$ induced by $\si_d$ on the
topological Julia set $J_{\sim_P}=\uc/\sim_P$. Denote by $\vp_P$ the
quotient map $\vp_P:\uc\to \uc/\sim_P$.

Take a point $z\in J(P)$, project it by the map $\psi_P$ to a point
$\psi_P(z)$ of the topological Julia set $J_{\sim_P}$, lift the point
$\psi_P(z)$ to the corresponding $\vp_P$-fiber $\vp_P^{-1}(\psi_P(z))$,
and then to its convex hull $\ch(\vp_P^{-1}(\psi_P(z)))$ denoted by
$G_z$. Clearly, $G_z$ is a gap or (possibly degenerate) leaf of
$\lam_{\sim_P}$; loosely, $G_z$ is the laminational counterpart of the
point $z$. This geometric association is important for a  combinatorial
interpretation of the dynamics of $P$. In particular, each critical
point $c$ of $P$ is associated with the critical gap or leaf $G_c$ of
$\lam_{\sim_P}$ (all dendritic invariant geodesic laminations have
finite critical sets).

We call polynomials with connected Julia sets, all of whose cycles are
repelling, \emph{dendritic}\index{polynomial!dendritic}. Let us
emphasize that we do not mean that the Julia sets of dendritic
polynomials are dendrites themselves; rather our terminology is
justified because by \cite{kiwi97} the Julia sets of dendritic
polynomials can be mapped to a non-trivial dendrite under a monotone
map. In particular, the Julia set of a dendritic polynomial may be
non-locally connected and, hence, not a dendrite.

Dealing with polynomials, we specify the order of their critical points
and talk about \emph{$($critically$)$ marked
polynomials}\index{polynomial!critically marked}. In that we follow
Milnor \cite{mil12}. More precisely, a \emph{$($critically$)$ marked
polynomial} is a polynomial $P$ with an ordered collection $C(P)$ of
its critical points, so that every multiple critical point is repeated
several times according to its multiplicity (thus, $C(P)$ is a
$(d-1)$-tuple, where $d$ is the degree of $P$). Marked polynomials do
not have to be dendritic (in fact, the notion is used by Milnor and
Poirier for hyperbolic polynomials, that is, in the situation
diametrically opposite to that of dendritic polynomials). However, we
consider only dendritic marked polynomials. Thus, speaking of a marked
polynomial, we mean a pair $(P, C(P))$.

In what follows, $C(P)$ is called an \emph{ordered critical
collection}\index{collection!ordered critical} of $P$; normally we use
the notation $C(P)=(c_1,$ $\dots,$ $c_{d-1})$, where a multiple
critical point $c$ is repeated in $C(P)$ according to its multiplicity.
Since we want to reflect convergence of polynomials, we allow for the
same critical point of multiplicity $k$ to be repeated $k-1$ times not
in a row. For example, let a polynomial $P$ of degree $5$ have two
critical points $c$ and $d$ of multiplicity $3$ each. Then we can mark
$P$ with any ordered collection of points $c$ and $d$ as long as each
point is repeated twice, such as $(c, d, c, d)$, or $(d, d, c, c)$, or
$(c, c, d, d)$ etc.

Endow the family of all critically marked polynomials with the natural
topology that takes into account the order among critical points so
that marked polynomials $(P_i, C(P_i))$ converge to a marked polynomial
$(P, C(P))$ if and only if $P_i\to P$ and $C(P_i)\to C(P)$. Our aim is
to provide local laminational models for some dendritic polynomials of
arbitrary degree $d$. In other words, we suggest a class of marked
dendritic polynomials $(P, C(P))$ with the following property. There
exists a neighborhood $\mathcal U$ of $(P, C(P))$ and a continuous map
from $\mathcal U$ to a special laminational parameter space. The
definition of this map is based upon information on laminational
equivalence relations $\sim_Q$ defined by dendritic polynomials $Q\in
\mathcal U$. This approach is close to Thurston's original approach
which led to the proof of the existence of a monotone map from the
entire quadratic Mandelbrot set onto its laminational counterpart, the
``pinched disk'' model $\Mc_2^{comb}$. We implement this approach on
open subsets of the space of marked polynomials of arbitrary degree
$d$.

As polynomials $P$, we choose dendritic polynomials with the following
additional property: the invariant dendritic geodesic lamination
$\lam_{\sim_P}$ has $d-1$ pairwise distinct critical sets. We will call
such polynomials \emph{simple dendritic}\index{polynomial!simple
dendritic}. If $(P, C(P))$ is a critically marked simple dendritic
polynomial, then all critical points in $C(P)$ must be distinct.
However, the mere fact that $P$ is dendritic and has $d-1$ distinct
critical points is not sufficient to conclude that $(P, C(P))$ is a
simple dendritic polynomial. This is because distinct critical points
of $P$ may belong to the same fiber of $\psi_P$ resulting in some
critical sets of $\lam_{\sim_P}$ being of multiplicity greater than
two. One can show that the space of simple dendritic critically marked
polynomials is open in the space of all critically marked dendritic
polynomials.

Denote the space of all degree $d$ simple critically marked dendritic
polynomials by $\cmd_d$. Consider the \emph{ordered postcritical
collection}\index{collection!ordered postcritical}$(P(c_1),$ $\dots,$
$P(c_{d-1}))$. The sets $G_{c_i}, 1\le i\le d-1$ are critical sets of
the invariant geodesic lamination $\lam_{\sim_P}$ and the sets
$G_{P(c_i)}, 1\le i\le d-1$ are their $\si_d$-images. Define the
following two maps from $\cmd_d$ to the space of compact subsets of
$\cdisk^{d-1}$. First, it is the map $\hPsi_d$ defined as follows:

$$\hPsi_d(P)=G_{c_1}\times G_{c_2}\times \dots \times
G_{c_{d-1}}.$$

\noindent Second, it is the map $\Psi_d$, defined as follows:

$$\Psi_d(P)=G_{P(c_1)}\times G_{P(c_2)}\times \dots \times
G_{P(c_{d-1})}.$$

\noindent These maps associate to any marked simple dendritic
polynomial a compact subset of $\cdisk^{d-1}$ (moreover, this subset
itself is the product of convex hulls of certain $\sim_P$-classes).
Notice that each set $G_{c_j}$ maps onto its image two-to-one. We call
the set $\Psi_d(P)$ the \emph{postcritical tag}\index{tags of
critically marked dendritic polynomials!postcritical} of the critically
marked polynomial $(P, C(P))$.

\begin{thmlc} Suppose that $(P,$  $C(P))$ is
a marked simple dendritic polynomial of degree $d$. Then there is a
neighborhood $\mathcal U$ of $(P, C(P))$ in $\cmd_d$ such that for any
two polynomials $(Q,C(Q)), (R, C(R))\in \mathcal U$ with $\hPsi(Q)\ne
\hPsi(R)$ we have that their $\Psi_d$-images $\Psi_d(Q)$ and
$\Psi_d(R)$ are disjoint.
\end{thmlc}

The Theorem on Local Charts for Dendritic Polynomials implies the
following corollary, in which the notation of the Theorem is used.

\begin{lppm}
Consider the union of all postcritical tags of polynomials in $\mathcal
U$ in $\ol{\disk}^{d-1}$ and its quotient space obtained by collapsing
these tags to points. The constructed space is separable and metric.
Moreover, the map $\Psi_d$ viewed as a map from $\mathcal U$ to this
space is continuous.
\end{lppm}

The second application extends the results of \cite{bopt15a} and can be
found in Subsection~\ref{ss:d-2}. In \cite{bopt15a}, we studied the
space $\prnp_3(\ol{ab})$ of all cubic invariant geodesic laminations
generated by cubic invariant laminational equivalence relations $\sim$
such that for some fixed critical leaf $D=\ol{ab}$ with non-periodic
endpoints we have $a\sim b$, and there are no gaps of capture type. The
main result of \cite{bopt15a} is that this family of cubic invariant
geodesic laminations is modeled by a lamination.  This result resembles
the description of the combinatorial Mandelbrot set.

More specifically, to each cubic invariant geodesic lamination $\lam$
from  $\prnp_3(D)$ we associate its critical set $C$ whose criticality
``manifests'' itself inside the circle arc $(b, a)$ of length
$\frac23$. We show that either $C$ is finite, or $C$ is a periodic
Fatou gap of degree two and period $k$. Now, if $\lam\in \prnp_3(D)$
then a pair of sets $\qcp=(Q, D)$ is called a \emph{quadratically
critical portrait privileged for $\lam$} if and only if $Q\subset C$ is
a critical leaf or a collapsing quadrilateral (by a \emph{collapsing}
quadrilateral\index{quadrilateral!collapsing} we mean a quadrilateral
whose boundary maps two-to-one to a chord). In the case when $C$ is a
critical periodic Fatou gap of period $k$, we require that $Q$ be a
collapsing quadrilateral obtained as the convex hull of a (possibly
degenerate) edge $\ell$ of $C$ of period $k$ and another edge $\hell$
of $C$ such that $\si_3(\ell)=\si_3(\hell)$.

In \cite{bopt15}, we show that for each $\lam\in\prnp_3(D)$ there are
only finitely many privileged quadratically critical portraits. Let
$\Ss_D$ denote the collection of all privileged for $\lam$
quadratically critical portraits $(Q, D)$. To each such $(Q, D)$ we
associate its \emph{minor} (a chord or a point) $\si_3(Q)\subset
\cdisk$. For each such chord we identify its endpoints, extend this
identification by transitivity and define the corresponding equivalence
relation $\simeq_D$ on $\uc$. The main result of \cite{bopt15} is that
$\simeq_D$ is itself a laminational equivalence (non-invariant!) whose
quotient is a parameterization of $\prnp_3(D)$.

In Subsection~\ref{ss:d-2} of the present paper, we generalize the
results of the paper \cite{bopt15} onto the degree $d$ case. In order
to do so we introduce the appropriate space analogous to $\prnp_3(D)$.
Namely, fix a collection $\mathcal Y$ of $d-2$ pairwise disjoint
critical chords of $\si_d$ with non-periodic endpoints. Let $\mathbb
L(\mathcal Y)$ be the space of all invariant geodesic laminations
generated by laminational equivalence relations that are
\emph{compatible} with this collection in the sense that $\lam_\sim$
belongs to $\mathbb L(\mathcal Y)$ if and only if the endpoints of each
critical leaf from $\mathcal Y$ are $\sim$-equivalent. Moreover,
similar to the case of $\prnp_3(D)$ we also require that $\lam_\sim$
has no gaps of capture type. We prove in \Cref{l:ld-nonempty} that
$\mathbb L(\mathcal Y)$ is non-empty.

Let $\mathcal Y^+$ be the union of all critical leaves from $\mathcal
Y$. There exists a unique component $A(\mathcal Y)=A$ of $\cdisk\sm
\mathcal Y^+$ on whose boundary the map $\si_d$ is two-to-one except
for its critical boundary edges ($\si_d$ is one-to-one in the same
sense on all other components of $\cdisk\sm \mathcal Y^+$). Moreover,
in \Cref{l:crit-set} we show that for each $\lam\in \mathbb L(\mathcal
Y)$, there exists a unique critical set $C$ that contains a critical
chord $\oc\subset \ol{A}$. We then use the set $C$ to define the
\emph{minor set} of $\lam$. Namely, it is shown that $C$ is either
finite, or a periodic (of period, say, $n$) Fatou gap such that
$\si_d^n: \bd(U)\to\bd(U)$ is two-to-one. In the former case, set
$m(\lam)=\si_d(C)$. In the latter case, choose a maximal finite gap
whose vertices are fixed under $\si_d^n$ and share an edge with $U$,
and let $m(\lam)$ be its $\si_d$-image. The main result of
Subsection~\ref{ss:d-2} is the following theorem.

\begin{thmcd} There exists a non-invariant laminational
equivalence relation $\sim_{\mathcal Y}$ such that minor sets of
invariant geodesic laminations from $\mathbb L(\mathcal Y)$  are convex
hulls of classes of equivalence of $\sim_{\mathcal Y}$; this gives rise
to the quotient space $\uc/\sim_{\mathcal Y}=\Mc_{\mathcal Y}$ that
parameterizes $\mathbb L(\mathcal Y)$.
\end{thmcd}

Yet another application of the results of this paper will be contained
in a forthcoming paper by the authors where we construct a higher
dimensional lamination of a subset of $\cdisk\times \cdisk$ whose
quotient space is a combinatorial model for the space of all marked
cubic polynomials with connected Julia set that only have repelling
cycles \cite{bopt16, bopt17}.

\section{Organization of the paper}

In Section \ref{s:basicdef}, we introduce invariant geodesic
laminations. In \Cref{s:laeqre}, we discuss laminational equivalence
relations in detail. General properties of invariant geodesic
laminations are considered in \Cref{s:gpgeo}. In \Cref{s:qc-por}, we
introduce and study our major tool, quadratically critical portraits,
for invariant geodesic laminations. The most useful results, based upon
quadratically critical portraits, can be obtained for some special
types of invariant geodesic laminations investigated in \Cref{s:pandp}.
In \Cref{s:acclam}, we introduce another major tool, so-called
\emph{accordions}\index{accordions}, which are basically sets of linked
leaves of distinct invariant geodesic laminations. We first study
accordions by postulating certain properties of them related to the
orientation of leaves comprising these accordions. In \Cref{s:smart},
we develop the principle of smart criticality and show that owing to
this principle we can apply the results of \Cref{s:acclam} to
accordions of two linked or essentially equal invariant geodesic
laminations.  Arguments based upon smart criticality yield that
accordions of linked or essentially equal geodesic laminations behave
much like gaps of a single invariant geodesic lamination. This is
established in Section \ref{s:qcrit}, where the method of smart
criticality is developed. Finally, in Section \ref{l:appli}, we will
prove the Main Theorems.

\section{Acknowledgments}

The authors are indebted to the referee for thoroughly reading this
text and making a number of very useful and thoughtful suggestions
leading to significant improvements in the paper.

\chapter{Invariant laminations: general properties}\label{c:invlam}

\section{Invariant geodesic laminations}\label{s:basicdef}

In this section, we give basic definitions, list some known results
concerning (invariant) geodesic laminations, and establish some less
known facts about them.

\subsection{Basic definitions}\label{ss:bd}

We begin with simple geometry.

\begin{dfn}[Chords]\label{d:chord}
A \emph{chord}\index{chord} is a closed segment connecting two points
of the unit circle, not necessarily distinct. If the two points
coincide, then the chord connecting them is said to be
\emph{degenerate}\index{chord!degenerate}.
\end{dfn}

Let us now consider collections of chords.

\begin{dfn}[Solids of chord collections]\label{d:solid}
Let $\rc$ be a collection of chords. Then we set $\bigcup\rc=\rc^+$ and
call $\rc^+$ the \emph{solid of $\rc$}.
\end{dfn}

We are mostly interested in collections of chords with specific
properties.

\begin{dfn}[Geodesic laminations]\label{d:geolam}
A \emph{geodesic lamination}\index{geodesic lamination} is a collection
$\lam$ of (perhaps degenerate) chords called \emph{leaves}\index{leaf}
such that the leaves are pairwise disjoint in $\disk$ (that is, in the
open unit disk), $\lam^+$ is closed, and all points of $\uc$ are
elements of $\lam$. \emph{Gaps}\index{gap} of $\lam$ are defined as the
closures of the components of $\disk\sm\lam^+$. The solid $\lam^+$ is
called the \emph{solid of the geodesic lamination $\lam$}.
\end{dfn}

The notion of geodesic lamination is \textbf{static} in the sense that
no map is even considered with resect to $\lam$. In order to relate it
to the dynamics of the map $\si_d$, it is useful to extend $\si_d$ as
described below.

\begin{dfn}[Extensions of $\si_d$]\label{d:extesig}
Extend $\si_d$ over leaves of $\lam$ so that the restriction of the
extended $\si_d$ to every leaf is an affine map. This extension is
continuous on $\lam^+$ and well-defined (provided that $\lam$ is
given). By Thurston \cite{thu85}, define a canonical
\emph{barycentric}\index{barycentric extension of the map $\si_d$}
extension of the map $\si_d$ to the entire closed disk $\cdisk$.
Namely, after $\si_d$ is extended affinely over all leaves of an
invariant geodesic lamination $\lam$, extend it piecewise affinely over
the interiors of all gaps of $\lam$, using the barycentric subdivision.
We will use the same notation for both $\si_d$ and all its extensions.
\end{dfn}

Observe that while the extensions of $\si_d$ can be defined for any
geodesic lamination, they are really sensible only in the case of
$\si_d$-invariant geodesic laminations considered below; thus, when
talking about $\si_d$ on $\cdisk$, we always have some invariant
geodesic lamination in mind and we extend $\si_d$ using Thurston's
barycentric extension (see \cite{thu85} for details).

\subsection{Sibling invariant geodesic laminations}\label{ss:sibgeo}

Let us introduce the notion of a (sibling) \emph{$\si_d$-invariant
geodesic lamination}\index{geodesic lamination!sibling
$\si_d$-invariant}, which is a slight modification of the notion of an
invariant geodesic lamination introduced by Thurston \cite{thu85}; in
the case when $d$ is fixed, we will often write ``invariant'' instead
of ``$\si_d$-invariant'' without causing ambiguity.

\begin{dfn}[Invariant geodesic laminations \cite{bmov13}]\label{d:sibli}
A geodesic la\-mination $\lam$ is (sibling) \emph{($\si_d$)-invariant}
provided that:
\begin{enumerate}
\item for each $\ell\in\lam$, we have $\si_d(\ell)\in\lam$,
\item \label{2}for each $\ell\in\lam$ there exists $\ell^*\in\lam$ so
    that $\si_d(\ell^*)=\ell$.
\item \label{3} for each $\ell\in\lam$ such that $\si_d(\ell)$ is a
    non-degenerate leaf, there exist $d$ \textbf{pairwise disjoint}
 leaves $\ell_1$, $\dots$, $\ell_d$ in $\lam$ such that $\ell_1=\ell$
 and $\si_d(\ell_i)=\si_d(\ell)$ for all
  $i=2$, $\dots$, $d$.
\end{enumerate}
\end{dfn}

Observe that since leaves are chords, and chords are closed segments,
pairwise disjoint leaves in part (3) of the above definition cannot
intersect even on the unit circle (that is, they cannot even have
common endpoints). Notice also, that Definition~\ref{d:sibli} can be
given without condition (2); in that case we will talk about
\emph{forward (sibling) invariant geodesic lamination}.  In particular,
forward (sibling) invariant geodesic laminations may well contain
finitely many non-degenerate leaves (in that case we will call it
\emph{finite}).

We call the leaf $\ell^*$ in (\ref{2}) a \emph{pullback}\index{pullback
of a leaf} of $\ell$ and the leaves $\ell_2$, $\dots$, $\ell_d$ in
(\ref{3}) \emph{sibling leaves}\index{leaves!sibling} or just
\emph{siblings} of $\ell=\ell_1$. In a broad sense, a \emph{sibling of
$\ell$} is a leaf with the same image but distinct from $\ell$.
Definition~\ref{d:sibli} is slightly more restrictive than Thurston's
definition of an invariant geodesic lamination, which we give below. In
what follows, given a set $A$, we let $\ch(A)$ denote the convex hull
of $A$.

\begin{dfn}[Invariant geodesic laminations in the sense of Thurston]\label{d:inva-thurston}
A geodesic lamination is said to be \emph{invariant $($in the sense of
Thurston$)$}\index{geodesic lamination!invariant in the sense of
Thurston} if the following holds:

\begin{enumerate}

\item for each non-degenerate $\ell\in\lam$, we have
    $\si_d(\ell)\in\lam$;

\item $\lam$ is \emph{gap invariant}\index{geodesic lamination!gap
    invariant}: if $G$ is a gap of $\lam$ and $H=\ch(\si_d(G\cap
    \uc))$ is the convex hull of $\si_d(G\cap \uc)$, then $H$ is a
    point, a leaf of $\lam$, or a gap of $\lam$, and, in the latter
    case, the map $\si_d|_{\bd(G)}:\bd(G)\to \bd(H)$ of the boundary
    of $G$ onto the boundary of $H$ is a positively oriented
    composition of a monotone map and a covering map (in fact the set
    $H$ as above will be called the $\si_d$-image of $G$ and will be
    denoted by $\si_d(G)$ in what follows).

\item there are $d$ \textbf{pairwise disjoint} leaves $\ell^*\in\lam$
    such that $\si_d(\ell^*)=\ell$.

\end{enumerate}

If $\lam$ satisfies conditions (1) and (2) only, then $\lam$ is called
\emph{forward invariant $($in the sense of Thurston$)$}\index{geodesic
lamination!forward invariant in the sense of Thurston}.
\end{dfn}

The above quoted result of \cite{bmov13} claims that if $\lam$ is
sibling $\si_d$-invariant, then it is $\si_d$-invariant in the sense of
Thurston. From now on, by $(\si_d$-$)$invariant geodesic laminations,
we mean \emph{sibling} $\si_d$-invariant geodesic laminations and
consider \textbf{only} such invariant geodesic laminations.

The next definition is crucial for our investigation and shows in what
ways different chords can coexist.

\begin{dfn}[Linked chords]\label{d:linchor}
Two \textbf{distinct} chords of $\disk$ are
\emph{linked}\index{chords!linked} if they intersect inside $\disk$ (we
will also sometimes say that these chords \emph{cross each
other}\index{chords!crossing}). Otherwise two chords are said to be
\emph{unlinked}\index{chords!unlinked}.
\end{dfn}

\Cref{d:gapsedges} deals with gaps of geodesic laminations and their
edges.

\begin{dfn}[Gaps and their edges]\label{d:gapsedges}
A gap $G$ is said to be \emph{infinite $($finite,
uncountable$)$}\index{gap!infinite}\index{gap!finite}\index{gap!uncountable}
if $G\cap \uc$ is infinite (finite, uncountable). Uncountable gaps are
also called \emph{Fatou} gaps. For a closed convex set $H\subset \C$,
straight segments from $\bd(H)$ are called \emph{edges} of $H$.
\end{dfn}

The \emph{degree}\index{degree of a gap} of a gap or leaf $G$ is
defined as follows.

\begin{dfn}[Degree of a gap or leaf]\label{d:degree-gap}
Let $G$ be a gap or a leaf. If $\si_d(G)$ is degenerate (that is, if
$\si_d(G)$ is a singleton), then the \emph{degree of $G$} is the
cardinality of $G\cap \uc$. Suppose now that $\si_d(G)$ is not a
singleton. Consider $\si_d|_{\bd(G)}$. Then the \emph{degree of $G$}
equals the number of components in the preimage of a point $z\in
\si_d(\bd(G))$ under the map $\si_d|_{\bd(G)}$.
\end{dfn}

Note that we talk about the number of components rather than the number
of points since, say, an entire critical leaf is mapped to a single
point, thus the full preimage of this point is infinite.

We say that $\ell$ is a \emph{chord of a geodesic lamination
$\lam$}\index{chord!of a geodesic lamination $\lam$} if $\ell$ is a
chord of $\ol\disk$ unlinked with all leaves of $\lam$.

\begin{dfn}[Critical sets]\label{d:cristuff}  A
\emph{critical chord $($leaf$)$}\index{chord!critical}
\index{leaf!critical} $\ol{ab}$ of $\lam$ is a chord (leaf) of $\lam$
such that $\si_d(a)=\si_d(b)$. A gap is \emph{all-critical}
\index{gap!all-critical} if all its edges are critical. An all-critical
gap or a critical leaf (of $\lam$) is called an \emph{all-critical
set}\index{gap!all-critical}\index{set!all-critical} (of $\lam$). A gap
$G$ is said to be \emph{critical}\index{gap!critical} if the degree of
$G$ is greater than one. A \emph{critical set}\index{set!critical} is
either a critical leaf or a critical gap.
\end{dfn}

Observe that a gap $G$ may be such that $\si_d|_{\bd(G)}$ is not
one-to-one, yet $G$ is not critical in the above sense. More precisely,
a gap may have critical edges while not being critical. Indeed, let $G$
be a triangle with one critical edge and two non-critical edges. Let
$G\cap \uc=\{x, y, z\}$ where $\ol{xy}$ is critical. Then
$\si_d(G)=\ol{\si_d(y) \si_d(z)}$ is a leaf of $\lam$ and
$\si_d|_{\bd(G)}$ is not one-to-one, but $G$ is not critical because
the degree of $G$ is one.

Finally, we need to define a metric on the set of geodesic laminations.
Heuristically two laminations should be close if for every leaf in one
lamination there is a leaf in the other lamination that is close to it.

We will use the Hausdorff metric $H$ to define the required metric.
Given a compact metric space $X$ with metric (distance function)
$\rho$, let $2^X$ denote the set of all non-empty closed subsets of
$X$. Let $\text{Ball}_\rho(A,\e)$ denote the set of all points $x\in X$
so that $\rho(x,A)<\e$. Given $A$, $B\in 2^X$, the metric
$$
H_\rho(A,B)=\inf\{\e>0\mid  A\subset \text{Ball}_\rho(B,\e) \hbox{ and } B\subset \text{Ball}_\rho(A,\e)\}
$$
is called the \emph{Hausdorff metric}. It is well known that with this
metric $2^X$ is a compact metric space. Since every point of $\uc$ is a
degenerate leaf of a geodesic lamination $\lam$, the solid $\lam^+$ is
a compact (and connected) subset of $\ol{\disk}$. It is tempting to
define the distance between $\lam_1$ and $\lam_2$ as
$H_\rho(\lam_1^+,\lam_2^+)$, where $\rho$ is the usual Euclidean metric
on the closed unit disk. Unfortunately, if $d>2$, there exist distinct
geodesic laminations $\lam_1$, $\lam_2$ such that
$\lam_1^+=\lam_2^+=\ol{\disk}$ and, hence, $H_\rho(\lam_1,\lam_2)=0$.
For example, in the cubic case, the two laminations consisting of all
vertical and of all horizontal chords are two such laminations. Hence
we need two refine the choice of the metric.

Clearly, every element $\ell$ of $\lam$ is a compact set and, hence, a
point in $2^{\ol\disk}$. Thus each geodesic lamination is a closed
subset of $2^{\ol\disk}$. Let $H_H$ denote the Hausdorff metric on
$2^{\ol\disk}$. Then the required distance on the set of geodesic
laminations is $H_H(\lam_1,\lam_2)$. With this metric, the set of
geodesic laminations is a compact metric space.
%We will denote the metric $H_H$ by $\rho$.

\begin{thm}[Theorem 3.21 \cite{bmov13}]\label{t:sibliclos}
The family of sets $\lam^+$ of all invariant geodesic laminations
$\lam$ is closed in the Hausdorff metric $H_H$. In particular, this
family is compact.
\end{thm}

Theorem~\ref{t:sibliclos} allows us to give the following definition.

\begin{dfn}\label{d:lamiconverg}
Suppose that a sequence $\lam_i^+$ of solids of invariant geodesic
laminations converges to a compact set $T$. Then by \Cref{t:sibliclos}
there exists an invariant geodesic lamination $\lam$ such that
$T=\lam^+$ is its solid. In this case we say that geodesic laminations
$\lam_i$ \emph{converge to} $\lam$. Thus, from now on we will write
$\lam_i\to \lam$ if $\lam^+_i\to \lam^+$ in the Hausdorff metric
$H_H$..
\end{dfn}

Clearly, $\lam_i^+\to \lam^+$ implies that the collections of chords
$\lam_i$ converge to the collection of chords $\lam$ (that is, each
leaf of $\lam$ is the limit of a sequence of leaves from $\lam_i$, and
each converging sequence of leaves of $\lam_i$ converges to a leaf of
$\lam$).

\section{Laminational equivalence relations}\label{s:laeqre}

In this section, we discuss (invariant) laminational equivalence
relations and (invariant) geodesic laminations generated by them. The
relation between certain polynomials with connected Julia sets and
laminational equivalence relations is also discussed. Finally, we
introduce a few useful concepts, which we will rely upon in the rest of
the paper.

\subsection{Laminational equivalence relations and their relations to complex polynomials}

A lot of geodesic laminations naturally appear in the context of
invariant equivalence relations on $\uc$ satisfying special conditions.
We will call such equivalence relations
\emph{laminational}\index{laminational equivalence relation}.

\begin{dfn}[Laminational equivalence relations]\label{d:lam}
An equi\-va\-lence re\-la- \newline tion $\sim$ on the unit circle
$\uc$ is said to be \emph{laminational} if either $\uc$ is one
$\sim$-equivalence class (such laminational equivalence relations are
called \emph{degenerate}\index{equivalence relation, laminational
degenerate}), or the following holds:

\noindent (E1) the graph of $\sim$ is a closed subset of $\uc \times
\uc$;

\noindent (E2) the convex hulls of distinct equivalence classes are
disjoint;

\noindent (E3) each equivalence class of $\sim$ is finite.
\end{dfn}

As with geodesic laminations, the above definition is static. However
for us the most interesting case is the dynamical case described below.

\begin{dfn}[Laminational equivalence relations and dynamics]\label{d:si-inv-lam}
A laminational equivalence relation $\sim$ is called ($\si_d$-){\em
invariant} if:

\noindent (D1) $\sim$ is {\em forward invariant}: for a
$\sim$-equivalence class $\g$, the set $\si_d(\g)$ is a
$\sim$-equivalence class;

\noindent (D2) for any $\sim$-equivalence class $\g$, the map $\si_d:
\g\to \si_d(\g)$ extends to $\uc$ as an orientation preserving covering
map such that $\g$ is the full preimage of $\si_d(\g)$ under this
covering map.
\end{dfn}

For an invariant laminational equivalence relation $\sim$ consider the
\emph{topological Julia set}\index{Julia set!topological}
$\uc/\hspace{-5pt}\sim\,=J_\sim$ and the \emph{topological
polynomial}\index{polynomial!topological} $f_\sim:J_\sim\to J_\sim$
induced by $\si_d$. The quotient map $\pi_\sim:\uc\to
\uc/\hspace{-5pt}\sim=J_\sim$ semi-conjugates $\si_d$ with
$f_\sim|_{J_\sim}$. A laminational equivalence relation $\sim$
\textbf{canonically extends} over $\C$: non-trivial classes of the
extension are convex hulls of classes of $\sim$. By Moore's Theorem,
the quotient space $\C/\hspace{-5pt}\sim$ is homeomorphic to $\C$.

The quotient map $\pi_\sim:\uc\to \uc/\hspace{-5pt}\sim$ extends to the
plane with the only non-trivial point-preimages
(\emph{fibers}\index{fiber of a map}) being the convex hulls of
non-degenerate $\sim$-equivalence classes. With any fixed
identification between $\C/\sim$ and $\C$, one extends $f_\sim$ to a
branched-covering map $f_\sim:\C\to \C$ of degree $d$ called a
\emph{topological polynomial} too. The complement $K_\sim$ of the
unique unbounded component $U_\infty(J_\sim)$ of $\C\sm J_\sim$ is
called the \emph{filled topological Julia set}\index{Julia set!filled
topological}. The \emph{(canonical) geodesic lamination $\lam_\sim$
generated by $\sim$}\index{geodesic lamination!generated by a
laminational equivalence relation} is the collection of edges of convex
hulls of all $\sim$-equivalence classes and all points of $\uc$.

\begin{lem}[Theorem 3.21 \cite{bmov13}]\label{t:qsib}
Geodesic laminations $\lam_\sim$ generated by $\si_d$-invariant
laminational equivalence relations are sibling invariant. If a sequence
of sets $\lam_{\sim_i}^+$ converges to a compact set $T$, then there
exists a sibling invariant geodesic lamination $\lam$ such that
$T=\lam^+$.
\end{lem}

We would like to motivate the usage of laminational equivalence
relations by showing in what way they are related to polynomials. Let
$P:\C\to\C$ be a polynomial of degree $d\ge 2$, let $A_\infty$ be the
basin of attraction of infinity, and let $J(P)=\bd(A_\infty)$ be the
Julia set of $P$.  When $J(P)$ is connected, $A_\infty$ is simply
connected and conformally isomorphic to $\C\sm\cdisk$ by a unique
isomorphism $\phi:\C\sm\ol\disk\to A_\infty$ asymptotic to the identity
at $\infty$. By a theorem of B\"ottcher (see, e.g., \cite[Theorem
9.1]{mil00}), the map $\phi$ conjugates $P|_{A_\infty}$ with
$z^d|_{\C\sm\ol\disk}$. If $J(P)$ is locally connected, then $\phi$
extends continuously to a semiconjugacy $\ol\phi$ between
$\si_d=z\mapsto z^d|_\uc$ and $P|_{J(P)}$:

\begin{equation}\label{e:lam_conj}\begin{CD}
\uc @>{\si_d|_{\uc}}>> \uc \\
@V{\ol\phi}VV @V{\ol\phi}VV \\
J(P) @>{P|_{J(P)}}>> J(P) \\
\end{CD}\end{equation}

The \textit{laminational equivalence generated by}\index{laminational
equivalence relation!generated by a polynomial} $P$ is the equivalence
relation $\sim_P$ on $\uc$ whose classes are \emph{$\ol\phi$-fibers},
i.e. point-preimages under $\ol\phi$.  Call $J_{\sim_P}=\uc/\sim_P$ the
\textit{topological Julia set associated with the polynomial }$P$. The
map $f_{\sim_P}$, induced on $J_{\sim_P}$ by $\si_d$, will be called
the \textit{topological polynomial associated to the
polynomial}\index{polynomial!topological associated to a polynomial}
$P$. Evidently $P|_{J(P)}$ and $f_{\sim_P}|_{J_{\sim_P}}$ are
topologically conjugate. The collection $\lam_P$ of chords of
$\ol\disk$ that are edges of convex hulls of $\sim_P$ classes is called
the \textit{geodesic lamination generated by the polynomial
}\index{geodesic lamination!generated by a polynomial}$P$.

In fact, this connection between polynomials and appropriately chosen
topological polynomials can be extended onto a wider class of
polynomials with connected Julia sets. The first steps in this
direction were made in a nice paper by Jan Kiwi \cite{kiwi97}.

\begin{dfn}[Irrationally indifferent periodic points]\label{d:irr-ind}
Let $x$ be a periodic point of a polynomial $P$ of period $n$. Then $x$
is said to be \emph{irrationally indifferent}\index{periodic
point!irrationally indifferent} if the multiplier $(P^n)'(x)$ of $P$ at
$x$ is of the form $e^{2\pi i\theta}$ for some irrational $\theta$. If
there exists an open $P^n$-invariant neighborhood of $x$, on which
$P^n$ is conjugate to an irrational rotation of an open unit disk, then
$x$ is said to be a periodic \emph{Siegel} point \index{periodic
point!Siegel}. If such a neighborhood of $x$ does not exist, then $x$
is said to be a \emph{periodic Cremer point}\index{periodic
point!Cremer}.
\end{dfn}

We also need to introduce a few topological concepts. Observe that in
our definition of a laminational equivalence relation we require that
classes of equivalence be finite. However the definitions may be given
without this requirement; in these cases we will talk about
laminational equivalence relations \emph{possibly with infinite
classes}.

\begin{dfn}[Dendrites and dendritic laminations]\label{d:dendr}\index{invariant geodesic lamination!dendritic}
A locally connected continuum is said to be a \emph{dendrite} if it
contains no subsets homeomorphic to the unit circle. If $\sim$ is a
laminational equivalence relation on the unit circle $\uc$ such that
the quotient space $\uc/\sim$ is a dendrite, then we call $\sim$ and
the corresponding geodesic lamination $\lam_\sim$ \emph{dendritic}.
Observe that, in this case, every $\sim$-class is finite, and hence
every point $x$ of the quotient space $\uc/\sim$ is such that $\uc/\sim
\sm \{x\}$ consists of finitely many components (in that case $x$ is
said to be of \emph{finite order}). If, however, $\sim$ is a
laminational equivalence relation on the unit circle $\uc$ possibly
with infinite classes such that $\uc/\sim$ is a dendrite, then we call
$\sim$ and $\lam_\sim$ \emph{dendritic possibly with infinite classes}.
\end{dfn}

In what follows when we talk about a
\emph{preperiodic}\index{object!preperiodic} object (point, set etc) we
mean that it is not periodic but maps to a periodic object after some
(positive) number of iterations of the map. On the other hand when we
talk about a \emph{$($pre$)$periodic}\index{object!(pre)periodic}
object, we mean that it is either preperiodic or periodic. In
particular, when talking about (pre)periodic points we mean points that
have finite forward orbits. Now we can state one of the important
results proven in \cite{kiwi97}.

\begin{thm}
\label{t:kiwi-dendr} Suppose that a polynomial $P$ with connected Julia
set $J=J(P)$ has no Siegel or Cremer periodic points. Then there exist
a laminational equivalence relation $\sim_P$, the corresponding
topological polynomial $f_{\sim_P}:J_{\sim_P}\to J_{\sim_P}$ restricted
to the topological Julia set, and a monotone semiconjugacy $\ph_P:J\to
J_{\sim_P}$. The semiconjugacy $\ph_P$ is one-to-one on all
$($pre$)$periodic points of $P$ belonging to the Julia set. If all
periodic points of $P$ are repelling, then $J_{\sim_P}$ is a dendrite.
\end{thm}

In what follows, denote by $\dc$ the space of all polynomials with
connected Julia sets and only repelling periodic points. Let $\dc_d$ be
the space of all such polynomials of degree $d$.

Theorem~\ref{t:kiwi-dendr} was extended \cite{bco11} onto all
polynomials with connected Julia sets. Call a monotone map $\ph_P$ of a
connected polynomial Julia set $J(P)=J$ onto a locally connected
continuum $L$ the \emph{finest monotone map of $J(P)$ onto a locally
connected continuum} if, for any monotone $\psi:J\to J'$ with $J'$
locally connected, there is a monotone map $h$ with $\psi=h\circ
\ph_P$. Then it is proven in \cite{bco11} that the finest monotone map
on a connected polynomial Julia set semiconjugates $P|_{J(P)}$ to the
corresponding topological polynomial $f_{\sim_P}$ on its topological
Julia set $J_{\sim_P}$ generated by the laminational equivalence
relation possibly with infinite classes $\sim_P$. It follows that the
following diagram is commutative (recall that by $\pi_{\sim_P}$ we
denote the quotient map corresponding to the lamination $\sim_P$).

\[\label{eq:commdiag}
\dgARROWLENGTH=5em
\dgARROWPARTS=8
\begin{diagram}
\node{J(P)} \arrow{e,t}{P|_{J(P)}} \arrow{se,t}{\varphi_P}
    \node{J(P)} \arrow{se,b,1}{\varphi_P}
    \node{\mathbb S^1} \arrow{e,t}{\sigma_d} \arrow{sw,b,1}{\pi_{\sim_P}}
    \node{\mathbb S^1} \arrow{sw,t}{\pi_{\sim_P}}\\
\node[2]{J_{\sim_P}} \arrow{e,b}{f_{\sim_P}|_{J_{\sim_P}}} \node{J_{\sim_P}}
\end{diagram}
\]

\subsection{Other useful notions}\label{ss:oun}

Considering objects related to geodesic laminations, we do not have to
fix these geodesic laminations. Recall that, given two points $a$,
$b\in\uc$, we write $\ol{ab}$ for the chord connecting $a$ with $b$.

\begin{dfn}\label{d:restuff}
By a \emph{periodic gap or
leaf}\index{gap!periodic}\index{leaf!periodic}, we mean a gap or a leaf
$G$, for which there exists the least number $n$ (called the
\emph{period} of $G$) such that $\si_d^n(G)=G$. Then we call the map
$\si_d^n:G\to G$ the \emph{remap}\index{remap}. An edge (vertex) of
$G$, on which the remap is the identity, is said to be
\emph{refixed}\index{edge of a periodic gap!refixed}\index{vertex of a
periodic gap!refixed}.
\end{dfn}

Given points $a$, $b\in \uc$, denote by $(a,b)$ the \emph{positively
oriented}\index{arc!positively oriented} open arc from $a$ to $b$ (that
is, moving from $a$ to be $b$ within $(a, b)$ takes place in the
counterclockwise direction). For a closed set $G'\subset \uc$, we call
components of $\uc\sm G'$ \emph{holes}\index{hole of a closed set or
its convex hull} (of $G'$ or of the convex hull $G=\ch(G')$ of $G'$).
If $\ell=\ol{ab}$ is an edge of the convex hull $G=\ch(G')$ of $G'$,
then we let $H_G(\ell)$ denote the component of $\uc\sm\{a, b\}$
disjoint from $G'$ and call it the hole of $G$ \emph{behind $\ell$} (it
is only unique if $G'$ contains at least three points). The
\emph{relative interior}\index{relative interior!of a gap} of a gap is
its interior in the plane; the \emph{relative interior}\index{relative
interior!of a segment} of a segment is the segment minus its endpoints.

\begin{dfn}\label{d:lamset}
If $A\subset \uc$ is a closed set and all the sets $\ch(\si_d^i(A))$
are pairwise disjoint, then $A$ is called
\emph{wandering}\index{wandering closed set}. If there exists $n\ge 1$
such that the sets $\ch(\si_d^i(A))$ with $i=0$, $\ldots$, $n-1$ have
pairwise disjoint relative interiors while $\si_d^n(A)=A$, then $A$ is
called \emph{periodic}\index{set!periodic} of period $n$. If there
exists a minimal $m>0$ such that all $\ch(\si_d^i(A))$ with $0\le i\le
m+n-1$ have pairwise disjoint relative interiors and $\si_d^m(A)$ is
periodic of period $n$, then we call $A$
\emph{preperiodic}\index{set!preperiodic} of period $n$ and preperiod
$m$. A set is called \emph{$($pre$)$periodic}\index{set!(pre)periodic}
if it is periodic or preperiodic. If $A$ is wandering, periodic or
preperiodic, and, for every $i\ge 0$ and every hole $(a, b)$ of
$\si_d^i(A)$, either $\si_d(a)=\si_d(b)$, or the positively oriented
arc $(\si_d(a),\si_d(b))$ is a hole of $\si_d^{i+1}(A)$, then we call
$A$ (and $\ch(A)$) a \emph{{\rm(}$\si_d${\rm )}-laminational
set}\index{set!laminational}. We call $\ch(A)$ \emph{finite} if $A$ is
finite. A {\em {\rm(}$\si_d$-{\rm)}stand alone gap}\index{gap!stand
alone} is defined as a laminational set with non-empty interior in the
plane.
\end{dfn}

Recall that when talking about a Jordan curve $K$ that encloses a
simply connected domain $W$ in the plane, by the \emph{positive
direction}\index{positive direction on a Jordan curve} on $K$ one means
the counterclockwise direction with respect to $W$, i.e., the direction
of a particle moving along $K$ so that $W$ remains on its left. When
considering a Jordan curve $K$ in the plane we \textbf{always} do so
with positive direction on it. In particulary, we consider the boundary
of a gap with positive direction on it. Accordingly, denote by $<$ the
\textbf{positive} (counterclockwise) circular order on $\uc=\mathbb
R/\mathbb Z$ induced by the usual order of $\mathbb R$. Note that this
order is only meaningful for sets of cardinality at least three. For
example, we say that $x<y<z$ provided that moving from $x$ in the
positive direction along $\uc$ we meet $y$ before meeting $z$.

\begin{dfn}[Order preserving maps of the circle] Let $X\subset \uc$ be a set with at
least three points. We call $\si_d$ \emph{order preserving on
$X$}\index{order preservation on a set under the map $\si_d$} if
$\si_d|_X$ is one-to-one and, for every triple $x$, $y$, $z\in X$ with
$x<y<z$, we have $\si_d(x)<\si_d(y)<\si_d(z)$.
\end{dfn}

Finally, we discuss in this section \emph{proper invariant geodesic
laminations}.

\begin{dfn}[Proper invariant geodesic lamination]\label{d:properlam}\index{invariant geodesic lamination!proper}
Two leaves with a common endpoint $v$ and the same image are said to
form a \emph{critical wedge} (the point $v$ the is said to be its
vertex). An invariant geodesic lamination $\lam$ is \emph{proper} if it
contains no critical leaf with a periodic endpoint and no critical
wedge with periodic vertex.
\end{dfn}

Given an invariant geodesic lamination $\lam$, define an equivalence
relation $\approx_\lam$ by declaring that $x \approx_\lam z$ if and
only if there exists a \emph{finite} concatenation of leaves of $\lam$
connecting $x$ and $z$.

\begin{thm}[Theorem 4.9 \cite{bmov13}]\label{t:nowander}
Let $\lam$ be a proper Thurston invariant lamination. Then
$\approx_\lam$ is a nonempty invariant laminational equivalence
relation.
\end{thm}

\section{General properties of invariant geodesic
laminations}\label{s:gpgeo}

Some results of this section are taken from \cite{bmov13}.

\begin{lem}[Lemma 3.7 \cite{bmov13}] \label{l:3.7}
If $\ol{ab}$ and $\ol{ac}$ are two leaves of an invariant geodesic
lamination $\lam$ such that $\si_d(a), \si_d(b)$ and $\si_d(c)$ are all
distinct points, then the order among points $a$, $b$, $c$ is preserved
under $\si_d$.
\end{lem}

We prove a few corollaries of Lemma~\ref{l:3.7}

\begin{lem}\label{l:sameperiod1}
If $\lam$ is an invariant geodesic lamination, $\ell=\ol{ab}$ is a leaf
of $\lam$, and $a$ is periodic of period $n$, then $b$ is
$($pre$)$periodic of period $n$.
\end{lem}

\begin{proof}
Assume that, while the point $a$ is of period $n$, the point $b$ is not
$\si_d^n$-fixed. Then, by Lemma \ref{l:3.7}, either the circular order
among the points $b_i=\si_d^{ni}(b)$ is the same as the order of
subscripts or $b_i=b_{i+1}$ for some $i$. In the former case $b_i$
converge to some limit point, a contradiction with the expansion
property of $\si^n_d$. Hence for some (minimal) $i$ we have
$b_i=b_{i+1}$. It follows that the period $m$ of $b_i$ cannot be less
than $n$ as otherwise we can consider $\si_d^m$ which fixes $b_i$ and
does not fix $a$ yielding the same contradiction with Lemma
\ref{l:3.7}.
\end{proof}

We will need the following elementary lemma. The notion of a
\emph{$($pre$)$\-cri\-tical}\index{object!(pre)critical} object is
similar to the notion of a (pre)periodic object; thus, a (pre)critical
point is either a precritical point, or a critical point.

\begin{lem}\label{l:concat}
If $x\in \uc$, and the chords $\ol{\si^i_d(x)\si_d^{i+1}(x)}$, $i=0$,
$1$, $\dots$ are pairwise unlinked, then the point $x$, and therefore
the leaf $\ol{x\si_d(x)}=\ell$, are $($pre$)$periodic.
\end{lem}

\begin{proof}
The sequence of leaves from the lemma is the $\si_d$-orbit of $\ell$,
in which consecutive images are concatenated and no two leaves are
linked. If, for some $i$, the leaf
$\ol{\si^i_d(x)\si_d^{i+1}(x)}=\si^i_d(\ell)$ is critical, then
$\si^{i+1}_d(\ell)=\{\si_d^{i+1}(x)\}$ is a $\si_d$-fixed point, which
proves the claim in this case. Assume now that the leaf $\ell$ is not
(pre)critical. If the point $x$ is not (pre)periodic, then, by
topological considerations, leaves $\si_d^n(\ell)$ must converge to a
limit leaf or a limit point. Clearly, this limit set is
$\si_d$-invariant. However, the map $\si_d$ is expanding, a
contradiction.
\end{proof}

Lemma~\ref{l:concat} easily implies Lemma~\ref{l:conconv}.

\begin{lem}\label{l:conconv}
Let $\lam$ be a geodesic lamination. Then the following holds.

\begin{enumerate}

\item If $\ell$ is a leaf of $\lam$ and, for some $n>0$, the leaf
    $\si_d^n(\ell)$ is concatenated to $\ell$, then $\ell$ is
    $($pre$)$periodic.

\item If $\ell$ has a $($pre$)$periodic endpoint, then $\ell$ is
    $($pre$)$periodic.

\item If two leaves $\ell_1$, $\ell_2$ from geodesic laminations
    $\lam_1$, $\lam_2$ share the same $($pre$)$\-pe\-rio\-dic
    endpoint, then they are $($pre$)$periodic with the same eventual
    period of their endpoints.

\end{enumerate}

\end{lem}

\begin{proof}
Let $\ell=\ol{uv}$. First, assume that $\si_d^n(u)=u$. Then the
statement (1) follows from Lemma \ref{l:sameperiod1}. Second, assume
that $\si^n_d(u)=v$. Then the statement (1) follows from
Lemma~\ref{l:concat}. The statements (2) and (3) follow from (1) and
Lemma~\ref{l:sameperiod1}. \end{proof}

A similar conclusion can be made for edges of periodic gaps.

\begin{lem}\label{l:cripe}
Suppose that $G$ is a gap of a geodesic lamination. Then, for every
edge $\ell$ of $G$, there exists an integer $k$ such that the length of
the hole $H_{\si_d^k(G)}(\si_d^k(\ell))$ exceeds $\frac{1}{d+1}$.
Moreover, suppose that a gap $G$ is periodic of period $m$. Then, for
every edge $\ell$ of $G$ which is not $($pre$)$\-cri\-ti\-cal, there
exists an edge $\ell_*$ of $G$ from the orbit of $\ell$ such that the
length of the hole $H_{G}(\ell_*)$ exceeds $\frac{1}{d^m+1}$. In
particular, any edge of a periodic gap is $($pre$)$pe\-ri\-odic or
$($pre$)$\-cri\-ti\-cal, and any periodic gap can have at most finitely
many non-degenerate periodic edges.
\end{lem}

\begin{proof}
To prove the first statement of the lemma, observe that the length
$s_n$ of the hole $H_{\si_d^n(G)}(\si_d^n(\ell))$ of $\si_d^n(G)$
behind the leaf $\si_d^n(\ell)$ grows with $n$ as long as $s_n$ stays
sufficiently small. In fact, it is easy to see that the correct bound
on $s_n$ is that $s_n<\frac{1}{d+1}$. Indeed, suppose that
$s_n<\frac{1}{d+1}$. Then the restriction
$\si_d|_{H_{\si_d^n(G)}(\si_d^n(\ell))}$ is one-to-one and the hole
$\si_d(H_{\si_d^n(G)}(\si_d^n(\ell)))$ is of length $ds_n>s_n$.
Clearly, this implies that for some $k$ the length of the hole
$H_{\si_d^k(G)}(\si_d^k(ell))$ will exceed $\frac{1}{d+1}$ as desired.

Now, suppose that $G$ is periodic of period $m$ and $\ell$ is not
(pre)critical. Then $G$ is $\si_d^m$-invariant, and the second claim of
the lemma follows from the first one. Observe that for any edge $\hell$
of $G$ such that $|H_G(\hell)|=s$ it is impossible that
$\frac1{d^m+1}\le s<\frac1{d^m}$ as in that case the arc $T$
complementary to the arc $\si^m_d(H_G(\hell))$ is of length $1-d^ms<s$,
a contradiction (all the vertices of $G$ must belong to $\ol{T}$ and
hence $\ol{T}$ must contain $H_G(\hell)$, a contradiction). The
remaining claims of the lemma now easily follow.
\end{proof}

Given $v\in \uc$, let $E(v)$ be the closure of the set $\{u\,|\,
\ol{uv}\in\lam\}$.

\begin{lem}\label{l:cones}
If $v$ is not $($pre$)$periodic, then $E(v)$ is at most finite. If $v$
is $($pre$)$periodic, then $E(v)$ is at most countable.
\end{lem}

\begin{proof}
The first claim is proven in \cite[Lemma 4.7]{bmov13}. The second claim
follows from Lemma~\ref{l:conconv} as by that lemma both vertices of
any leaf with an endpoint $v$ must be preperiodic.
\end{proof}

Properties of individual wandering polygons were studied in
\cite{kiw02}; properties of collections of wandering polygons were
studied in \cite{bl02}; their existence was established in \cite{bo08}.
The most detailed results on wandering polygons and their collections
are due to Childers \cite{chi07}.

Let us describe the entire $\si_d$-orbit of a finite periodic
laminational set.

\begin{prop}\label{p:forconcat}
Let $T$ be a $\si_d$-periodic finite laminational set and $X$ be the
union of the forward images of\, $T$. Then, for every connected
component $R$ of $X$, there is an $m$-tuple of points
$a_0<a_1<\dots<a_{m-1}<a_m=a_0$ in $\uc$ such that $R$ consists of
eventual images of $T$ containing $\ol{a_ia_{i+1}}$ for $i=0$,
$\ldots$, $m-1$. If $m>1$, then the remap of $R$ is a transitive
combinatorial rotation on the collection of all images of $T$ in $R$.
\end{prop}

Loosely speaking, one can say that, under the appropriate power of
$\si_d$, the set $T$ ``rotates'' around the convex hull of  $\{a_0,
\dots, a_{m-1}\}$. Note that the case $m=1$ is possible. In this case,
$R$ consists of several images of $T$ sharing a common vertex $a_0$,
there is a natural cyclic order among the images of $T$, and the remap
of $R$ is a cyclic permutation of these images, not necessarily a
combinatorial rotation. In particular, it may happen that there is a
unique image of $T$ containing $a_0$; in this case we deal with more
standard dynamics where the sets $T, \si_d(T), \dots, \si_d^{k-1}(T)$
are pairwise disjoint while $\si_d^k(T)=T$.

\begin{proof}
Set $T_k=\si_d^k(T)$. Let $k$ be the smallest positive integer such
that $T_k$ intersects $T_0$; we may suppose that $T_k\ne T_0$. There is
a vertex $a_0$ of $T_0$ such that $a_1=\si_d^k(a_0)$ is also a vertex
of $T_0$. Clearly, both $a_1$ and $a_2=\si_d^k(a_1)$ are vertices of
$T_k$. Set $a_i=\si_d^{ki}(a_0)$. Then we have $a_m=a_0$ for some
minimal $m>0$. Let $Q$ be the convex hull of the points $a_0$, $\dots$,
$a_{m-1}$. Then $Q$ is a convex polygon, or a chord, or a point. If
$m>1$, then $a_i$ and $a_{i+1}$ are the endpoints of the same edge of
$Q$ (otherwise some edges of the polygons $T_{ki}$ would cross in
$\disk$). Set $R=\cup_{i=0}^{m-1}T_{ki}$. If $m=1$, then the sets
$T_{ki}$ share the vertex $a_0$. If $m=2$ then $Q$ is a leaf that flips
under the action of $\si_d^k$. Finally, if $m>2$, then it follows from
the fact that the boundary of $Q$ is a simple closed curve that every
chord $\ol{a_ia_{i+1}}$ is an edge of $T_{ki}, i=0, \dots, m-1$ shared
with $Q$ and the sets $T_{ki}$ are disjoint from the interior of $Q$.

Since the case $m=2$ is straightforward, let us assume now that $m>2$.
Notice that, by the construction, the map $\si_d^k$ sends each set
$T_{ki}$ to the  set $T_{k(i+1)}$ adjacent to $T_{ki}$, and
$\si_d^k(\ol{a_ia_{i+1}})=\ol{a_{i+1}a_{i+2}}$. Now, let $s$ be the
least integer such that $\si_d^s(\ol{a_0a_1})=\ol{a_ja_{j+1}}$ is an
edge of $Q$ for some $j$. Evidently, the number $s$ does not have to be
equal to $k$. Still, it follows that
$\si_d^s(\ol{a_1a_2})=\ol{a_{j+1}a_{j+2}},$ $\dots,$
$\si_d^s(\ol{a_{m-1}a_0})=\ol{a_{j-1}a_j}$. Thus, the map $\si_d^s|_R$
is a combinatorial rotation. Moreover, the choice of $s$ and the fact
that $\ol{a_1a_2}=\si_d^k(\ol{a_0a_1})$ imply that the $\si_d^s$-orbit
of $\ol{a_0a_1}$ is the collection of all edges of $Q$, i.e. that
$\si_d^s$ is transitive on the collection of all images of $T$ forming
$R$.

It remains to prove that $R$ is disjoint from $R_j=\si_d^j(R)$ for
$j<s$. By way of contradiction suppose that $R_j$ intersects some
$T_{si}$. Note that the ``shape'' of the set $R_j$ mimics that of $R$:
the set $R_j$ consists of $m$ sets that are $\si_d^j$-images of sets
$T_{ki}, i=0, \dots, m-1$ adjacent to a convex polygon $\si_d^j(Q)$ in
the same way the sets $T_{ki}, i=0, \dots, m-1$ are adjacent to $Q$.

Let us show that sets $Q,$ $\si_d(Q),$ $\dots,$ $\si_d^{s-1}(Q)$ have
at most a vertex in common and that each set $R_j$ is contained in one
component of $\cdisk\sm Q$. Indeed, consider the set $\si_d^j(Q)$. If
all images of $T$ adjacent to $\si_d^j(Q)$ are distinct from the images
of $T$ adjacent to $Q$, then all the images of $T$ adjacent to
$\si_d^j(Q)$ are contained in the same component of $\cdisk\sm Q$.
Hence $\si_d^j(Q)$ can have at most a common vertex with $Q$. Moreover,
suppose that a set $\si_d^j(T)$, adjacent to $\si_d^j(Q), j<s$, in fact
coincides with set $\si_d^{is}(T)$ adjacent to $Q$. Then by the choice
of $s$ it follows that the edge $\si_d^j(\ol{a_0a_1})$ cannot coincide
with the edge of $\si_d^j(T)=\si_d^{is}(T)$ shared by this set and $Q$.
Since $\si_d^j(Q)$ is adjacent to $\si_d^j(T)$ along the edge
$\si_d^j(\ol{a_0a_1})$, it follows again that $\si_d^j(Q)$ and $Q$
cannot have more than one vertex in common, and that $R_j$ is contained
in the corresponding component of $\cdisk\sm Q$.

Now, suppose that $x\in R_j\cap R$. Assume that $x$ belongs to a
component $A$ of $\cdisk\sm Q$. Then on the one hand $\si_d^s$ sends
$x$ to a component $B$ of $\cdisk\sm Q$ distinct from $A$. On the other
hand, by the previous paragraph the entire $R_j$ must be contained in
$A$ which implies that $\si_d^s(x)$ must belong to $A$, a
contradiction.
\end{proof}

It is well-known (see \cite{kiw02}) that any infinite gap $G$ of an
invariant geodesic lamination $\lam$ is (pre)periodic. By a
\emph{vertex}\index{vertex of a gap or leaf} of a gap or\, leaf $G$ we
mean any point of $G\cap\uc$.

\begin{lem}\label{l:perinf}
Let $G$ be a periodic gap of period $n$ and set $K=\bd(G)$. Then
$\si_d^n|_K$ is the composition of a covering map and a monotone map of
$K$. If $\si_d^n|_K$ is of degree one, then either statement {\rm (1)}
or statement {\rm (2)} below holds.

\begin{enumerate}
\item The gap $G$ has at most countably many vertices, only finitely
    many of which are periodic, and no edge is preperiodic. All
    non-periodic edges of $G$ are $($pre$)$critical and isolated.

\item The map $\si_d^n|_K$ is monotonically semiconjugate to an
    irrational circle rotation so that each fiber of this
    semiconjugacy is a finite concatenation of $($pre$)$critical
    edges of $G$.
\end{enumerate}

\end{lem}

\begin{proof}
The first claim of the lemma holds since by \cite{bmov13} sibling
invariant geodesic laminations are invariant in the sense of Thurston
(see the beginning of Subsection~\ref{ss:sibgeo}). Consider now the
case when $G$ is of degree one. Then it follows that no edge of $G$ can
have two preimages under $\si_d^n$. In particular, $G$ has no
preperiodic edges. All other claims in the statement (1) of the lemma
follow from Lemma~\ref{l:cripe}. Observe that a critical edge $\ell$ of
$G$ must be isolated because by definition of an invariant geodesic
lamination (more precisely, because geodesic invariant laminations are
gap-invariant) there is another gap of $\lam$ sharing the edge $\ell$
with $\lam$ and mapping onto $\si_d(G)$ under $\si_d$.

In the statement (2) of the lemma we will prove only the very last
claim. Denote by $\ph$ the semiconjugacy from (2). Let $T\subset K$ be
a fiber of $\ph$. By Lemma~\ref{l:cripe} all edges of $G$ are
(pre)critical. Hence if $T$ contains infinitely many edges, then the
forward images of $T$ will hit critical leaves of $\si_d^n$ infinitely
many times as $T$ cannot collapse under a finite power of $\si_d^n$.
This would imply that an irrational circle rotation has periodic
points, a contradiction that completes the proof.
\end{proof}

We can now recall the notion of a \emph{periodic Siegel gap}; we will
also introduce a useful notion of the \emph{skeleton}\index{skeleton of
an infinite gap} of an infinite gap.

\begin{dfn}\label{d:siegel}
Let $G$ be an infinite gap of a geodesic lamination $\lam$. If $G\cap
\uc$ is at most countable, then we say that the \emph{skeleton} of $G$
is empty. Otherwise the \emph{skeleton} of $G$ is defined as the convex
hull of the maximal Cantor subset of $G\cap \uc$. Periodic infinite
gaps $G$ of geodesic laminations such that the remap on the boundary of
$G$ is monotonically semiconjugate to an irrational rotation are said
to be \emph{$($periodic$)$ Siegel gaps}.
\end{dfn}

Observe that the skeleton of a periodic Siegel gap is non-empty.
Observe also that edges of the skeleton of a periodic Siegel gap $G$ do
not have to be edges of $G$ itself.

By \cite{bl02}, if $\sim$ is an invariant laminational equivalence
relation possibly with infinite classes then there are no countable
infinite gaps of $\sim$, and the skeleton of a Siegel gap $G$ coincides
with $G$. In other words, in this case infinite gaps of $\sim$ are
either periodic Siegel gaps or periodic Fatou gaps of degree greater
than one, or their preimages.

It is known that periodic Siegel gaps must have critical edges that are
isolated. Therefore, both countable and Siegel gaps must have isolated
edges. Let us investigate other properties of periodic Siegel gaps.

\begin{lem}\label{l:sieg-str}
Suppose that $\lam$ is a geodesic lamination. Let $G$ be a periodic
Siegel gap of $\lam$ of period $n$. Let $H$ be the skeleton of $G$ and
$\ell=\ol{ab}$ be an edge of $H$. Consider the union of \emph{all}
finite concatenations of leaves of $\lam$ coming out of $a$ or $b$ and
let $L$ be the convex hull of all such leaves. Then $L$ is a finite
polyhedron so that every two vertices can be connected by a chain of
leaves from $\lam$ and $\ell$ is an edge of $L$. Moreover, for a
minimal $m>0$, the $\si_d^m$-image of $L$ is a singleton in $H$ that is
a limit from both sides of points of $H$. The semiconjugacy $\varphi$
between $\si_d^n$ and the corresponding irrational rotation can be
extended onto the union of $G$ and all sets $L$ by collapsing each set
$L$ to the point $\varphi(a)=\varphi(b)$. Moreover, at each edge
$\ell_*=\ol{a_*b_*}$  of $L$, there is an infinite gap of $\lam$ that
has on its boundary a finite concatenation of leaves of $\lam$
connecting $a_*$ and $b_*$ and that maps onto $G$ under $\si_d^m$.
\end{lem}

\begin{proof} It follows from Lemma~\ref{l:perinf} that for every edge
$\ell=\ol{ab}$ of $H$ there exists $m$ such that
$\si_d^m(a)=\si_d^m(b)=x$ is a point of the Cantor set $H\cap \uc$ that
is a limit point of $H\cap\uc$ from both sides. This implies that there
are no leaves of $\lam$ coming out of $x$. Therefore, any finite
concatenation of leaves of $\lam$ coming out of $a$ or $b$ maps to $x$
under $\si_d^m$ and must be contained in the appropriate finite polygon
mapped to $x$ under $\si_d^m$. This implies the first two claims of the
lemma. The existence of infinite gaps at edges of $L$ follows now from
the definition of an invariant geodesic lamination.
\end{proof}

Observe that, by Lemma~\ref{l:sieg-str}, the $\si_d^{m-1}$-image of $L$
is an all-critical gap. We will need Lemma~\ref{l:sieg-str} in what
follows, in particular, when we study Siegel gaps of two linked
geodesic laminations.

\begin{dfn}\label{d:exte-sieg}
In what follows sets $L$ defined in Lemma~\ref{l:sieg-str} are said to
be \emph{decorations} of $G$. The union of $G$ with all its decorations
is said to be the \emph{extension of $G$}.
\end{dfn}

In particular, Lemma~\ref{l:sieg-str} shows that the semiconjugacy
$\varphi$ can be defined on the extension of the corresponding Siegel
gap. Then the fibers of $\varphi$ (i.e., point-preimages under
$\varphi$) are either decorations of $G$ or single points of the set
$G\cap \uc$ that are limit points of $G\cap \uc$ from both sides.

Lemma~\ref{l:perinf} implies Corollary~\ref{c:noinf}.

\begin{cor}\label{c:noinf}
Suppose that $G$ is a periodic gap of an invariant geodesic lamination
$\lam$, whose remap has degree one. Then at most countably many
pairwise unlinked leaves of other invariant geodesic laminations can be
located inside $G$.
\end{cor}

We say that a chord is located \emph{inside}\index{chord!inside a gap}
$G$ if it is a subset of $G$ and intersects the interior of $G$.

\begin{proof}
Any chord located inside $G$ has its endpoints at vertices of $G$.
Since in case (1) of Lemma~\ref{l:perinf} there are countably many
vertices of $G$, we may assume that case (2) of Lemma~\ref{l:perinf}
holds. Applying the semiconjugacy $\varphi$ from this lemma, we see
that if a leaf $\ell$ is located in $G$ and its endpoints do not map to
the same point by $\varphi$, then an iterated image of $\ell$ will
eventually cross $\ell$. If there are uncountably many leaves of
geodesic laminations inside $G$, then among them there must exist a
leaf $\ell$ with endpoints in distinct fibers of $\ph$. By the above,
some forward images of $\ell$ cross each other, a contradiction.
\end{proof}

\chapter{Special types of invariant laminations}\label{c:spety}

\section{Invariant geodesic laminations with quadratically critical
portraits}\label{s:qc-por}

Here we define invariant geodesic laminations with quadratically
critical portraits and discuss linked or essentially equal invariant
geodesic laminations with quadratically critical portraits. First we
motivate our approach and study families of collections of quadratic
quadrilaterals with certain natural properties. Then we discuss
properties of families of invariant geodesic laminations, for which the
corresponding collections of critical quadrilaterals can be defined.

\subsection{Collections of critical quadrilaterals and their
properties}\label{ss:coll-quad}

Thurston defines the \emph{minor}\index{minor} $m$ of a
$\si_2$-invariant geodesic lamination $\lam$ as the image of a longest
leaf $M$ of $\lam$. Any longest leaf of $\lam$ is said to be a
\emph{major}\index{major} of $\lam$. If $m$ is non-degenerate, then
$\lam$ has two disjoint majors, both mapping to $m$; if $m$ is
degenerate, then $\lam$ has a unique major that is a critical leaf. In
the quadratic case, the majors are uniquely determined by the minor.
Thus, a quadratic invariant geodesic lamination is essentially defined
by its minor. Even though, in the cubic case, one could define majors
and minors similarly, unlike in the quadratic case, these ``minors'' do
not uniquely determine the corresponding majors.

The simplest way to see that is to consider distinct pairs of critical
leaves with the same images. More precisely, choose an all-critical
triangle $\Delta_1$ with non-periodic vertices so that the common image
$x_1$ of the vertices of $\Delta_1$ is periodic (alternatively, has a
dense orbit in $\uc$). Choose a different all-critical triangle
$\Delta_2$ with similar properties. Now, choose an edge $\ol{c}$ of
$\Delta_1$. Clearly, there is a unique edge $\ol{d}$ of $\Delta_2$
disjoint from $\ol{c}$. Under the assumptions made about $\Delta_1$ and
$\Delta_2$ it is easy to see that the two critical leaves $\ol{c}$ and
$\ol{d}$ have so-called \textbf{aperiodic kneadings} as defined by Kiwi
in \cite{kiwi97}. Therefore, by \cite{kiwi97}, these critical leaves
generate the corresponding cubic invariant geodesic lamination. Any
other similar choice of critical edges of $\Delta_1$ and $\Delta_2$
gives rise to a cubic invariant geodesic lamination too; clearly, these
two invariant geodesic laminations are very different even though they
have the same images of their critical leaves, that is, the same minors
(see Figure~\ref{fig:2all-cri}).
%Figure~2).
Thus, in the cubic case we should be concerned with critical sets, not
only their images.

\begin{figure}
  \includegraphics[height=6cm]{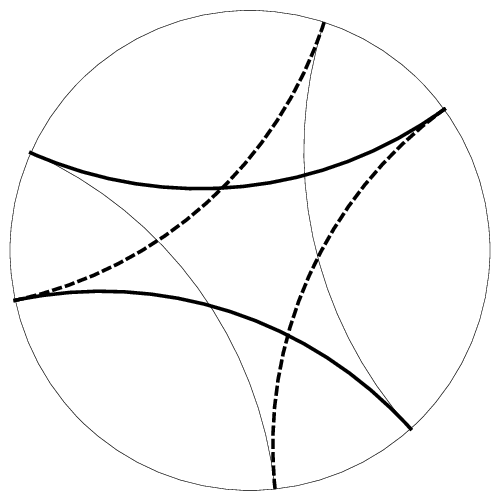}
  {\caption{Two all-critical triangles in the cubic case. Taking similar
  outlined edges (thick and thick, thin and thin, or dashed and dashed),
  one from each of the two triangles, generates three different laminations with the same pair of minors.}\label{fig:2all-cri}}
\end{figure}

We study how ordered collections of critical sets of invariant geodesic
laminations are located with respect to each other. The fact that
critical sets may have different degrees complicates such study. So, it
is natural to adjust our invariant geodesic laminations to make sure
that the associated critical sets of two invariant geodesic laminations
are of the same type. As associated critical sets we choose
(generalized) \emph{critical quadrilaterals}.

\begin{dfn}\label{d:qll}
A (generalized) \emph{critical
quadrilateral}\index{quadrilateral!critical (generalized)} $Q$ is the
circularly ordered 4-tuple $[a_0,a_1,a_2,a_3]$ of marked points $a_0\le
a_1\le a_2\le a_3\le a_0$ in $\uc$ so that $\ol{a_0a_2}$ and
$\ol{a_1a_3}$ are critical chords (called \emph{spikes}\index{spike});
here critical quadrilaterals $[a_0,a_1,a_2,a_3]$, $[a_1,a_2,a_3,a_0]$,
$[a_2,a_3,a_0,a_1]$ and $[a_3,a_0,a_1,a_2]$ are viewed as equal.
\end{dfn}

We want to comment upon our notation. By $(X_1, \dots, X_k)$, we always
mean a $k$-tuple, that is, an \textbf{ordered} collection of elements
$X_1$, $\dots$, $X_k$. On the other hand, by $\{X_1, \dots, X_k\}$ we
mean a collection of elements $X_1, \dots, X_k$ with no fixed order.
Since, for critical quadrilaterals, we need to emphasize the
\textbf{circular} order among its vertices, we choose the notation
$[a_0,a_1, a_2, a_3]$ distinct from either of the two just described
notations.

For brevity, we will often use the expression ``critical
quadrilateral'' when talking about the convex hull of a critical
quadrilateral. Clearly, if all vertices of a critical quadrilateral are
distinct or if its convex hull is a critical leaf, then the
quadrilateral is uniquely defined by its convex hull. However, if the
convex hull of a critical quadrilateral is a triangle, this is no
longer true. Indeed, let $T=\ch(a, b, c)$ be an all-critical triangle.
Then $[a,a,b,c]$ is a critical quadrilateral, but so are $[a,b,b,c]$
and $[a,b,c,c]$.

A \emph{collapsing quadrilateral}\index{quadrilateral!collapsing} is a
critical quadrilateral, whose $\si_d$-image is a leaf. A critical
quadrilateral $Q$ has two intersecting spikes and is either a
collapsing quadrilateral, or a critical leaf, or an all-critical
triangle, or an all-critical quadrilateral. If all its vertices are
pairwise distinct, then we call $Q$ \emph{non-degenerate}, otherwise
$Q$ is called \emph{degenerate}\index{quadrilateral!(non)-degenerate}.
Vertices $a_0$ and $a_2$ ($a_1$ and $a_3$) are called
\emph{opposite}\index{vertices!opposite}. Considering invariant
geodesic laminations, all of whose critical sets are critical
quadrilaterals, is not very restrictive: we can ``tune'' a given
invariant geodesic lamination by inserting new leaves into its critical
sets in order to construct a new invariant geodesic lamination with all
critical sets being critical quadrilaterals.

\begin{lem}\label{l:qls}
The family of all critical quadrilaterals is closed. The family of all
critical quadrilaterals that are critical sets of invariant geodesic
laminations is closed too.
\end{lem}

\begin{proof}
The first claim follows easily because if a sequence of critical
quadrilaterals converges then the limit is again a critical
quadrilateral. The second one follows from Theorem~\ref{t:sibliclos}
and the fact that if $\lam_i\to \lam$, then the critical quadrilaterals
of invariant geodesic laminations $\lam_i$ converge to critical
quadrilaterals that are critical sets of $\lam$.
\end{proof}

In the quadratic case, we have less variety of critical quadrilaterals:
only collapsing quadrilaterals and critical leaves. As mentioned above,
each quadratic invariant geodesic lamination $\lam$ either\, already
has a critical quadrilateral, or\, can be tuned to have one. The latter
can be done in several ways if $\lam$ has a finite critical set (on
which $\si_2$ acts in a two-to-one fashion). If, however, $\lam$ does
not have a finite critical set, then its critical set must be a
periodic Fatou gap $U$ of degree two. It follows from \cite{thu85} that
it has a unique refixed edge $M$; then one can tune $\lam$ by inserting
into $U$ the quadrilateral that is the convex hull of $M$ and its
sibling leaf.

Thurston's parameterization \cite{thu85} can be viewed as associating
to every quadratic invariant geodesic lamination $\lam$ with critical
quadrilateral $Q$ its minor $m$. It is easy to see that $m$ is the
$\si_2$-image of $Q$ and that $Q=\si_2^{-1}(m)$ is the full
$\si_2$-preimage of $m$. We would like to translate some crucial
results of Thurston's into the language of critical quadrilaterals of
quadratic invariant geodesic laminations.

To this end, observe, that, by the above, two minors cross if and only
if their full pullbacks (which are collapsing quadrilaterals coinciding
with convex hulls of pairs of majors) have a rather specific mutual
location: their vertices alternate on the circle. A major result of
Thurston's from \cite{thu85} is that \textbf{minors of different
quadratic invariant geodesic laminations are unlinked}; in the language
of critical quadrilaterals this can be restated as follows:
\textbf{critical quadrilaterals of distinct quadratic invariant
geodesic laminations cannot have vertices that alternate on the
circle}. All this motivates Definition~\ref{d:strolin}.

\begin{dfn}[cf with \cite{thu85}]\label{d:strolin}
Let $A$ and $B$ be two quadrilaterals. Say that $A$ and $B$ are
\emph{strongly linked}\index{strongly linked quadrilaterals} if the
vertices of $A$ and $B$ can be numbered so that the following holds:
$$a_0\le b_0\le a_1\le b_1\le a_2\le b_2\le a_3\le b_3\le a_0$$
where $a_i$, $0\le i\le 3$, are vertices of $A$ and $b_i$, $0\le i\le
3$ are vertices of $B$. Equivalently, $A$ and $B$ are \emph{strongly
linked} if no hole of either quadrilateral contains more than one
vertex of the other one.
\end{dfn}

Strong linkage is a closed condition: if two variable critical
quadrilaterals are strongly linked and converge, then they must
converge to two strongly linked critical quadrilaterals. An obvious
case of strong linkage is between two non-degenerate critical
quadrilaterals, whose vertices alternate on the circle so that all the
inequalities in Definition~\ref{d:strolin} are strict. Yet even if both
critical quadrilaterals are non-degenerate, some inequalities may be
non-strict, which means that some vertices of both quadrilaterals may
coincide.

For example, two coinciding critical leaves can be viewed as strongly
linked critical quadrilaterals. Otherwise, an all-critical triangle $A$
with vertices $x$, $y$, $z$ and its edge $B=\ol{yz}$ can be viewed as
strongly linked quadrilaterals if the vertices are chosen as follows:
$a_0=x, a_1=a_2=y, a_3=z$ and $b_0=b_1=y, b_2=b_3=z$. Observe that if a
critical quadrilateral $Q$ is a critical leaf or has all vertices
distinct, then $Q$ as a critical quadrilateral has a well-defined
assignment of vertices; the only ambiguous case is when $Q$ is an
all-critical triangle.

If an ordered collection of a few chords can be concatenated to form a
Jordan curve, or if there are two identical chords, then we say that
they form a \emph{loop}\index{loop of chords}. In particular, one chord
does not form a loop while two equal chords do. If an ordered
collection of chords $(\ell_1, \dots, \ell_k)$ contains no chords
forming a loop, then we call it a \emph{no loop collection}\index{no
loop collection of chords}.

\begin{lem}\label{l:noloop}
The family of no loop collections of critical chords is closed.
\end{lem}

\begin{proof}
Suppose that there is a sequence of no loop collections of critical
chords $\nc^i=(\ell_1^i, \dots, \ell_s^i)$ such that $\nc^i\to
\nc=(\ell_1, \dots, \ell_s)$. Clearly, all chords $\ell_i$ are
critical. We need to show that $\nc$ is a no loop collection. By way of
contradiction assume that, say, chords $\ell_1=\ol{a_1a_2}, \dots,
\ell_k=\ol{a_ka_1}$ form a loop $\widehat \nc$, in which the order of
points $a_1, \dots, a_k$ is positive. We claim that $\widehat \nc$
cannot be the limit of no loop collections of critical chords,
contradicting the convergence assumption that $\nc^i\to \nc$. This
follows from the fact that if $G'\subset\uc$ is a union of finitely
many sufficiently \emph{small} circle arcs such that all
\emph{straight} edges in the boundary of the convex hull $G=\ch(G')$
are critical, then in fact all circle arcs in $G'$ are degenerate, so
that $G$ is a finite polygon.

A more formal proof follows. Consider chords $\ell^i_1=\ol{b^i_1d^i_1},
\dots, \ell^i_k=\ol{b^i_kd^i_k}$ such that points $b^i_j$ converge to
$a_j$ and points $d^i_j$ converge to $a_{j+1}$ ($j+1$ is understood
here and in the rest of the argument modulo $k$) as $i\to \infty$. Then
for a well-defined collection of integers $m_1, \dots, m_k$ we have
that $a_{j+1}=a_j+m_j \cdot \frac{1}d$ and $d^i_j=b^i_j+m_j\cdot
\frac{1}d$. Moreover, since $\widehat \nc$ is a loop, then
$m_1+\dots+m_k=d$. Now, since $\nc^i$ is a no loop collection and all
leaves in $\nc^i$ are unlinked, for each $1\le j\le k$ we have
$d^i_j=b^i_j+m_j\cdot \frac{1}{d}\le b^i_{j+1}$ and there exists at
least one $1\le j\le k$ such that $d^i_j<b^i_{j+1}$. Since by the above
$m_1+\dots+m_k=d$ it follows that after we follow $k$ chords
$\ell^i_1=\ol{b^i_1d^i_1}, \dots, \ell^i_k=\ol{b^i_kd^i_k}$ along the
circle considering $b_j^i$ as the initial point of $\ell^i_j$ and
$d_j^i$ as the terminal point of $\ell^i_j$ we see that the terminal
point $d^i_k$ of $\ell^i_k$ is located slightly beyond the initial
point $b^i_1$ of $\ell_1^i$ which implies that $\ell^i_k$ crosses
$\ell_1^i$, a contradiction.
\end{proof}

\begin{comment}

Adding up these inequalities we see that $d^i_1+(m_1+\dots+m_k)\cdot
\frac{1}{d}<b^i_k$.  Since by the above $m_1+\dots+m_k=d$ we discover
that $d^i_1+1=d^i_1<b^i_k$. On the other hand, the fact that all the
chords $\ell^i_1=\ol{b^i_1d^i_1}, \dots, \ell^i_k=\ol{b^i_kd^i_k}$ are
unlinked implies that $b^i_k\le d^i_1$, a contradiction.
\end{proof}

\end{comment}

We will need the following definition.

\begin{dfn}[Full collections and complete samples of spikes]\label{d:fulco}
Call a no loop collection of $d-1$ pairwise unlinked critical chords a
\emph{full collection (of critical chords)}\index{full collection of
critical chords}. Given a collection $\mathcal Q$ of $d-1$ distinct
critical quadrilaterals of an invariant  geodesic lamination $\lam$, we
choose one spike in each of them and call this collection of $d-1$
critical chords a \emph{complete sample of spikes (of $\mathcal
Q$)}\index{complete sample of spikes}.
\end{dfn}

Now we are ready to investigate invariant geodesic laminations for
which the appropriate collections of critical quadrilaterals can be
defined.

\subsection{Quadratically critical invariant geodesic laminations}\label{ss:qua-cri-lam}

Suppose that $\lam$ is the invariant geodesic lamination generated by a
laminational equivalence relation all of whose critical sets are
critical quadrilaterals. Then any complete sample of spikes is a full
collection because in this case distinct critical sets are disjoint.
Observe that, by Lemma~\ref{l:noloop}, full collections of critical
chords form a closed family. It follows that the fact that complete
samples of spikes form a full collection survives limit transition
(unlike pairwise disjointness). This inspires another definition.

\begin{dfn}[Quadratic criticality]\label{d:quacintro}
Let $(\lam, \qcp)$ be an invariant  geodesic lamination with a
$(d-1)$-tuple $\qcp$ of critical quadrilaterals that are gaps or leaves
of $\lam$ such that any complete sample of spikes is a full collection.
Then $\qcp$ is called a \emph{quadratically critical portrait for
$\lam$}\index{quadratically critical portrait} while the pair $(\lam,
\qcp)$ is called an \emph{invariant geodesic lamination with
quadratically\, critical portrait}\index{invariant geodesic
lamination!with quadratically critical portrait} (if the appropriate
invariant geodesic lamination $\lam$ for $\qcp$ exists but is not
emphasized we simply call $\qcp$ a \emph{quadratically critical
portrait}\index{quadratically critical portrait}). The space of all
quadratically critical portraits is denoted by $\fqcp_d$. The family of
all invariant geodesic laminations with quadratically critical
portraits is denoted by $\L\fqcp_d$.
\end{dfn}

Observe that any full collection of critical chords is a quadratically
critical portrait. Notice also, that if $C$ is a complementary
component of a full collection of critical chords in $\disk$, then
$\si_d$ is one-to-one on the boundary of $C$ except for critical chords
contained in the boundary of $C$. Therefore, if $\lam$ admits a
quadratically critical portrait then there are no gaps of $\lam$ of
degree greater than one that are different from critical quadrilaterals
from this quadratically critical portrait. In particular, $\lam$ cannot
have infinite gaps of degree greater than one.

\begin{cor}\label{c:qcpoclose}
The spaces $\fqcp_d$ and $\L\fqcp_d$ are compact.
\end{cor}

\begin{proof}
Let $(\lam^i, \qcp^i)\to (\lam, \cp)$, where convergence as always is
understood in the Hausdorff ($H_H$) sense. By Theorem~\ref{t:sibliclos}
and Lemma~\ref{l:qls}, here in the limit we have an invariant geodesic
lamination $\lam$ and an ordered collection $\cp$ of $d-1$ critical
quadrilaterals. Let $\cp=(C_j)^{d-1}_{j=1}$ be these limit critical
quadrilaterals. Choose a collection of spikes $\ell_j, j=1, \dots, d-1$
of quadrilaterals of\, $\cp$. Suppose that there is a loop formed by
some of these spikes. By construction, there exist collections of
spikes from quadratically critical portraits $\qcp^i$ converging to
$(\ell_1, \dots, \ell_{d-1})$. Since by definition these are full
collections of critical chords, this contradicts Lemma~\ref{l:noloop}.
Hence $(\ell_1, \dots, \ell_{d-1})$ is a full collection of critical
chords too. That implies that $\cp$ is a quadratically critical
portrait for $\lam$ and proves that $\fqcp_d$ and $\L\lp_d$ are compact
spaces.
\end{proof}

The following lemma describes invariant geodesic laminations admitting
a quadratically critical portrait. Recall that by a \emph{collapsing
quadrilateral}\index{quadrilateral!collapsing} we mean a critical
quadrilateral that maps to a non-degenerate leaf.

\begin{lem}\label{l:qcpoexi}
An invariant  geodesic lamination $\lam$ has a quadratically critical
portrait if and only if all its critical sets are collapsing
quadrilaterals or all-critical sets.
\end{lem}

\begin{proof}
If $\lam$ has a quadratically critical portrait, then the claim of the
lemma follows by definition. Assume that the critical sets of $\lam$
are collapsing quadrilaterals and all-critical sets. Then $\lam$ may
have several critical leaves (some of them are edges of all-critical
gaps, some are edges of other gaps, some are not edges of any gaps at
all). Choose a no loop collection of critical leaves of $\lam$ which is
maximal by cardinality. Add to them the collapsing quadrilaterals of
$\lam$. Include all selected sets in the family of pairwise distinct
sets $\cp=(C_1, \dots, C_m)$ consisting of critical leaves and
collapsing quadrilaterals.

We claim that $\cp$ is a quadratically critical portrait. To this end
we need to show that $m=d-1$ and that any collection $\nc$ of spikes of
sets from $\cp$ is a no loop collection. First of all, let us show that
any such collection $\nc$ contains no loops. Indeed, suppose that $\nc$
contains a loop $\ell_1\in C_1$, $\dots$, $\ell_r\in C_r$. By
construction there must be a collapsing quadrilateral among sets $C_1,
\dots, C_r$. We may assume that, say, $C_1=[a, x, b, y]$ is a
collapsing quadrilateral and $\ell_1=\ol{ab}$ is contained in the
interior of $C_1$ except for points $a$ and $b$. The spikes $\ell_2,
\dots, \ell_r$ form a chain of concatenated critical chords which has,
say, $b$ as its initial point and $a$ as its terminal point. Since
these spikes come from sets $C_2, \dots, C_r$ distinct from $C_1$, they
have to pass through either $x$ or $y$ as a vertex, a contradiction
with $C_1$ being collapsing. Thus, $\nc$ contains no loops, which
implies that the number $m$ of chords in $\nc$ is at most $d-1$.

Assume now that $m<d-1$ and show that this leads to a contradiction.
Indeed, if $m<d-1$, then we can find a component $U$ of $\disk\sm
\nc^+$ with boundary including some circle arcs such that $\si_d$ on
the boundary of $U$ is $k$-to-$1$ or higher with $k>1$ (images of
critical edges of $U$ may have more than $k$ preimages). We claim that
there exists a critical chord $\ell$ of $\lam$ inside $U$ that connects
points in $\bd(U)$ not connected by a chain of critical edges in
$\bd(U)$. Observe that an arc on $\bd(U)$ may include several critical
chords from the collection $\nc$. Consider all open arcs $A\subset
\bd(U)$ such that $\si_d$ is non-monotone on $A$, and the endpoints of
$A$ are connected by a leaf of $\lam$. Call such open arcs and the
corresponding closed arcs \emph{non-monotone}. Non-monotone arcs exist;
indeed, by the assumptions there exist leaves $\ell$ of $\lam$ inside
$U$, and at least one of the two arcs in the boundary of $U$ that
connect the endpoints of $\ell$ must be non-monotone.

The intersection of a decreasing sequence of non-monotone arcs is a
closed arc $A_0$ with endpoints connected with a leaf $\ell_0\in \lam$
such that either $\ell_0$ is the desired critical leaf of $\lam$ (the
leaf $\ell_0$ cannot connect two points otherwise connected by a chain
of critical edges from $\bd(U)$ as this would contradict the fact that
arcs approaching $A_0$ are non-monotone), or $A_0$ is still
non-monotone. Thus, it will be enough to show that if $A_0$ is a closed
non-monotone arc which is minimal by inclusion, then there exists the
desired critical chord of $\lam$.

Clearly, $A_0\cup \ell_0$ is a Jordan curve enclosing a Jordan disk
$T$, and $A_0$ is not a union of spikes. If $\ell_0$ is not critical,
then, by the assumption of minimality of $A_0$, the leaf $\ell_0$
cannot be approached by leaves of $\lam$ from within $T$, thus $\ell_0$
is an edge of a gap $G\subset T$. Take a component $W$ of $T\sm G$ that
shares an edge $\m$ with $G$. Then, by minimality of $A_0$, either
$\bd(W)$ collapses to a point or $\bd(W)$ maps in a monotone fashion to
the hole of $\si_d(G)$ located ``behind'' $\si_d(\m)$ united with
$\si_d(\m)$. This implies that $G$ is critical as otherwise the quoted
properties of components $W$ of $T\sm G$ and the fact that $\si_d$ maps
$G$ onto $\si_d(G)$ in a one-to-one fashion show that $\si_d|_{A_0}$ is
(non-strictly) monotone, a contradiction. The gap $G$ cannot be
all-critical, since $\ell_0$ is an edge of $G$. Therefore, $G$ is a
collapsing quadrilateral, which contradicts our choice of $\Cc$.
\end{proof}

Observe that there might exist several quadratically critical portraits
for an invariant geodesic lamination $\lam$ from Lemma~\ref{l:qcpoexi}.
For example, consider a $\si_4$-invariant geodesic lamination $\lam$
with two all-critical triangles $\Delta_1=\ch(a,b,c), \Delta_2=\ch(a,c,
d)$ sharing an edge $\ell=\ol{ac}$. The proof of Lemma~\ref{l:qcpoexi}
leads to a quadratically critical portrait consisting of any three
edges of $\Delta_1$, $\Delta_2$ not equal to $\ell$ in some order
(recall that for each critical leaf its structure as a quadrilateral is
unique). However it is easy to check that the collection $([a,b,b,c]$,
$[a,a,c,c]$, $[a,c,d, d])$ is a quadratically critical portrait too.
Notice that, in the definition of a complete sample of spikes, we do
not allow to use more than one spike from each critical set, hence the
pair of coinciding spikes in $[a,a,c,c]$ does not form a loop of
spikes.

\begin{figure}
  \includegraphics[width=10cm]{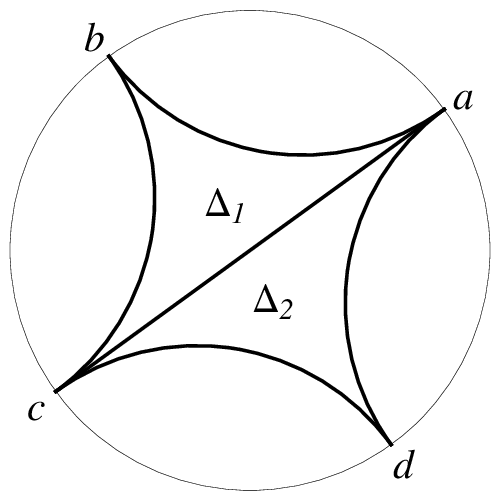}
  \caption{This figure illustrates the case with two all-critical triangles
  $\Delta_1$ and $\Delta_2$ described in the text.}
  \label{f:qc-portrait}
\end{figure}

Given a quadratically critical portrait $\qcp$, any complete sample of
spikes is a full collection of critical chords. If $\qcp$ includes sets
that are not leaves, then there are several complete samples of spikes
as the choice of spikes is ambiguous. This is important for
Section~\ref{s:smart}, where we introduce and study the so-called
\textbf{smart criticality} and its applications to \textbf{linked
invariant geodesic laminations with quadratically critical portraits}
introduced below. First we need a technical definition.

\begin{dfn}\label{d:compat}\index{cluster!critical}
A \emph{critical cluster} of $\lam$ is a convex subset of $\ol\disk$
which is maximal by inclusion, whose boundary is a union of critical
leaves of $\lam$.
\end{dfn}

A critical leaf disjoint from all other leaves is itself a critical
cluster. Consider also the example discussed after
Lemma~\ref{l:qcpoexi}. There, a $\si_4$-invariant geodesic lamination
$\lam$ has two all-critical triangles sharing a critical edge; the
union of these triangles is a critical cluster of $\lam$.

\begin{dfn}[Linked invariant geodesic laminations]\label{d:qclink1}\index{invariant geodesic laminations!linked or essentially equal}
Let $\lam_1$ and $\lam_2$ be geola\-mi\-na\-tions with quadratically
critical portraits $\qcp_1=(C^i_1)_{i=1}^{d-1}$ and
$\qcp_2=(C^i_2)_{i=1}^{d-1}$ and a number $0\le k\le d-1$ such that:

\begin{enumerate}
\item for each $j>k$ the sets $C^j_1$ and $C^j_2$ are contained in a
    common critical cluster of $\lam_1$ and $\lam_2$ (in what follows
    these clusters will be called \emph{special critical
    clusters}\index{clusters, special critical} and leaves contained
    in them will be called \emph{special critical
    leaves}\index{leaf!special critical}).
\item for every $i$ with $1\le i\le k$, the sets $C^i_1$ and $C^i_2$
    are either strongly linked critical quadrilaterals or share a
    spike.
\end{enumerate}

\noindent Then we use the following terminology:

\begin{enumerate}

\item[(a)] if in {\rm (1)} for every $i$ with $1\le i\le k$, the
    quadrilaterals $C^i_1$ and $C^i_2$ share a spike, we say that
    $\qcp_1$ and $\qcp_2$, (as well as $(\lam_1, \qcp_1)$ and
    $(\lam_2, \qcp_2)$) are \emph{essentially
    equal})\index{essentially equal},

\item[(b)] if in {\rm (1)} there exists $i$ with $1\le i\le k$ such
    that the quadrilaterals $C^i_1$ and $C^i_2$ are strongly linked
    and do not share a spike, we say that $\qcp_1$ and $\qcp_2$ (as
    well as $(\lam_1, \qcp_1)$ and $(\lam_2, \qcp_2)$) are
    \emph{linked}\index{linked}.

\end{enumerate}

\noindent The critical sets $C^i_1$ and $C_2^i$, $1\le i\le d-1$ are
called \emph{associated $($critical sets of invariant geodesic
laminations with quadratically critical portraits $(\lam_1, \qcp_1)$
and $(\lam_2, \qcp_2))$}\index{critical set}.
\end{dfn}

\section{Some special types of invariant geodesic
laminations}\label{s:pandp} Below, we discuss perfect invariant
geodesic laminations and dendritic invariant geodesic laminations.

\subsection{Perfect invariant geodesic laminations}\label{ss:pergeo}

The following is a natural basic definition.

\begin{dfn}\label{d:perfect} An invariant geodesic lamination $\lam$ is
said to be \emph{perfect}\index{invariant geodesic lamination!perfect}
if no leaf of $\lam$ is isolated. Given \textbf{any} invariant geodesic
lamination $\lam$, we can consider it with the Hausdorff metric;
clearly this makes $\lam$ a compact metric space. Define the
\emph{perfect part} of $\lam$\index{perfect part} as the maximal
perfect subset $\lam^p$ of $\lam$.
\end{dfn}

Since points of $\uc$, considered as degenerate leaves, belong to
$\lam$, it follows from Definition~\ref{d:perfect} that $\lam^p$ must
contain all singletons of $\uc$. Lemma~\ref{l:perfect-sub-lam} easily
follows from the definitions. Recall that by a \emph{collapsing
polygon}\index{polygon!collapsing} we mean a critical polygon that maps
onto a non-degenerate leaf. In other words, if $G$ is a collapsing
polygon, then all its edges map to the image leaf.

\begin{lem}\label{l:perfect-sub-lam}
The collection $\lam^p$ is an invariant perfect geodesic lamination.
For every $\ell\in \lam^p$ and every neighborhood $U$ of $\ell$, there
exist uncountably many leaves of $\lam^p$ in $U$.
\end{lem}

\begin{proof}
The fact that no leaf of $\lam^p$ is isolated follows immediately. To
see that $\lam^p$ is invariant, notice that by definition only edges of
collapsing quadrilaterals or their sibling leaves can have ambiguous
collections of $d$ pairwise disjoint sibling leaves. Indeed, if a leaf
$\ell$ has more than $d-1$ sibling leaves, then two sibling leaves of
$\ell$ must have a common vertex. This implies the claimed. It follows
that there are at most finitely many leaves for which the choice of a
collection of pairwise disjoint sibling leaves is ambiguous.

Now, let $\ell=\ol{xy}$ be a non-critical leaf of $\lam^p$. Choose a
sequence of leaves $\ell_i=\ol{x_iy_i}$ of $\lam^p$ such that $x_i\to
x, y_i\to y$ and for every $i$ there are exactly $d$ leaves in $\lam$
(including $\ell_i$) with the $\si_d$-image
$\ol{\si_d(x_i)\si_d(y_i)}$. Moreover, their images are of length close
to that of $\si_d(\ell)$ and therefore are bounded away from zero. This
easily implies that the distance between any two endpoints of any two
leaves from the collection of all sibling leaves of $\ell_i$ is bounded
away from zero too. Therefore, the limits of these leaves form a
collection of $d$ leaves of $\lam^p$ (including $\ell$) with the same
image $\si_d(\ell)$. By definition, this shows that $\lam^p$ is
invariant, as desired.
\end{proof}

While the existence of the perfect part $\lam^p$ of $\lam$ is thus
established, the actual construction of it in the dynamical case is not
obvious at all. The process of finding $\lam^p$ was described in detail
in \cite{bopt10}. In what follows we will need a few facts and concepts
established in \cite{bopt10}.

\begin{dfn}[Supergaps \cite{bopt10}] Consider a (periodic) infinite gap $U$ of
$\lam^p$. Then $U$ is said to be a (periodic) \emph{super-gap} of
$\lam$.
\end{dfn}

Clearly, the part of the invariant geodesic lamination $\lam$ contained
in the orbit of $U$ consists of no more than countably many
non-degenerate leaves.

Suppose that $\sim$ is an invariant laminational equivalence relation
and $\lam_\sim$ is the corresponding invariant geodesic lamination. Let
$\lam^0_\sim=\lam_\sim$ and define $\lam^k_\sim$ inductively by
removing all isolated leaves from $\lam^{k-1}_\sim$ (one may call each
step in this process a \emph{countable cleaning}). It is proven in
\cite{bopt10} that after finitely many countable cleaning steps we will
obtain the invariant perfect geodesic lamination $\lam^p_\sim$. This
result is then used in proving the following lemma.

\begin{lem}[Lemma 3.2 \cite{bopt10}]\label{l:leafinsgap}
If $\sim$ is a laminational equivalence relation, then the following
holds.

\begin{enumerate}

\item Every leaf of $\lam_\sim$ inside a super-gap $G$ of $\sim$ is
    (pre)periodic or (pre)critical; every edge of a super-gap is
    (pre)periodic.

\item Every edge of any gap $H$ of $\lam^p_\sim$ is  not isolated in
    $\lam^p_\sim$ from outside of $H$; all gaps of $\lam^p_\sim$ are
    pairwise disjoint. Moreover, gaps of $\lam^p_\sim$ are disjoint
    from leaves that are not their edges.

\item There are no infinite concatenations of leaves in
    $\lam^p_\sim$. Moreover, the invariant geodesic lamination
    $\lam^p_\sim$ is generated by a laminational equivalence relation
    $\sim^p$ except that there may be the following leaves of
    $\lam_{\sim^p}$ which by definition should not be included in
    $\lam^p_\sim$: it is possible that one edge of certain finite gap
    of $\sim^p$ is a leaf passing inside an infinite gap of
    $\lam^p_\sim$.

\item Any periodic Siegel gap is a proper subset of its super-gap.

\end{enumerate}

\end{lem}

The next lemma specifies some properties of any perfect invariant
geodesic lamination $\lam^p$.

\begin{lem}\label{l:lamp}
If $\lam^p$ is a perfect invariant geodesic lamination, then at most
two leaves of $\lam^p$ share an endpoint. Moreover, any leaf of
$\lam^p$ is a limit of uncountably many leaves of $\lam^p$ disjoint
from $\ell$. If a leaf $\ell$ is critical, then $\si_d(\ell)$ is a
point separated from the rest of the circle by images of those leaves
converging to $\ell$ so that $\ell$ is either disjoint from all other
leaves or gaps of $\lam^p$ or is an edge of an all-critical gap of
$\lam^p$ disjoint from all other leaves or gaps of $\lam^p$.
\end{lem}

\begin{proof}
Suppose that there are more than two leaves of $\lam^p$ coming out of
the same point. Then, since, by Lemma~\ref{l:cones}, there are at most
countably many leaves of $\lam^p$ sharing an endpoint, $\lam^p$ has
isolated leaves, a contradiction. This implies that any leaf of
$\lam^p$ is a limit of an uncountably many leaves of $\lam^p$ disjoint
from $\ell$. The rest of the lemma easily follows.
\end{proof}

Clearly, \Cref{l:lamp} implies the following corollary.

\begin{cor}\label{c:cridisj}
Let $\lam$ be a perfect invariant geodesic lamination. Then the
critical sets of $\lam$ are pairwise disjoint and are either
all-critical sets, or critical sets mapping exactly $k$-to-$1$, $k>1$,
onto their images.
\end{cor}

In this paper, we study, in particular, perfect invariant geodesic
laminations and dendritic invariant geodesic laminations, which are a
particular case of perfect ones. By \Cref{c:cridisj}, all such geodesic
laminations have critical sets with certain natural properties. To
avoid unnecessary complications, we will consider invariant geodesic
laminations with similar properties even if we do not necessarily
assume that they are perfect.

\begin{dfn}\label{d:regula}
Let $\lam$ be an invariant geodesic lamination. Suppose that their
critical sets are pairwise disjoint except for the case when a critical
leaf is a boundary edge of an all-critical set. Then we say that $\lam$
is \emph{regular}\index{invariant geodesic lamination!regular}.
\end{dfn}

If an invariant geodesic lamination $\lam$ is regular, then all its
critical leaves are boundary leaves of all-critical sets. In
particular, if $C$ is a critical set of $\lam$ which is not an
all-critical set then it maps onto its image in exactly $k$-to-$1$
fashion. By \Cref{c:cridisj}, perfect invariant geodesic laminations
are regular. However, it is easy to give examples of regular invariant
geodesic laminations that are not perfect. Indeed, all quadratic
invariant geodesic laminations corresponding to parabolic quadratic
polynomials from the main cardioid are regular while it is well-known
that they have only countably many non-degenerate leaves. Therefore,
these laminations are not perfect, and the perfect part in any such
lamination is the set of all points of $\uc$.

We will use quadratically critical portraits to parameterize (``tag'')
certain classes of regular invariant geodesic laminations. An obstacle
to this is the fact that an invariant  geodesic lamination $\lam$ with
a $k$-to-$1$ critical set such that $k>2$ does not admit a
quadratically critical portrait. However, using Lemma~\ref{l:qcpoexi},
it is easy to see that in this case one can insert critical
quadrilaterals in critical sets of higher degree in order to ``tune''
$\lam$ into an invariant  geodesic lamination with a quadratically
critical portrait. This motivates the following definition.

\begin{dfn}\label{d:cripatt1}
Let $\lam$ be a regular invariant geodesic lamination with pairwise
disjoint critical sets (gaps or leaves) $D_1$, $\dots$, $D_k$. Let
$\lam\subset \lam_1$ and $\qcp=(E_1,\dots,E_{d-1})$ be a quadratically
critical portrait for $\lam_1$. Clearly, there is a unique
$(d-1)$-tuple $\zc$ $=$ $(C_1,\dots,C_{d-1})$ such that for every $1\le
i\le d-1$ we have $E_i\subset C_i$ and there is $1\le j(i)\le k$ with
$C_i=D_{j(i)}$. Then $\zc$ is called the \emph{critical pattern of
\,$\qcp$ in $\lam$}\index{critical pattern}; we will also say that
$\qcp$ \emph{generates} $\zc$. Observe that each $D_{j(i)}$ is repeated
in $\zc$ exactly $m_{j(i)}-1$ times, where $m_{j(i)}$ is the degree of
$D_{j(i)}$.

In general, given a regular invariant geodesic lamination $\lam$ with
(pairwise disjoint) critical sets $D_1$, $\dots$, $D_k$, by an
\emph{invariant geodesic lamination with critical
pattern}\index{invariant geodesic lamination!with critical pattern} we
mean a pair $(\lam, \zc)$, where $\zc$ $=$ $(C_1,\dots,C_{d-1})$ is a
$(d-1)$-tuple of sets such that every $C_i$ coincides with some $D_j$,
and every $D_j$ is repeated in $\zc$ exactly $m_j-1$ times, where $m_j$
is the degree of $D_j$. Then $\zc$ is called a \emph{critical pattern
for $\lam$}\index{critical pattern}.
\end{dfn}

Let us show that critical patterns from the second part of
\Cref{d:cripatt1} are always generated by quadratically critical
portraits.

\begin{lem}\label{l:non-vac}
Given a regular invariant geodesic lamination with a critical pattern
$(\lam, \zc)$, where $\lam$ has (pairwise disjoint) critical sets
$D_1$, $\dots$, $D_k$ of degrees $m_1$, $\dots$, $m_k$ respectively,
there exists a full collection of critical chords of sets $D_1$,
$\dots$, $D_k$ that generates $\zc$.
\end{lem}

Recall that a full collection of critical chords is a collection of
$d-1$ pairwise unlinked critical chords with no loops. As was noticed
before, a full collection can be viewed as a quadratically critical
portrait. Observe also that $\sum^{k}_{i=1}m_i-1=d-1$.

\begin{proof}
If $D_i$ is an all-critical set, then we can simply choose any $m_i-1$
of its (critical) edges. If the critical set $D_i$ is not all-critical,
then we can still choose $m_i$ points of $D_i\cap \uc$ with the same
image, take the convex hull of this collection of points, and, finally,
choose $m_i-1$ edges of this convex hull. Putting together the
collections of critical chords just constructed, we will create a
desired full collection of critical chords. It is easy to see now that
one can order them so that they generate $\zc$ as desired. It remains
to apply Thurston's pullback construction and this way construct the
geodesic lamination with the critical sets just chosen as required.
\end{proof}

Observe that the choice of a full collection that generates a given
critical pattern as explained above is far from unique. Notice also
that by changing the order of the critical sets in which they show in a
critical pattern, various critical patterns for the same invariant
geodesic lamination can be obtained.

\subsection{Dendritic invariant geodesic laminations with critical
patterns}\label{ss:fulden} Be\-low we introduce a useful notation.
Recall that in Definition~\ref{d:dendr} we define dendritic laminations
(in which case we do not allow for infinite classes of the
corresponding laminational equivalence relations) and dendritic
laminations possibly with infinite classes (in which case we do allow
for infinite classes). Observe that by Kiwi \cite{kiw02} an infinite
class of a lamination possibly with infinite classes must be
(pre)periodic.

\begin{dfn}\label{d:dendrilam}
The family of all dendritic invariant geodesic laminations is denoted
by $\L\mathcal D_d$. The family of all dendritic invariant geodesic
laminations possibly with infinite classes is denoted by $\L\mathcal
D_d^\infty$. The space of all dendritic invariant geodesic laminations
with critical patterns is denoted by $\L\cpd_d$. The space of all
dendritic invariant geodesic laminations possibly with infinite classes
and critical patterns is denoted by $\L\cpd_d^\infty$.
\end{dfn}

Observe that if $\lam=\lam_\sim$ is a dendritic invariant geodesic
lamination then all its gaps are finite (by definition) and correspond
to $\sim$-equivalence classes. The situation is a little more
complicated if $\lam=\lam_\sim$ is a dendritic invariant geodesic
lamination possibly with infinite classes. Lemma~\ref{l:perfectd} deals
with these cases.

\begin{lem}\label{l:perfectd}
Dendritic invariant geodesic laminations $\lam$ possibly with infinite
clas\-ses are perfect. On the other hand, every perfect geodesic
lamination can be viewed as a dendritic invariant lamination possibly
with infinite classes.
\end{lem}

\begin{proof}
Let $\ell$ be a leaf of $\lam$. By way of contradiction, suppose that
$\ell$ is isolated. Then $\ell$ is a common edge of two gaps. Denote
these gaps by $G$ and $H$. Since $\lam$ is dendritic, it follows that
in fact $\lam=\lam_\sim$ is generated by a dendritic laminational
equivalence relation $\sim$ possibly with infinite classes. On the
other hand, both $G$ and $H$ must be $\sim$-classes because $\sim$ is
dendritic. This shows that $G$ and $H$ must be forming one $\sim$-class
and therefore, by definition of the geodesic lamination generated by a
laminational equivalence relation, we see that $\ell$ cannot be a leaf
of $\lam_\sim$, a contradiction.

Now, suppose that $\lam$ is a perfect invariant geodesic lamination.
Then, by \Cref{l:lamp}, gaps of $\lam$ are pairwise disjoint. Hence the
set $\lam^+$ can be partitioned into pairwise disjoint leaves or gaps.
Declaring these sets as classes of equivalence of $\sim$ we see that
$\lam$ is generated by $\sim$ in the usual sense. Moreover, the fact
that the corresponding quotient map collapses all gaps of $\lam$ to
points implies that the corresponding quotient space is a dendrite.
This completes the proof.
\end{proof}

Since Siegel gaps and countable gaps have isolated edges, it follows
that the only gaps of a perfect geodesic lamination $\lam$ are either
finite gaps, or periodic Fatou gaps of degree greater than one, or
their preimages. As we noticed in the proof of \Cref{l:perfectd}, the
fact that $\lam$ is perfect implies that no two gaps of $\lam$ can
intersect. In particular, all Fatou gaps of $\lam$ are disjoint from
other gaps, both finite and Fatou. Moreover, Fatou gaps of a perfect
geodesic lamination have no critical edges as by the properties of
invariant geodesic laminations such edges would be isolated.

Thus, by \Cref{l:cripe}, if $U$ is a periodic Fatou gap of an invariant
perfect geodesic lamination $\lam$, then there are finitely many
periodic edges of $U$ and all other edges are their preimages. Observe
that if a perfect geodesic lamination $\lam$ has some Fatou gaps, then
$\lam$ can be generated by several laminational equivalence relations
depending on whether the corresponding quotient map collapses certain
grand orbits of Fatou gaps. In particular, all these gaps must be
collapsed under the quotient map in the case when the corresponding
quotient space is a dendrite.

Strong conclusions about the topology of the Julia sets of
non-re\-nor\-ma\-li\-zable polynomials $P\in \dc$ follow from
\cite{kvs06}. Building upon earlier results by Jeremy Kahn and Misha
Lyubich \cite{kl09a, kl09b} and by Oleg Kozlovskii, Weixiao Shen and
Sebastian van Strien \cite{ksvs07a, ksvs07b}, Kozlovskii and van Strien
generalized results of Artur Avila, Kahn, Lyubich and Shen
\cite{akls09} and proved in \cite{kvs06} that if all periodic points of
$P$ are repelling, and $P$ is non-renormalizable, then $J(P)$ is
locally connected; moreover, by \cite{kvs06}, two such polynomials that
are topologically conjugate are in fact quasi-conformally conjugate.
Thus, in this case $f_{\sim_P}|_{J_{\sim_P}}$ is a precise model of
$P|_{J(P)}$. Finally, for a given dendritic laminational equivalence
relation $\sim$, it follows from another result of Jan Kiwi
\cite{kiw05} that there exists a polynomial $P$ with $\sim=\sim_P$.
This polynomial does not have to have a locally connected Julia set.
Thus, by \cite{kiw05} associating polynomials from $\dc$ with their
laminational equivalence relations $\sim_P$ and invariant geodesic
laminations $\lam_P=\lam_{\sim_P}$, one maps polynomials from $\dc_d$
\textbf{onto} $\L\dc_d$.

To study the association of polynomials with their\, invariant geodesic
laminations, we need Lemma~\ref{l:gm} (it is stated as a lemma in
\cite{gm93} but goes back to Douady and Hubbard \cite{hubbdoua85}).

\begin{lem}[\cite{gm93, hubbdoua85}]\label{l:gm}
Let $P$ be a polynomial of degree $d>1$, and let $R$ be an external ray
of $P$ landing at an iterated preimage $y$ of a repelling periodic
point $x$. We write $n$ for the minimal non-negative integer such that
$P^n(y)=x$. Then, for every polynomial $P^*$ of degree $d$ that is
sufficiently close to $P$, the external ray $R^*$ of $P^*$ with the
same argument as $R$ lands at a point $y^*$ that is close to $y$ and
such that $x^*=P^{*n}(y)$ is a repelling periodic point of $P^*$ close
to $x$.
\end{lem}

In what follows we need a result from \cite{bl02} that deals with
subcontinua of topological Julia sets.

\begin{thm}[No Wandering Continua \cite{bl02}]\label{t:nwc}
Let $\sim$ be a laminational equivalence relation possibly with
infinite classes and $f_\sim:J_\sim\to J_\sim$ be the corresponding
topological polynomial. Then for any non-degenerate continuum $K\subset
J_\sim$ there exist $0\le l<m$ such that $f_\sim^l(K)\cap
f_\sim^m(K)\ne \0$.
\end{thm}

We will also need the following result, which combines Theorem 7.2.6 of
\cite{bfmot10} and a part of Theorem 7.2.7 of \cite{bfmot10}.

\begin{thm}[Theorems 7.2.6 and 7.2.7 of \cite{bfmot10}]\label{t:726}
Let $D\subset J_\sim$ be a dendrite such that $f_\sim(D)\subset D$.
Then $f_\sim$ has infinitely many periodic cutpoints in $D$.
\end{thm}

We are ready to prove the following lemma. Similar to the previously
introduced terminology, by a precritical point we mean a non-critical
point that eventually maps to a critical point while by a (pre)critical
point we mean a critical or precritical point.

\begin{lem}\label{l:condense}
Suppose that $\sim$ is a dendritic laminational equivalence relation
possibly with infinite classes. Then the following holds.

\begin{enumerate}

\item Each subcontinuum of $J_\sim$ contains a $($pre$)$periodic
    non-$($pre$)$critical point.

\item Each subcontinuum of $J_\sim$ contains a $($pre$)$critical
    point.

\item Each leaf of $\lam_\sim$ can be approximated by
    $($pre$)$periodic leaves that will never map to a critical set of
    $\lam_\sim$.

\end{enumerate}

\end{lem}

\begin{proof}
(1) Consider the topological polynomial $f_\sim$. Choose a continuum
$I\subset J_\sim$. Assume that the sets $I$ and $f^{k}_\sim(I)$ are
non-disjoint. Consider the union $T$ of all iterated $f^k_\sim$-images
of $I$ (this union is connected) and take its closure $\ol T$. Then
$\ol T\subset J_\sim$ is an $f^k_\sim$-invariant dendrite. By
\Cref{t:726}, there are infinitely many periodic cutpoints in $\ol T$.
Since $f_\sim^k$ has only finitely many critical points, there are
infinitely many periodic non-(pre)critical cutpoints in $\ol T$. Since
$T$ is connected and dense in $\ol T$, it follows that $T$ contains
periodic non-(pre)critical points. Hence $I$ contains (pre)periodic
non-(pre)critical points as desired.

Now, suppose that $K\subset J_\sim$ is any subcontinuum of $J_\sim$.
Then, by \Cref{t:nwc}, there exist an eventual image $I$ of $K$ such
that, for some $k>0$, the sets $I$ and $f^k_\sim(I)$ are non-disjoint.
By the previous paragraph, it follows that $I$, and therefore $K$,
contains (pre)periodic points as desired.

(2) The arguments are similar to those in the proof of the statement
(1). Suppose that a continuum $I\subset J_\sim$ does not contain
(pre)critical points of $f_\sim$. Then its forward images do not
contain critical points. By \Cref{t:nwc} we may assume that $I$ and
$f_\sim^k(I)$ are non-disjoint. Consider the union $T$ of all
$f^k_\sim$-images of $I$ (this union is connected) and take its closure
$K$. Then $K\subset J_\sim$ is an $f^k_\sim$-invariant dendrite. Since
by the construction $K$ and its images can only contain critical points
of $f_\sim$ as its endpoints, it follows that $f^k_\sim|_K$ is
one-to-one. By \Cref{t:726} there are infinitely many periodic
cutpoints of $K$. Hence we can find two periodic points $x, y\in K$
such that $f_\sim^n(x)=x, f_\sim^n(y)=y$ and $f_\sim^n|_[x, y]$ is
one-to-one where by $[x, y]$ we denote the unique arc inside $J_\sim$
with endpoints $x$ and $y$. Moreover, we may assume that there are no
$f^n_\sim$-fixed points in $(x, y)$. This implies that either $x$ or
$y$ attracts points of $[x, y]$ close to it, a contradiction.

(3) Let $\ell$ be a leaf of $\lam_\sim$. Since by \Cref{l:perfectd} the
invariant geodesic lamination $\lam_\sim$ is perfect, we can find a
side of  $\ell$ from which this leaf is non-isolated. Applying the
quotient map, we can find an arc in $J_\sim$ and its preimage under the
quotient map that is a connected union of a family of pairwise disjoint
leaves and gaps of $\lam_\sim$ (they are convex hulls of
$\sim$-equivalence classes) including $\ell$. These convex hulls of
$\sim$-equivalence classes are approaching $\ell$ from the side from
which $\ell$ is not isolated. By (1) there are (pre)periodic leaves or
gaps in this family; moreover, we can choose them so that they never
map to a critical set of $\lam_\sim$. Since by the previous results all
edges of a (pre)periodic gap of $\lam$ are (pre)periodic themselves, it
follows that $\ell$ can be approximated (from the above chosen side) by
(pre)periodic leaves that never map to a critical set of $\lam_\sim$ as
desired.
\end{proof}

In the dendritic case, the connection between critical patterns and
invariant geodesic laminations can be studied using results of Jan Kiwi
\cite{kiwi97}. One of the results that can be easily deduced from
\cite{kiwi97} is the following theorem. We provide a sketch of an
alternative geometric proof here.

\begin{thm}[cf \cite{kiwi97}]\label{t:fromkiwi}
If $\lam$ is a dendritic invariant geodesic lamination and $\lam'$ is
an invariant geodesic lamination such that $\lam$ and $\lam'$ share a
collection of $d-1$ critical chords with no loops among them, then
$\lam'\supset \lam$ and $\lam'\sm \lam$ consists of at most countably
many leaves inserted in certain gaps of $\lam$.
\end{thm}

\begin{proof}
Denote by $\sim$ a laminational equivalence relation generating $\lam$.
The critical chords shared by $\lam$ and $\lam'$ define $d-1$
complementary components to the closed unit disk $\ol{\disk}$. Clearly,
the closure $\ol{A}$ of each such component $A$ intersected with $\uc$
maps (under $\si_d$) onto the entire circle $\uc$ in a one-to-one order
preserving fashion (except for the endpoints). The boundary of $\ol{A}$
consists of circle arcs and concatenations of critical chords.

This allows one to consider pullbacks of chords into each such set
$\ol{A}$. Indeed, given a chord $\ell$ and a set $\ol{A}$ as above, we
can consider a set of all points in $\ol{A}\cap \uc$ that map to the
endpoints of $\ell$. Generically, either endpoint will have exactly one
preimage there. However if exactly one endpoint of $\ell$ equals the
image of a boundary critical chord of $\ol{A}$ (or of a concatenation
of boundary chords of $A$) then $\ell$ will have two preimages in
$\ol{A}$. Finally, if both endpoints of $\ell$ are images of boundary
concatenations of critical chords of $\ol{A}$, then we choose two
preimages of $\ell$ that are disjoint (it is easy to see that such
choice is unique).

The fact that the critical chords are shared by $\lam$ and $\lam'$ and
the definition of a invariant geodesic lamination imply that all
pullback chords constructed like that, except possibly for finitely
many chords, are shared by $\lam'$ and $\lam$. Therefore, limits of
these pullback chords are leaves of both $\lam$ and $\lam'$. As follows
from \cite{thu85}, these limits form an invariant geodesic lamination
$\lam''$. Moreover, by \Cref{l:condense}, each subcontinuum of $J_\sim$
contains (pre)critical points, which implies that $\lam''=\lam$. Since
all gaps of $\lam$ are finite, it follows that $\lam'\sm \lam$ consists
of at most countably many leaves inserted in certain gaps of $\lam$.
\end{proof}

Critical patterns were introduced in \Cref{d:cripatt1}. We are ready to
consider critical patterns of quadratically critical portraits in
dendritic geodesic laminations. This notion is closely related to that
of \textbf{critically marked (dendritic) polynomial}, which was
introduced in the Introduction as we discussed there the Theorem on
Local Charts for Dendritic Polynomials (in that we follow Milnor
\cite{miln93, M}). Recall that the space of all dendritic invariant
geodesic laminations with critical patterns is denoted by $\L\cpd_d$.

\begin{dfn}\label{d:psi}
To each marked dendritic polynomial $(P, C(P))$ of degree $d$ we
associate the corresponding dendritic invariant geodesic lamination
with  critical pattern $(\lam_{\sim_P}, \zc(P, C(P)))$ by defining
$\zc(P, C(P))=\zc$ as the ordered collection of convex hulls of
$\sim_P$-classes associated to critical points of $P$ in the order they
appear in $C(P)$; in the notation from the Introduction $C(P)=(c_1,
\dots, c_{d-1})$ and $\zc=(G_{c_1}, \dots, G_{c_{d-1}})$. Also, define
the map $\hPsi_d$ so that $\hPsi_d(P, C(P))=G_{c_1}\times \dots \times
G_{c_{d-1}}.$
\end{dfn}

Suppose that a sequence of regular invariant geodesic laminations with
critical patterns $(\lam^i, \zc^i)$ converges in the Hausdorff sense.
Then, by Theorem~\ref{t:sibliclos}, the limit $\lam^\infty$ of
$\si_d$-invariant geodesic laminations $\lam^i$ is itself a
$\si_d$-invariant geodesic lamination. Moreover, then critical patterns
$\zc^i$ converge to the limit collection of $d-1$ critical sets of
$\lam^\infty$. We are interested in the case when the $\si_d$-invariant
geodesic lamination $\lam^\infty$ is in a sense compatible with a
dendritic $\si_d$-invariant geodesic lamination.

\begin{lem}\label{l:uppers}
Suppose that a sequence of regular invariant geodesic laminations with
critical patterns $(\lam^i, \zc^i)$ converges in the sense of the
Hausdorff metric $H_H$ to an invariant  geodesic lamination
$\lam^\infty$ with a collection of limit critical sets $C_1, \dots,
C_{d-1}$ and there exists a dendritic invariant geodesic lamination
$\lam$ with a critical pattern $\zc=(Z_1, \dots, Z_{d-1})$ such that
$C_i\subset Z_i, 1\le i\le d-1$. Then $\lam^\infty\supset \lam$.
\end{lem}

\begin{proof}
By \Cref{l:non-vac}, for every $i$, we can choose a full collection
$F^i=(\oc^i_1, \dots, \oc^i_d)$ of critical chords that generates
$\zc^i$. By \Cref{l:noloop}, we may assume that these full collections
converge to a full collection $F=(\oc_1,$ $\dots,$ $\oc_d)$ as $i$
tends to infinity. Clearly, elements of $F$ are critical chords
compatible with $\lam^\infty$. On the other hand, by the assumptions
they are compatible with dendritic invariant geodesic lamination
$\lam$. Therefore by \Cref{t:fromkiwi} $\lam^\infty\supset \lam$ as
desired.
\end{proof}

For an integer $m>0$, we use a partial order by inclusion among
$m$-tuples: $(A_1, \dots, A_m)\succ (B_1, \dots, B_m)$ (or $(B_1,
\dots, B_m)\prec (A_1, \dots, A_m)$) if and only if $A_i\supset B_i$
for all $i=1$, $\dots$, $m$. Thus $m$-tuples and $k$-tuples with $m\ne
k$ are always incomparable. Lemma~\ref{l:uppers} says that if critical
patterns of regular invariant geodesic laminations converge
\textbf{into} a critical pattern of a dendritic invariant geodesic
lamination $\lam$, then the corresponding regular invariant geodesic
laminations themselves converge \textbf{over} $\lam$.

\begin{dfn}\label{d:uppers}\index{upper semicontinuous}
Let $F$ be a map from a topological space $A$ to the space $2^B$
%$\mc(B)$
of compact subsets of a compactum $B$. Then $F$ is said to be
\emph{upper semicontinuous}\index{upper semicontinuous!map} if $x_i\to
x$ in $A$ implies that the limit of every convergent subsequence
$y_{i_k}\in F(x_{i_k})$ belongs to $F(x)$. Equivalently, for any
neighborhood $\mathcal U$ of $F(x)$ there exists a neighborhood $V$ of
$x$ such that $F(y)\subset \mathcal U$ if $y\in V$.
\end{dfn}

The fact that $F$ is upper semicontinuous does not necessarily mean
that sets $F(x_i)$ must converge in the Hausdorff sense whenever
$x_i\to x$. However all existing Hausdorff limits of subsequences of
the sets $F(x_i)$ are contained in the set $F(x)$ as long as $x_i\to
x$.

Corollary~\ref{c:crista} easily follows from Lemmas~\ref{l:gm},
~\ref{l:condense} and ~\ref{l:uppers}.

\begin{cor}\label{c:crista}
Suppose that a sequence $(P_i, C(P_i))$ of critically marked dendritic
polynomials converges to a critically marked dendritic polynomial $(P,
C(P))$. Consider invariant geodesic laminations with critical patterns
$(\lam_{\sim_{P_i}}, \zc(P_i, C(P_i)))$ and $(\lam_{\sim_P}, \zc(P,
C(P)))$. If $(\lam_{\sim_{P_i}}, \zc(P_i, C(P_i)))$ converge in the
sense of the Hausdorff metric to $(\lam^\infty,\zc_\infty)$, then
$\lam^\infty\supset \lam_{\sim_P}$ and $\zc_\infty\prec \zc(P, C(P))$.
In particular, the map $\hPsi_d$ is upper semicontinuous.
\end{cor}

By Corollary~\ref{c:crista}, critical sets of dendritic invariant
geodesic laminations $\lam_{\sim_P}$ associated with polynomials $P\in
\dc_d$ cannot explode under perturbation of $P$ (they may implode
though). Provided that a geometric (visual) way to parameterize
$\L\cpd_d$ is given, the map $\hPsi_d$ may yield the corresponding
parameterization of the space of all dendritic critically marked
polynomials and gives an important application of our tools. This
justifies the introduction and studying dendritic geodesic laminations
with critical patterns, which are natural counterparts of critically
marked dendritic polynomials \cite{mil12}.

\section{Accordions of invariant geodesic laminations}\label{s:acclam}

In the Introduction, we mentioned that some of Thurston's tools from
\cite{thu85} fail in the cubic case. This motivates us to develop new
tools (so-called \emph{accordions}\index{accordion}), which basically
track linked leaves from different invariant geodesic laminations. In
this section, we study accordions in detail. In Sections~\ref{s:acclam}
- \ref{s:qcrit}, we assume that $\lam_1$, $\lam_2$ are
$\si_d$-invariant geodesic laminations, and $\ell_1$, $\ell_2$ are
leaves of $\lam_1$, $\lam_2$, respectively.

\subsection{Motivation}\label{ss:motiv}
For a quadratic invariant geodesic lamination $\lam$ and a leaf $\ell$
of $\lam$ that is not a diameter, let $\ell'$ be the sibling of $\ell$.
Denote by $C(\ell)$ the open strip of $\cdisk$ between $\ell$ and
$\ell'$ and by $L(\ell)$ the length of the shorter component of $\uc\sm
\ell$. Suppose that $\frac13\le L(\ell)$ and notice that by the
assumptions we always have $L(\ell)<\frac12$. Denote by $k$ the
smallest number such that $\si_2^k(\ell)\subset C(\ell)$ except perhaps
for the endpoints (depending on the dynamics of $\ell$, the number $k$
is not necessarily defined, so the forthcoming conclusions should be
understood conditionally). The Central Strip Lemma (Lemma II.5.1 of
\cite{thu85}) claims that provided the number $k$ \textbf{is} defined,
we have $\si_2^k(\ell)$ separates $\ell$ and $\ell'$. In particular, if
$\ell=M$ is a \emph{major}\index{major}, that is, a longest leaf of
some quadratic invariant geodesic lamination, then an eventual image of
$M$ cannot enter $C(M)$.

Let us list Thurston's results for which the Central Strip Lemma is
crucial. A \emph{$\si_2$-wandering
triangle}\index{triangle!($\si_2$-)wandering} is a triangle with
vertices $a$, $b$, $c$ on $\uc$ such that the convex hull $T_n$ of
$\si_2^n(a)$, $\si_2^n(b)$, $\si_2^n(c)$ is a non-degenerate triangle
for every $n=0$, $1$, $\dots$, and all these triangles are pairwise
disjoint.

\begin{thm}[No Wandering Triangle Theorem \cite{thu85}]\label{t:nwt}
Wan\-de\-ring trian\-gles for $\si_2$ do not exist.
\end{thm}

Theorem~\ref{t:gaptrans} stated below follows from the Central Strip
Lemma and is due to Thurston \cite{thu85} for $d=2$. For arbitrary $d$,
it is due to Jan Kiwi \cite{kiw02} who used different tools. Observe
that by definition a gap must have at least three vertices (this
trivial observation is important for the last claim of the theorem
dealing with the quadratic case).

\begin{thm}[\cite{thu85, kiw02}]\label{t:gaptrans}
If $A$ is a finite $\si_d$-periodic gap of period $k$, then either $A$
is a $d$-gon, and $\si^k_d$ fixes all vertices of $A$, or there are at
most $d-1$ orbits of vertices of $A$ under $\si_d^k$. Thus, for $d=2$,
the remap is transitive on the vertices of any finite periodic gap.
\end{thm}

\begin{figure}
  \includegraphics[width=10cm]{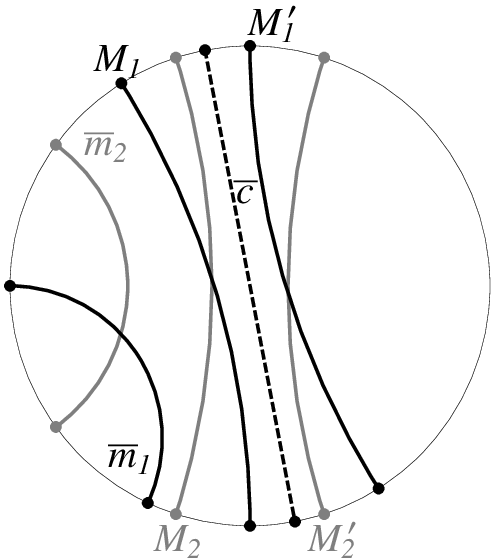}
  \caption{This figure illustrates Thurston's proof that quadratic minors are unlinked.
  The Central Strip Lemma forces orbits of both minors to not cross $\ovc$.}
  \label{f:central-strip}
\end{figure}

Another crucial result of Thurston is that minors of distinct quadratic
invariant geodesic laminations are disjoint in $\disk$. A sketch of the
argument follows. Let $\m_1$ and $\m_2$ be the minors of two invariant
geodesic laminations $\lam_1\ne \lam_2$ that cross in $\disk$. Let
$M_1$, $M'_1$ and $M_2$, $M'_2$ be the two pairs of corresponding
majors. We may assume that $M_1$, $M_2$ cross in $\disk$ and $M'_1$,
$M'_2$ cross in $\disk$, but $(M_1\cup M_2)\cap (M'_1\cup M'_2)=\0$
(see Figure \ref{f:central-strip}) so that there is a diameter $\ovc$
with strictly preperiodic endpoints separating $M_1\cup M_2$ from
$M_1'\cup M_2'$. Thurston shows that there is a unique invariant
geodesic lamination $\lam$, with only finite gaps, whose major is
$\ovc$. By the Central Strip Lemma, forward images of $\m_1$, $\m_2$ do
not intersect $\ovc$. Hence $\m_1\cup \m_2$ is contained in a finite
gap $G$ of $\lam$. By the No Wandering Triangle Theorem, $G$ is
eventually periodic. By Theorem~\ref{t:gaptrans}, some images of $\m_1$
intersect inside $\disk$, a contradiction.

Examples indicate that statements analogous to the Central Strip Lemma
fail in the cubic case. Indeed, Figure~\ref{f:ex-intro} shows a leaf
$M=\ol{\frac{342}{728}\frac{579}{728}}$ of period $6$ under $\si_3$ and
its $\si_3$-orbit together with the leaf $M'$ (which has the same image
as $M$ forming together with $M$ a narrower ``critical strip'' $S_n$)
and the leaf $N'$ (which has the same image as $N=(\si_3)^4(M)$ forming
together with $N$ a wider ``critical strip'' $S_w$). Observe that
$\si_3(M)\subset S_w$, which shows that the Central Strip Lemma does
not hold in the cubic case (orbits of periodic leaves may give rise to
``critical strips'' containing some elements of these orbits of
leaves). This apparently makes a direct extension of the arguments from
the previous paragraph impossible leaving the issue of whether and how
minors of cubic invariant geodesic laminations can be linked
unresolved.

\begin{figure}
\includegraphics[width=5.2cm]{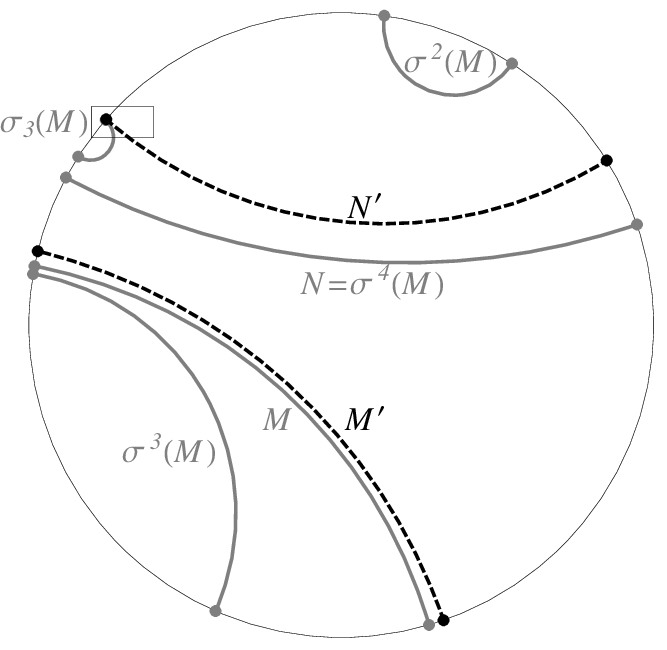}
\hspace{2cm}
\includegraphics[width=5.2cm]{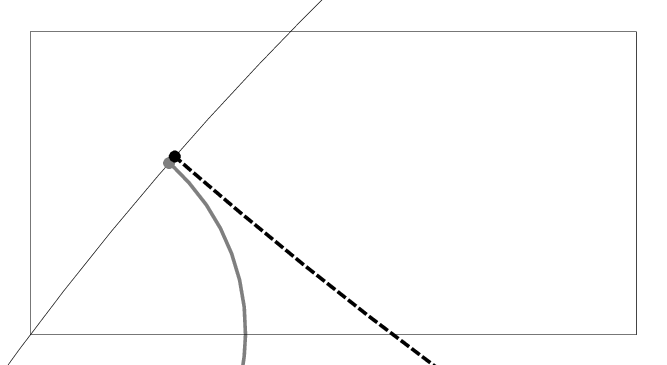}
{\caption{This figure shows that the Central Strip Lemma fails in the
cubic case. Its left part has a fragment in which two endpoints of
leaves are located very close to each other. Its right part is
the zoomed-in version of the fragment indicating that the periodic
points do not coincide.}\label{f:ex-intro}}
\end{figure}

Another consequence of the failure of the Central Strip Lemma in the
cubic case is the failure of the No Wandering Triangle Theorem (a
counterexample was given in \cite{bo08}; in fact, it was shown in
\cite{bco12, bco13} that there exists a large set of dendritic
invariant geodesic laminations with wandering triangles). Properties of
wandering polygons were studied in \cite{kiw02, bl02, chi07}.

\subsection{Properties of accordions}\label{ss:propacc}
We now give a definition of accordions.

\begin{dfn}\label{d:accord}
Let $A_{\lam_2}(\ell_1)$ be the collection of leaves of $\lam_2$ linked
with $\ell_1$, together with $\ell_1$. Let $A_{\ell_2}(\ell_1)$ be the
collection of leaves from the forward orbit of $\ell_2$ that are linked
with $\ell_1$, together with $\ell_1$. The sets defined above are
called \emph{accordions $($of $\ell_1)$}\index{accordion!of a leaf}
while $\ell_1$ is called the \emph{axis}\index{axis of an accordion}
(of the appropriate accordion). Sometimes we will also use
$A_{\lam_2}(\ell_1)$ and $A_{\ell_2}(\ell_1)$ to mean the union of the
leaves constituting these accordions.
\end{dfn}

In general, accordions do not behave nicely under $\si_d$ as linked
leaves may have unlinked images. To avoid these problems, for the rest
of this section, we will impose the following conditions on accordions.

\begin{dfn}\label{d:opacc}
A leaf $\ell_1$ is said to \emph{have order preserving accordions with
respect to $\lam_2$ $($respectively, to a leaf
$\ell_2)$}\index{accordions!order preserving} if $A_{\lam_2}(\ell_1)\ne
\{\ell_1\}$ (respectively, $A_{\ell_2}(\ell_1)\ne \{\ell_1\}$), and,
for each $k\ge 0$, the map $\si_d$ restricted to
$A_{\lam_2}(\si_d^k(\ell_1))\cap\uc$ (respectively, to
$A_{\ell_2}(\si_d^k(\ell_1))\cap\uc$) is order preserving (in
particular, it is one-to-one). Say that $\ell_1$ and $\ell_2$ have
\emph{mutually order preserving accordions}\index{accordions!mutually
order preserving} if $\ell_1$ has order preserving accordions with
respect to $\ell_2$, and vice versa (in particular, $\ell_1$ and
$\ell_2$ are not precritical).
\end{dfn}

Though fairly strong, these conditions naturally arise in the study of
linked or essentially equal invariant geodesic laminations. In
Section~\ref{s:qcrit}, we show that they are often satisfied by pairs
of linked leaves of linked or essentially equal invariant geodesic
laminations (Lemma~\ref{l:indepcrit}) so that there are at most
countably many pairs of linked leaves that do not have mutually order
preserving accordions. If invariant geodesic laminations are perfect,
this will imply that every accordion consisting of more than one leaf
contains a pair of leaves with mutually order preserving accordions.
Understanding the rigid dynamics of such pairs is crucial to our main
results.

\begin{prop}\label{p:accimage}
If $\si_d$ is order preserving on an accordion $A$ with axis $\ell_1$
and $\ell\in A$, $\ell\ne\ell_1$, then $\si_d(\ell)$ and
$\si_d(\ell_1)$ are linked. In particular, if $\ell_1$ has order
preserving accordions with respect to $\ell_2$ then $\si_d^k(\ell)\in
A_{\ell_2}(\si_d^k(\ell_1))$ for every $\ell\in A_{\ell_2}(\ell_1)$,
$\ell\ne\ell_1$, and every $k\ge 0$.
\end{prop}

\begin{proof} The proof of Proposition~\ref{p:accimage} immediately follows from
the definitions and is left to the reader.
\end{proof}

We now explore more closely the orbits of leaves from
Definition~\ref{d:opacc}.

\begin{prop}\label{p:sameori}
Suppose that $\ell_1$ and $\ell_2$ are linked, $\ell_1$ has order
preserving accordions with respect to $\ell_2$, and $\si_d^k(\ell_2)\in
A_{\ell_2}(\ell_1)$ for some $k>0$. In this case, if $\ell_2=\ol{xy}$,
then either $\ell_1$ separates $x$ from $\si_d^k(x)$ and $y$ from
$\si_d^k(y)$, or $\ell_2$ has $\si_d^k$-fixed endpoints.
\end{prop}

\begin{proof}
Suppose that $\ell_2$ is not $\si_d^k$-fixed. Denote by $x_0=x,y_0=y$
the endpoints of $\ell_2$; set $x_i=\si_d^{ik}(x_0)$,
$y_i=\si_d^{ik}(y_0)$ and $A_t=A_{\ell_2}(\si_d^t(\ell_1))$, where
$t=0$, $1$, $\ldots$. If $\ell_1$ does not separate $x_0$ and $x_1$,
then either $x_0\le x_1<y_1\le y_0<x_0$ or $x_0<y_0\le y_1<x_1\le x_0$.
We may assume the latter (cf. Figure \ref{f:sameori}).

Since $\si_d^k$ is order preserving on $A_0\cap\uc$, then $x_0<y_0\le
y_1\le y_2<x_2 \le x_1 \le x_0$ while the leaves $\ol{x_1y_1}$ and
$\ol{x_2y_2}$ belong to the accordion $A_k$ so that the above
inequalities can be iterated. Inductively we see that

$$x_0< y_0 \le \dots \le y_{m-1}\le y_m<x_m\le x_{m-1}\le \dots \le x_0.$$

\noindent All leaves $\ol{x_i y_i}$ are pairwise distinct as otherwise
there exists $n$ such that $\ol{x_{n-1} y_{n-1}}\ne \ol{x_n
y_n}=\ol{x_{n+1} y_{n+1}}$ contradicting $\si_d^k$ being order
preserving on $A_{k(n-1)}$. Hence the leaves $\ol{x_i y_i}$ converge to
a $\si_d^k$-fixed point or leaf, contradicting the expansion property
of $\si_d^k$.
\end{proof}

\begin{figure}[H]
\includegraphics[width=7cm]{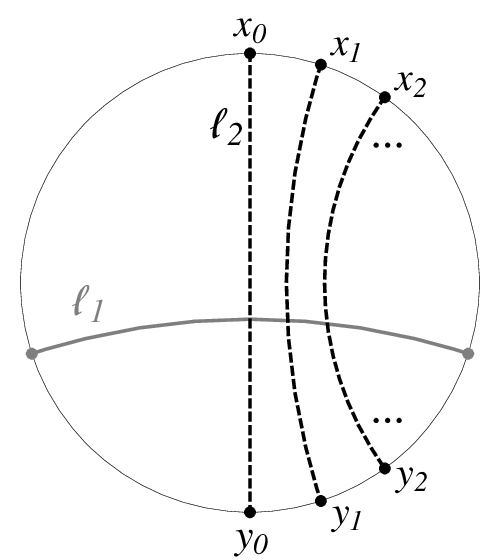}
{\caption{This figure illustrates Proposition~\ref{p:sameori}.
Although in the figure $\ol{x_2y_2}$ is linked with $\ell_1$, the
argument does not assume this. In this and forthcoming figures,
leaves marked in the same fashion belong to the same grand orbits of
leaves. }\label{f:sameori}}
\end{figure}

In what follows, we often use one of the endpoints of a leaf as the
subscript in the notation for this leaf.

\begin{lem}\label{l:sameperiod}
If $\ell_a=\ol{ab}$ and $\ell_x=\ol{xy}$, where $a<x<b<y$, are linked
leaves with mutually order preserving accordions, and $a$, $b$ are of
period $k$, then $x$, $y$ are also of period $k$.
\end{lem}

\begin{proof}
By the order preservation, $\si_d^k(x)$ is not separated from $x$ by
$\ell_a$. It follows from Proposition~\ref{p:sameori} that
$x=\si^k_d(x)$, $y=\si^k_d(y)$. Since, by Lemma~\ref{l:sameperiod1},
the points $x$ and $y$ have the same period (say, $m$), then $m$
divides $k$. Similarly, $k$ divides $m$. Hence $k=m$.
\end{proof}

We will mostly use the following corollary of the above results.

\begin{cor}\label{c:nothreelink}
Suppose that $\ell_a=\ol{ab}$ and $\ell_x=\ol{xy}$ with $x<a<y<b$ are
linked leaves. If $\ell_a$ and $\ell_x$ have mutually order preserving
accordions, then there are the following possibilities for
$A=A_{\ell_x}(\ell_a)$.

\begin{enumerate}

\item $A=\{\ell_a, \ell_x\}$ and no forward image of $\ell_x$ crosses
    $\ell_a$.

\item $A=\{\ell_a, \ell_x\}$, the points $a$, $b$, $x$, $y$ are of
    period $2j$ for some $j$, $\si_d^j(x)=y, \si_d^j(y)=x$, and
    either $\si^j_d(a)=b$, $\si^j_d(b)=a$, or $\si_d^j(\ell_a)\ne
    \ell_a$, and $\ell_x$ separates the points $a$, $\si_d^j(b)$ from
    the points $b$, $\si_d^j(a)$.

\item $A=\{\ell_a, \ell_x\}$, the points $a$, $b$, $x$, $y$ are of
    the same period, $x$, $y$ have distinct orbits, and $a$, $b$ have
    distinct orbits.

\item There exists $i>0$ such that $A=\{\ell_a, \ell_x,
    \si^i_d(\ell_x)\}$ and either $x<a<y\le
    \si_d^i(x)<b<\si_d^i(y)\le x$ or $x\le \si^i_d(y)<a<\si^i_d(x)\le
    y<b$, as shown in Figure \ref{fig:3link}.
\end{enumerate}

\end{cor}

\begin{proof}
Three \textbf{distinct} images of $\ell_x$ cannot cross $\ell_a$ as if
they do, then it is impossible for the separation required in
Proposition~\ref{p:sameori} to occur for all of the pairs of images of
$\ell_x$. Hence at most two images of $\ell_x$ cross $\ell_a$.

If two distinct leaves from the orbit of $\ell_x$ cross $\ell_a$, then,
by Proposition~\ref{p:sameori} and the order preservation, case (4)
holds. Thus we can assume that $A=\{\ell_a, \ell_x\}$. If no forward
image of $\ell_x$ is linked with $\ell_a$, then we have case (1).

In all remaining cases we have $\si_d^k(\ell_x)=\ell_x$ for some $k>0$.
By Lemma~\ref{l:sameperiod1}, points $x$ and $y$ are of the same
period. Suppose that $x$, $y$ belong to the same periodic orbit. Choose
the least $j$ such that $\si^j_d(x)=y$.

Let us show that then $\si^j_d(y)=x$. Indeed, assume that
$\si^j_d(y)\ne x$. Since by the assumption the only leaf from the
forward orbit of $\ell_x$, linked with $\ell_a$, is $\ell_x$, we may
assume (for the sake of definiteness) that $y<\si^j_d(y)\le b$. Then a
finite concatenation of further $\si^j_d$-images of $\ell_x$ will
connect $y$ with $x$. Again, since $A=\{\ell_a, \ell_x\}$, one of their
endpoints will coincide with $b$. Thus, $y<\si^j_d(y)\le
b<\si_d^j(b)\le x$, see Figure \ref{f:nothreelinka2i}. Let us now apply
$\si_d^j$ to $A$; by the order preservation $y<\si^j_d(a)<\si^j_d(y)\le
b<\si^j_d(b)\le x<a$. Hence, $\si^j_d(\ell_a)$ is linked with $\ell_a$,
a contradiction.

\begin{figure}
\includegraphics[width=10cm]{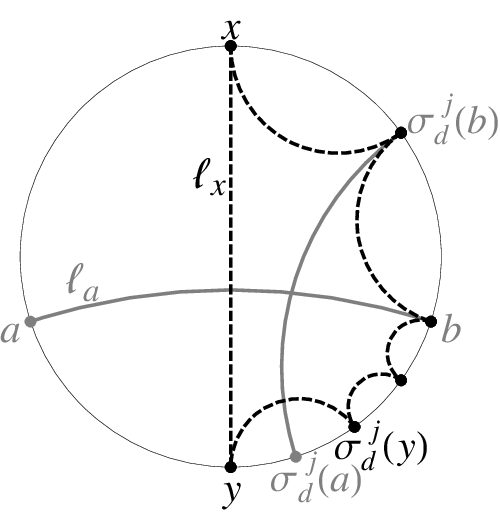}
{\caption{This figure illustrates the proof of
Corollary~\ref{c:nothreelink}.}\label{f:nothreelinka2i}}
\end{figure}

Thus, $\si^j_d(y)=x$ (that is, $\si_d^j$ flips $\ell_x$ onto itself),
$k=j$, the points $x$ and $y$ are of period $2j$ and, by
Lemma~\ref{l:sameperiod}, the points $a$ and $b$ are also of period
$2j$. If $\si_d^j(a)=b$, then $\si_d^j(b)=a$, and if $\si_d^j(b)=a$,
then $\si_d^j(a)=b$ (since both points have period $2j$). Now, if
$\si_d^j(a)\ne b$ and $\si_d^j(b)\ne a$, then, by the order
preservation, $\ell_x$ separates the points $a$, $\si_d^j(b)$ from the
points $b$, $\si_d^j(a)$. So, case (2) holds.

Assume that $x$ and $y$ belong to distinct periodic orbits of period
$k$. By Lemma~\ref{l:sameperiod}, the points $a$, $b$ are of period
$k$. Let points $a$ and $b$ have the same orbit. Then, if $k=2i$ and
$\si_d^i$ flips $\ell_a$ onto itself, it would follow from the order
preservation that $\si_d^i(\ell_x)$ is linked with $\ell_a$. Since
$\ell_x$ is the unique leaf from the orbit of $\ell_x$ linked with
$\ell_a$ this would imply that $\si_d^i$ flips $\ell_x$ onto itself, a
contradiction with $x, y$ having disjoint orbits. Hence we may assume
that, for some $j$ and $m>2$, we have that $\si_d^j(a)=b$, $jm=k$, and
a concatenation of leaves $\ell_a$, $\si_d^j(\ell_a)$, $\dots$,
$\si_d^{j(m-1)}(\ell_a)$ forms a polygon $P$.

If one of these leaves distinct from $\ell_a$ (say,
$\si_d^{js}(\ell_a)$) is linked with $\ell_x$, we can apply the map
$\si_d^{j(m-s)}$ to $\si_d^{js}(\ell_a)$ and $\ell_x$; by order
preservation we will see then that $\ell_a$ and
$\si_d^{j(m-s)}(\ell_x)\ne \ell_x$ are linked, a contradiction with the
assumption that $A=\{\ell_a, \ell_x\}$. If none of the leaves
$\si_d^j(\ell_a)$, $\dots$, $\si_d^{j(m-1)}(\ell_a)$ is linked with
$\ell_x$, then $P$ has an endpoint of $\ell_x$ as one of its vertices.
As in the argument given above, we can then apply $\si_d^j$ to $A$ and
observe that, by the order preservation, the $\si_d^j$-image of
$\ell_x$ is forced to be linked with $\ell_x$, a contradiction. Hence
$a$ and $b$ have disjoint orbits, and case (3) holds.
\end{proof}

\begin{figure}
\includegraphics[width=5.2cm]{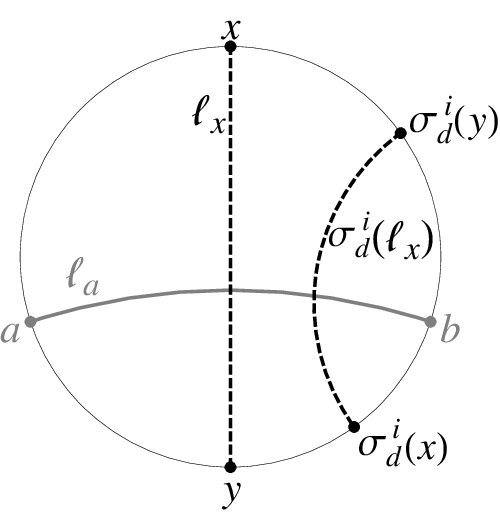}
\hspace{2cm}
\includegraphics[width=5.2cm]{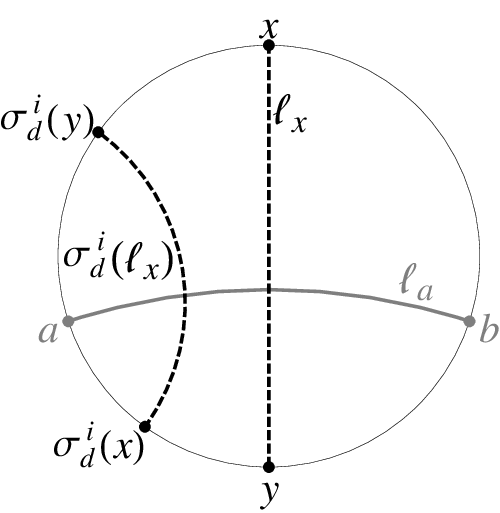}
{\caption{This figure shows two cases listed in
Corollary~\ref{c:nothreelink}, part (4).}\label{fig:3link}}
\end{figure}

\subsection{Accordions are (pre-)periodic or wandering}\label{ss:accwape}
Here we prove Theorem~\ref{t:compgap}, which is the main result of
Section~\ref{s:acclam}.

\begin{dfn}\label{d:positor}
A finite sequence of points $x_0, \dots, x_{k-1}\in \uc$ is
\emph{positively ordered}\index{sequence of points!positively ordered}
if $x_0<x_1<\dots<x_{k-1}<x_0$. If the inequality is reversed, then we
say that points $x_0$, $\dots$, $x_{k-1}\in \uc$ are \emph{negatively
ordered}\index{sequence of points!negatively ordered}. A sequence
$y_0$, $y_1$, $\dots$ is said to be \emph{positively circularly
ordered}\index{sequence of points!positively circularly ordered} if it
is either positively ordered or there exists $k$ such that $y_i=y_{i
\mod k}$ and $y_0<y_1<\dots<y_{k-1}<y_0$. Similarly we define sequences
that are \emph{negatively circularly ordered}\index{sequence of
points!negatively circularly ordered}.
\end{dfn}

A positively (negatively) \textbf{circularly} ordered sequence that is
not positively (negatively) ordered is a sequence, whose points repeat
themselves after the initial collection of points that are positively
(negatively) ordered.

\begin{dfn}\label{d:leavineq}
Suppose that the chords $\bt_1,$ $\dots,$ $\bt_n$ are edges of the
closure $Q$ of a single component of $\disk\sm \bigcup \bt_i$. For each
$i$, let $m_i$ be the midpoint of the hole $H_Q(\bt_i)$. We write
$\bt_1<\bt_2<\dots<\bt_n$ if the points $m_i$ form a positively ordered
set and call the chords $\bt_1$, $\dots$, $\bt_n$ \emph{positively
ordered}\index{chords!positively ordered}. If the points $m_i$ are
positively circularly ordered, then we say that $\bt_1,$ $\dots,$
$\bt_n$ are \emph{positively circularly
ordered}\index{chords!positively circularly ordered}. \emph{Negatively
ordered} and \emph{negatively circularly ordered} chords are defined
similarly.\index{chords!negatively ordered} \index{chords!negatively
circularly ordered}
\end{dfn}

Lemma~\ref{l:linkstruct} is used in the main result of this section.

\begin{lem}\label{l:linkstruct}
If $\ell_a$ and $\ell_x$ are linked, have mutually order preserving
accordions, and $\si_d^k(\ell_x)\in A_{\ell_x}(\ell_a)$ for some $k>0$,
then, for every $j>0$, the leaves $\si_d^{ki}(\ell_x)$, $i=0$, $\dots$,
$j$, are circularly ordered, and $\ell_a$, $\ell_x$ are periodic with
endpoints of the same period.
\end{lem}

\begin{figure}
\includegraphics[width=10cm]{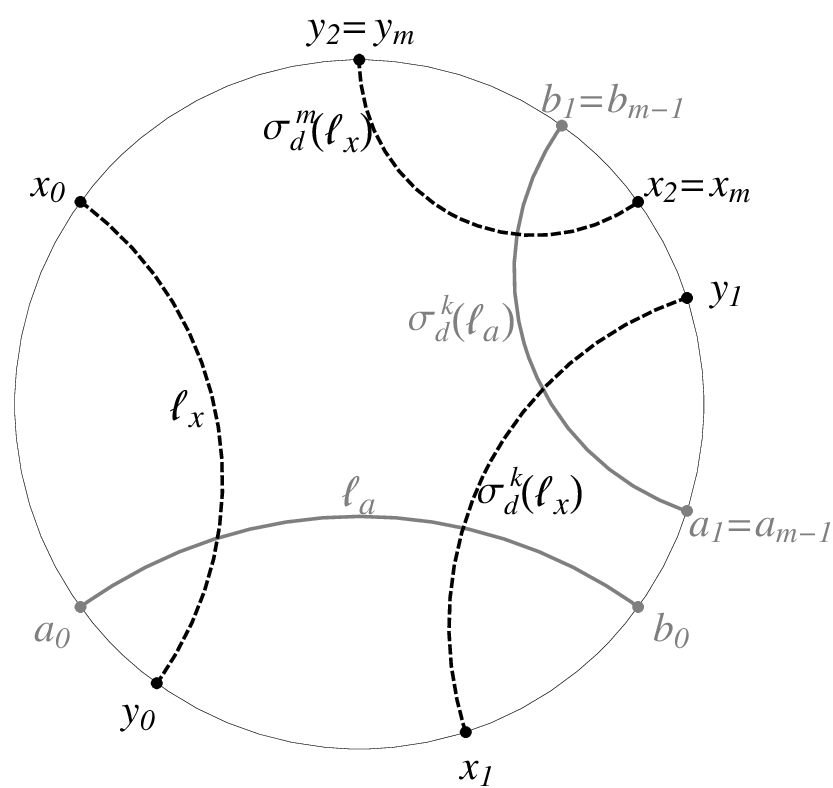}
{\caption{This figure illustrates Lemma~\ref{l:linkstruct}. Images
of $\ell_a$ cannot cross other images of $\ell_a$, neither can they
cross images of $\ell_x$ that are already linked with two images of
$\ell_a$ (by Corollary~\ref{c:nothreelink}). Similar claims hold for
$\ell_x$.}\label{f:linkstruct}}
\end{figure}

\begin{proof}
By Lemma~\ref{l:sameperiod}, we may assume that case (4) of
Corollary~\ref{c:nothreelink} holds (and so $\si_d^k(\ell_x)\ne
\ell_x$). Set $B=\{\ell_a, \ell_x\}$, $\ell_a=\ol{ab}, \ell_x=\ol{xy}$
and let $a_i$, $b_i$, $x_i$, $y_i$ denote the $\si_d^{ik}$-images of
$a$, $b$, $x$, $y$, respectively ($i\ge 0$). We may assume that the
first possibility from case (4) holds and $x_0<a_0<y_0\le
x_1<b_0<y_1\le x_0$ (see the left part of Figure \ref{fig:3link} and
Figure \ref{f:linkstruct}). By the assumption of mutually order
preserving accordions applied to $B$, we have $x_i<a_i<y_i\le
x_{i+1}<b_i<y_{i+1}\le x_i$ ($i\ge 0$), in particular $x_1<a_1<y_1$.

There are two cases depending on the location of $a_1$. Consider one of
them as the other one can be considered similarly. Namely, assume that
$b_0<a_1<y_1$ and proceed by induction for $m$ steps observing that

$$x_0<a_0<y_0\le x_1<b_0\le a_1<\dots\le x_m<b_{m-1}<a_m<y_m\le x_0.$$

\noindent Thus, the first $m$ iterated $\si_d^k$-images of $\ell_x$ are
circularly ordered and alternately linked with the first $m-1$ iterated
images of $\ell_a$ under $\si_d^k$ (see Figure \ref{f:linkstruct}). In
the rest of the proof, we exploit the following fact.

\smallskip

\noindent\textbf{Claim A.} \emph{Further images of $\ell_a$ or $\ell_x$
distinct from the already existing ones cannot cross the leaves
$\ell_a, \si_d^k(\ell_x), \dots, \si_d^{k(m-1)}(\ell_a),
\si_d^{km}(\ell_x)$ because either it would mean that leaves from the
same invariant geodesic lamination are linked, or it would contradict
Corollary~\ref{c:nothreelink}.}

\smallskip

By Claim A, we have $b_m\in (y_m, a_0]$. Consider possible locations of
$b_m$.

(1) If $x_0<b_m\le a_0$, then $\ol{a_mb_m}$ is linked with
$\ol{x_my_m},$ $\ol{x_{m+1}y_{m+1}}$ and $\ol{x_0y_0}$, which, by
Corollary~\ref{c:nothreelink}, implies that
$\ol{x_{m+1}y_{m+1}}=\ol{x_0y_0}$, and we are done (observe that, in
this case, by Lemma~\ref{l:sameperiod}, points $a_0, b_0$ are periodic
of the same period as $x_0, y_0$).

(2) The case $x_0=b_m$ is impossible because if $x_0=b_m$, then, by the
order preservation and by Claim A, the leaf
$\ol{x_{m+1}y_{m+1}}=\si_d^{k(m+1)}(\ell_x)$ is forced to be linked
with $\ell_a$, a contradiction.

(3) Otherwise we have $y_m<b_m<x_0$ and hence, by the order
preservation, $y_m\le x_{m+1}<b_m$. Then, by Claim A and because images
of $\ell_x$ do not cross, $b_m<y_{m+1}\le x_0$. Suppose that
$y_{m+1}=x_0$ while $y_0\ne x_1$. Applying $\si_d^k$ to leaves
$\ol{x_{m+1}x_0}$ and $\ol{x_0y_0}$ and using Claim A we see that
$y_0\le x_{m+2}<x_1$. However, the order preservation then implies that
$\ol{a_{m+1}b_{m+1}}$ crosses both $\ol{x_{m+1}x_0}$ and
$\ol{x_{m+2}x_1}$ and therefore crosses $\ell_a$ itself, a
contradiction. Hence the situation when $y_{m+1}$ coincides with $x_0$
can only happen if $y_0=x_1$. It follows that then
$\si^k_d(\ol{x_{m+1}y_{m+1}})=\ol{x_0y_0}$, and we are done (as before,
we need to rely on Lemma~\ref{l:sameperiod} here).

Otherwise $b_m<y_{m+1}<x_0$ and the arguments can be repeated as leaves
$\si_d^{ki}(\ell_x), i=0, \dots, m+1$ are circularly ordered. Thus,
either $\ell_x$ is periodic, $\ol{x_ny_n}=\ol{x_0y_0}$ for some $n$,
and all leaves in the $\si^k_d$-orbit of $\ell_x$ are circularly
ordered, or the leaves $\ol{x_iy_i}$ converge monotonically to a point
of $\uc$. The latter is impossible since $\si_d^k$ is expanding.  By
Lemma~\ref{l:sameperiod}, the leaf $\ell_a$ is periodic and its
endpoints have the same period as the endpoints of $\ell_x$.
\end{proof}

Theorem~\ref{t:compgap} is the main result of this section.

\begin{thm}\label{t:compgap}
Consider linked chords $\ell_a=\ol{ab}$, $\ell_x=\ol{xy}$ with mutually
order preserving accordions, and set $B=\ch(\ell_a,\ell_x)$. Suppose
that not all forward images of $B$ have pairwise disjoint interiors.
Then there exists a finite periodic stand alone gap $Q$ such that all
vertices of $Q$ are in the forward orbit of $\si^r_d(B)$ for some
minimal $r$, they belong to two, three, or four distinct periodic
orbits of the same period, and the remap of $Q\cap\uc$ is not the
identity unless $Q=\si^r_d(B)$ is a quadrilateral.
\end{thm}

\begin{proof}
We may assume that there are two forward images of $B$ with
non-disjoint interiors. Choose the least $r$ such that the interior of
$\si^r_d(B)$ intersects some forward images of $B$. We may assume that
$r=0$ and, for some (minimal) $k>0$, the interior of the set
$\si^k_d(B)$ intersects the interior of $B$ so that $\si_d^k(\ell_x)\in
A_{\ell_x}(\ell_a)$. We write $x_i$, $y_i$ for the endpoints of
$\si_d^{ik}(\ell_x)$, and $a_i$, $b_i$ for the endpoints of
$\si_d^{ik}(\ell_a)$.

By Lemma~\ref{l:linkstruct} applied to both leaves, by the assumption
of mutually order preserving accordions, and because leaves in the
forward orbits of $\ell_a, \ell_x$ are pairwise unlinked, we may assume
without loss of generality that, for some $m\ge 1$,
$$x_0<a_0<y_0\le x_1<b_0\le a_1< \dots \le x_m<b_{m-1}\le a_m<y_m<b_m$$
and $x_m=x_0$, $y_m=y_0$, $a_m=a_0$, $b_m=b_0$, that is, we have the
situation shown in Figure \ref{f:linkstruct}. Thus, for every $i=0$,
$\dots$, $k-1$, there is a loop $L_i$ of alternately linked
$\si_d^k$-images of $\si_d^i(\ell_a)$ and $\si_d^i(\ell_x)$. If the
$\si_d^k$-images of $\si_d^i(\ell_a)$ are concatenated to each other,
then their endpoints belong to the same periodic orbit, otherwise they
belong to two distinct periodic orbits.

A similar claim holds for $\si_d^k$-images of $\si_d^i(\ell_x)$. Thus,
the endpoints of $B$ belong to two, three or four distinct periodic
orbits of the same  period (the latter follows by
Corollary~\ref{c:nothreelink} and Lemma~\ref{l:linkstruct}). Set
$\ch(L_i)=T_i$ and consider some cases.

\begin{figure}
\includegraphics[width=10cm]{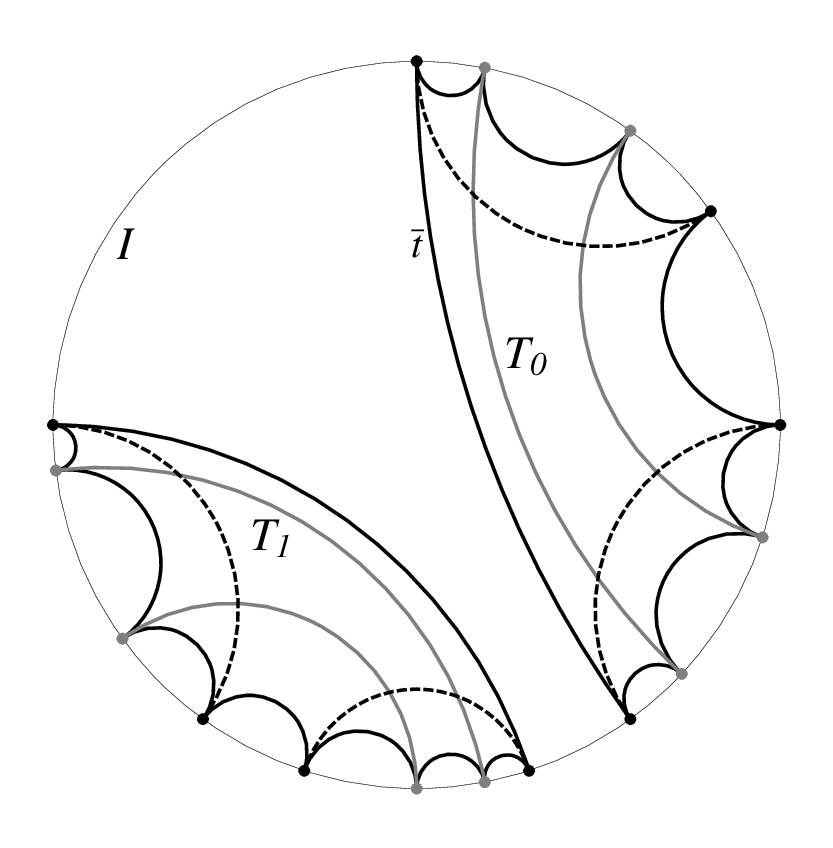}
{\caption{This figure illustrates the proof of
Theorem~\ref{t:compgap} in the case $m>1$.}\label{f:compgap1}}
\end{figure}

(1) Let $m>1$ (this includes the ``flipping'' case from part (2) of
Corollary~\ref{c:nothreelink}). Let us show that the sets $T_i$ either
coincide or are disjoint. Every image $\hell$ of $\ell_a$ in $L_i$
crosses two images of $\ell_x$ in $L_i$ (if $m=2$ and $\ell_x$ is
``flipped'' by $\si^k_d$, we still consider $\ell_x$ and
$\si^k_d(\ell_x)$ as distinct leaves). By
Corollary~\ref{c:nothreelink}, no other image of $\ell_x$ crosses
$\hell$.

Suppose that interiors of $T_i$ and $T_j$ intersect. Let $\bt$ be an
edge of $T_i$ and $I=H_{T_i}(\bt)$ be the corresponding hole of $T_i$.
Then the union of two or three images of $\ell_a$ or $\ell_x$ from
$L_i$ \textbf{separates} $I$ from $\uc\sm I$ in $\cdisk$ (meaning that
any curve connecting $I$ with $\uc\sm I$ must intersect the union of
these two or three images of $\ell_a$ or $\ell_x$, see Figure
\ref{f:compgap1}). Hence if there are vertices of $T_j$ in $I$
\textbf{and} in $\uc\sm I$ then there is a leaf of $L_j$ crossing
leaves of $L_i$, a contradiction with the above and
Corollary~\ref{c:nothreelink}.

Thus, the only way $T_i\ne T_j$ can intersect is if they share a vertex
or an edge. We claim that this is impossible. Indeed, $T_i\ne T_j$
cannot share a vertex as otherwise this vertex must be
$\si_d^k$-invariant while all vertices of any $T_r$ map to other
vertices (sets $T_r$ ``rotate'' under $\si_d^k$). Finally, if $T_i$ and
$T_j$ share an edge $\ell$ then the same argument shows that $\si_d^k$
cannot fix the endpoints of $\ell$, hence it ``flips'' under $\si_d^k$.
However this is impossible as each set $T_r$ has at least four vertices
and its edges ``rotate'' under $\si_d^k$.

So, the component $Q_i$ of $X=\bigcup^{k-1}_{i=0} T_i$ containing
$\si_d^i(\ell_a)$ is $T_i$. By Lemma~\ref{l:linkstruct}, the map
$\si_d|_{T_i\cap\uc}$ is order preserving or reversing. As $\si_d$
preserves order on any single accordion, $\si_d|_{T_i\cap\uc}$ is order
preserving. The result now follows; note that the first return map on
$Q$ is not the identity map.

(2) Let $m=1$. This corresponds to part (3) of
Corollary~\ref{c:nothreelink}: both $\ell_a$ and $\ell_x$ have
endpoints of minimal period $k$, and the orbit of $\ell_a$ ($\ell_x$)
consists of $k$ pairwise disjoint leaves. Note that $T_0$ is a
quadrilateral, and the first return map on $T_0$ is the identity map.
Consider the case when not all sets $T_i$ are pairwise disjoint. Note
that, by the above, $T_0$ is a periodic stand alone gap satisfying the
assumptions of Proposition \ref{p:forconcat}. It follows that every
component of the union of $T_i$ is a concatenation of gaps sharing
edges with the same polygon. See Figure \ref{f:compgap2}, in which the
polygon is a triangle.

\begin{figure}
\includegraphics[width=10cm]{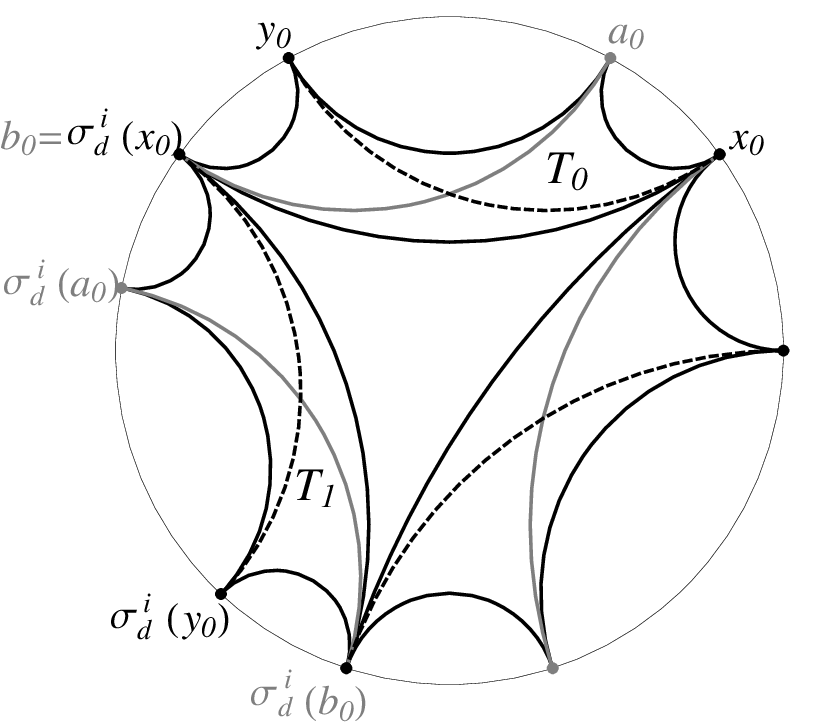}
{\caption{This figure illustrates the proof of Theorem~\ref{t:compgap}
in the case $m=1$.}\label{f:compgap2}}
\end{figure}
\end{proof}

For a leaf $\ell_1\in \lam_1$, let $\B_{\lam_2}(\ell_1)$ be the
collection of all leaves $\ell_2\in \lam_2$ that are linked with
$\ell_1$ and have mutually order preserving accordions with $\ell_1$.
Observe that if $\ell_1$ is (pre)critical, then
$\B_{\lam_2}(\ell_1)=\0$ by Definition~\ref{d:opacc}. Similarly, no
leaf from $\B_{\lam_2}(\ell_1)$ is (pre)critical.

\begin{cor}\label{c:finaccord}
The collection $\B_{\lam_2}(\ell_1)$ is finite.
\end{cor}

\begin{proof}
Suppose first that $\ell_1$ is not (pre)periodic. Let us show that the
convex hull $B$ of $\ell_1$ and leaves $\n_1, \dots, \n_s$ from
$\B_{\lam_2}(\ell_1)$ is wandering. By Theorem~\ref{t:compgap}, for
each $i$, the set $B_i=\ch(\ell_1, \n_i)$ is wandering (because
$\ell_1$ is not (pre)periodic). This implies that if $i\ne j$ then
$\si^i_d(\ell_1)$ and $\si_d^j(\n_t)$ are disjoint (otherwise
$\si_d^i(B_t)$ and $\si_d^j(B_t)$ are non-disjoint). Moreover,
$\si_d^i(\ell_1)$ and $\si_d^j(\ell_1)$ are disjoint as otherwise, by
Lemma~\ref{l:conconv}, the leaf $\ell_1$ is (pre)periodic. Therefore
$\si_d^j(\ell_1)$ is disjoint from $\si_d^i(B)$.

Suppose that $\si_d^i(B)$ and $\si_d^j(B)$ are non-disjoint. By the
just proven then, say, $\si_d^j(\n_1)$ is non-disjoint from
$\si_d^i(B)$. Again by the just proven $\si_d^j(\n_1)$ is disjoint from
$\si_d^i(\ell_1)$. Hence the only possible intersection is between
$\si_d^j(\n_1)$ and, say, $\si_d^i(\n_2)$. Moreover, since
$\si_d^j(\ell_1)$ is disjoint from $\si_d^i(B)$, then $\si_d^j(\n_1)\ne
\si_d^i(\n_2)$ and, moreover, as distinct leaves of the same invariant
geodesic lamination, the leaves $\si_d^j(\n_1), \si_d^i(\n_2)$ cannot
cross. Hence the only way $\si_d^j(\n_1)$ and $\si_d^i(\n_2)$ are
non-disjoint is that $\si_d^j(\n_1)$ and $\si_d^i(\n_2)$ are
concatenated.

Assume that $\si^t_d(\n_2)$ is concatenated with $\n_1$ at an endpoint
$x$ of $\n_1$. Clearly, $x$ is a common vertex of $B$ and of
$\si^t_d(B)$. Hence $\si^t_d(x)$ is a common vertex of $\si^t_d(B)$ and
$\si^{2t}_d(B)$, etc. Connect points $x$, $\si^t_d(x),$
$\si^{2t}_d(x),$ $\dots$ with consecutive chords $\m_0,$ $\m_1,$
$\dots$. These chords are pairwise unlinked because, as it follows from
the above, the sets $\si_d^r(B)$, $r=0$, $1$, $\dots$ have pairwise
disjoint interiors. Hence, by Lemma~\ref{l:concat}, the point $x$ is
(pre)periodic, a contradiction with the fact that all sets
$B_i=\ch(\ell_1, \n_i)$ are wandering. Thus, $B$ is wandering. Hence,
by \cite{kiw02}, the collection $\B_{\lam_2}(\ell_1)$ is finite. In
fact, \cite{kiw02} implies a nice upper bound on the number of vertices
of $B$. Indeed, it is proven in \cite{kiw02} that a wandering
non-(pre)critical gap of a lamination has at most $d$ vertices; in
particular, $B$ has at most $d$ vertices (notice that by the
assumptions any power of $\si_d$ in $B$ is one-to-one).

Suppose now that $\ell_1$ is periodic. Then by Theorem~\ref{t:compgap}
any leaf of $\B_{\lam_2}(\ell_1)$ is periodic with the same periods of
endpoints. This implies that in this case the collection
$\B_{\lam_2}(\ell_1)$ is finite. Finally, if $k>0$ is the minimal
number such that $\si_d^k(\ell_1)$ is periodic and $\ell_2\in
\B_{\lam_2}(\ell_1)$ then $\si_d^k(\ell_2)$ is linked with
$\si_d^k(\ell_1)$, which implies that $\ell_2$ is a $\si_d^k$-preimage
of one of finitely many leaves from $\B_{\lam_2}(\si_d^k(\ell_1))$.
Thus, in this case $\B_{\lam_2}(\ell_1)$ is finite too.
\end{proof}

\section{Smart criticality}\label{s:smart}

\textbf{Throughout this section, we assume that $\lam_1$ and $\lam_2$
are linked or essentially equal geodesic invariant laminations with
quadratically critical portraits, see Definition \ref{d:qclink1}.} Our
aim in Section~\ref{s:smart} is to introduce \emph{smart criticality},
a principle that allows one to use a flexible choice of critical chords
of $\lam_1$ and $\lam_2$ in order to treat certain sets of linked
leaves of $\lam_1$ and $\lam_2$ as if they were sets of one invariant
geodesic lamination. However, first we need simple claims dealing with
critical clusters and special critical leaves; these claims follow from
the definitions almost immediately. Critical clusters are defined in
\ref{d:compat} and special critical leaves and special critical
clusters in \ref{d:qclink1}.

\begin{lem}\label{l:special}
Suppose that $\ell_1$ is a special critical leaf of $\lam_1$. Then the
only leaves of $\lam_2$ it can be linked with are special critical
leaves of $\lam_2$. Moreover these leaves have the same image as
$\ell_1$. Otherwise $\ell_1$ may have a common endpoint with some
leaves of $\lam_2$, in which case its forward images are endpoints of
the corresponding images of these leaves.
\end{lem}

\begin{proof} By definition, if $\ell_1$ is a special critical leaf
then $\ell_1\subset C$ where $C$ is a critical cluster common for both
$\lam_1$ and $\lam_2$. Since edges of $C$ are leaves of $\lam_2$, it
follows that the only leaves of $\lam_2$ that are linked with $\ell_1$
are chords of $C$ connecting vertices of $C$. This implies the first
claim of the lemma. The second claim is left to the reader.
\end{proof}

In the next several lemmas we study the dynamics of a leaf $\ell_1$ of
$\lam_1$ assuming that $\ell_1$ is not a special critical leaf of
$\lam_1$.

\begin{lem}\label{l:accorder}
If $\ell_1\in\lam_1$ is not a special critical leaf, then each critical
set $C$ of\, $\qcp_2$ has a spike $\ol{c}$ unlinked with $\ell_1$;
these spikes form a full collection\, $\mathcal E$ of spikes of
$\lam_2$ unlinked with $\ell_1$. If\, an\, endpoint $x$ of $\ell_1$ is
neither a vertex of a special critical cluster nor a common vertex of
associated critical quadrilaterals of the invariant geodesic
laminations $\lam_1$ and $\lam_2$, then $\mathcal E$ can be chosen so
that $x$ is not an endpoint of a spike from $\mathcal E$.
\end{lem}

\begin{proof}
Since $\ell_1$ is not a special critical leaf, spikes of $\lam_2$ from
special critical clusters are unlinked with $\ell_1$. Otherwise take a
pair of associated critical quadrilaterals $A\in \lam_1, B\in \lam_2$
with vertices alternating non-strictly on $\uc$

$$a_0\le b_0\le a_1\le b_1\le a_2\le b_2\le a_3\le b_3\le a_0$$

\noindent and observe, that $\ell_1$ is contained, say, in $[a_0, a_1]$
and hence is unlinked with the spike $\ol{b_1b_3}$ of $B$.

The second claim follows because by the assumptions, as we choose a
spike from a critical quadrilateral of $\qcp_2$, we can always choose
it to avoid $x$. This completes the proof.
\end{proof}

%Sasha 12-21-15

We apply Lemma~\ref{l:accorder} to studying accordions. Denote by
$\mathcal E_{\lam_2}(\ell_1)$ a full collection of spikes from
Lemma~\ref{l:accorder}.

\begin{cor}\label{c:accorder}
If $\ell_1=\ol{ab}\in\lam_1$ is not a special critical leaf, then
$A=A_{\lam_2}(\ell_1)$ is contained in the closure of a component of
$\disk\sm \mathcal E_{\lam_2}(\ell_1)^+$, and $\si_d|_{A\cap\uc}$ is
$($non-strictly$)$ monotone. Let $\ell_2=\ol{xy}\in \lam_2$ and
$\ell_1\cap \ell_2\ne \0$. Then:

\begin{enumerate}

\item if $\ell_1$ and $\ell_2$ are concatenated at a point $x$ that
    is neither a vertex of a special critical cluster nor a common
    vertex of associated critical quadrilaterals of our invariant
    geodesic laminations, then $\si_d$ is (non-strictly) monotone on
    $\ell_1\cup \ell_2$;

\item if $\ell_2$ crosses $\ell_1$, then, for each $i$, we have
    $\si_d^i(\ell_1)\cap\si_d^i(\ell_2)\ne \0$, and one the following
    holds:

\begin{enumerate}

\item $\si_d^i(\ell_1)=\si_d^i(\ell_2)$ is a point or a leaf shared
    by $\lam_1, \lam_2$;

\item $\si_d^i(\ell_1), \si_d^i(\ell_2)$ share an endpoint;

\item $\si_d^i(\ell_1), \si_d^i(\ell_2)$ are linked and have the
    same order of endpoints as $\ell_1, \ell_2$;

\end{enumerate}

\item points $a$, $b$, $x$, $y$ are either all
    $($pre$)$\-pe\-ri\-odic of the same eventual period, or are all
    not $($pre$)$periodic.

\end{enumerate}

\end{cor}

\begin{proof}
Set $\mathcal E=\mathcal E_{\lam_2}(\ell_1)$. If $\ell_1$ coincides
with one of spikes from $\mathcal E$, then the claim follows (observe
that then by definition $A=\ell_1$ as spikes of sets of $\lam_2$ do not
cross leaves of $\lam_2$). Otherwise there exists a unique
complementary component $Y$ of $\mathcal E^+$ with $\ell_1\subset Y$
(except perhaps for the endpoints). The fact that each leaf of $\lam_2$
is unlinked with spikes from $\mathcal E$ implies that
$A_{\lam_2}(\ell_1)\subset \ol{Y}$. This proves the main claim of the
lemma.

(1) By Lemma~\ref{l:accorder}, the collection $\mathcal E$ can be
chosen so that $x$ is not an endpoint of a chord from $\mathcal E$. The
construction of $Y$ then implies that $\si_d$ is monotone on
$\ell_1\cup \ell_2$.

(2) We use induction. By Definition~\ref{d:qclink1}, if a critical leaf
$\n_1\in \lam_1$ crosses a leaf $\m_2\in \lam_2$ and comes from a
special critical cluster (see \Cref{d:compat}), then both $\n_1$ and
$\m_2$ come from a special critical cluster and have the same image.
Thus we may assume that neither $\si^i_d(\ell_1)$ nor $\si^i_d(\ell_2)$
are from a special critical cluster. We may also assume that
$\si^i_d(\ell_1)$ and $\si^i_d(\ell_2)$ do not share an endpoint as
otherwise the claim is obvious. Hence it remains to consider the case
when $\si^i_d(\ell_1)$ and $\si^i_d(\ell_2)$ are linked and are not
special critical leaves. Then by the main claim either their images are
linked or at least they share an endpoint.

(3) By Lemma~\ref{l:conconv}, if an endpoint of a leaf of an invariant
geodesic lamination is (pre)\-pe\-ri\-odic, then so is the other
endpoint of the leaf. Consider two cases. Suppose first that an image
of $\ell_1$ and an image of $\ell_2$ ``collide'' (that is, have a
common endpoint $z$). By the above, if $z$ is (pre)periodic, then all
endpoints of our leaves are, and if $z$ is not (pre)periodic, then all
endpoints of our leaves are not (pre)periodic. Suppose now that no two
images of $\ell_1, \ell_2$ collide. Then it follows that $\ell_1$ and
$\ell_2$ have mutually order preserving accordions, and the claim
follows from Theorem~\ref{t:compgap}.
\end{proof}

Lemma~\ref{l:accorder} and Corollary~\ref{c:accorder} implement smart
criticality. Indeed, given an invariant  geodesic lamination $\lam$, a
finite gap or leaf $G$ of it is such that the set $G\cap \uc$ (loosely)
consists of points whose orbits avoid critical sets of $\lam$. It
follows that any power of the map is order preserving on $G\cap \uc$.
It turns out that we can treat sets $X$ formed by linked leaves of two
linked or essentially equal invariant geodesic laminations similarly by
varying our choice of the full collection of spikes at each step so
that the orbit of $X$ avoids \textbf{that particular} full collection
of spikes at \textbf{that particular} step (thus \textbf{smart}
criticality). Therefore, similarly to the case of one invariant
geodesic lamination, any power of the map is order preserving on $X$.
This allows one to treat such sets $X$ almost as sets of one invariant
geodesic lamination.

Combining Corollary~\ref{c:accorder} and Corollary~\ref{l:special} we
obtain Corollary~\ref{c:intersect}.

\begin{cor}\label{c:intersect} Suppose that $\ell_1\in \lam_1, \ell_2\in
\lam_2$; moreover, let $\ell_1$ and $\ell_2$ be non-disjoint. Then
$\si^n_d(\ell_1)$ and $\si^n_d(\ell_2)$ are non-disjoint for any $n\ge
0$.
\end{cor}

The proof of Corollary \ref{c:intersect} is left to the reader.

The purpose of our investigation is to see how much two linked (or
essentially equal) geodesic laminations can differ. In other words, we
study the rigidity of geodesic laminations with respect to their
quadratically critical portraits (we consider quadratically critical
portraits as critical data associated with the corresponding geodesic
lamination). In fact, we can already discuss the extent to which
geodesic laminations $\lam_1$ and $\lam_2$ differ in the particular
case of periodic Siegel gaps. Recall the notions of the skeleton,
decorations and the extension of an infinite gap $G$ introduced in
Definitions \ref{d:siegel} and \ref{d:exte-sieg}. The skeleton of $G$
is the convex hull of the maximal Cantor subset of $G\cap\uc$.
Decorations of $G$ are the convex hulls of maximal connected unions of
leaves attached to edges of $G$. The extension of $G$ is the union of
$G$ and all its decorations. Recall also that, by Lemma
\ref{l:sieg-str}, for every edge $\ell$ of a decoration of $G$, there
is a gap in the grand orbit of $G$ that has $\ell$ on its boundary.

\begin{lem}\label{l:same-sieg}
Gaps from the grand orbits of periodic Siegel gaps of $\lam_1$ and
$\lam_2$ can be paired up so that gaps in the same pair have the same
skeletons and the same decorations.
\end{lem}

\begin{proof}
Let $G$ be a periodic Siegel gap of $\lam_1$ of period $n$. Let $H$ be
the skeleton of $G$. Below when talking about fibers we mean fibers
(point-preimages) of the semiconjugacy between $\si_d^n$ restricted
onto the extension of $G$ and the corresponding irrational rotation.
Suppose that $\ell_2$ is a leaf of $\lam_2$ that intersects the
extension of $G$. If $\ell_2$ intersects two distinct fibers of the
extension of $G$, then, by Corollary~\ref{c:intersect}, the
$\si_d^n$-images of $\ell_2$ will keep intersecting the
$\si_d^n$-images of these fibers. The fact that $\si_d^n$ restricted
onto the extension of $G$ is semiconjugate to an irrational rotation
implies then that some images of $\ell_2$ are linked with each other, a
contradiction.

Thus, $\ell_2$ and all its images intersect exactly one fiber. For
geometric reasons this is equivalent to the fact that no leaf of
$\lam_2$ intersects the interior of $H$. Therefore, there exists a gap
$G_2$ of $\lam_2$ that contains $H$. Since $\lam_2$ is a quadratically
critical geodesic lamination, it cannot have infinite gaps of degree
two. It follows that $G_2$ is also a periodic Siegel gap of period $n$
with the same skeleton $H$ as $G$.

Consider a gap $T$ of $\lam_1$ such that for some minimal $m\ge 0$ we
have that $\si_d^m(T)=G$. The properties of geodesic laminations, the
fact that $\lam_1$ has no infinite critical gaps, and the fact that
$H\cap \uc$ is a Cantor set imply that $T$ maps onto $G$ with degree
one and the maximal Cantor subset $S$ of $T\cap \uc$ maps onto $H\cap
\uc$ one-to-one except, perhaps, for the endpoints $a, b$ of
$\si_d^m$-critical chords such that $(a, b)$ is a complementary arc of
$S$ (clearly, there are at most finitely many such pairs of points $a,
b$). Each hole $(a, b)$ of $S$ corresponds to a finite concatenation of
leaves of $T$ connecting $a$ and $b$ with endpoints in $[a, b]$.
Properties of geodesic laminations imply that there are maximal
connected finite concatenations of leaves growing from $a$ and $b$, and
their unions map onto the corresponding decorations of $G$.

Now, if there exists a leaf $\ell_2$ of $\lam_2$ that intersects the
convex hull of $S$, then by Corollary~\ref{c:intersect} the leaf
$\si_d^m(\ell_2)$ connects two distinct fibers on the boundary of $G$,
a contradiction. Therefore there exists a gap $T_2$ of $\lam_2$ that
contains $S$ in $T_2\cap \uc$ as its maximal Cantor subset (observe
that the arguments can be repeated in the opposite direction, which
shows that $S$ is a common maximal Cantor set in $T\cap \uc$ and in
$T_2\cap \uc$). It follows that $\si_d^m$ maps $T_2$ onto $G_2$ with
degree one. Since this argument does not depend on the choice of $T$
and the corresponding choice of $m$, we see that the grand orbit of the
gap $G$ and the grand orbit of the gap $G_2$ consist of pairs of
infinite gaps that share the same skeleton (because they share maximal
Cantor subsets of their intersections with the unit circle).

The description of the dynamics of extensions of periodic Siegel gaps
implies that given a decoration $A$ of a periodic Siegel gap $Q$, we
see a finite collection of eventual preimages of $Q$ attached to this
decoration so that the following holds: the convex hull of $A$ has
finitely many edges at each of which a skeleton of the corresponding
preimage of $Q$ is attached.  Since by the above the family of
skeletons of gaps from the grand orbits of periodic Siegel disks is the
same for both $\lam_1$ and $\lam_2$, we conclude that the family of
convex hulls of decorations of these gaps
is also the same. %This completes the proof of the lemma.
\end{proof}

Let us now continue studying orbits of pairs of non-disjoint leaves of
geodesic laminations $\lam_1$ ad $\lam_2$. Lemma~\ref{l:accorder2}
describes how $\si_d$ can be \textbf{non-strictly} monotone on
$A\cap\uc$ taken from Corollary~\ref{c:accorder}. A concatenation $\rc$
of spikes of an invariant  geodesic lamination $\lam$ such that the
endpoints of its chords are monotonically ordered on the circle will be
called a \emph{chain of spikes $($of $\lam)$}\index{chain of spikes}.
Recall that for a collection of chords of $\disk$ such as $\rc$ we use
$\rc^+$ to denote $\bigcup\rc$.

\begin{lem}\label{l:accorder2}
Suppose that $\ell_a=\ol{ab}\in \lam_1$ and $\ell_x=\ol{xy}\in \lam_2$,
where $a<x<b\le y<a$ $($see Figure \ref{f:accorder2}$)$ and, if $b=y$,
then $b$ is neither a vertex of a special critical cluster nor a common
vertex of associated critical quadrilaterals of our invariant geodesic
laminations. Suppose also that $\si_d(a)=\si_d(x)$. Then either both
$\ell_a$, $\ell_x$ are contained inside the same special critical
cluster, or there are chains of spikes $\rc_1$ of $\lam_1$ and $\rc_2$
of $\lam_2$ connecting $a$ with $x$. If one of the leaves $\ell_a$,
$\ell_x$ is not critical, then we may assume that
$\rc_1^+\cap\uc\subset [a, x]$ and that $\rc_2^+\cap\uc\subset [a, x]$.
In any case, the points $a$ and $x$ belong to the critical sets of both
laminations.
\end{lem}

Recall that, according to our terminology, a chord is contained
\emph{inside}\index{chord!inside a special critical cluster} $S$ if it
is a subset of $S$ intersecting the interior of $S$.

\begin{proof}
First assume that one of the leaves $\ell_a, \ell_x$ (say, $\ell_a$) is
a special critical leaf. Then both $a$ and $b$ are vertices of a
special critical cluster. By the assumptions, this implies that $b\ne
y$ and hence $\ell_a$ and $\ell_x$ are linked and are inside a special
critical cluster. Assume from now on that neither $\ell_a$ nor $\ell_x$
is a special critical leaf.

\begin{figure}
\includegraphics[width=10cm]{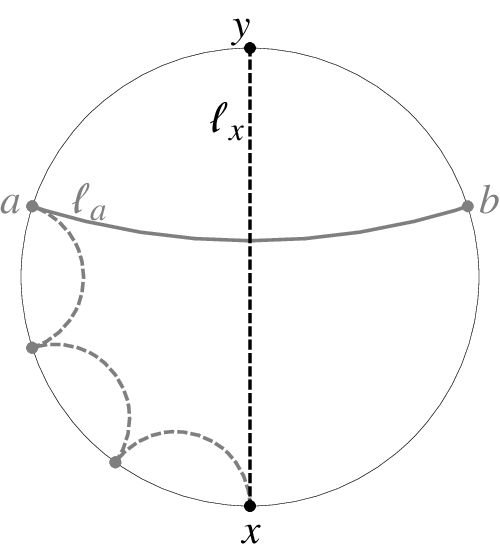}
{\caption{This figure illustrates
Lemma~\ref{l:accorder2}. Here the leaves $\ell_a,\ell_x$
collapse around a chain of spikes shown as dashed grey geodesics.}\label{f:accorder2}}
\end{figure}

By Lemma~\ref{l:accorder}, choose a full collection $\mathcal A_2$ of
spikes of $\lam_2$ unlinked with $\ell_a$ and a full collection
$\mathcal A_1$ of spikes of $\lam_1$ unlinked with $\ell_x$. By the
assumptions and Lemma~\ref{l:accorder}, we may choose these collections
so that if $b=y$, then $b=y\nin \mathcal A_1^+\cup \mathcal A_2^+$.
Thus in any case the point $\ell_a\cap \ell_x=w\in \disk$ does not
belong to $\mathcal A_1^+\cup \mathcal A_2^+$.

It follows that there is a well-defined component $Y$ of $\cdisk\sm
[\mathcal A_1^+\cup \mathcal A_2^+]$ containing $\ell_a\cup \ell_x$
except perhaps for the endpoints. Since $\si_d(a)=\si_d(x)$, there is a
chain of spikes $\rc_2\subset \mathcal A_2$ of $\lam_2$ and a chain of
spikes $\rc_1\subset \mathcal A_1$ of $\lam_1$ connecting $a$ and $x$.
In particular, $a\in \mathcal A_1,$ $x\in \mathcal A_2,$ and both $a$
and $x$ must belong to the critical sets of both laminations.

Suppose that, say, $\rc_1^+\cap\uc\subset [x, a]$. Since all spikes are
critical chords that cross neither $\ell_a$ nor $\ell_x$, this implies
that both $\ell_a$ and $\ell_x$ are critical. Therefore, if at least
one of the leaves $\ell_a$, $\ell_x$ is not critical, then we may
assume that $\rc_1^+\cap\uc\subset [a, x]$ and that
$\rc_2^+\cap\uc\subset [a, x]$.
\end{proof}

The assumptions of Lemma~\ref{l:accorder2} automatically hold if leaves
$\ell_a$, $\ell_x$ are linked and one of them (say, $\ell_a$) is
critical; in this case, by Corollary~\ref{c:accorder}, the point
$\si_d(\ell_a)$ is an endpoint of $\si_d(\ell_x)$, and, renaming the
points, we may assume that $\si_d(a)=\si_d(x)$.

\begin{dfn}\label{d:accorder2}
Non-disjoint leaves $\ell_1\ne \ell_2$ are said to \emph{collapse
around chains of spikes}\index{leaves!collapsing around chains of
spikes} if there are two chains of spikes, one in each of the two
invariant geodesic laminations, connecting two adjacent endpoints of
$\ell_1$, $\ell_2$ as in Lemma~\ref{l:accorder2}.
\end{dfn}

Smart criticality allows one to treat accordions as gaps of one
invariant geodesic lamination provided images of leaves do not collapse
around chains of spikes.

\begin{lem}\label{l:indepcrit}
Let $\ell_1$, $\ell_2$ be linked leaves from $\lam_1$, $\lam_2$ such
that there is no $t$ with $\si_d^t(\ell_1)$, $\si_d^t(\ell_2)$
collapsing around chains of spikes (in particular, this holds if the
endpoints of $\si_d^t(\ell_1)$ are disjoint from the endpoints of
$\si_d^t(\ell_2)$ for all $t$). Then there exists an $N$ such that the
$\si_d^N$-images of $\ell_1$, $\ell_2$ are linked and have mutually
order preserving accordions. Conclusions of Theorem~\ref{t:compgap}
hold for $\ell_1$, $\ell_2$, and $B=\ch(\ell_1, \ell_2)$ is either
wandering or $($pre$)$periodic so that $\ell_1$, $\ell_2$ are
$($pre$)$periodic of the same eventual period of endpoints.
\end{lem}

\begin{proof}
By way of contradiction, suppose that there exists the minimal $t$ such
that $\si^{t+1}_d(\ell_1)$ is not linked with $\si^{t+1}_d(\ell_2)$.
Then $\si_d^t(\ell_1)$ crosses $\si_d^t(\ell_2)$ while their images
have a common endpoint. Hence Lemma~\ref{l:accorder2}, applied to
$\si_d^t(\ell_1)$ and $\si_d^t(\ell_2)$, implies that
$\si_d^t(\ell_1)$, $\si_d^t(\ell_2)$ collapse around a chain of spikes,
a contradiction. Thus, $\si_d^t(\ell_1)$ and $\si_d^t(\ell_2)$ cross
for any $t\ge 0$. In particular, no image of either $\ell_1$ or
$\ell_2$ is ever critical.

By Lemma~\ref{l:conconv}, choose $N$ so that leaves
$\si_d^N(\ell_1)=\ol{ab}$ and $\si_d^N(\ell_2)=\ol{xy}$ are periodic or
have no (pre)\-pe\-ri\-odic endpoints. If $\ol{ab}$ and $\ol{xy}$ are
periodic, then no collapse around chains of critical leaves on
\textbf{any} images of $\ol{ab}, \ol{xy}$ is possible (for
set-theoretic reasons). Hence $\si_d^N(\ell_1), \si_d^N(\ell_2)$ are
linked and have mutually order preserving accordions as desired.

Suppose now that our leaves have non-(pre)\-pe\-ri\-odic endpoints.
Evidently, the set $E$ of all endpoints of all possible chains of
spikes is finite. Thus, there exists an $N$ such that if $n\ge N$, then
$\si_d^n(a)$ is disjoint from $E$ as otherwise by the pigeonhole
principle $a$ would have to be (pre)periodic. The same holds for $b$,
$x$ and $y$, so we may assume that, for $n\ge N$, no endpoint of
$\si_d^n(\ell_1)$ or $\si_d^n(\ell_2)$ is in  $E$. Hence, the
$\si_d^N$-images of $\ell_1$, $\ell_2$ are linked and have mutually
order preserving accordions.
\end{proof}

\section{Linked quadratically critical invariant geodesic laminations}\label{s:qcrit}

The main results of Section \ref{s:qcrit} are based on the principle of
Smart Criticality and the results describing the dynamics of
accordions. Basically, we are studying two linked or essentially equal
invariant geodesic laminations with quadratically critical portraits
and establish the extent to which they must resemble each other.
Therefore our results can be viewed as rigidity results of certain
subsets (or certain dynamical properties) of geodesic invariant
laminations with respect to their linked perturbations. For instance,
we show that two linked or essentially equal invariant geodesic
laminations with quadratically critical portraits have the same
\emph{perfect parts} (see Definition~\ref{d:perfect}). We also show
that two linked or essentially equal invariant geodesic laminations
with quadratically critical portraits have the same \emph{Siegel parts}
defined below (see also page~\pageref{spart}).

\begin{dfn}\label{d:siegel-part}
The closure of the union of the grand orbits of all periodic Siegel
gaps of an invariant geodesic lamination $\lam$ is denoted by
$\lam^{Sie}$ and is called the \emph{Siegel part}\index{Siegel part} of
$\lam$.
\end{dfn}

However the relations between the remaining parts of two linked or
essentially equal invariant geodesic laminations are less rigid. We
study them in the next section concentrating upon the case when
invariant geodesic laminations are generated by invariant laminational
equivalence relations.

In this section, we will always assume that the invariant geodesic
laminations with quadratically critical portraits $(\lam_1,\qcp_1)$ and
$(\lam_2,\qcp_2)$ are linked or essentially equal.

The next lemma studies cardinalities of certain collections of leaves.

\begin{lem}\label{l:countable}
Suppose that invariant geodesic laminations $(\lam_1,\qcp_1)$ and
$(\lam_2,\qcp_2)$ with quadratically critical portraits are linked or
essentially equal. The set $\Tc$ of all leaves of $\lam_2$ non-disjoint
from a leaf $\ell_1$ of $\lam_1$ is at most countable. Thus, if $\ell$
is a leaf on which uncountably many leaves of one of the geodesic
laminations $\lam_1$ or $ \lam_2$ accumulate then $\ell$ is unlinked
with any leaf of the other geodesic lamination.
\end{lem}

\begin{proof}
If $\ell_1$ has (pre)periodic endpoints, then, by
Corollary~\ref{c:accorder}, any leaf of $\lam_2$ non-disjoint from
$\ell_1$ has (pre)periodic endpoints implying the first claim of the
lemma in this case. Let $\ell_1$ have no (pre)periodic endpoints. Then,
by Corollary~\ref{c:accorder}, leaves of $\lam_2$ non-disjoint from
$\ell_1$ have no (pre)\-pe\-ri\-odic endpoints. By Lemma~\ref{l:cones},
for every eventual image $x$ of an endpoint of $\ell_1$ there are
finitely many leaves with endpoint $x$. Hence the set of all leaves of
$\lam_2$ with endpoint whose orbit collides with the orbit of an
endpoint of $\ell_1$ is countable. If we remove them from $\Tc$, then
we obtain a new collection $\Tc'$ of leaves; by \Cref{l:indepcrit},
they have mutually order preserving accordions with $\ell_1$. By
Corollary~\ref{c:finaccord}, the collection $\Tc'$ is finite. This
completes the proof of the first claim of the lemma. The second claim
follows immediately. \end{proof}

Let $\qcp$ be a quadratically critical portrait of an invariant
geodesic lamination $\lam$. Since, by Corollary~\ref{c:cridisj},
distinct critical sets of the perfect part $\lam^p$ are disjoint, each
critical set of $\lam$ is contained in a unique critical set of
$\lam^p$. Hence $\qcp$ generates the \textbf{critical pattern
$\zc(\qcp)$ of $\qcp$ in $\lam^p$}, and so each invariant geodesic
lamination with critical portrait $(\lam, \qcp)$ gives rise to the
perfect invariant geodesic lamination with critical pattern $(\lam^p,
\zc(\qcp))$.

\begin{dfn}[Perfect-Siegel part]\label{d:pspart}\index{perfect-Siegel part}
The union $\lam^{pS}$ of the perfect and the Siegel parts of an
invariant geodesic lamination $\lam$ is called the \emph{perfect-Siegel
part} of $\lam$ (it is easy to see $\lam^{pS}$ is a geodesic
lamination).
\end{dfn}

In fact, $\lam^{pS}$ is a proper geodesic lamination (see
\Cref{d:properlam}) because critical leaves with periodic endpoints or
critical wedges with periodic vertices are impossible in the
perfect-Siegel part of $\lam$ (hence, they are not present in
$\lam^{pS}$). Hence $\lam^{pS}$ induces the corresponding laminational
equivalence relation $\approx_{\lam^{pS}}$, which in turn defines its
geodesic lamination $\lam_{\approx_{\lam^{pS}}}$.

Periodic Fatou gaps of $\lam^{pS}$ and $\lam_{\approx_{\lam^{pS}}}$ may
differ. Indeed, let $U$ be a periodic Fatou gap of $\lam^{pS}$ of
degree greater than one. There may exist a finite chain of edges of
$U$. Since $U$ is a gap of $\lam^{pS}$, these edges must be
non-isolated from the outside of $U$. It follows that they all are
(pre)periodic. On the other hand, by definition the initial and the
terminal points of this chain of edges are connected by a leaf of
$\lam_{\approx_{\lam^{pS}}}$ that is \textbf{not} a leaf of
$\lam^{pS}$.

With respect to the Siegel parts, the geodesic lamination
$\lam_{\approx_{\lam^{pS}}}$ contains convex hulls of the skeletons of
the periodic Siegel gaps and their pullbacks and convex hulls of
decorations attached to these gaps and their pullbacks while the
geodesic lamination $\lam^{pS}$ may contain finite chains of leaves
inside the decorations and their pullbacks, but contains no leaves
inside convex hulls of the periodic Siegel gaps and their pullbacks.
These are the only two types of differences between $\lam^{pS}$ and
$\lam_{\approx_{\lam^{pS}}}$. Observe that if the original two geodesic
laminations are generated by laminational equivalence relations, the
latter phenomenon (concerning the Siegel parts) is impossible because
by definition there is no erasing of leaves related to it.

\begin{dfn}\label{d:per-sie-lam} A laminational equivalence relation
$\sim$ is said to be \emph{perfect-Siegel} if
$\sim=\approx_{\lam_{\sim}^{pS}}$.
\end{dfn}

Some conditions immediately imply that for an equivalence $\sim$ its
perfect-Siegel part generates the equivalence relation
$\approx_{\lam_{\sim}^{pS}}$ that coincides with $\sim$.

\begin{lem}\label{l:per-sie-lam}
If all critical sets of $\sim$ are finite, then $\sim$ is a
perfect-Siegel equivalence relation.
\end{lem}

\begin{proof}
If $\sim\ne \approx_{\lam_{\sim}^{pS}}$, then there must exist a
periodic Fatou gap $U$ of $\approx_{\lam_{\sim}^{pS}}$ of degree $k>1$
such that all its edges are leaves of $\lam_\sim$ and there is a
countable non-empty set of leaves of $\lam_\sim$ inside $U$; moreover,
$U$ contains no Siegel gaps of $\sim$. However then we can consider
$\lam_\sim$ restricted on $U$, and by the assumptions it would follow
that there are no infinite gaps of $\lam_\sim$ inside $U$ at all. It is
well-known that in this case $U/\sim$ is a dendrite and $\lam_\sim$
must have uncountably many leaves inside $U$, a contradiction.
\end{proof}

Theorem~\ref{t:noesli} studies the perfect-Siegel parts of invariant
geodesic laminations with quadratically critical portraits. Recall that
for brevity we call a gap $G$ uncountable (countable, finite) if $G\cap
S^1$ is uncountable (countable, finite).

\begin{thm}\label{t:noesli}
If $(\lam_1, \qcp_1)$ and $(\lam_2, \qcp_2)$ are invariant geodesic
laminations with quadratically critical portraits that are linked or
essentially equal, then we have the following equality: $$(\lam^p_1,
\zc(\qcp_1))=(\lam^p_2, \zc(\qcp_2)).$$ Also, the Siegel parts
$\lam^{Sie}_1$ of $\lam_1$ and $\lam^{Sie}_2$ of $\lam_2$ coincide, and
so $\lam^{pS}_1=\lam^{pS}_2$.
\end{thm}

\begin{proof} By way of contradiction, assume that
$\lam^p_1\not\subset \lam^p_2$; then $\lam^p_1\not\subset \lam_2$, and
there exists a leaf $\ell^p_1\in \lam^p_1\sm \lam_2$. Then, by
Lemma~\ref{l:countable}, the leaf $\ell^p_1$ is inside a gap $G$ of
$\lam_2$. Since $\lam^p_1$ is perfect, from at least one side all
one-sided neighborhoods of $\ell^p_1$ contain uncountably many leaves
of $\lam^p_1$. Hence $G$ is uncountable (if $G$ is finite or countable,
then there must exist edges of $G$ that cross leaves of $\lam^p_1$, a
contradiction as above). We claim that this is impossible. Indeed, by
\cite{kiw02} $G$ is (pre)periodic. Hence we may assume that $G$ is
periodic and still contains uncountably many leaves from $\lam^p_1$.
Since our geodesic laminations have quadratically critical portraits,
it follows that $G$ is a Siegel gap. This contradicts
Corollary~\ref{c:noinf}. Finally, the claim of the lemma dealing with
Siegel parts of geodesic laminations $\lam_1$ and $\lam_2$ follows from
\Cref{l:same-sieg}.
\end{proof}

Jan Kiwi showed in \cite{kiwi97} that if all critical sets of an
invariant  geodesic lamination $\lam$ are critical leaves with
\textbf{aperiodic kneading}, then its perfect part $\lam^p$ is
completely determined by these critical leaves (he also showed that
this defines the corresponding laminational equivalence relation $\sim$
such that $\lam^p=\lam_\sim$ and that $\sim$ is dendritic). Our results
are related to Kiwi's. Indeed, by Theorem~\ref{t:noesli}, if $\lam$ is
an invariant  geodesic lamination with a quadratically critical
portrait $\qcp$, then $\lam^p\subset \lam$ is completely defined by
$\qcp$; in other words, if there is another invariant geodesic
lamination $\hlam$ with the same quadratically critical portrait
$\qcp$, then still $\hlam^p=\lam^p$.

Theorem~\ref{t:noesli} takes the issue of how critical data impacts the
perfect part of an invariant geodesic lamination further as it
considers the dependence of the perfect parts upon critical data while
relaxing the conditions on critical sets and allowing for ``linked
perturbation'' of the critical data. Therefore, Theorem~\ref{t:noesli}
could be viewed as a rigidity result: ``linked perturbation'' of
critical data does not change the perfect invariant geodesic
lamination.

\section{Invariant geodesic laminations generated by laminational
equivalence relations}\label{s:geo-equi}

In the previous section, we investigated the relation between two given
quadratically critical invariant geodesic laminations that are linked
or essentially equal. However in general geodesic laminations are not
quadratically critical. Thus we need to develop tools allowing us to
adjust arbitrary geodesic laminations to produce quadratically critical
ones. In this section we do so first without any restrictions upon the
degree of the map and its Fatou gaps, and then adding some restrictions
to obtain more precise results.

\subsection{Linked invariant geodesic
laminations of any degree}\label{ss:general-link} The above analysis
justifies the next definition.

\begin{dfn}\label{d:linoreq}
Let $\lam_1$ and $\lam_2$ be invariant geodesic laminations. Suppose
that there are invariant geodesic laminations with quadratically
critical portraits $(\lamm_1, \qcp_1)$, $(\lamm_2, \qcp_2)$ such that
$\lam_1\subset \lamm_1$, $\lam_2\subset \lamm_2$.
%and $(\lamm_1, \qcp_1)$.VT:???
Then we say that the  modifications $\lamm_1$ of $\lam_1$ and $\lamm_2$
of $\lam_2$ are \emph{induced}\index{modifications of geodesic
laminations!induced} by quadratically critical portraits $\qcp_1$ and
$\qcp_2$, respectively. If this can be done so that $(\lamm_1, \qcp_1)$
and $(\lamm_2, \qcp_2)$ are linked or essentially equal, then we say
that $\lam_1$ and $\lam_2$ are \emph{intrinsically linked}
(\emph{essentially equal}, respectively)\index{invariant geodesic
laminations!intrinsically linked or essentially equal}.
\end{dfn}

Two invariant geodesic laminations are intrinsically linked or
essentially equal if and only if we can ``tune'' them into two
quadratically critical geodesic laminations by inserting into their
critical sets  critical quadrilaterals in a dynamically consistent way
so that the thus constructed quadratically critical portraits of the
two geodesic laminations are linked/essentially equal.

However arbitrary quadratically critical modifications of invariant
geodesic laminations may yield a significant increase of the
corresponding perfect-Siegel parts of these invariant geodesic
laminations. Thus, in order to implement our results we need to agree
upon the way arbitrary invariant geodesic laminations should be
modified (tuned) into geodesic laminations with quadratically critical
portraits. This can only be done by inserting critical quadrilaterals
into critical sets of given geodesic laminations. Moreover, ideally
this quadratically critical tuning should not increase the size of the
original geometric lamination too much.

Recall that laminational equivalence relations $\sim$ appear in complex
dynamics in a very natural way (see Section~\ref{s:laeqre}). For many
polynomials $P$ with connected Julia sets, they give rise to
semiconjugacies between $P|_{J(P)}$ and the topological polynomial
$f_{\sim_P}:\uc/\sim_P\to \uc/\sim_P$ with a special choice of
$\sim_P$.

Therefore, from the point of view of complex dynamics, laminational
equivalence relations are of main interest. We want to use our tools to
study them, in particular to study the mutual location of their
critical sets. Thus, we need to develop methods of quadratically
critical tuning for invariant geodesic laminations generated by
laminational equivalences.

\begin{dfn}[Invariant non-capture geodesic
lamination]\label{d:non-capture} Con\-si\-der an invariant geodesic
lamination $\lam_\sim$ generated by an invariant laminational
equivalence relation $\sim$. Suppose that there are no
\textbf{preperiodic} Fatou gaps that map onto their image $k$-to-$1$
with $k>1$. Then both $\sim$ and $\lam_\sim$ are said to be of
\emph{non-capture} type. Otherwise $\sim$ and $\lam_\sim$ are said to
be of \emph{capture} type. Any periodic Fatou gap $U$ with a
preperiodic pullback that maps forward $k$-to-$1$, where $k>1$, is
itself said to be of \emph{capture type}.
\end{dfn}

The reason for our  interest in invariant non-capture geodesic
laminations is the following. An invariant geodesic lamination
$\lam_\sim$ of capture type will have at least one preperiodic Fatou
gap $U$ that maps onto its image in a $k$-to-$1$ fashion. This allows
for a variety of ways critical quadrilaterals can be inserted in $U$.
It is therefore impossible to associate with $\lam_\sim$ a unique (or
finitely many) quadratically critical portrait(s).

Similarly, there exists an ambiguity related to the issue of how
critical quadrilaterals can be inserted into \textbf{periodic} Fatou
gaps of degree greater than one. We will tackle this issue later on,
however first we want to simplify the picture and consider the case
with no  periodic Fatou gaps of degree $k>1$. An easier version here is
that of two dendritic geodesic laminations $\lam_{\sim_1}$ and
$\lam_{\sim_2}$ generated by laminational equivalence relations
$\sim_1$ and $\sim_2$. However, by \Cref{t:noesli}, we can work with a
wider class of geodesic laminations.

\begin{lem}[Laminations with finite critical sets]\label{l:fin-cri-set}
Let $\lam_\sim$ be an invariant geodesic lamination. Then all critical
sets of $\lam_\sim$ are finite if and only if $\lam_\sim$ has no Fatou
gaps of degree greater than one. Moreover, then
$\lam_\sim=\lam_\sim^{pS}$.
\end{lem}

\begin{proof} The first claim of the lemma is left to the reader.
To prove the last claim of the lemma, recall that by \cite{bopt10} the
perfect part $\lam_\sim^p$ of $\lam_\sim$ is itself an invariant
geodesic lamination with some (possibly empty) collection of periodic
Fatou gaps of degree greater than one (so-called \emph{super-gaps} of
$\lam_\sim$) and their pairwise disjoint preimages. By the assumption,
$\lam_\sim$ can have neither periodic Fatou gaps of degree greater than
one nor Siegel gaps of capture type. Therefore, if $U$ is a periodic
super-gap of $\lam_\sim$, then $U$ lies entirely in the Siegel part of
$\lam_\sim$. Clearly, pullbacks of the intersection of $U$ and the
Siegel part of $\lam_\sim$ fill up pullbacks of periodic super-gaps of
$\lam_\sim$. We conclude that $\lam_\sim=\lam^{pS}_\sim$.
\end{proof}

Observe that if $\lam_\sim$ is an invariant geodesic lamination
generated by a laminational equivalence relation $\sim$ such that all
critical sets of $\lam_\sim$ are finite, then $\lam_\sim$ is regular
because finite critical sets correspond to $\sim$-classes of
equivalence and hence either coincide or are disjoint. Therefore, by
Definitions~\ref{d:regula} and \ref{d:cripatt1} one can talk about
critical patterns of quadratically critical portrait in $\lam_\sim$ or
of invariant geodesic lamination $\lam_\sim$ with critical pattern.

Combining \Cref{l:fin-cri-set} with \Cref{t:noesli}, we obtain
\Cref{c:dend-sie-unli} (the last claim of \Cref{c:dend-sie-unli} is
left to the reader).

\begin{cor}\label{c:dend-sie-unli} If geodesic laminations
$\lam_{\sim_1}$ and $\lam_{\sim_2}$ with no Fatou gaps of degree
greater than one are intrinsically linked or essentially equal, then
$\sim_1=\sim_2=\sim$ and $\lam_{\sim_1}=\lam_{\sim_2}=\lam_\sim$ are
equal. If $\qcp_1$ and $\qcp_2$ are two quadratically critical
portraits of $\lam_{\sim_1}$ and $\lam_{\sim_2}$ that are linked or
essentially equal, then the critical patterns of $\qcp_1$ and $\qcp_2$
in $\lam_\sim$ coincide.
\end{cor}

Observe that this generalizes  results by Kiwi \cite{kiwi97}. In our
terms his results state that if two dendritic geodesic laminations are
essentially equal then they coincide. We weaken the assumptions here
and allow for \textbf{linked} geodesic laminations generated by
laminational equivalence relations from a \textbf{wider} class while
the conclusion remains the same.

The general case is more complicated. Consider two perfect non-empty
geodesic laminations $\lam_{\sim_1}$ and $\lam_{\sim_2}$ that are
intrinsically linked or essentially equal. By definition, there are
modifications $\lam^m_{\sim_1}$ of $\lam_{\sim_1}$ and
$\lam^m_{\sim_2}$ of $\lam_{\sim_2}$ that are quadratically critical
and linked. This means that critical quadrilaterals were inserted into
critical sets of both geodesic laminations forming two linked
quadratically critical{} portraits. However this may have resulted in a
significant growth of the corresponding geodesic lamination as together
with the inserted quadrilaterals we have to add their images, preimages
and their limits. Therefore in general we cannot conclude that
$\lam_{\sim_1}$ and $\lam_{\sim_2}$ are equal if they are intrinsically
linked or essentially equal.

Thus, when studying two invariant geodesic laminations $\lam_1$ and
$\lam_2$ with infinite gaps of degree greater than one, it may be
useful to consider a restricted family of their modified invariant
geodesic laminations $\lamm_1$ and $\lamm_2$, designed to increase the
basic invariant geodesic laminations $\lam_1$ and $\lam_2$ as little as
possible so that the fact that $\lamm_1$ and $\lamm_2$ are linked
implies more information about $\lam_1$ and $\lam_2$ themselves.
Accordingly, we would like to specify quadratically critical tuning of
invariant geodesic laminations.

\begin{dfn}[Admissible quadratically critical
modifications]\label{d:admi-mod} Suppose \newline that $\lam_\sim$ is
an invariant geodesic lamination generated by an invariant laminational
equivalence relation $\sim$ and $\lamm$ is its quadratically critical
modification. We say that $\lamm$ is an \emph{admissible quadratically
critical modification} of $\lam$ if the set $\lamm\sm \lam$ consists of
a countable family of leaves.
\end{dfn}

Then the following theorem easily follows.

\begin{thm}\label{t:admi-link}
Suppose that invariant geodesic laminations $\lam_{\sim_1}$ and
$\lam_{\sim_2}$ have admissible quadratically critical modifications
$\lamm_{\sim_1}$ of $\lam_{\sim_1}$ and $\lamm_{\sim_2}$ of
$\lam_{\sim_2}$ that are linked or essentially equal. Then
$\lam^p_{\sim_1}=\lam^p_{\sim_2}$. If either lamination has no eventual
preimages $U$ of periodic Siegel gaps with $\si_d|_U$ being $k$-to-$1$
with $k>1$, then their Siegel parts coincide too so that
$\lam^{pS}_{\sim_1}=\lam^{pS}_{\sim_2}$
\end{thm}

\begin{proof}
By \Cref{t:noesli} the perfect parts and the Siegel parts of the
admissible modified invariant geodesic laminations $\lamm_{\sim_1}$ and
$\lamm_{\sim_2}$ coincide. By definition, these perfect parts coincide
with perfect parts of original invariant geodesic laminations
$\lam_{\sim_1}$ and $\lam_{\sim_2}$. Clearly, this implies the first
claim of the theorem. Suppose now that either lamination does not have
eventual preimages $U$ of periodic Siegel gaps such that $\si_d|_U$ is
$k$-to-$1$ with $k>1$. Then Siegel parts of modified geodesic
laminations and of the original geodesic laminations coincide, which
implies the last claim of the theorem.
\end{proof}

\Cref{c:admi-link} easily follows from definitions and
\Cref{t:admi-link}.

\begin{cor}\label{c:admi-link}
In the situation of \Cref{t:admi-link}, two linked or essentially equal
quadratically critical portraits $\qcp_1$ and $\qcp_2$ of
$\lam_{\sim_1}$ and $\lam_{\sim_2}$, respectively, generate the same
critical pattern in the common perfect part of both geodesic
laminations. If either lamination does not have eventual preimages $U$
of periodic Siegel gaps such that $\si_d|_U$ is $k$-to-$1$ with $k>1$,
then $\qcp_1$ and $\qcp_2$ generate the same critical pattern in
$\lam^{pS}_{\sim_1}=\lam^{pS}_{\sim_2}$.
\end{cor}

The proof is left to the reader.

Observe that some assumption concerning eventual preimages of periodic
Siegel gaps is necessary for the conclusion of the theorem to hold.
Indeed, consider an invariant geodesic lamination $\lam_{\sim_1}$
generated by an invariant laminational equivalence relation $\sim_1$.
Assume that there exists an eventual preimage $U$ of a periodic Siegel
gap that maps onto its image in the $k$-to-$1$ fashion with $k>1$.
Consider a different laminational equivalence relation $\sim_2$ that
identifies $k$ (possibly degenerate) edges of $U$ with the same image,
and, accordingly, identifies preimages of these edges, which themselves
are edges of the same pullbacks of $U$. In terms of invariant geodesic
laminations, this means that a critical set (gap or leaf) that
coincides with the convex hull $A$ of $k$ newly identified edges of $U$
is inserted in $U$, and then this set is pulled back according to
Thurston's pullback construction.

Clearly, $\sim_1$ and $\sim_2$ can be supplied with two quadratically
critical portraits that are admissible for both and, in fact, coincide
themselves. Indeed, choose an appropriate collection of critical
quadrilaterals in $A$ and use it as a part of a quadratically critical
portrait; then define the remaining critical quadrilaterals so that
altogether we will get an admissible quadratically critical portrait
$Q$ for $\lam_{\sim_1}$; it follows that $Q$ serves as an admissible
quadratically critical portrait for $\sim_2$ as well. Thus, we must
make assumptions clarifying the way periodic Siegel gaps are pulled
back.

On the other hand, no assumption concerning periodic Fatou gaps of
perfect parts of $\lam_\sim$ is necessary. Indeed, suppose that $U$ is
a periodic Fatou gap of $\lam^p_\sim$. Then any eventual preimage-gaps
of $U$ cannot share edges or finite gaps separating them, unlike in the
Siegel case, because if they do, then this will give rise to isolated
leaves in $\lam^p_\sim$, a contradiction. Hence the ambiguity described
above for Siegel parts is impossible in the case of perfect parts of
invariant geodesic laminations.

\subsection{Counterexamples of type B}\label{ss:type-B}

There are examples showing that the assumptions of \Cref{t:admi-link}
in its current form cannot be relaxed. That is, in general, we cannot
make conclusions concerning the coincidence of the given invariant
geodesic laminations as a whole if they have linked or essentially
equal admissible modifications. Indeed, let us consider a class of
invariant geodesic laminations studied in \cite{bopt13a}. Every
geodesic lamination of this class has a unique invariant finite gap
$G$.  By \cite{kiw02}, there are at most two periodic orbits (of the
same period) forming the set of vertices of $G$, and if vertices of $G$
form two periodic orbits, points of these orbits alternate on $\uc$.
For each edge $\ell$ of $G$ denote by $H_G(\ell)$ the circle arc with
the same endpoints as $\ell$, separated from $G$ by $\ell$, and call it
the \emph{hole of $G$ behind $\ell$}\index{hole of a gap behind an
edge}.

Moreover, assume that at each edge $\ell$ of $G$ there exists a Fatou
gap, say, $U$, attached to $G$ (that is, sharing with $G$ a common edge
$\ell$) and having the maximal possible degree depending on its
location (i.e., if the map $\si_3$ is two-to-one, respectively,
three-to-one on $H_G(\ell)$, then the map $\si_3$ is two-to-one,
respectively, three-to-one on $U$). It is not difficult to explicitly
construct such Fatou gaps. Indeed, let $G$ have $m$ edges $\ell_0$,
$\dots$, $\ell_{m-1}$. For each $i$, let $\fg_i$ be the convex hull of
all points $x\in \ol{H_G(\ell_i)}$ with $\si_3^j(x)\in
\ol{H_G(\si_3^j(\ell_i))}$ for every $j\ge 0$. It is straightforward to
see that $\fg_i$ are infinite gaps such that $\fg_i$ maps to $\fg_j$ if
$\ell_j=\si_3(\ell_i)$. These gaps are called the {\em canonical Fatou
gaps attached to $G$}\index{canonical Fatou gap attached to an
invariant finite gap}.

It is shown in \cite{bopt13a} that, given a gap $G$, the corresponding
invariant geodesic lamination with the listed properties exists and is
unique. It is then called the \emph{canonical invariant geodesic
lamination of $G$}\index{canonical lamination of an invariant gap} and
is denoted by $\lam_G$. By \cite{bopt13a}, the invariant geodesic
lamination $\lam_G$ is generated by an invariant laminational
equivalence relation, which we will denote $\sim_G$. Observe that $G$
can also be an invariant leaf $\ol{0\frac12}$ in which case the
definitions are similar to the above.

Finite invariant gaps $G$ are classified in \cite{bopt13a} into several
categories called gaps of type A, B and D.
\begin{comment}
V: there are no type C finite gaps!
\end{comment}
This classification mimics Milnor's classification of hyperbolic
components in slices of cubic polynomials and quadratic rational
functions \cite{miln93, M}. In the present paper we are interested in
gaps of type B (from ``Bi-transitive'') and their canonical invariant
geodesic laminations.

\begin{dfn}[Gaps of type B]\label{d:typeb}
Suppose that $G$ is a $\si_3$-invariant gap. As\-sume that its vertices
form one periodic orbit (so that the edges of $G$ form one periodic
orbit of edges too). Moreover, suppose that there is an edge of $G$,
denoted by $M_1$, that separates $G$ from $0$, and another edge of $G$,
denoted by $M_2$, that separates $G$ from $\frac12$. The edges $M_1$
and $M_2$ are said to be \emph{major edges (leaves)} or simply
\emph{majors} of $G$. Then $G$ is said to be an \emph{invariant finite
gap of type B}\index{invariant finite gap of type B}.
\begin{comment}
V: ALREADY DEFINED!! Denote by $H_G(M_1)$ the arc with the same
endpoints as $M_1$ containing $0$ and call it the \emph{major hole of
$G$ behind $M_1$}; similarly, denote by $H_G(M_2)$ the arc with the
same endpoints as $M_2$ containing $\frac12$, and call it the
\emph{major hole of $G$ behind $M_2$}\index{major edge}\index{major
hole}.
\end{comment}
\end{dfn}

It is easy to see that major holes of an invariant gap $G$ of type B
are of length greater than $\frac13$ but less than $\frac23$. The next
example illustrates the definitions just given.

\begin{example}\label{fingap2}
Consider the finite gap $G$ with vertices $\frac7{26}$, $\frac{11}{26}$
and $\frac{21}{26}$. This is a gap of type B. The first major leaf
$M_1$ connects $\frac{21}{26}$ with $\frac{7}{26}$ and the second major
leaf $M_2$ connects $\frac{11}{26}$ with $\frac{21}{26}$. The edges of
$G$ form one periodic orbit to which both $M_1$ and $M_2$ belong. The
major hole $H_G(M_1)$ contains $0$ and the major hole $H_G(M_2)$
contains $\frac12$.
\end{example}

\begin{figure}[htp]
\includegraphics[width=6cm]{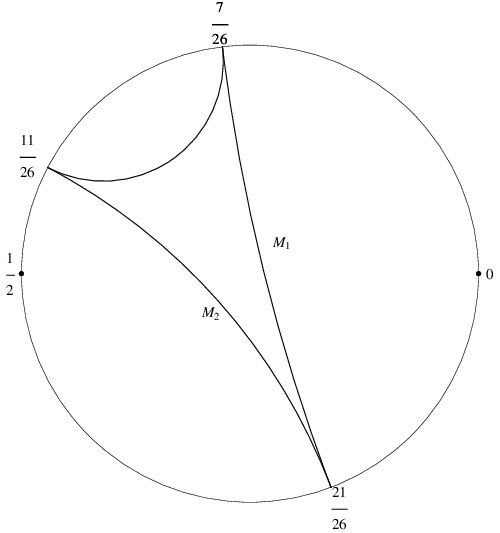}
\includegraphics[width=6cm]{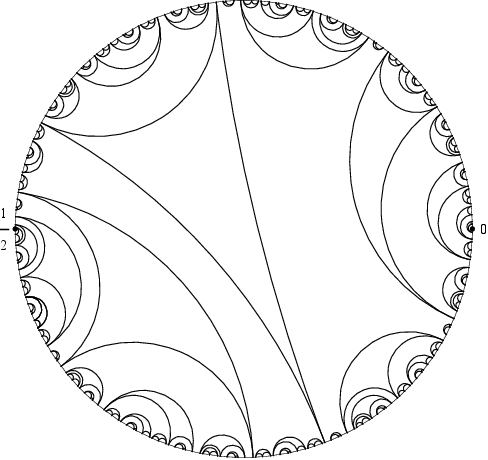}
\caption{The rotational gap described in Example \ref{fingap2} and
its canonical lamination.}
\end{figure}

Consider an invariant finite gap $G$ of type $B$. Denote by $T$ and $N$
its majors and by $U_T$ and $U_N$ the corresponding canonical Fatou
gaps of its canonical invariant geodesic lamination $\lam_G$. Let $G$
have $m$ edges $\ell_0$, $\dots$, $\ell_{m-1}$ ordered in the positive
direction around the circle. For each $0\le i\le m-1$, let $\fg_i$ be
the canonical Fatou gap attached to $G$ at $\ell_i$. Let $T=\ell_t$ and
$N=\ell_n$ with $0<t<n<m-1$ (we may always achieve this by renumbering
the edges of $G$). Then there are $n-t-1$ edges of $G$ in the positive
direction from $T$ and $N$ and $m-n-1+t$ edges of $G$ in the positive
direction from $N$ to $T$.

Let $T=\ol{ab}$ with $a<0<b$. Let $a', b'$ be the points on the
boundary of $U_T$ such that $\si_3(a')=\si_3(a)$ and
$\si_3(b')=\si_3(b)$. It is easy to see \cite{bopt13a} that we have
$a<b'<0<a'<b$. Choose two $\si_3^m$-fixed points on the boundary of
$U_T$ and denote them $x$ and $y$ so that $a<x<y<b$. Clearly, $x<0<y$,
and, moreover, $a<x<b'<0<a'<y<b'$. Then the orbit of the point $y$ has
exactly one point in every hole of $G$. If we connect them in the
positive order, then we obtain a new $\si_3$-invariant finite gap $H$.

It is easy to see that $H$ is of type B. Indeed, let $\hat y\in
H_G(\ell_{t-1})$ be a unique point from the orbit of $y$ in
$H_G(\ell_{t-1})$. Then $\ol{\hat y y}$ is an edge of $H$, which
separates $0$ from $H$ because $\hat y<0<y$. A similar edge of $H$ can
be found on the opposite side of it cutting $\frac12$ off $H$. By
definition this implies that $H$ is of type $B$. Clearly, the canonical
invariant geodesic laminations $\lam_G$ and $\lam_H$ are distinct.
However it is not difficult to show that they have admissible
quadratically critical modifications that are linked. Moreover those
modifications are very natural, if not the only natural, quadratically
critical modifications of the respective invariant geodesic
laminations.

More precisely, insert a critical quadrilateral $R_G=\ch(a, b', a', b)$
in $U_T$. Insert a similarly defined critical quadrilateral $L_G$ in
$U_N$. Using Thurston's pullback construction we can complete the
non-invariant geodesic lamination formed by $\lam_G$ together with
$R_G$ and $L_G$ to the invariant geodesic lamination $\lamm_G$. In the
same fashion we can modify the canonical invariant geodesic lamination
$\lam_H$ into an invariant geodesic lamination $\lamm_H$. We claim that
$\lamm_G$ and $\lamm_H$ are linked.

Indeed, to observe that we need to figure out where on $(\hat y, y)$
the other vertices of $R_H$ are located. Now, since the point
$\si_3(y)$ belongs to the arc $(\si_3(a), \si_3(b'))$, it follows that
the vertex $y'$ of $R_H$ with the same $\si_3$-image as $y$ must belong
to $(a, b')$. Similarly the vertex $\hat y'$ of $R_H$ with the same
image as $\hat y$ must belong to $(b', a')$. Therefore the
quadrilaterals $R_G$ and $R_H$ are strongly linked. In the same fashion
one can show that the quadrilaterals $L_G$ and $L_H$ are strongly
linked. Therefore, $\lamm_G$ and $\lamm_H$ are strongly linked whereas
they are certainly not equal.
 This shows that apparently there are no general claims analogous
to \Cref{t:admi-link} applicable to all invariant geodesic laminations
(with admissible linked or essentially equal modifications) rather than
only to their perfect or Siegel parts.

\subsection{Quadratically almost perfect-Siegel non-capture case}\label{ss:partic}

As we explained above, an important issue is that of assigning
admissible quadratically critical portraits to geodesic laminations
with periodic Fatou gaps of degree greater than one (as before we
consider \emph{non-capture} geodesic laminations). In what follows we
define \emph{legal quadrilaterals}; they give rise to quadratically
critical portraits with desired properties. A quadratically critical
pattern formed by legal quadrilaterals will be called a \emph{legal
quadratically critical pattern}. Let us emphasize that the choice of
legal quadratically critical patterns should be laminational in the
following sense: together with a legal quadratically critical pattern
we should be able to choose an invariant geodesic lamination that has
this quadratically critical pattern but also is in essence the same or
very close to the original geodesic lamination. Two laminational
equivalence relations will be called \emph{linked} if they have linked
legal quadratically critical patterns. We aim at defining the above
notions so that two laminational equivalence relations that are linked
will have to coincide. We achieve this goal in the current section in
particular case of so-called \emph{quadratically almost perfect-Siegel}
laminational equivalence relations.

Now we need to define the concepts listed above. In doing so we are
motivated by the case of $\si_2$, which we will now consider. To begin
with, choose a critical leaf $\di=\ol{0\frac12}$ with fixed endpoint
$0$. Then apply classic Thurston's step-by-step pullback construction
agreeing that, in case of ambiguity,  we  will choose \textbf{all}
possible consistent pullbacks of the existing leaves. Thus, at the
first step we will add to $\di$ the leaves, $\ol{0\frac14},$
$\ol{\frac14 \frac12},$ $\ol{\frac12 \frac34},$ $\ol{\frac34 0}$. This
creates a quadrilateral (actually a square) with vertices $0, \frac14,
\frac12$ and $\frac34$ and diagonal $\di$ so that the quadrilateral is
represented as a union of two triangles.

If we continue in the same fashion, we will add to these two triangles
four more triangles adjacent to the original two at their ``outer''
edges, i.e. at edges not equal to $\di$. Continuing in this fashion we
will in the end \textbf{tile} the closed disk $\cdisk$ into triangles,
each of which eventually maps to one of the original ``big'' triangles,
then collapses to $\di$ and then collapses further to the singleton
$\{0\}$. The leaves that we add form a null sequence and accumulate to
points of the unit circle $\uc$. The constructed invariant geodesic
lamination has countably many leaves so that its perfect part coincides
with the empty geodesic lamination. We will call this invariant
geodesic lamination \emph{basic quadratic geodesic lamination} and
denote it by $\lam_2^{bas}$. One can consider other versions of
Thurston's pullback construction starting with $\di$, but they will all
be subsets of $\lam_2^{bas}$ and will all have the empty geodesic
lamination as the perfect part.

On the other hand, it is well-known that with any other choice of a
critical leaf $\ell$ there exists a non-trivial invariant laminational
equivalence $\sim$ such that the corresponding invariant geodesic
lamination $\lam_\sim$ is compatible with $\ell$ in the following
sense: leaves of $\lam_\sim$ are not linked with $\ell$. In other
words, $\di$ is the unique critical leaf of $\si_2$ that is compatible
only with the trivial invariant laminational equivalence relation.

A similar construction can be implemented inside any $\si_2$-periodic
critical Fatou gap $U$ of period, say, $n$, however there will be an
important distinction. Indeed, it is well-known that $U$ has a refixed
edge $M_0=M(0)$ and all other edges of $U$ are appropriate pullbacks of
$M(0)$. To clarify the picture, we will denote those pullbacks of
$M(0)$ using the semiconjugacy between $\si_2^n|_{\ol{U}}$ and $\si_2$
that collapses all edges of $U$ to points of the unit circle so that
$M(\frac12)$ is the pullback of $M(0)$ since $\frac12$ is the $\si_2$
pullback of $0$. Insert in $U$ a critical quadrilateral $Q$ coinciding
with the convex hull of $M(0)$ and its sibling edge $M(\frac12)$ of $U$
(clearly, $M(\frac12)$ is the unique edge of $U$ distinct from $M(0)$
and such that $\si^n_2(M(\frac12))=\si^n_2(M(0))$.

Then there are two edges of $U$, namely, $M(\frac14)$ and $M(\frac34)$,
that map to $M(\frac12)$ under $\si_2^n$. We connect $M(\frac14)$ with
the closest edge of $Q$ to it  (evidently, this edge connects endpoints
of $M(0)$ and $M(\frac12)$) and thus construct a quadrilateral that
maps onto $Q$ under $\si_2^n$. In the same fashion we treat
$M(\frac34)$ and construct one more quadrilateral that maps onto $Q$
under $\si_2^n$. Then we continue to implement Thurston's pullback
construction and in the end construct a countable invariant geodesic
sublamination inside $U$. In fact, this sublamination is the preimage
of the invariant geodesic lamination constructed in the previous
paragraph under the semiconjugacy collapsing edges of $U$ to points of
the unit circle.

Suppose now that a $\si_d$-invariant geodesic lamination $\lam_\sim$
generated by a laminational equivalence relation $\sim$ is given.
Recall that a periodic Fatou gap $U$ of period $n$ of a laminational
equivalence relation $\sim$ is said to be \emph{quadratic}\index{Fatou
gap!quadratic} if $\si_d^n|_{\bd(U)\cap \uc}$ is two-to-one (except
perhaps for points of $\bd(U)\cap \uc$ that are images of
$\si_d^n$-critical edges of $\bd(U)$; the latter may have more than two
$\si_d^n$-preimages). We need a construction similar to the above in
the case of critical quadratic periodic Fatou gaps $U$ of period $n$.
The construction is inspired by the fact that $\si_d^n:\bd(U)\to
\bd(U)$ is monotonically semiconjugate to $\si_2$ and in a lot of cases
the semiconjugacy can be extended onto finite gaps attached to $U$.
First though we study the dynamics of periodic Fatou gaps of degree
greater than one.

In the following lemma, we talk about topological polynomials and their
bounded Fatou domains instead of talking about laminations and their
Fatou gaps. However, a translation from one language into the other is
straightforward.

\begin{lem}[Dynamics of periodic Fatou domains]\label{l:period-fatou}
Let $\Omega$ be a periodic bounded Fatou domain of a topological
polynomial $f$ of degree greater than one. Suppose that the period of
$\Omega$ is $n$. Then one of the following holds.

\begin{enumerate}

\item All the sets $f^i(\ol{\Omega})$, where $i=0$, $\dots$, $n-1$,
    are pairwise disjoint.

\item There exists $m<n$ such that $n=mk$ for some integer $k>1$, the
    set $Y=\bigcup^{k-1}_{l=0}f^{lm}(\ol{\Omega})$ is connected, and
    the sets $Y$, $f(Y)$, $\dots$, $f^{m-1}(Y)$ are pairwise
    disjoint. Moreover, the intersection $\bigcap^{k-1}_{l=0}
    f^{ml}(\ol{\Omega})$ is a singleton.
\end{enumerate}

\end{lem}

\begin{proof}
Consider the orbit of $\ol{\Omega}$, that is the union
$X=\bigcup^{n-1}_{i=0} f^i(\ol{\Omega})$. Clearly,
$X=\bigcup^{m-1}_{j=0}f^j(Y)$ where $Y$ is the component of $X$
containing $\Omega$, and $n=mk$ for some integer $k\ge 1$; it follows,
that $f^m(Y)=Y$. If $k=1$, then $Y=\ol{\Omega}$, which simply means
that the sets $\ol{\Omega}$, $f(\ol{\Omega})$, $\dots,$
$f^{n-1}(\ol{\Omega})$ are pairwise disjoint. This corresponds to case
(1) of the lemma.

Suppose that $k>1$. Then $Y=\bigcup^{k-1}_{l=0} f^{lm}(\ol{\Omega})$.
Now, it is well-known that $f^m|_Y$ must have a fixed point, say, $a$;
since by definition the sets $Y$,  $\dots$, $f^{m-1}(Y)$ are pairwise
disjoint, the point $a$ is of period $m$. Clearly, $a\in
\bigcap^{k-1}_{l=0} f^{lm}(\ol{\Omega})$. Geometrically this means that
$\ol{\Omega}$ ``rotates'' around $a$ under iterations of $f^m$ so that
after $k$ steps it maps back onto itself. Observe that sets
$f^{im}(\ol{\Omega})$ and $f^{jm}(\ol{\Omega})$ with $0\le i<j<k$
cannot intersect other than at $a$ because otherwise this will create
points of $\bd(\Omega)$ ``shielded' from infinity, which is impossible.
\end{proof}

Based upon \Cref{l:period-fatou}, we can give the following definition.

\begin{dfn}\label{d:fatou-types}
In case (1) of \Cref{l:period-fatou} we say that the Fatou domain $U$
(and the corresponding Fatou gap of the corresponding invariant
geodesic lamination/laminational equivalence relation) is of
\emph{non-rotational type}\index{Fatou domain/gap of non-rotational
type}. In case (2) of \Cref{l:period-fatou} we say that the Fatou
domain $U$ (and the corresponding Fatou gap of the corresponding
invariant geodesic lamination/laminational equivalence relation) is of
\emph{rotational type}\index{Fatou domain/gap of rotational type}.
\end{dfn}

Now let us discuss well-known facts concerning finite periodic gaps of
invariant geodesic laminations and stated here without proof. Let $U$
be a Fatou gap of an invariant geodesic lamination. Say that \emph{a
finite gap $G$ is attached to $U$ at a leaf $M$} if $U$ and $G$ share
an edge $M$\index{attached gaps}. Recall also that a gap (leaf) $G$ of
$\lam_\sim$ is said to be \emph{periodic (of period $s$)} if $s$ is the
least number such that $\si_d^s(G)=G$. The nature of $\si_d^s|_G$ may
be different though.

\begin{dfn}[Types of finite periodic gaps and leaves]\label{d:fin-gap-typ}
Let $G$ be a finite periodic gap of period $s$. Then $\si_d^s:G\to G$
may be the identity map. In this case we say that $G$ is a \emph{fixed
return} gap \index{gap!fixed return}. Otherwise $\si_d^s:G\to G$ is
``cyclic'', and we say that $G$ is a \emph{cyclic return
gap}\index{gap!cyclic return}. Similar analysis and terminology apply
to periodic leaves. If a periodic leaf $\ell$ of period $s$ is such
that $\si_d^s:\ell\to \ell$ is the identity map, then we say that
$\ell$ is a \emph{fixed return}\index{leaf!fixed return} leaf. If
$\si_d^s:\ell\to \ell$ flips $\ell$, then we say that $\ell$ is a
\emph{flip return leaf}\index{leaf!flip return}. If
$\si_d^i(\ell)\cap\ell\ne\0$ for some $0<i<s$, then $\ell$ is an edge
of a finite cyclic return gap $G$ and we say that $\ell$ is a
\emph{fixed return leaf of cyclic type}\index{leaf!fixed return, of
cyclic type}.
\end{dfn}

The analysis above shows that if a finite periodic gap $G$ is attached
to a periodic Fatou gap $U$ then the following cases are possible.

\begin{enumerate}

\item The gap $G$ is fixed return. In this case the edge of $U$, at
    which $G$ is attached to $U$, has period that is a multiple of
    the period of $U$.

\item The gap $G$ is cyclic return. Then such gap $G$ is unique for
    $U$ and $U$ is a Fatou gap of rotational type. Moreover, the edge
    of $U$ at which $G$ is attached to $U$, must be a fixed return
    leaf of cyclic type of the same period as $U$.

\end{enumerate}

We are ready to define legal critical quadrilaterals associated with a
critical quadratic $n$-periodic Fatou gap $U$ of a $\si_d$-invariant
laminational equivalence relation $\sim$. Recall that we also want to
define the geodesic lamination containing the critical quadrilaterals
in question in such a way that it is  not very different from the
original geodesic lamination containing $U$. Let $M$ be the refixed
edge of $U$, and let $M'$ be edge of $M$ such that $\si_d(M)=\si_d(M')$
($M'$ is the \emph{sibling} edge of $M$). Critical quadrilaterals $Q$
that we associate with $U$ will be either the convex hull of $M$ and
$M'$ (the construction of $Q$ is then analogous to the case of $\si_2$,
and $Q$ is called \emph{trivial}) or convex hulls of certain edges of a
finite gap $G$ attached to $U$ at $M$ and their sibling edges coming
from the sibling gap of $G$ attached to $U$ at $M'$ (then $Q$ is called
\emph{non-trivial}). In the latter case the construction of $Q$ and the
corresponding new geodesic lamination containing $Q$ must involve
erasing  the entire grand orbit of $M$.

Let us now proceed with the construction. First assume that there is no
finite gap attached to $U$ at $M$. In this case, we associate with $U$
only the critical quadrilateral obtained as the convex hull of $M$ and
its sibling edge $M'$. This is similar to the case of $\si_2$
considered above. Otherwise suppose that $G$ is a finite periodic gap
attached to $U$ at $M$. Then there is a sibling-gap $G'$ of $G$ that is
attached to $U$ at $M'$. Clearly, $\si_d^n$ maps $G'$ to $G$. Consider
two cases.

(1) If $G$ is a fixed return gap, we first erase the grand orbit of $M$
from $\lam$. Then any remaining edge of $G$ can be connected to its
sibling edge of $G'$ to create a critical quadrilateral $Q$. Finally,
$Q$ must be pulled back to create the corresponding geodesic
lamination. Then such critical quadrilateral $Q$ is said to be
\emph{legal} and the corresponding invariant geodesic lamination is
called a \emph{legal modification} of $\lam_\sim$. Observe that if we
simply erase the grand orbit of $M$ we get a new geodesic lamination
$\hlam$ that has a periodic Fatou gap containing $U$ and of the same
period and the same degree as $U$; basically, in the new gap finite
concatenations of leaves replace appropriate leaves from the grand
orbit of $M$. Clearly, $\hlam$ is proper and generates $\sim$ (and
$\lam_\sim$) in the usual way: two points are $\sim$-equivalent if and
only if they can be connected with finite chain of leaves of $\hlam$.
And in the sense of $\hlam$ we do literally the same as was done in the
$\si_2$-case: construct a critical quadrilateral based upon a refixed
edge of a gap and its sibling.

(2) Assume that the period of $G$ is $m<n$ and $n=mk$ for some $k>1$.
By \Cref{l:period-fatou}, the gap $G$ is attached to $U$ at the refixed
edge $M$ of $U$, the gap $G$ is of cyclic type, $\si_d^m$ acts on $G$
as ``rotation'', and only after $\si_d^m$ is $k$ times applied to $G$
will we have the identity map on $G$. Thus, each edge of $G$
``rotates'' under $\si_d^m$, and there are $k$ edges in its orbit under
this ``rotation''. Hence the number $kl$ of edges of $G$ is a multiple
of $k$. Consider now two separate cases: $l=1$ and $l>1$.

(a) Suppose that $l=1$. Then, as usual, we insert a critical
quadrilateral $Q$ based upon $M$ and $M'$, and pull $Q$ back to
construct the corresponding geodesic lamination. In particular, in this
case there is a unique critical quadrilateral $Q$ associated with $M$.

(b) Suppose that $l>1$. We can think of $\si_d^m$ and its action on $G$
as follows. Choose $l$ consecutive edges of $G$; they will all be in
different orbits, and under the action of $\si_d^m$ this segment of
$\bd(G)$ maps so that its images are pairwise disjoint (except for the
endpoints) segments of $\bd(G)$ until under $\si_d^n=(\si_d^m)^k$ is
maps back to itself as the identity map.

One can insert an $l+1$-gon into $G$ as follows. Choose a segment $I$
of $\bd(G)$ concatenating $l$ consecutive edges of $G$ so that one of
these edges is $M$. Take the convex hull $\ch(I)$ of $I$. It follows
that $\ch(I)$ is an $(l+1)$-gon, which ``rotates'' inside $G$ under the
action of $\si_d^m$ until it comes back to itself under $\si_d^n$, the
first return map being the identity. Observe that the choice of $I$ is
by no means unique. The orbit of $\ch(I)$ under $\si_d^m$ consists of
$k$ gaps with pairwise disjoint interior; these gaps are
``concatenated'' at their appropriate common vertices. The complement
in $G$ to the union of $\si_d^m$-images of $\ch(I)$ is another finite
gap $T$ with $k$ edges, which form one cycle under $\si_d^m$.

This construction is not unique as one can choose a segment $I$ in
several ways. In fact, it is easy to see that there are $l$ distinct
choices of a gap $T$ inside $G$, and, accordingly, $l$ distinct cycles
of sets like $\ch(I)$ ``rotating'' around $T$. In each case there is
exactly one segment $I$ of the boundary of $G$ that contains $M$. If we
now erase $M$ and its entire grand orbit, then we obtain an invariant
geodesic lamination similar to the one described above in case (1). The
gap $U$ will again be enlarged, and all leaves from the grand orbit of
$M$ will be replaced by finite concatenations of leaves (for example,
$M$ itself will be replaced by the remaining edges of $\ch(I)$, etc).
As before, the period of $U$ and the period of the newly constructed
gap containing $U$ are equal, and the same can be said about their
degrees.

Thus, we can erase the grand orbit of $M$ but on the other hand add the
gap $T$ as above and its grand orbit. This yields a new geodesic
lamination $\hlam$ similar to the geodesic lamination $\hlam$ from case
(1). Observe that in $\hlam$ $U$ is enlarged at the expense of $G$, and
$G$ is replaced by a smaller gap $T$. As before, $\hlam$ is a proper
geodesic lamination, which generates $\sim$ in the usual way; we simply
insert a critical quadrilateral $Q$ based upon an edge of $\ch(I)$ and
its sibling edge, and pull $Q$ back to construct the corresponding
geodesic lamination. This can be done in $2l-1$ ways.

Since in case (2)(b) we are replacing the gap $G$ by a smaller gap $T$,
then in all pullbacks of $G$ the corresponding pullbacks of $T$ will
have to appear. The original gap $U$ is enlarged by adding to it
$\ch(I)$ and all its pullbacks.
\begin{comment}
Thus, the structure of the geodesic lamination does not change after
its modification as it will have only some slightly enlarged Fatou gaps
and only some slightly shrunk finite gaps.
\end{comment}
It follows that the new geodesic lamination (which is evidently proper
as any geodesic lamination remains proper after one erases some of its
leaves) generates the same laminational equivalence relation.

However, in case (1), the picture is more sensitive. In this case,
there is a situation in which erasing $M$ and its grand orbit leads to
a significant change in the geodesic lamination in question and
contradicts our desire to not change it too much. Indeed, suppose that
$G$ is a finite gap attached to $U$ at $M$, and there exists a critical
gap $H$ that eventually maps to $G$. Then there must exist several
Fatou gaps attached to $H$ at its edges that are pullbacks of $M$.
Erasing $M$ and its pullbacks will result into these Fatou gaps merging
into one Fatou gap of higher degree. Thus, the structure of the new
geodesic lamination with the critical quadrilateral $Q$ described above
will be very different from the original. Hence the construction above
is not applicable if there exist critical gaps that are preimages of
$G$. Therefore when defining legal modifications of periodic quadratic
Fatou gaps we always assume that \emph{no critical gap is mapped to a
fixed return gap attached to $U$ at its refixed edge $M$ or to $M$
itself}.

We are ready to define legal quadrilaterals and legal modification of
geodesic laminations. However first we need a useful general
definition.

\begin{dfn}[Geolaminational collections]\label{d:lam-col}
If $\mathcal Y=(Y_1, \dots, Y_k)$ is a collection of gaps or leaves and
there exists a $\si_d$-invariant geodesic lamination $\lam$ such that
$Y_1$, $\ldots$, $Y_k$ are gaps or leaves of $\lam$, then we will call
$\mathcal Y$
\emph{($\si_d$-)geolaminational}\index{collection!geolaminational}.
\end{dfn}

Recall that by a \emph{full}\index{full collection of critical
quadrilaterals} collection of critical quadrilaterals we mean a
collection such that on the boundaries of components of its complement
the map $\si_d$ is one-to-one except perhaps for boundary critical
chords (clearly, such a collection must consist of $d-1$ critical
quadrilaterals).

\begin{dfn}[Legal objects]\label{d:canon} Critical quadrilaterals
constructed in (1) and (2) for $\sim$ and $\lam_\sim$  are said to be
\emph{legal}\index{legal quadrilateral}. A critical quadrilateral is
also called \emph{legal} if it is contained in a finite critical set of
a geodesic lamination. An full ordered  geolaminational collection of
legal critical quadrilaterals of $\sim$ (and $\lam_\sim$) is called a
\emph{legal quadratically critical portrait} of $\sim$ (and
$\lam_\sim$). The corresponding geodesic pullback laminations are
called \emph{legal modification}\index{legal modification} of
$\lam_\sim$.

Recall that legally modifying a lamination is a two step process. At
the first step, we replaced all periodic quadratic Fatou gaps by
possibly larger gaps. A quadratic gap $U$ gets larger if a refixed edge
of it and all edges in its grand orbit are erased. In this case the
larger gap $\widetilde U$ is said to be a \emph{legal modification} of
$U$.
\end{dfn}

Let us again emphasize that  legal quadratically critical portraits of
geodesic laminations do not always exist. However they can definitely
be constructed if \emph{no critical gap is mapped to a fixed return gap
attached to $U$ at its refixed edge $M$ or to $M$ itself}.

If they do exist then, by definition,  the corresponding geodesic
lamination is a  quadratically critical geodesic lamination. It is easy
to see that legal modifications differ from the original geodesic
lamination $\lam_\sim$ in that they tune critical gaps of $\lam_\sim$
and gaps from their grand orbit, but otherwise $\lam_\sim$ remains the
same. Even though legal quadrilaterals and legal modifications of
geodesic laminations are not always well-defined, we can define the
class of laminational equivalence relations and their geodesic
laminations for which legal quadrilaterals and legal modifications are
not only well-defined but also have the following crucial property: for
them the fact that legal modifications are linked implies that the
geodesic laminations coincide.

%[Quadratically almost perfect-Siegel non-capture laminations]

\begin{dfn}\label{d:alm-per-lam} Let $\sim$ be an invariant non-capture
laminational equivalence relation such that all periodic Fatou gaps of
$\lam^{pS}_\sim$ of degree greater than one are quadratic. Then we say
that $\lam_\sim$ and $\sim$ are \emph{quadratically almost
perfect-Siegel non-capture} invariant geodesic lamination
\index{invariant geodesic lamination!quadratically almost
perfect-Siegel non-capture} and \emph{quadratically almost
perfect-Siegel non-capture} \index{laminational equivalence
relation!quadratically perfect-Siegel non-capture}  laminational
equivalence relation.
\end{dfn}

It is easy to see that $\sim$ is quadratically almost perfect-Siegel
non-capture if and only if all critical sets of $\lam^{pS}_\sim$ are
either finite sets or periodic quadratic Fatou gaps. The terminology is
a little awkward but explicit; indeed, in the above definition
$\lam_\sim$ is not itself perfect-Siegel as $\lam^{pS}_\sim$ does not
have to coincide with $\lam_\sim$ but the difference between the two is
of ``quadratic periodic nature'': $\lam^{pS}$ has no capture Fatou
domains and all periodic Fatou domains $U$ of $\lam^{pS}$ are
quadratic. Clearly, $U$ may contain leaves of the original geodesic
lamination $\lam_\sim$, but by definition there are no Siegel gaps of
$\lam_\sim$ in $U$, and no more than countably many leaves of
$\lam_\sim$ in $U$.

Observe that if an invariant laminational equivalence relation $\sim$
is quadratically almost perfect-Siegel non-capture then all its
periodic Fatou gaps are either Siegel or quadratic, and there are no
critical preperiodic Fatou gaps. However the opposite statement is not
true. Indeed, consider a cubic invariant geodesic lamination obtained
as follows. Let $\ol{0\frac12}=\di$ be the unique chord in $\disk$ with
$\si_3$-invariant endpoints. Let $\fg_a$ be the convex hull of all
points with orbits \textbf{a}bove $\di$ and $\fg_b$ be the convex hull
of all points with orbits \textbf{b}elow $\di$. Then $\di$ is a common
edge of both gaps. Using Thurston's pullback scheme one can construct a
unique cubic invariant geodesic lamination $\lam_{ab}$ that has both
gaps. Clearly, $\lam_{ab}$ has two invariant critical quadratic Fatou
gaps and no other critical sets. In particular, all its periodic Fatou
gaps are either Siegel or quadratic, and there are no critical
preperiodic Fatou gaps. However, the perfect part of $\lam_{ab}$ is the
empty geodesic lamination, which, evidently, has an invariant cubic gap
(the entire closed disk). Hence $\lam_{ab}$ is \textbf{not}
quadratically almost perfect-Siegel non-capture.

\emph{Let us fix a quadratically almost perfect-Siegel non-capture
geodesic lamination $\lam_\sim$.} Say that a gap $G$ is \emph{almost
attached to $U$}\index{almost attached gaps} if either $G$ is attached
to $U$, or $G$ is attached to a gap $H$ and $H$ is attached to $U$.

\begin{lem}\label{l:relief}
Suppose that $U$ is a quadratic periodic Fatou gap of $\lam_\sim$. Then
there are no critical sets of $\lam_\sim$ mapped to edges of $U$ or to
gaps attached to $U$; in particular, there are no critical gaps
attached to $U$ and $U$ has no  critical edges. Moreover, if $W$ is a
Fatou gap of $\approx_{\lam_\sim^{pS}}$ then no Fatou gap of
$\lam_\sim$ is almost attached to $W$. In particular, only finite gaps
$G$ of fixed return type can be attached to a Fatou gap $W$ of
$\approx_{\lam_\sim^{pS}}$ at edges of $W$, and all their other edges
are non-isolated in $\lam_\sim$ from outside of $G$.
\end{lem}

Observe that $W$ above must be a quadratic gap as follows from the
definitions and assumptions.

\begin{proof}
Let a critical set (gap or leaf) $H$ map to a leaf or to a gap $G$
attached to $U$. We may assume that $U$ is critical. If $H$ is attached
to $U$, then there must exist a Fatou gap $V$ attached to $H$ and such
that $\si_d(V)=\si_d(U)$. Thus, the Fatou gap $W$ of $\lam_\sim^{pS}$
containing $U\cup H\cup V$ maps forward under $\si_d$ in at least
three-to-one fashion, a contradiction. Hence $H$ is not attached to
$U$. Then there are at least two Fatou gaps attached to $H$ and mapped
to $U$ under the appropriate power of $\si_d$. Similar to the above,
consider the Fatou gap $W$ of $\lam_\sim^{pS}$ containing $H$ and these
two gaps. Clearly, $W$ is critical. If it contains $U$, then the degree
of $\si_d|_{\bd(W)}$ is at least three, a contradiction. If $W$ does
not contain $U$, then either $W$ is not periodic, or $W$ is a critical
periodic Fatou gap of $\lam_\sim^{pS}$ such that in its orbit there is
another critical gap, namely the one that contains $U$. In all these
cases we arrive at a contradiction.

Now, suppose that $W$ is a Fatou gap of $\approx_{\lam_\sim^{pS}}$.
Thus there are at most countably many leaves of $\lam_\sim$ inside $W$.
If there exists a Fatou gap $V$ of $\lam_\sim$ attached to $W$ at an
edge $G$, or a Fatou gap $V$ attached to a gap $G$ that is attached to
$W$, then we can unite $W$, $G$ and $V$ to create a ``non-dynamic'' gap
containing $W$, $G$ and $V$ and therefore containing at most countably
many leaves of $\lam_\sim$. This would show that $W$ cannot be a gap of
$\approx_{\lam_\sim^{pS}}$ %$\lam_\sim$,
a contradiction. The last claim of the lemma now easily follows.
\end{proof}

Let $U$ be a critical Fatou gap of $\approx_{\lam^{pS}_\sim}$ of degree
greater than one. By the assumptions, $U$ is periodic, say, of period
$n$. By \Cref{l:relief} all periodic gaps attached to $U$ are finite
fixed return gaps. Such periodic gaps $G$ give rise to the difference
between the geodesic lamination generated by $\approx_{\lam^{pS}_\sim}$
and the geodesic lamination $\lam^{pS}$. Indeed, a gap $G$ attached to
$U$ has an edge $\ell$ separating $U$ and $G$ that is not a part of
$\lam^{pS}$; the same holds for similar leaves and their pullbacks.
Otherwise the geodesic lamination generated by
$\approx_{\lam^{pS}_\sim}$ and $\lam^{pS}$ coincide.

\begin{cor}\label{c:legal-ok}
Legal modifications of $\lam_\sim$ are well-defined.
\end{cor}

This justifies the next definition.

\begin{dfn}\label{d:link-our-lams}
Suppose that $\sim_1$ and $\sim_2$ are quadratically almost
perfect-Siegel non-capture laminational equivalence relations such that
some legal modifications $\lam_{\sim_1}^{leg}$ and
$\lam_{\sim_2}^{leg}$ are linked or essentially equal. Then we say that
$\sim_1$ and $\sim_2$ are \emph{linked} or \emph{essentially equal}.
\end{dfn}

Now we can state the main theorem of this section.

\begin{thm}\label{t:alm-qua-lam-link} Suppose that $\sim_1$ and
$\sim_2$ are quadratically almost perfect-Siegel non-capture
laminational equivalence relations that are linked or essentially
equal. Then  $\sim_1=\sim_2$. Moreover, if $\qcp_1$ and $\qcp_2$ are
linked or essentially equal legal quadratically critical portraits of
$\lam_{\sim_1}$ and $\lam_{\sim_2}$ then critical patterns of $\qcp_1$
and $\qcp_2$ in $\lam_{\sim_1}=\lam_{\sim_2}$ coincide.
\end{thm}

\begin{proof} By \Cref{t:admi-link}, we have $\lam^{pS}_{\sim_1}=\lam^{pS}_{\sim_2}$.
Therefore,
$\approx_{\lam^{pS}_{\sim_1}}=\approx_{\lam^{pS}_{\sim_2}}=\approx$.
Consider a periodic gap $U$ of period $n$ of $\approx$. Collapse all
edges of $U$ by a map $\psi$ that semiconjugates $\si_d^n|_{\bd(U)}$
and $\si_2$. Consider the two induced under the action of $\psi$ by the
restrictions of $\sim_1$ and $\sim_2$ onto $U$ \textbf{quadratic}
laminational equivalence relations $\approx_1$ and $\approx_2$ and the
corresponding \textbf{quadratic} invariant geodesic laminations
$\lam_{\approx_1}$ and $\lam_{\approx_2}$. Since $\lam^{leg}_1$ and
$\lam^{leg}_2$ are linked or essentially equal, it follows that
$\lam_{\approx_1}$ and $\lam_{\approx_2}$ are linked or essentially
equal. By \cite{thu85}, we have $\lam_{\approx_1}=\lam_{\approx_2}$ and
hence the restrictions of $\sim_1$ and $\sim_2$ on $U$ coincide.
Applying this argument to all periodic gaps of $\lam^{pS}_{\sim_1}=
\lam^{pS}_{\sim_2}$, we conclude that $\sim_1=\sim_2$ as desired. The
last claim of the lemma follows from definitions, so we leave it to the
reader.
\end{proof}

We can also establish a simple but useful version of
\Cref{t:alm-qua-lam-link}. To this end we need
Definition~\ref{d:link-patt}. Recall that invariant geodesic
laminations are called \emph{regular}\index{invariant geodesic
lamination!regular} if their critical sets are pairwise disjoint except
for the case when a critical leaf is a boundary edge of an all-critical
set.

\begin{dfn}\label{d:link-patt}
Suppose that $(\lam_{\sim_1}, \mathcal Z_1)$ and $(\lam_{\sim_2},
\mathcal Z_2)$ are regular invariant geodesic laminations with critical
patterns. Then we say that $(\lam_{\sim_1}, \mathcal Z_1)$ and
$(\lam_{\sim_2}, \mathcal Z_2)$ are \emph{linked} or \emph{essentially
equal} if there are two linked or essentially equal quadratically
critical portraits $\qcp_1$ and $\qcp_2$ inserted, respectively, in
sets from $\mathcal Z_1$ and $\mathcal Z_2$ (by definition then it
follows that $\lam_{\sim_1}$ and $\lam_{\sim_2}$ are intrinsically
linked or essentially equal).
\end{dfn}

\Cref{c:alm-qua-lam-link} now easily follows.

\begin{cor}\label{c:alm-qua-lam-link}
If $(\lam_{\sim_1}, \mathcal Z_1)$ and $(\lam_{\sim_2}, \mathcal Z_2)$
are quadratically almost per\-fect-Siegel non-capture invariant
geodesic laminations with critical patterns that are linked or
essentially equal, then $(\lam_{\sim_1}, \mathcal Z_1)= (\lam_{\sim_2},
\mathcal Z_2)$.
\end{cor}

\chapter{Applications: spaces of topological polynomials}\label{l:appli}

In this section we apply the tools developed above.

\section{The local structure of the space of all simple dendritic
polynomials}\label{ss:loc-dendr}

The results obtained in this subsection were described in the
Introduction in Subsection~\ref{ss:main-app}. Recall that a
\emph{$($critically$)$ marked polynomial}\index{polynomial!critically
marked} $P$ is a polynomial of degree $d$ equipped with an ordered
$(d-1)$-tuple $C(P)$ of critical points of $P$ such that the number of
entries of every critical point in $C(P)$ reflects its multiplicity,
i.e., equals the multiplicity minus one. In what follows when talking
about a marked polynomial we use the notation $(P, C(P))$ with
$C(P)=(c_1, \dots, c_{d-1})$ being critical points of $P$. The space of
all marked polynomials $(P, C(P))$ is endowed with the natural product
topology.

If $P$ is dendritic, then  by \Cref{t:kiwi-dendr} (due to Kiwi
\cite{kiwi97}) there exists an invariant laminational equivalence
relation $\sim_P$ such that the filled Julia set $J(P)$ of the
polynomial $P|_{J(P)}$ is monotonically semiconjugate by a map $\psi_P$
to the associated topological polynomial $f_{\sim_P}:J_{\sim_P}\to
J_{\sim_P}$ induced by $\si_d$ on the topological Julia set
$J_{\sim_P}=\uc/\sim_P$; let $\vp_P:\uc\to \uc/\sim_P$ be the
corresponding quotient map . For every point $z\in J(P)$ we set
$G_z=\vp_P^{-1}(\psi_P(z))$; the set $G_z$ is a laminational
counterpart of the point $z$.

A \emph{simple dendritic}\index{polynomial!simple dendritic} polynomial
is defined as a dendritic polynomial $P$ with only simple critical
points and the following property: every pair of distinct critical
points of $P$ can be separated by a pair of (pre)periodic external rays
together with their common landing point. In more combinatorial terms,
a dendritic polynomial $P$ is simple dendritic if there are $d-1$
distinct (and hence disjoint) critical sets $G_{c_1}$, $\dots$,
$G_{c_{d-1}}$ of $\lam_P$. The equivalence of the two definitions
follows from \Cref{l:condense} and \Cref{t:kiwi-dendr}. It follows that
if $(P, C(P))$ is a simple marked dendritic polynomial, then all points
$c_1$, $\dots$, $c_{d-1}$ in $C(P)$ must be distinct. Denote by
$\cmd_d$  the family of all simple (critically) marked dendritic
polynomials of degree $d$.

\begin{lem}\label{l:simpled}
The family $\cmd_d$ of all simple marked dendritic polynomials is an
open subset of the space of all marked dendritic polynomials.
\end{lem}

\begin{proof}
Suppose that $(P, C(P))$ is a simple marked dendritic polynomial. We
need to show that all marked dendritic polynomials in a sufficiently
small neighborhood $\mathcal U$ of $(P, C(P))$ are simple. Consider a
dendritic topological Julia set $J_{\sim_P}$. By definition all $d-1$
critical points of the topological polynomial $f_{\sim_P}$ are
distinct. By definition, there is a collection of (pre)periodic
external rays of $P$ such that the union $\Gamma_P$ of these rays and
their landing points divides the plane into finitely many pieces, and
each piece contains no more than one critical point of $P$. We may
assume that the landing points are all (pre)periodic but not
(pre)critical.
\begin{comment}
By \Cref{l:condense} there are (pre)periodic non-(pre)critical points
separating critical points of $f_{\sim_P}$. Also, by Kiwi's
\Cref{t:kiwi-dendr} each such (pre)periodic non-(pre)critical point is
the landing point of a finite collection of (pre)periodic external rays
of $P$; the arguments of the rays from this finite collection form a
(pre)periodic gap of $\lam_{\sim_P}$.
\end{comment}

By Lemma \ref{l:gm}, if an open neighborhood $\mathcal U$ of $(P,
C(P))$ is sufficiently small, then, for any marked polynomial $(Q,
C(Q))\in \mathcal U$, there is a union $\Gamma_Q$ of (pre)periodic
external rays and their landing points that is close to $\Gamma_P$ and
has the following property: the external rays in $\Gamma_Q$ have the
same arguments as the external rays in $\Gamma_P$, and pairs of rays in
$\Gamma_Q$ land together if and only if the corresponding pairs of rays
in $\Gamma_P$ land together. Since $C(Q)$ is also close to $C(P)$, it
follows that any two elements of $C(Q)$ are separated by $\Gamma_Q$,
hence $Q$ is simple dendritic.
\begin{comment}
points in its Julia set on which the external rays of $Q$ land, and
these external rays have the same arguments as the rays landing at the
(pre)periodic non-(pre)critical points from the previous paragraph. The
corresponding gaps partition the remaining subset of the closed unit
disk into the same components as before, of which $d-1$ must contain
critical sets while others cannot; moreover, each component must
contain a critical set which maps forward by $\si_d$ in the two-to-one
fashion. Therefore if $(Q, C(Q))\in \mathcal U$ is a marked dendritic
polynomial then it must be simple as desired.
\end{comment}
\end{proof}

\begin{dfn}[Local parameterization of dendritic
polynomials]\label{d:loc-para} Define \newline the following map
$\Psi_d$ from $\cmd_d$ to the space %$\mc(\cdisk^{d-1})$
$2^{\cdisk^{d-1}}$ of compact subsets of $\cdisk^{d-1}$:

$$\Psi_d(P)=G_{P(c_1)}\times G_{P(c_2)}\times \dots \times
G_{P(c_{d-1})}$$

\noindent and call the sets $\Psi_d(P)$ \emph{postcritical tags of
critically marked dendritic polynomials}\index{tags of critically
marked dendritic polynomials!postcritical}.
\end{dfn}

We show that if we fix a simple critically marked dendritic polynomial
$P$ and a sufficiently small  neighborhood $\mathcal U$ of $P$ in
$\cmd_d$, then the space of all corresponding tags has nice properties.
However first we need to introduce a few new notions and quote a useful
topological result.

\begin{dfn}\label{d:usc} A collection $\mD=\{D_\alpha\}$ of compact
and disjoint subsets of a metric space $X$ is \emph{upper
semicontinuous $($USC$)$}\index{upper semicontinuous!collection} if,
for every $D_\alpha$ and every open set $U\supset D_\alpha$, there
exists an open set $V$ containing $D_\alpha$ so that for each
$D_\beta\in \mD$, if $D_\beta\cap V\ne\0$, then $D_\beta\subset U$. A
decomposition of a metric space is said to be \emph{upper
semicontinuous $($USC$)$}\index{upper semicontinuous!decomposition} if
the corresponding collection of sets is upper semicontinuous.
\end{dfn}

Upper semicontinuous decompositions of separable metric spaces are
studied in \cite[p. 13]{dave86}.

\begin{thm}[\cite{dave86}]\label{t:dav} If $\mD$ is an
upper se\-mi\-con\-ti\-nuous decomposition of a separable metric space
$X$, then the quotient space $X/\mD$ is also a separable metric space.
\end{thm}

In the above situation we call $X/\mD$ the space \emph{generated by
$\mD$} and denote by $\pi_\mD:X\to X/\mD$ the corresponding quotient
map. In what follows we will use a well-known fact given below without
a proof. Recall that by \Cref{d:uppers} a map $F:A\to 2^B$ from a
topological space $A$ to the space of all compact subsets of a
compactum $B$ is \emph{upper semicontinuous}\index{upper
semicontinuous!map} if for every $x\in A$ and every neighborhood
$\mathcal U$ of $F(x)$ there exists open neighborhood $V$ of $x$ such
that $y\in V$ implies $F(y)\subset \mathcal U$.

\begin{lem}\label{l:through}
Let $F:A\to 2^B$ be an upper semicontinuous map from a topological
space $A$ to the space $2^B$ of all compact subsets of a compactum $B$.
Suppose that for any two points $x, y\in A$ either $F(x)=F(y)$, or
$F(x)\cap F(y)=\0$. Then the partition $\mD$ of the set $\bigcup_{x\in
A} F(x)$ into sets $F(x)$ is upper semicontinuous and the map
$\pi_{\mD}\circ F:A\to X/\mD$ is continuous.
\end{lem}

The next theorem is the main theorem of this subsection. It combines
the Theorems on Local Charts for Dendritic Polynomials and Local
Pinched Polydisk Model for Dendritic Polynomials.

%\begin{thm}[Local charts for all marked dendritic
%polynomials]\label{t:local-chart}

\begin{thm}\label{t:local-chart} For a simple critically
marked dendritic polynomial $(P,$ $C(P))\in \cmd_d$ of degree $d$ there
exists a neighborhood $\mathcal U$ of $(P, C(P))$ such that for any two
critically marked dendritic  polynomials $(Q, C(Q)), (R,$ $C(R))\in
\mathcal U$ either $\Psi_d(Q)$ and $\Psi_d(R)$ are disjoint or
$\Psi_d(Q)=\Psi_d(R)$.

Moreover, the map $\Psi_d$ from $\mathcal U$ to the space
%$\mc(\cdisk^{d-1})$
$2^{\cdisk^{d-1}}$ of all compact and connected subsets of the polydisk
$\cdisk^{d-1}$ is upper semicontinuous, the partition $\mD$ of the set
$\bigcup_{T\in \mathcal U}\Psi_d(T)$ into subsets $\Psi_d(T), T\in
\mathcal U,$ is upper semicontinuous, and the map $\pi_\mD\circ
\Psi_d:\mathcal U\to \bigcup_{T\in \mathcal U}\Psi_d(T)/\mD$ is
continuous.
\end{thm}

\begin{proof}
Recall that the map $\hPsi_d$ was introduced in
Subsection~\ref{ss:fulden}; this map associates to $(P, C(P))$ the
dendritic invariant geodesic lamination with critical pattern
$(\lam_{\sim_P}, \mathcal Z(P, C(P)))$ where $\mathcal Z(P,
C(P))=(G_{c_1},$ $G_{c_2},$ $\dots,$ $G_{c_{d-1}})$. By \Cref{c:crista}
and by definition the map $\hPsi_d$ is upper semicontinuous. Since by
definition

$$\Psi_d(P, C(P))=\si_d(G_{c_1})\times \dots \si_d(G_{c_{d-1}})$$

\noindent it follows that $\Psi_d$ is upper semicontinuous too. By
\Cref{l:through} to prove the theorem it remains to prove the first
claim of the theorem, that is to prove that there exists a neighborhood
$\mathcal U$ of $(P, C(P))$ such that for any two critically marked
dendritic polynomials $(Q, C(Q)), (R, C(R))\in \mathcal U$ either
$\Psi_d(Q)$ and $\Psi_d(R)$ are disjoint or $\Psi_d(Q)=\Psi_d(R).$ By
\Cref{l:simpled} we may assume that $\mathcal U$ is chosen sufficiently
small so that any marked dendritic polynomial in $\mathcal U$ is
simple. Moreover, since $\Psi$ is upper semi-continuous, we may assume
that for each critical set $C_i\in C(P)$ and each $Q\in\mathcal U$
there exists an open set $W_i$ so that $\si_d|_{W_i}$ is two-to-one and
the corresponding critical set $D_i\in C(Q)$ is contained in $W_i$.

By way of contradiction assume that two simple critically marked
dendritic polynomials $(Q, C(Q)), (R, C(R))\in \mathcal U$ have
non-disjoint sets \linebreak $\Psi_d(Q, C(Q))$ and $\Psi_d(R, C(R))$.
Let us show that if $Q$ and $R$ are sufficiently close to $P$ then this
actually implies that $\Psi_d(Q)=\Psi_d(R)$. Choose a critical set $C$
of $\sim_Q$. Then $\si_d$ maps $C$ forward in a two-to-one fashion. The
vertices of $C$ can be divided between two segments, each of which
contains, say, $m$ vertices of $C$. We can order vertices of $C$ on the
circle so that their collection is the set
$x_1<\dots<x_m<y_1<\dots<y_m$ where the segments mentioned above are
$[x_1,x_m]$ and $[y_1,y_m]$. Thus, for any given $j$ with $1\le j\le
m$, we have that $\si_d(x_j)=\si_d(y_j)=z_j$. Then the $\si_d$-image
$\si_d(C)$ of $C$ coincides with the convex hull of the points
$z_1<z_2<\dots<z_m$.

Now, by the assumption $\Psi_d(Q, C(Q))$ and $\Psi_d(R, C(R))$ are
non-disjoint. Choose a critical set $D$ of $\lam_{\sim_Q}$ and a
critical set $E$ of $\lam_{\sim_R}$ that have the same position in
$\mathcal Z(Q,C(Q))$, respectively, in $\mathcal Z(R,C(R))$, as $C$
does in $\mathcal Z(P,C(P))$. Since $\Psi_d(Q, C(Q))$ and $\Psi_d(R,
C(R))$ are non-disjoint, we have $\si_d(D)\cap \si_d(E)\ne \0$. By the
assumptions on $\mathcal U$ there exists a tight neighborhood $W$ of
$C$ such that $D\cup E\subset W$ and  $W$  maps onto its image exactly
two-to-one.

The fact that $\si_d(D)\cap \si_d(E)\ne \0$ implies that there is an
edge $\ell_D$ of $\si_d(D)$ and an edge $\ell_E$ of $\si_d(E)$ such
that $\ell_D\cap \ell_E\ne \0$. Thus, either $\ell_D$ and $\ell_E$
share an endpoint, or they are linked. Since $D\cup E\subset W$, it
follows that there exists a unique critical quadrilateral $Q_D\subset
W$ that maps two-to-one onto $\ell_D$. Clearly, $Q_D\subset D$.
Similarly, there exists a unique critical quadrilateral $Q_E\subset W$
that maps two-to-one onto $\ell_E$. Moreover, $Q_E\subset E$. If
$\ell_D$ and $\ell_E$ share an endpoint, then the full preimage of this
endpoint inside $W$ is a critical diagonal shared by $Q_D$ and $Q_E$.
If $\ell_D$ and $\ell_E$ are linked, then it easily follows that $Q_D$
and $Q_E$ are strongly linked.

This argument can be repeated for all critical sets of $\lam_{\sim_P}$.
Therefore, we see that if $\Psi_d(Q, C(Q))$ and $\Psi_d(R, C(R))$ are
non-disjoint, then the critical sets from $\mathcal Z(Q, C(Q))$ and
$\mathcal Z(R, C(R))$ that occupy the same position in the respective
critical patterns contain critical quadrilaterals that are strongly
linked or share a critical chord. By definition this implies that
$(\lam_{\sim_Q}, \mathcal Z(Q, C(Q)))$ and $(\lam_{\sim_R}, \mathcal
Z(R, C(R)))$ are linked or essentially equal. By
\Cref{c:alm-qua-lam-link} $(\lam_{\sim_Q}, \mathcal Z(Q, C(Q))) =
(\lam_{\sim_R}, \mathcal Z(R, C(R)))$
\end{proof}

\section{Two-dimensional spaces of $\si_d$-invariant geodesic
laminations}\label{ss:d-2}

The second application of our tools extends the results of
\cite{bopt15a} where we studied the space $\prnp_3(\ol{ab})$ of all
cubic invariant geodesic laminations generated by cubic invariant
laminational equivalence relations $\sim$ compatible with fixed
critical leaf $D=\ol{ab}$ with non-periodic endpoints; in other words,
we consider all laminational equivalence relations $\sim$ with $a\sim
b$. The main result of \cite{bopt15a} is that this family of cubic
invariant geodesic laminations is itself a lamination. This result
resembles a laminational description of the combinatorial Mandelbrot
set. First we study Thurston's invariant geodesic pullback laminations.

\subsection{Invariant geodesic pullback laminations}

We use \cite{bmov13} where \emph{proper} invariant geodesic laminations
were introduced (see the necessary definitions and claims in
Subsection~\ref{ss:oun} of the present paper).

\begin{lem}\label{l:per-crit-appr}
Suppose that $\lam$ is an invariant geodesic lamination. Then the
following claims hold.

\begin{enumerate}

\item There is no critical leaf $\ell=\ol{ab}\in \lam$ with a
    periodic endpoint $a$ that is approximated by leaves of $\lam$
    disjoint from $a$.

\item There is no critical wedge $W$ of $\lam$ with a periodic vertex
    $v$ such that both sides of $W$ are approximated by leaves of
    $\lam$ disjoint from $v$.

\end{enumerate}

\end{lem}

\begin{proof}
(1) If a critical leaf $\ell=\ol{ab}\in \lam$ with $a$
    $n$-periodic is approximated by leaves of $\lam$ disjoint from $a$, then
    the fact that $a$ repels close by points under $\si_d^n$ implies
    that leaves approximating $\ell$ and disjoint from $a$ will have
    $\si_d^n$-images crossing these leaves, a contradiction.

(2) Let a critical wedge $W$ consist of leaves $\ol{vu}$ and $\ol{vt}$
where $v$ is $n$-periodic. We may assume that $v<u<t<v$ and that
$\si_d^n(v)=v<t\le \si_d^n(u)=\si_d^n(t)\le v$. Then leaves
approximating %$\ol{vu}$
$\ol{vt}$ and disjoint from $v$ will have $\si_d^n$-images crossing
these leaves, a contradiction.
\end{proof}

We will need the following definition and notation.

\begin{dfn}[Admissible critical collection]\label{d:admi-col}
Let $\mathcal C=\{\oc_1=\ol{a_1b_1},$ $\dots,$
$\oc_{d-1}=\ol{a_{d-1}b_{d-1}}\}$ be a full collection of $d-1$
%pairwise disjoint
critical chords. If each chord $\oc_1, \dots, \oc_r$ has a periodic
endpoint while $\oc_{r+1}, \dots, \oc_{d-1}$ have non-periodic
endpoints then we call $\mathcal C$ an \emph{($r$-)admissible critical
collection}\index{admissible critical collection}. Also, the closure of
a component of $\cdisk\sm \bigcup_{i=1}^{d-1}\oc_i$ is called a
\emph{$\mathcal C$-domain}\index{$\mathcal C$-domain}.
\end{dfn}

Let us fix an $r$-admissible critical collection $\mathcal C$ and use
the notation from \Cref{d:admi-col}. We want to associate to $\mathcal
C$ a laminational equivalence relation. To this end we mimic Thurston's
pullback construction \footnote{We are indebted to Gao Yan for drawing
our attention to an inaccuracy in the original construction of a
pullback lamination.} (Proposition II.4.5 \cite{thu85}) and define
pullback (geodesic) laminations generated by $\mathcal C$.

By \cite{bmov13} geodesic laminations can be easily associated to
laminational equivalence relations if they are proper (see
Subsection~\ref{ss:oun}). Therefore, taking into account the definition
and properties of proper laminations, it is natural to pullback only
leaves $\oc_{r+1}, \dots, \oc_{d-1}$ that have non-periodic endpoints.
It is natural to expect that this will result in a proper lamination
and thus will lead to a laminational equivalence relation.

However this construction may involve ambiguities. Consider one such
possibility. Assume that on the $n$-th step a finite forward invariant
(under $\si_d$) lamination $\lam_n$ is obtained. Let a critical chain
of leaves $\ell_1=\ol{a_1a_2}, \ell_2=\ol{a_2a_3}, \dots,
\ell_k=\ol{a_ka_{k+1}}$ be a part of the boundary of a $\mathcal
C$-component $U$, and that $\si_d(a_1)=x$ is an endpoint of several
leaves $\ol{xy_1}, \dots, \ol{xy_s}$ of $\lam_n$. % whose pullbacks are
%not contained in $U$ (i.e., these leaves were added on the previous
%step).
For simplicity assume that points $y_1, \dots, y_s$ have unique
$\si_d$-preimages $z_1,$ $\dots,$ $z_s\in \partial U$. Then when on the
next step we pull back the leaves $\ol{xy_1}, \dots, \ol{xy_s}$ into
$U$ the point $x$ can be pulled to any of the points $a_1, \dots,
a_{k+1}$ which creates an ambiguous situation (despite the fact that
points $y_1, \dots, y_s$ can be pulled in $U$ uniquely only to points
$z_1, \dots, z_s\in \partial U$). In other words, a number of points
(namely, $k+1$ points $a_1, \dots, a_{k+1}$) can be connected with
$z_1, \dots, z_s$, and this creates ambiguity. This ambiguity surfaces
even if $k=1$ (the boundary critical chain of $U$ consists of one leaf)
as even in that case the two endpoints of this critical leaf can be
connected to $z_1, \dots, z_s$ in various ways. We will have to take
care of this ambiguity as we define pullback laminations generated by
$\mathcal C$.

Let us now give precise definitions. Suppose that on some step $n$ a
collection $\lam_n$ of pullback leaves is constructed. On step $n+1$
each leaf $\ell=\ol{xy}\in \lam_n$ can be pulled back into $\mathcal
C$-domains $B$. The boundary of $B$ maps forward covering $\uc$ in the
one-to-one fashion except for critical edges of $B$, which map to one
point each. Then the pullback of $\ell$ in $B$ is well-defined and
unique in all cases except for the following.

Let us call maximal concatenations of (critical) edges of $B$
\emph{critical chains}\index{critical chain}. Endpoints of critical
leaves in a critical chain $Z$ are called
\emph{vertices}\index{critical chain! vertices of} of $Z$. If
$\ell=\ol{xy}$ has exactly one endpoint (say, $x$) that is the
immediate image of a critical chain $T$ while $y$ pulls back to just
one point $w\in B$, then the leaf $\ell$ can pull back to various
leaves connecting $w$ to vertices of $T$. On the other hand, if both
endpoints $x$ and $y$ pull back to critical chains of $B$, then
$\ol{xy}$ can pull back to various leaves connecting vertices of the
first chain to vertices of the second chain. In what follows we call
the convex hull of the union of all pullbacks of a leaf $\ell=\ol{xy}$
into a $\mathcal C$-domain $U$ a \emph{$\mathcal C$-maximal} pullback
of $\ell$; clearly, this is the convex hull of the full pullbacks of
$x$ and $y$ to the boundary of $U$.

Thus, the ambiguity stems from the fact that critical chains can map to
the endpoints of leaves of $\lam_n$. We could resolve it by postulating
our choices. However we prefer a different approach. Namely, we show
that there \emph{exists} a well-defined way of constructing the
pullback $\lam_{n+1}$ of $\lam_n$ so that $\lam_{n+1}$ is sibling
forward invariant. In other words, we prove the \emph{existence} of
pullback laminations. Then we consider \emph{any} sequence
$\lam_0\subset \lam_1\subset \dots$ of finite invariant pullback
laminations with $\lam_0=\{\oc_{r+1}, \dots, \oc_{d-1}\}$ and show that
their limits generate the same equivalence relation.

Recall that by a \emph{forward invariant lamination} we mean a
collection of leaves satisfying Definition~\ref{d:sibli} (see remark
right after this definition), in particular it has to satisfy our
\emph{sibling} condition: for each non-critical leaf $\ell$ of the
collection there must exist $d-1$ other leaves of this collection so
that all $d$ leaves in question are \emph{pairwise disjoint} and have
the same image as $\ell$.

\begin{dfn}\label{d:pullam} Suppose that there exists a sequence of
finite forward invariant laminations $\lam_0=\{\oc_{r+1}, \dots,
\oc_{d-1}\}, \lam_1, \dots$ such that $\si_d(\lam_{n+1})=\lam_n$ for
every $n>0$. Then we say that $\lam_0, \lam_1, \dots$ is a sequence of
\emph{finite pullback laminations generated by $\mathcal C$}.
\end{dfn}

Lemma~\ref{l:ineach} follows from definitions.

\begin{lem}\label{l:ineach} Suppose that $\ell_1, \dots, \ell_d$ are
arbitrary pairwise disjoint leaves that do not cross leaves from
$\mathcal C$ and have the same image. Then each $\mathcal C$-domain
contains exactly one leaf $\ell_i$.
\end{lem}

Observe that the leaves $\ell_1, \dots, \ell_d$ are not assumed to be
members of a forward invariant pullback lamination generated by
$\mathcal C$.

\begin{proof}
Set $\mathcal T=\{\ell_i, i=1, \dots, d\}$. Let
$\si_d(\ell_1)=\ol{xy}$; for each $i$ let $\ell_i=\ol{x_iy_i}$ with
$\si_d(x_i)=x, \si_d(y_i)=y$. Call $x_i$'s \emph{$x$-points} and
$y_i$'s \emph{$y$-points}. By the assumptions all preimages of $x$ and
all preimages of $y$ form the set of all endpoints of leaves from
$\mathcal T$.

Let $U$ be a $\mathcal C$-domain and show that it contains a leaf from
$\mathcal T$. Since all preimages of $x$ and $y$ are endpoints of
leaves from $\mathcal T$, the claim is immediate if an $x$-point or a
$y$-point belongs to $\partial U$ but is not an endpoint of a critical
leaf from $\mathcal C$. Hence we may assume that there are two critical
chains in $\partial U$ mapping to $x$ and $y$ respectively. Let
$X=\hell_1\cup \dots \hell_r$ be the critical chain in $\partial U$
that maps to $x$. Let $\hell_j=\ol{a_ja_{j+1}}, j=1, \dots, r$. Assume
that $a_1<a_2<\dots<a_{r+1}$ and $(a_1, a_{r+1})$ is the circle arc
which intersects  $\partial U$ in exactly $\{a_2,\dots, a_{r}\}$.

Clearly, $a_j$'s are $x$-points. Since the number of $x$-points in the
closed circle arc $[a_1, a_2]$ is by one greater than the number of
$y$-points in that arc, then at least one leaf from $\mathcal T$ has an
$x$-endpoint inside $[a_1, a_2]$ and the other endpoint outside $[a_1,
a_2]$. If this leaf is inside $U$, we are done. Otherwise the only
possibility is that its endpoints are $a_2$ and some $y$-point inside
$[a_2, a_3]$. In that case we can repeat the argument and continue it
until we find the desired leaf.
\end{proof}

Next we show that Definition~\ref{d:pullam} is not vacuous.

\begin{lem}\label{l:pullam} Sequences of finite pullback laminations generated by $\mathcal C$ exist.
\end{lem}

\begin{proof}
As in Thurston's pullback construction, we define finite forward
invariant laminations $\lam_n, n=0, 1, \dots$ step-by-step. We always
include all points of the unit circle in them, so below we will only
describe their non-degenerate leaves. Set $\lam_0=\{\oc_{r+1}, \dots,
\oc_{d-1}\}$. Clearly, $\lam_0$ is a finite forward invariant
lamination. Assume now that $\lam_n$
is constructed and describe how $\lam_{n+1}$ is constructed. %Recall
%that a component of $\cdisk\sm \bigcup_{i=1}^{d-1}\oc_i$ is called a
%\emph{$\mathcal C$-domain};
Observe that there are $d$ $\mathcal C$-domains $U$, and on the
boundary $\partial U$ of each such $U$ the map $\si_d$ is one-to-one
except for critical chains in $\partial U$ that collapse to points.

Following Thurston, we pullback into $U$ every leaf $\ell=\ol{xy}$ of
$\lam_n$. If neither endpoint of $\ell$ is the image of a critical
chain from $\partial U$ then such a pullback is unique. However one or
both endpoints of $\ell$ may be images of critical chains in $\partial
U$. Thus, to have a well-defined pullback of $\ell$ to $U$ we need a
more elaborate algorithm. Here is how we want to do it.

Choose the positive (counterclockwise) direction on $\partial U$; given
an arc $T\subset \partial U$, we call this direction on $T$ (on the
entire $\partial U$) the \emph{$U$-direction}. This induces specific
direction on each critical chain $T$ from $\partial U$ and on each
critical leaf from $\partial U$. Observe that positive direction on a
critical leaf depends on the choice of $U$. Indeed, each critical leaf
$\oc=\ol{ab}\in \mathcal C$ is an edge of two $\mathcal C$-components,
say, $U$ and $V$; then if viewed as an edge of $U$ it will have, say,
initial endpoint $a$ and terminal endpoint $b$ while if viewed as an
edge of $V$ it will then have initial endpoint $b$ and terminal
endpoint $a$.

Clearly, $\partial U$ maps onto $\uc$ in a monotone fashion, with
critical chains in $\partial U$ being exactly the non-degenerate fibers
of $\si_d|_{\partial U}$. Given a point $x\in \uc$, denote by $I_x(U)$
the arc-preimage of $x$ in $\partial U$, denote by $i_x(U)$ the
\emph{initial} point of $I_x(U)$ and by $t_x(U)$ the \emph{terminal}
point of $I_x(U)$ understood in terms of the $U$-direction on $\partial
U$. Observe that the arcs $I_x(U)$ are in fact either points or
critical chains on the boundary of $U$.

\smallskip

\noindent \textbf{Claim.} \emph{Let $U$ and $V$ be two distinct
$\mathcal C$-domains. Then for every point $x\in \uc$ we have
$t_x(U)\ne t_x(V), i_x(U)\ne i_x(V)$.}

\smallskip

\begin{proof}[Proof of the Claim] We may assume that $Z=\ol{U}\cap
\ol{V}\ne \0$. Then either $Z$ is a critical leaf shared by the
boundaries of $U$ and $V$, or $Z=\{z\}$ is a point of the circle which
is a common point of two critical leaves $\ell_U\subset \partial U$ and
$\ell_V\subset \partial V$. In the former case the $U$-direction on $Z$
is opposite to the $V$-direction on $Z$, hence $t_x(U)\ne t_x(V)$ as
desired.

Now, suppose that $Z=\{z\}$ is just a point. Then $z$ is a common
vertex of a critical chain $I_x(U)$ and of a critical chain $I_x(V)$.
Suppose that $z=t_x(U)=t_x(V)$. Then there are distinct critical leaves
$\ol{az}\subset I_x(U), \ol{bz}\subset I_x(V)$, and we may assume,
without losing generality, that $a<b<z$. However, since the
$V$-direction on $\ol{bz}$ must be from $b$ to $z$ it follows from the
definitions that $V$ must be located between $\ol{az}$ and $\ol{bz}$
which, because of the existence of the leaf $\ol{az}$, means that the
critical chain $I_x(V)$ cannot terminate at $z$ and must include at
least one more critical edge $\ol{zc}$ of $V$ growing out of $z$ with
$a\le c<b$. This shows that $z$ is \emph{not} the terminal point of
$I_x(V)$, a contradiction. Similar arguments show that $i_x(U)\ne
i_x(V)$.
\end{proof}

Now, for a leaf $\ell=\ol{xy}\in \cdisk$, we postulate that as the
endpoints of the pullback leaf of $\ell$ we choose the points $t_x(U)$
and $t_y(U)$. We claim that this yields a sequence of finite forward
invariant laminations $\lam_n$. Moreover, we will show inductively (it
follows immediately from the construction) that for each point $x\in
\uc$ and each leaf $\ell$ of one of our laminations with endpoint $x$
all leaves of our laminations contained in $U$ and mapped to leaves
with endpoint $x$ have as the corresponding endpoint $t_x(U)\in
\partial U\cap \uc$.

Evidently, the base of induction holds. Suppose that $\lam_n$ satisfies
all the declared conditions and consider $\lam_{n+1}\sm \lam_n$. Let us
show that $\lam_{n+1}$ is sibling invariant. To this end we need to
verify that a non-critical leaf $\ell\in \lam_{n+1}\sm \lam_n$ has
$d-1$ sibling leaves, and all these $d$ leaves are pairwise disjoint.
We may assume that $\ell\in \lam_{n+1}\sm \lam_n$. Hence $\ell$ was
added on the last step in the construction. By construction this
implies that each $\mathcal C$-domain contains exactly one preimage of
$\si_d(\ell)=\ol{xy}$. Moreover, there are no other preimages of
$\ol{xy}$ in $\lam_{n+1}$ and by the Claim all these preimages of
$\ol{xy}$ are pairwise disjoint. This implies the desired.
\end{proof}

\begin{comment}

All possible pullbacks of $\ell$ are contained in the convex hull of
$s$ and the critical chain in question. We can think of pulling
$\ol{xy}$ back to this convex hull and call such a pullback
\emph{($\mathcal C$-)maximal}\index{pullback of leaves!maximal}.

2) Suppose that both endpoints $x$ and $y$ pull back to critical chains
of $B$. Then the \emph{($\mathcal C$-)maximal}\index{pullback of
leaves!maximal} pullback of $\ell$ is the convex hull of both critical
chains in $\bd(B)$ that map to $x$ and $y$.

We will use this analysis when we define \emph{admissible pullback
laminations} generated by an $r$-admissible critical collection
$\mathcal C$.

\end{comment}

The next definition is a step towards constructing a fully invariant
(not just forward invariant) lamination generated by $\mathcal C$.

\begin{dfn}\label{d:pul-prelam}
Consider a sequence $\lam_0=\{\oc_{r+1}, \dots, \oc_{d-1}\}, \lam_1,
\lam_2, \dots$ of finite pullback laminations generated by $\mathcal
C$. Call the union of all leaves from all laminations $\lam_n, n=0, 1,
\dots$ a \emph{pullback prelamination generated by $\mathcal C$}.
\end{dfn}

By Lemma~\ref{l:pullam}, the family of such pullback prelaminations is
non-empty. Since in Definition~\ref{d:pul-prelam} we mean \emph{any}
sequence of pullback laminations, not just a particular sequence
constructed in Lemma \ref{l:pullam}, there may exist several pullback
prelaminations generated by $\mathcal C$.

\begin{dfn}\label{d:pullam1} The closure of a pullback prelamination generated by
$\mathcal C$ is called a \emph{pullback lamination generated by
$\mathcal C$} \index{pullback geodesic lamination generated by a full
critical collection}
\end{dfn}

Observe that by Definition~\ref{d:pullam1}, the pullbacks of leaves
from $\mathcal C$ are dense in any pullback lamination $\lam$ generated
by $\mathcal C$ while no leaf of $\lam$ crosses a leaf from $\mathcal
C$.

\begin{thm}\label{t:pullam} Any pullback lamination generated by
$\mathcal C$ is invariant.
\end{thm}

\begin{proof} The claim follows from Corollary 3.20
\cite{bmov13}.
\end{proof}

By \cite{bmov13}  a \emph{proper} invariant geodesic lamination $\lam$
defines a laminational equivalence relation $\approx_\lam$ such that
$a\approx_\lam b$ if and only if there exists a finite chain of leaves
of $\lam$ connecting $a$ and $b$. We will show that all pullback
laminations  are proper and that they all generate \emph{the same
laminational equivalence}.

\begin{comment}

This proves Lemma~\ref{l:max-pul}.

\begin{lem}\label{l:max-pul}
Maximal pullbacks of leaves are well-defined. Thurston's pullback
construction applied to leaves from $\mathcal C$ yields pullback leaves
contained in $\mathcal C$-maximal pullbacks.
\end{lem}

We are ready to prove the next lemma.

\end{comment}

Recall that in the pullback construction we only pullback the leaves
$\oc_{r+1},$ $\dots,$ $\oc_{d-1}$ (by definition these are exactly the
leaves from our critical collection which do not have periodic
endpoints).

\begin{lem}\label{l:pull}
Let $\mathcal C=(\ol c_1,\dots,\ol c_{d-1})$ be an $r$-admissible
critical
collection. Let $\lam(\mathcal C)$ be a %invariant
geodesic lamination generated by $\mathcal C$. Then:

\begin{enumerate}

\item $\lam(\mathcal C)$ is proper and generates a laminational
    equivalence relation $\approx_{\lam(\mathcal C)}$ (in particular,
    $a_j\approx_{\lam(\mathcal C)} b_j$ for each $j, r+1\le j\le
    d-1$);

\item if $\ell$ is a critical leaf of $\lam(\mathcal C)$, then its
    endpoints must be non-periodic vertices of one of the critical
    chains from $\mathcal C$ (in particular, if all critical chains
    are just leaves - e.g., if all leaves in $\mathcal C$ are
    pairwise disjoint - then $\ell$ must be $\oc_j$ for some $j,
    r+1\le j\le d-1$);

\item if $\hlam$ is another pullback lamination generated by
    $\mathcal C$ then it is proper and
    $\approx_{\hlam}=\approx_{\lam(\mathcal C)}$ (and so the
    laminational equivalence relation $\approx_{\lam(\mathcal C)}$
    depends only on $\mathcal C$ and is from now on denoted by
    $\approx_\mathcal C$);

\item if $\hlam$ is an invariant geodesic lamination with leaves
    $\oc_{r+1}, \dots, \oc_{d-1}$ whose pullbacks do not cross leaves
    from $\mathcal C$, then $\hlam$ contains all limit leaves of
    $\lam(\mathcal C)$ and at least one pullback of each $\oc_i,
    r+1\le i\le d-1$ inside each $\mathcal C$-maximal pullback of
    $\oc_i$;

\item if $\sim$ is a laminational equivalence relation such that
    $a_j\sim b_j, r+1\le j\le d-1$ and no leaf of $\lam_\sim$ crosses
    $\oc_i, 1\le i\le r$, then for any two points $a, b$ such that
    $a\approx_{\mathcal C} b$ we have $a\sim b$.

\end{enumerate}

\end{lem}

\begin{comment}

Apply Thurston's pullback construction to leaves $\oc_{r+1}$, $\dots$,
$\oc_{d-1}$. By \Cref{l:pullam}, the $\mathcal C$-maximal pullbacks of
these leaves are well-defined. In the end we will have constructed the
entire family of pullbacks of leaves $\oc_{r+1}$, $\dots$, $\oc_{d-1}$;
taking its closure we obtain the desired invariant geodesic lamination
$\lam(\mathcal C)$.

\end{comment}

\begin{proof}
(1) The fact that $\lam(\mathcal C)$ is proper follows from the way we
define it and \Cref{l:per-crit-appr}.

(2) Since a critical leaf $\ell$ of $\lam(\mathcal C)$ must not cross
leaves from $\mathcal C$, it must be contained in a $\mathcal
C$-domain. This implies that its endpoints must be vertices of one of
the critical chains from $\mathcal C$ because distinct critical chains
from the boundary of the same $\mathcal C$-domain have distinct images.
Using Lemma~\ref{l:per-crit-appr}, the particular cases listed in the
rest of the claim now easily follow.

%coincide with a leaf $\oc_i$ with $1\le i\le d-1$. Since $\lam(\mathcal
%C)$ is proper, $\ell=\oc_j$ for some $j$ with $r+1\le j\le d-1$.

(3) By definition, $\lam(\mathcal C)$ is the closure of an invariant
prelamination that is the union of leaves from a sequence of finite
pullback laminations $\lam_0\subset \lam_1\subset \dots$ generated by
$\mathcal C$. For each $n$ define the equivalence relation $\approx_n$
on $\uc$ as follows: two points $x, y\in \uc$ are
$\approx_n$-equivalent if there exists a finite chain of concatenated
leaves of $\lam_n$ such that $x$ and $y$ are the endpoints of the first
and the last leaf in the chain respectively. In cases like that we will
simply say that the chain of leaves in question \emph{connects} $x$ and
$y$.

Suppose that $\ell_1\cup \ell_2\cup \dots \ell_k$ is a chain of
concatenated leaves of $\lam_n$. Let $x$ be the initial endpoint of
$\ell_1$; let $y$ be the terminal point of $\ell_k$. Choose a $\mathcal
C$-domain $U$ and points $x'', y''\in \partial U$ such that
$\si_d(x'')=x, \si_d(y'')=y$. Then, by Lemma~\ref{l:ineach}, there
exists a chain of leaves $\ell'_1\cup \ell'_2\cup \dots \cup \ell'_k$
of $\lam_{n+1}$ such that $\ell'_i\subset U, \si_d(\ell'_i)=\ell_i$ and
for each $i, 1\le i\le k-1$ the terminal point of $\ell'_i$ and the
initial endpoint of $\ell'_{i+1}$ are connected with a chain of
critical edges of $U$; moreover, if $x'$ is the initial endpoint of
$\ell'$ and $y'$ is the terminal point of $\ell'_k$ then there exists
also a chain of critical edges of $U$ connecting $x''$ and $x'$ and a
chain of critical ledges of $U$ connecting $y'$ and $y''$. This amounts
to the following claim: \emph{if $x$ and $y$ are connected with a chain
of leaves of $\lam_n$ and $x'', y''\in \partial U$ are such that
$\si_d(x'')=x, \si_d(y'')=y$ then there exists a chain of leaves of
$\lam_{n+1}$ contained in $U$ and connecting $x''$ and $y''$.}

Now, let $\hlam$ be another pullback lamination generated by $\mathcal
C$; denote by $\hlam_n$ the finite laminations whose sequence gives
rise to $\hlam$. For each $n$ the finite pullback lamination $\hlam_n$
generates the equivalence relation $\happrox_n$ in the same fashion as
above for $\approx_n$. We claim that $\approx_n=\happrox_n$. Indeed, it
is clear that $\approx_0=\happrox_0$ (after all,
$\lam_0=\hlam_0=\{\oc_{r+1}, \dots, \oc_{d-1}\}$). Assume that
$\approx_n=\happrox_n$ and prove that $\approx_{n+1}=\happrox_{n+1}$.
It suffices to show that if $x \approx_{n+1} y$, then $x
\happrox_{n+1}\, y$.  By definition, to this end it is enough to show
that if $\ell=\ol{xy}$ is a leaf of $\lam_{n+1}$ then there exists a
chain of leaves of $\hlam_{n+1}$ connecting $x$ and $y$.

First, suppose that $\ell$ is critical. Then by definition $\ell\in
\lam_0=\hlam_0$ and we are done. Second, suppose that $\ell$ is not
critical. Then $\ell\subset U$ for some $\mathcal C$-domain $U$, and
$\si_d(\ell)$ is a leaf of $\lam_n$. By induction there exists a chain
of leaves of $\hlam_n$ connecting $\si_d(x)$ and $\si_d(y)$. By the
above claim there exists a chain of leaves of $\hlam_{n+1}$ connecting
$x$ and $y$ (and contained in $U$). Hence $x \happrox_{n+1} y$ as
desired. By induction this implies that $\approx_n=\happrox_n$ for
every $n$. By definition of the laminational equivalence relation
determined by a proper invariant lamination it then follows that
$\approx_\hlam=\approx_{\lam(\mathcal C)}$ as desired.

(4) By %the given and by
definition of invariant lamination both $\hlam$ and $\lam(\mathcal C)$
have leaves contained in each $\approx_\mathcal C$-class. Therefore any
limit leaf of $\lam(\mathcal C)$ is a limit leaf of $\hlam$ and vice
versa; moreover, any such limit leaf is the limit of a sequence of
convex hulls of $\approx_\mathcal C$-classes. The fact that $\hlam$
contains at least one pullback of each $\oc_i, r+1\le i\le d-1$ inside
each $\mathcal C$-maximal pullback of $\oc_i$ is immediate.

(5) The leaves $\oc_j$ with $r+1\le j\le d-1$ either are leaves of
$\lam_\sim$ or are contained in finite gaps of $\lam_\sim$. Add them to
$\lam_\sim$ and then pull them back as leaves of $\lam_\sim$ or inside
pullbacks of gaps of $\lam_\sim$; do it according to
Definition~\ref{d:pullam}.
%following the rules introduced above when we
%constructed $\lam(\mathcal C)$.
In this way, we can get an invariant pullback prelamination generated
by $\mathcal C$. When we close it, by (4) the new leaves added all
belong to $\lam_\sim$. Thus, $\lam(\mathcal C)$ consists of either
leaves of $\lam_\sim$ or leaves contained inside finite gaps of
$\lam_\sim$. By definition of $\approx_{\mathcal C}$, this implies the
desired.
\end{proof}

The following definition is central for this section.

\begin{dfn}[Pullback laminational equivalence relations]\label{d:pull}
%The invariant geodesic lamination $\lam(\mathcal C)$ from \Cref{l:pull}
%is called the \emph{pullback geodesic lamination generated by $\mathcal
%C$}\index{pullback geodesic lamination generated by a full critical
%collection}.
The laminational equivalence relation $\approx_{\mathcal C}$ is called
the \emph{laminational equivalence relation generated by $\mathcal
C$}\index{laminational equivalence relation!generated by a full
critical collection}.
\end{dfn}

In the next lemma, we study possible differences between certain
invariant geodesic pullback laminations.

\begin{lem}\label{l:big-deal}
In the notation of Lemma~\ref{l:pull} the following holds.

\begin{enumerate}

\item There may exist leaves of $\lam(\mathcal C)$ which intersect
    the interiors of finite gaps of $\lam_{\approx_{\mathcal C}}$;
    all such leaves are pullbacks of the leaves $\oc_{r+1}, \dots,
    \oc_{d-1}$. This is the only situation when a leaf of
    $\lam(\mathcal C)$ does not belong to $\lam_{\approx_{\mathcal
    C}}$.

\item Suppose that a leaf $\ell$ of $\lam_{\approx_{\mathcal C}}$ is
    a common edge of an infinite Fatou gap and a finite gap. Then
    either $\ell$ is a pullback of a leaf $\oc_j$ with $r+1\le j\le
    d-1$, or $\ell$ does not belong to $\lam(\mathcal C)$. The latter
    is the only situation when a leaf from $\lam_{\approx_{\mathcal
    C}}$ does not belong to $\lam(\mathcal C)$.

\item A leaf $\ell$ of $\lam(\mathcal C)$ that is not a pullback of
    one of the leaves $\oc_i$ with $r+1\le i\le d-1$ (in particular,
    this holds if $\ell$ is periodic) is non-isolated in
    $\lam(\mathcal C)$, non-isolated in $\lam_{\approx_{\mathcal
    C}}$, and belongs to $\lam_{\approx_{\mathcal C}}$.
\end{enumerate}

\end{lem}

\begin{proof}
(1) If a leaf $\ell$ of $\lam(\mathcal C)$ is not a leaf of
$\lam_{\approx_{\mathcal C}}$, then there must exist a finite gap of
$\lam_{\approx_{\mathcal C}}$ containing $\ell$. Moreover, since in
that case $\ell$ is isolated in $\lam(\mathcal C)$, it itself must be a
pullback of a leaf $\oc_j$ with $r+1\le j\le d-1$.

(2) Suppose that a leaf $\ell=\ol{ab}$ of $\lam_{\approx_{\mathcal C}}$
is not a leaf of $\lam(\mathcal C)$. By definition there exists a
finite chain of leaves of $\lam(\mathcal C)$ connecting $a$ and $b$.
Hence $\ell$ is an edge of a finite gap $G$ of $\lam_{\approx_{\mathcal
C}}$. If $\ell$ is not isolated in $\lam(\mathcal C)$ from the outside
of $G$, then it itself must belong to $\lam(\mathcal C)$, a
contradiction. If there is a finite gap of $\lam(\mathcal C)$ outside
of $G$ that shares $\ell$ with $G$, then this gap and $G$ must be
united into one bigger gap of $\lam_{\approx_{\mathcal C}}$, a
contradiction to the definition of $\approx_{\mathcal C}$. Hence there
is a Fatou gap $U$ of $\approx_{\mathcal C}$ such that $\ell$ is an
edge of $U$. Since $\ell$ is not a leaf of $\lam(\mathcal C)$, it
cannot be a pullback of a leaf $\oc_j, r+1\le d-1$. On the other hand,
if a leaf $\ell$ of $\lam_{\approx_{\mathcal C}}$ is a common edge of
an infinite Fatou gap and a finite gap, then, in case this leaf is not
a pullback of a leaf $\oc_j, r+1\le d-1$, it follows from definitions
that it cannot be a leaf of $\lam(\mathcal C)$ as desired.

(3) If a leaf $\hell$ does not belong to the union of the grand orbits
of leaves $\oc_i$ with $r+1\le i\le d-1$, then by definition it is
approximated by pullbacks of these leaves disjoint from $\hell$ (as the
endpoints of leaves $\oc_i$ with $r+1\le i\le d-1$ are non-periodic,
only finitely many their pullbacks share an endpoint). By
\Cref{l:pull}, each leaf $\oc_i, r+1\le i\le d-1$ is contained in the
convex hull $G$ of a $\approx_{\mathcal C}$-class where $G$ is a leaf
or a finite gap. Hence $\hell$ is a limit leaf of leaves of
$\lam_{\approx_{\mathcal C}}$. This implies that $\hell$ is a leaf of
$\lam_{\approx_{\mathcal C}}$.
\end{proof}

Let us study periodic Fatou gaps of $\lam_{\approx_{\mathcal C}}$ and
$\lam(\mathcal C)$.

\begin{lem}\label{l:gener-equiv}
Let $U$ be a periodic Fatou gap of $\approx_{\mathcal C}$ of degree
$s>1$. Then all finite periodic gaps $G$ attached to $U$ at edges
$\ell$ of $U$ are fixed return so that $U$ is of non-rotational type;
moreover, all other edges of gaps $G$ which are attached to $U$ are
non-isolated in $\lam(\mathcal C)$. Each critical chord $\oc_i$ with
$i\le r$ is contained in a periodic critical Fatou gap $W$ of
$\lam(\mathcal C)$ of degree greater than one. Finally, there are no
preperiodic critical Fatou gaps of $\approx_{\mathcal C}$ (or of
$\lam(\mathcal C)$).
\end{lem}

\begin{proof}
Let $G$ be a finite periodic gap attached to $U$ at an edge $M$. Then
$M$ is isolated in $\lam_{\approx_{\mathcal C}}$ and periodic. This
implies that $M$ is not a leaf of $\lam(\mathcal C)$. All other edges
of $G$ must be non-isolated in $\lam_{\approx_{\mathcal C}}$ because
otherwise these edges are not leaves of $\lam(\mathcal C)$. This shows
that vertices of $G$ are not connected with finite chains of leaves of
$\lam(\mathcal C)$, and $G$ is not a gap of $\lam_{\approx_{\mathcal
C}}$, a contradiction. Since all other edges $\ell\ne M$ of $G$ are
non-isolated in $\lam_{\approx_{\mathcal C}}$, then, clearly, $G$ must
be fixed return. Similarly, $U$ cannot have periodic edges of flip
type. This implies that $U$ is of non-rotational type.

By \Cref{l:pull}, the chord $\oc_i$ (with $i\le r$) is not a leaf of
$\lam(\mathcal C)$. Hence $\oc_i$ is contained in a critical gap of
$\lam_{\approx_{\mathcal C}}$. Since critical gaps of laminational
equivalence relations can only have periodic vertices if they are
infinite, it follows that $\oc_i$ is contained is a critical periodic
Fatou gap $W$. Suppose that $W$ is not periodic. Let $\oc_i$ have a
periodic endpoint $a$ of period $m$. Since $W$ is not periodic then $a$
cannot be non-isolated on $\bd(W)$ from either side. Hence $a$ is an
edge of a periodic leaf which implies that $W$ is actually periodic.
Thus,  $\oc_i$ is contained in a periodic critical Fatou gap $W$ of
$\lam(\mathcal C)$ of degree greater than one. The last claim is left
to the reader.
\end{proof}

The following is a corollary of our results.

\begin{cor}\label{c:finite}
Let $\mathcal C$ be a $0$-admissible critical collection of chords.
Then $\lam_{\approx_{\mathcal C}}\subset \lam(\mathcal C)$, all
critical sets of $\lam_{\approx_{\mathcal C}}$ are finite, all its
infinite gaps (if any) are periodic Siegel gaps and their degree one
preimages. If $\ell=\ol{ab}$ is a chord such that all its forward
images do not cross leaves from $\mathcal C$ and one another then $a
\approx_{\mathcal C} b$.
\end{cor}

\begin{proof}
The proof of all the claims but the last one is left to the reader.
Now, let $\ell=\ol{ab}$ be a chord such that all its forward images do
not cross leaves from $\mathcal C$. Let us show that then $a
\approx_{\mathcal C} b$. Observe that if $\ell$ is non-disjoint from a
leaf $\hell$ of $\lam(\mathcal C)$ then the corresponding images
$\si_d^m(\ell)$ and $\si_d^m(\hell)$ are also non-disjoint (this is
because both leaves do not have images crossing critical leaves from
$\mathcal C$). Now, suppose that $\ell$ crosses infinitely many
pullbacks of critical leaves from $\mathcal C$. Since by the assumption
images of $\ell$ do not cross leaves from $\mathcal C$ this would imply
that for each $\si_d^N$-pullback of a leaf of $\mathcal C$ crossing
$\ell$ the leaf $\si_d^N(\ell)$ has a common endpoint with a leaf from
$\mathcal C$. Since we assume the existence of infinitely many
pullbacks of critical leaves from $\mathcal C$ crossing $\ell$, then
there are some leaves from $\mathcal C$ that have periodic endpoints, a
contradiction.

Thus, there are only finitely many pullbacks of critical leaves from
$\mathcal C$ that cross $\ell$. This implies that in fact there is a
finite string of gaps of $\approx_{\mathcal C}$, say, $G_1, G_2, \dots,
G_k$ such that $a$ is a vertex of $G_1$, $b$ is a vertex of $G_k$, and
each two gaps $G_i$ and $G_{i+1}$ share a common edge that crosses
$\ell$. This implies that at least one of these gaps is infinite and
hence is a preimage of a periodic Siegel gap of $\approx_{\mathcal C}$
in which $\ell$ connects two edges. Hence a forward image $\ell'$ of
$\ell$ connects edges of a periodic Siegel gap $U$. Since the canonical
map that collapses edges of gaps semiconjugates the first return map
and an irrational rotation, it follows that the image of $\ell'$ under
this semiconjugacy will cross one of its eventual images, a
contradiction. Hence the situation described above is impossible and $a
\approx_{\mathcal C} b$ as desired.
\end{proof}

\begin{dfn}[Tuning of laminations]\label{d;tune}
Suppose that $\lam$ is an invariant geodesic lamination. Suppose that
there exists a periodic Fatou gap $U$ of $\lam$ of degree greater than
one. Finally, suppose that there exists an invariant geodesic
lamination $\hlam\supset \lam$ such that $\hlam\sm \lam$ consists of
leaves contained in the grand orbit of $U$. Then we say that $\hlam$
\emph{tunes $\lam$ (in the grand orbit of $U$}). If $\lam$ generates a
laminational equivalence relation $\approx_\lam$ while $\hlam$
generates a laminational equivalence relation $\approx_\hlam$, then we
say that $\approx_\hlam$ \emph{tunes $\approx_\lam$}.
\end{dfn}

\Cref{c:finite-1} deals with more specific pullback geodesic
laminations.

\begin{cor}\label{c:finite-1}
Let $\mathcal C$  be an admissible critical collection with $r=1$; let
$\oc_1=\ol{a_1b_1}$ where $a_1$ is of period $n$. %
Then the following holds.

\begin{enumerate}

\item The chord $\oc_1$ is a subset of a periodic quadratic critical
    Fatou gap $V$ of $\lam(\mathcal C)$ of period $m$ such that
    $n=mk$ is a multiple of $m$; the gap $V$ contains a periodic
    quadratic critical Fatou gap $U$ of $\approx_{\mathcal C}$ of
    period $m$.

\item Except for $U$, all critical sets of $\approx_{\mathcal C}$ are
    finite.  All infinite gaps of $\approx_{\mathcal C}$ are either
    in the grand orbit of  $U$ or in the grand orbits of (possibly
    existing) periodic Siegel gaps.
\end{enumerate}

\noindent Moreover, there exists a laminational equivalence relation
$\sim$ that tunes $\approx_{\mathcal C}$ inside the grand orbit of\,
$U$, has a critical quadratic periodic Fatou gap $T\subset U$, and is
such that $a_1$ is a refixed vertex of $T$ or of a finite periodic
gap\, $G$ attached to $T$ at the refixed edge of\, $T$.
\end{cor}

Observe that the gap $V$ may contain finite concatenations of edges.
The gap $U\subset V$ is obtained from $V$ by replacing every such
maximal concatenation of edges with a single leaf (an edge of the
convex hull of the concatenation).

\begin{proof} (1) This claim follows immediately from \Cref{l:gener-equiv}.

(2) This claim follows from \Cref{l:pull}.

To prove the last claim, consider $\si_d^m|_V$. Since the point $a_1$
is an $n$-periodic vertex of $V$, then $a_1$ is of period $k$ under the
action of $\si_d^m$. Apply the standard monotone semiconjugacy $\psi$
between $\si_d^m|_V$ and $\si_2$ (the map $\psi$ simply collapses edges
of $V$ to points of $\uc$). It follows that $\psi(a_1)=x$ is a
$k$-periodic point of $\si_2$. Now, it is well-known that there exists
a $\si_2$-invariant laminational equivalence relation $\simeq$ such
that there exists a critical periodic Fatou gap $W_2$ of $\simeq$ and
$x$ is an endpoint of the major edge of $W_2$ (in particular, $x$ is a
refixed vertex of $W_2$).

Using $\psi$, we can lift the equivalence relation $\simeq$ to the
entire $V$. This gives rise to a $\si_d^m$-equivalence relation
$\sim^V$ on $V$ and to the corresponding $\si_d^m$-invariant geodesic
sublamination $\lam^V_\sim$ of $V$. Note that we write $\lam^V_\sim$
instead of $\lam_{\sim^V}$. We can then pull this back under
the action of $\si_d$ %to complement $\approx_{\mathcal C}$ and
to obtain a laminational equivalence relation $\sim$ on the entire
circle. Points that are $\approx_{\mathcal C}$-equivalent will be
declared $\sim$-equivalent while otherwise two points are declared
$\sim$-equivalent if the chord connecting them lies in a pullback of a
finite gap of $\lam^V_\sim$. This defines a $\si_d$-invariant
laminational equivalence relation $\sim$ on the entire circle.
Basically, we tune $V$ and then pull back this tuning. The equivalence
relation $\sim$ is closed because pullbacks of leaves of $\lam^V_\sim$
can only accumulate on leaves of $\lam^V_\sim$ or in the boundary of
the grand orbit of $V$.

Let us show that $\sim$ satisfies all the necessary conditions.
Clearly, all of them are automatically satisfied except for the claims
concerning $a_1$. Now, the construction implies that the
$\psi$-preimage of the gap $W_2$ is a stand-alone Fatou gap $H\subset
V$ of period $n$. Observe that $H$ is not necessarily a gap of $\sim$
because some vertices of $W_2$ may have non-degenerate
$\psi$-preimages. However $H$ contains a gap $T$ of $\sim$, and one can
obtain $T$ by completing convex hulls of finite chains of edges on the
boundary of $H$ with extra edges. The relation of $H$ and $T\subset H$
is like the relation of $V$ and $U\subset V$. Now, if the major $M_2$
of $W_2$ has endpoints with degenerate $\psi$-preimages then
$\psi^{-1}(M_2)=M$ is just a leaf of $\sim$ and $a_1$ is an endpoint of
$M$ as desired. If, however, one or both endpoints of $M_2$ have
non-degenerate $\psi$-preimages, then the $\psi$-preimage of $M_2$ is
the desired finite gap $G$.
\end{proof}

\begin{thm}\label{t:finite}
Let $\qcp=\{C_1, \oc_2=\ol{a_2b_2}$, $\ldots$,
$\oc_{d-1}=\ol{a_{d-1}b_{d-1}}\}$ be a geolaminational quadratically
critical portrait such that $C_1$ is a critical quadrilateral and
$\oc_2$, $\ldots$, $\oc_{d-1}$ are critical leaves with non-periodic
endpoints. Then there exists a non-empty laminational equivalence
relation $\sim$ such that $a_2\sim b_2$, $\dots$, $a_{d-1}\sim b_{d-1}$
while $C_1$ is such that either {\rm (1)} all its vertices are
$\sim$-equivalent and non-periodic, or {\rm (2)} $C_1$ has a periodic
edge, and if $\ell'$, $\ell''$ are diagonals of $C_1$ and $\mathcal
T'=(\ell',$ $\oc_2,$ $\dots,$ $\oc_{d-1})$, $\mathcal T''=(\ell'',$
$\oc_2,$ $\dots,$ $\oc_{d-1})$, then $\lam(\mathcal T')=\lam(\mathcal
T'')=\lam$, $\approx_{\mathcal T'}=\approx_{\mathcal T''},$ and $C_1$
is contained in a critical quadratic periodic Fatou gap of $\lam$.
\end{thm}

\begin{proof}
If $C_1$ has a diagonal $\ell$ with non-periodic endpoints, then we add
$\ell$ to $\oc_2, \dots, \oc_{d-1}$ to form a collection $\mathcal C$.
By \Cref{c:finite}, there exists a laminational equivalence relation
$\approx_{\mathcal C}$ such that $\ell$, $\oc_2$, $\dots$, $\oc_{d-1}$
connect pairs of $\approx_{\mathcal C}$-equi\-va\-lent points.
Moreover, by the assumptions and by \Cref{c:finite}, all vertices of
$C_1$ are $\approx_{\mathcal C}$-equivalent and non-periodic, as
desired. Otherwise we may assume that there is an edge $\ol{ab}$ of
$C_1$ such that $a$ and $b$ are periodic. Choose a diagonal $\ell'$ of
$C_1$, say, the one that contains $a$. Suppose that $a$ is of period
$n'$. Set $\mathcal T'=\{\ell', \oc_2, \dots, \oc_{d-1}\}$. By
\Cref{l:pull} there exist an invariant geodesic pullback lamination
$\lam(\mathcal T')$ and the associated laminational equivalence
relation $\approx_{\lam_{\mathcal T'}}$. Moreover, by
\Cref{c:finite-1}, the leaf $\ell'$ is contained in a periodic critical
Fatou gap $V'$ of $\lam(\mathcal T')$. On the other hand, $V'$ contains
a Fatou gap $U'$ of $\approx_{\mathcal T'}$. We may assume that the
period of $U'$ and $V'$ is $m'$ while $n'=m'k'$.

Let us discuss the location of $b$ with respect to this picture; we may
assume that $b\ne a$ and that $C_1$ is non-degenerate. We claim that
$b$ is a vertex of $V'$. Indeed, suppose otherwise. Then either
$\ol{ab}$ intersects $V'$ at only one point $a$, or $\ol{ab}$ crosses
an edge of $V'$. Since all edges of $V'$ are either pullbacks of leaves
$\oc_i$ with $2\le i\le d-1$, or limits of such pullbacks, it follows
that in either case there exists a pullback $N$ of one of the leaves
$\oc_2=\ol{a_2b_2}$, $\dots$, $\oc_{d-1}=\ol{a_{d-1}b_{d-1}}$ that
crosses $\ol{ab}$. Let us show that this leads to a contradiction.

Observe that Thurston's pullback construction implies the existence of
a pullback geodesic lamination $\lam'$ that contains \textbf{all}
leaves from $\mathcal T'$. This geodesic lamination strictly contains
$\lam(\mathcal T')$ because (1) it must contain all pullbacks of leaves
$\oc_2$, $\ldots$, $\oc_{d-1}$ that generate $\lam(\mathcal T')$, and
(2) by \Cref{l:per-crit-appr}, the leaf $\ell'$ is not a limit leaf of
$\lam'$. By definition, $\lam^{\qcp}$ with quadratically critical
portrait $\qcp$ and $\lam'$ with quadratically critical portrait
$\mathcal T'$ are essentially equal. Since the leaves $\ol{ab}$ and $N$
are linked, they will either (1) stay linked under any iteration of
$\si_d$, or (2) there will exist the minimal $i+1$ such that
$\si_d^{i+1}(\ol{ab})$ and $\si_d^{i+1}(N)$ are not linked while
$\si_d^i(\ol{ab})$ and $\si_d^i(N)$ are linked.

Now, in case (1), we will eventually obtain that the image of $N$ that
coincides with $\oc_j$ for some $2\le j\le d-1$ is linked with an image
of $\ol{ab}$. This contradicts the fact that $\qcp$ is geolaminational.
Consider case (2). Then $\si_d^i(N)$ cannot be critical again because
$\qcp$ is geolaminational. Hence $\si_d^i(N)$ is precritical. This
implies that $\si_d^{i+1}(N)$ is non-degenerate, (pre)critical, and has
a periodic endpoint. Applying a suitable iteration of $\si_d$, we will
observe that a certain image of $N$ is a leaf $\oc_j$ with $2\le j\le
d-1$ with a periodic endpoint, a contradiction. Thus, $b$ is a vertex
of $V'$, which implies that $C_1\subset V'$.

Clearly, the same construction can be implemented for $\mathcal T''$
based upon the other diagonal of $C_1$ passing through $b$. We may
assume that the period of $b$ is $n''$. It leads to a invariant
geodesic pullback lamination $\lam(\mathcal T'')$ and the associated
laminational equivalence relation $\approx_{\lam_{\mathcal T''}}$.
Moreover, $\ell''$ is contained in a periodic critical Fatou gap $U''$
of $\approx_{\mathcal T''}$, which is contained in the corresponding
Fatou gap $V''$ of $\lam(\mathcal T'')$. We may assume that the period
of $U''$ and $V''$ is $m''$ while $n''=m''k''$. As before for $\mathcal
T'$ we will also have that $C_1\subset V''$.

We need to show that $\approx_{\mathcal T'}=\approx_{\mathcal T''}$. To
this end, observe that, since $C_1\subset V'$, all pullbacks of leaves
$\oc_2$, $\dots$, $\oc_{d-1}$ chosen for $\approx_{\mathcal T'}$ can be
described as pullbacks outside of $C_1$, and the same can be said about
pullbacks of $\oc_2$, $\dots$, $\oc_{d-1}$ chosen for
$\approx_{\mathcal T''}$. Therefore, these pullbacks coincide. Since
they are dense in $\lam(\mathcal T')$ and in $\lam(\mathcal T'')$, we
have $\lam(\mathcal T')=\lam(\mathcal T'')$, which implies the other
claims in the end of the lemma.
\end{proof}

\subsection{The space of $\si_d$-invariant geodesic laminations
compatible with a given generic collection of $d-2$ critical chords}

We will now describe the results of \cite{bopt15} omitting technical
details. Consider cubic geodesic laminations $\lam$ with a critical
leaf $D=\ol{ab}$ whose endpoints are non-periodic. Without loss of
generality, we may assume that $(a, b)$ is a positively oriented circle
arc of length $\frac13$. Then there are several possibilities
concerning critical sets of $\lam$. First, $\lam$ can have a finite
critical set $C\ne D$ contained in the convex hull of the circle arc
$[b, a]$. By properties of invariant geodesic laminations, $C$ is a gap
or a leaf, on which $\si_3$ acts two-to-one (unless $D$ is an edge of
$C$ and so the point $\si_3(D)$ has all three of its preimages in $C$).
Thus, if $C$ is finite, then there are two cases. First, $C$ can be a
$2n+1$-gon with $D$ being one of its edges such that one can break down
all its remaining edges into pairs of ``sibling edges'' (one can say
that the ``sibling edge'' of $D$ is the vertex of $C$ not belonging to
$D$ and with the same image as $D$). Second, $C$ can be a $2m$-gon such
that $D$ is not an edge of $C$; in this case $\si_3|_C$ is two-to-one.

Now, $C$ could also be an infinite gap. Then it may be a periodic Fatou
gap of period $k$ and degree $2$ (in this case $D$ may well be an edge
of $C$). Otherwise $C$ may be preperiodic; then it cannot be eventually
mapped onto a periodic gap of degree greater than one because we deal
with the cubic case. Hence there must exist a periodic Siegel gap $U$
with $D$ being one of its edges and an infinite gap $C$ such that
$\si_3|_{\bd(C)}$ is two-to-one and $C$ eventually maps onto $U$. In
other words, an invariant geodesic lamination with leaf $D$ and of
capture type must have a periodic Siegel gap $U$ and a critical gap $C$
that eventually maps onto $U$.

Now, let $\prnp_3(D)$ be the family of all cubic geodesic laminations
with a critical leaf $D$ with non-periodic endpoints \emph{except for
geodesic laminations of capture type}. If $\lam\in \prnp_3(D)$, then a
quadratically critical portrait $\qcp=(Q, D)$ is said to be
\emph{privileged for $\lam$} if $Q\subset C$, where $C\ne D$ is defined
above, and the additional requirement mentioned below is fulfilled. By
the above analysis, either $C$ is finite, or $C$ is a periodic Fatou
gap of degree two and period $k$. In the former case, the critical
quadrilateral $Q\subset C$ can be arbitrary. In the latter case, we
require that $Q$ be a collapsing quadrilateral that is the convex hull
of a (possibly degenerate) edge $\ell$ of $C$ of period $k$ and its
sibling edge $\ell^*$ of $C$.

In \cite{bopt15} we show that for each $\lam\in\prnp_3(D)$ there are
only finitely many privileged quadratically critical portraits. Let
$\Ss_D$ denote the collection of all privileged for $\lam$
quadratically critical portraits $(Q, D)$ with $D$ as the second
element. To each such quadratically critical portrait $(Q, D)$ we
associate its \emph{minor} (a chord or a point) $\si_3(Q)\subset
\cdisk$. For each such chord we identify its endpoints, extend this
identification by transitivity and define the corresponding equivalence
relation $\simeq_D$ on $\uc$. The main result of \cite{bopt15} is that
$\simeq_D$ is itself a laminational equivalence (non-invariant!) whose
quotient is a parameterization of $\prnp_3(D)$.

The tools used in \cite{bopt15} are based upon accordions and smart
criticality rather than upon Thurston's tools \cite{thu85}. Indeed, the
main technical lemma used in \cite{thu85} is the Central Strip Lemma
showing how ``long'' (based upon circle arcs longer than $\frac13$)
leaves of invariant geodesic laminations may enter the central strips
between themselves and their siblings. The lemma has a multitude of
consequences, including the fact that there are no wandering (i.e.
non-preperiodic and non-precritical) triangles of quadratic invariant
geodesic laminations, and the construction of $\qml$. However, the
Central Strip Lemma fails already in the cubic case (in particular,
there are wandering triangles of cubic invariant geodesic laminations
\cite{bo08}). This shows the necessity of using new techniques in
\cite{bopt15}.

In order to generalize the results of \cite{bopt15} to the degree $d$
case, we introduce appropriate spaces of laminations analogous to
$\prnp_3(D)$; these spaces depend not on one critical leaf but on a
suitable \emph{collection} of critical leaves.

\begin{dfn}\label{d:d-space}
Fix a collection $\mathcal Y$ of $d-2$ pairwise disjoint critical
chords with non-(pre)periodic endpoints and pairwise disjoint forward
orbits. Define a space $\mathbb L(\mathcal Y)$ of \emph{invariant
geodesic laminations} as follows: $\lam\in \mathbb L(\mathcal Y)$ if
$\lam$ is generated by a \emph{laminational equivalence relation}
$\sim$ such that the endpoints of each critical chord from $\mathcal Y$
are $\sim$-equivalent, and $\lam$ has no gaps of Siegel capture type.
\end{dfn}

Let $\mathcal Y^+$ be the union of all critical chords from $\mathcal
Y$. It is easy to see that there is a unique component $A(\mathcal
Y)=A$ of $\cdisk\sm \mathcal Y^+$ such that $\si_d|_{\bd(A)}$ is
two-to-one except for its critical boundary edges (this map is
one-to-one in the same sense on all other components of $\cdisk\sm
\mathcal Y^+$). Indeed, $d-2$ critical chords of $\mathcal Y$ split the
disc into $d-2$ connected sets each of which has the boundary whose
intersection with the circle maps onto the entire circle in almost
one-to-one fashion (except for the endpoints of boundary edges that are
critical chords). Hence the length of each such intersection is
$\frac{a}{d}$ for some $a>0$. Clearly, this implies that $d-3$ of them
has the boundary whose intersection with the circle is of length
$\frac{1}{d}$ while one of them has the boundary whose intersection
with the circle is of length $\frac{2}{d}$. This is exactly the desired
component $A$.

Denote by $\oy_1$, $\ldots$, $\oy_k$ all critical chords from $\mathcal
Y$ contained in the boundary of $A$ (clearly, $1\le k\le d-2$).
Consider $\oy_1$; there exists exactly one point $a\in \bd(A)\sm \oy_1$
with $\si_d(a)=\si_d(\oy_1)$ (a chord $\oy_t\subset \bd(A)$ with
$\si_d(\oy_1)=\si_d(\oy_t)$ would contradict the assumption of pairwise
disjointness of forward orbits of critical chords from $\mathcal Y$)
while other points of $\bd(A)$ have images disjoint from
$\si_d(\oy_1)$. The same holds for other boundary critical chords of
$A$. For any other component $T$ of $\cdisk\sm \mathcal Y^+$ it is easy
to see that except for the collapsing of the boundary chords of $T$ all
other points of $\bd(T)$ map forward in the one-to-one fashion.

\begin{cor}\label{l:ld-nonempty}
The family %of invariant geodesic laminations
$\mathbb L(\mathcal Y)$ is non-empty.
\end{cor}

\begin{proof}
Insert a critical chord $\oc$ in $A$ so that both endpoints of $\oc$
are non-periodic, and $\oc$ is disjoint from all chords from $\mathcal
Y$. Then \Cref{c:finite} implies the existence of the desired invariant
laminational equivalence $\sim$ and the geodesic lamination generated
by $\sim$ .
\end{proof}

Let us study the critical sets of geodesic laminations from $\mathbb
L(\mathcal Y)$.

\begin{lem}\label{l:good-class}
Let $X$ be a critical set of $\lam_\sim\in \mathbb L(\mathcal Y)$. If
$X$ is infinite, then $X\subset \ol{A}$ is a periodic quadratic Fatou
gap and all other critical sets of $\lam$ are finite and
non-preperiodic. If $X$ is finite and preperiodic, then $X\subset A$
and all critical sets of $\lam$ are finite.
\end{lem}

\begin{proof}
By definition, the only possibly existing infinite critical set $X$ of
a geodesic lamination $\lam\in \mathbb L(\mathcal Y)$ is a critical
Fatou gap contained in $A$. Clearly, $X$ cannot be a preperiodic gap
that maps onto a periodic Fatou gap of degree greater than one because
then there will be at least two infinite critical sets of $\lam$
(indeed, the orbit of a periodic Fatou gap must contain a critical
Fatou gap \cite{bl02}). On the other hand, by definition of the family
of laminations $\mathbb L(\mathcal Y)$, the set $X$ cannot be a
pullback of a periodic Siegel gap. Hence, the only possibility is that
$X$ is a periodic quadratic Fatou gap contained in $A$ as desired. The
rest of the lemma is immediate.
\end{proof}

Let us define tags for geodesic laminations from $\mathbb L(\mathcal
Y)$. Our approach is different from Thurston's: instead of considering
\emph{minor leaves}, or \emph{minors}, of geodesic laminations we work
with their \emph{minor sets} basically defined as the images of
critical sets.

\begin{lem}\label{l:crit-set}
If $\lam_\sim\in \mathbb L(\mathcal Y)$, then there is a unique
critical set $C_\sim$ of $\lam_\sim$ containing a critical chord $\oc$,
where $\oc\subset A(\mathcal Y)$ except for the endpoints. Any infinite
gap non-disjoint from $A$ is contained in $\ol{A}$. Finally, if $x\in
\si_d(C_\sim)\cap \uc$ then the entire set $\si_d^{-1}(x)\cap \ol{A}$
is contained in $C_\sim$.
\end{lem}

\begin{proof}
Clearly, at least one critical set $C$ of $\lam_\sim$ contains a
critical chord $\oc$ contained in $A(\mathcal Y)$ except for the
endpoints. Let us show that this set $C$ is unique. Indeed, it is easy
to see that any two critical chords contained in $\ol{A}$ and
non-disjoint from $A$ are linked. Therefore, two \emph{distinct
critical} sets $C_1$ and $C_2$ of $\lam_\sim$ with the properties from
the lemma cannot exist. By \Cref{l:good-class}, if $C_\sim$ is infinite
then $C_\sim$ is a periodic quadratic Fatou gap.

Let $U$ be an infinite gap non-disjoint from $A$. Since all boundary
chords of $A$ are contained in finite gaps of $\lam_\sim$ or are
themselves leaves of $\lam_\sim$, it follows that $U\subset \ol{A}$. In
particular, this holds for $C_\sim$ if it is an infinite gap.

Recall that  $\oy_1$, $\ldots$, $\oy_k$ are all critical chords from
$\mathcal Y$ contained in the boundary of $A$. We claim that for each
$\oy_j, 1\le j\le k$ either $C_\sim$ is disjoint from $\oy_j$ or
$\oy_j\subset C_\sim$. This is clear if $C_\sim$ is a finite gap or
leaf. Let $C_\sim$ be a periodic quadratic gap. If $\oy_1\cap
C_\sim=\{z\}$ is a singleton, then the convex hull $H$ of the
$\sim$-class of $z$ contains $\oy_1$ and the edge $\ell$ of $\C_\sim$
with endpoint $z$. Since $z$ is not (pre)periodic by the assumptions,
then $H$ cannot be (pre)periodic. By \Cref{l:cripe} this implies that
$\ell$ must be (pre)critical. Thus, $z$ is an endpoint of $\oy_1$ that
eventually maps to an endpoint of a critical leaf of $\lam_\sim$, i.e.
an endpoint of a leaf $\oy_s$, a contradiction with our assumptions.
Hence for each $\oy_j$ either $C_\sim$ is disjoint from $\oy_j$ or
$\oy_j\subset C_\sim$.

The fact that $C_\sim$ is critical implies that $\si_d|_{C_\sim\cap
\ol{A}}$ is in fact the composition of the map that collapses all
boundary chords of $C_\sim\cap \ol{A}$ to points and then an exactly
two-to-one map. Let $x\in \si_d(C_\sim)\cap \uc$. If $x$ is not the
image of one of the boundary chords of $A$, then it has exactly two
preimages in $\ol{A}$, and both must belong to $C_\sim$. Otherwise, set
$x=\si_d(\oy_i)$, where $1\le i\le k$. Then, by the above, there must
still exist one more point in $C_\sim\cap A$ that maps to $x$. This
proves that $\si_d^{-1}(x)\cap \ol{A}$ is contained in $C_\sim$, as
desired.
\end{proof}

Observe that, while the set $C_\sim$ is typically contained in
$\ol{A}$, some parts of it may ``stick out'' of $\ol{A}$. For example,
it may happen that $\oy_1$ is a diagonal of an all critical
quadrilateral that has one vertex in $A$ and the other one in a
component of $\cdisk\sm \mathcal Y^+$ adjacent to $A$ at $\oy_1$;
clearly, the same can be said about $\oy_2, \dots, \oy_k$. In fact, any
critical set $C_\sim$ not contained in $\ol{A}$ is finite and must
contain $\oy_i$ as a chord for some $i$. The set $C_\sim$ is important
in defining \emph{minor sets} of geodesic laminations from $\mathbb
L(\mathcal Y)$. Observe that any geodesic lamination from $\mathbb
L(\mathcal Y)$ admits legal modifications. Indeed, recall that legal
modifications are well-defined if \emph{no critical gap is mapped to a
fixed return gap attached to a periodic critical Fatou $U$ at its
refixed edge $M$ or to $M$ itself}. However, by \Cref{l:good-class} in
case $U$ exists there are no (pre)periodic critical sets, and the
desired follows.

\begin{dfn}[Minor sets of laminations from $\mathbb
L(\mathcal Y)$]\label{d:minor-set} For $\lam_\sim\in \mathbb L(\mathcal
Y)$ we define the \emph{minor set $m_\sim$} of $\lam_\sim$ as follows.

\begin{enumerate}

\item If $C_\sim$ is finite, set $m_\sim=\si_d(C_\sim)$.

\item Suppose that $C_\sim$ is a quadratic periodic Fatou gap of
    period $n$. Then there is either one or several legal
    modifications $U$ of $C_\sim$ associated with the corresponding
    legal modifications of $\lam_\sim$ and corresponding legal
    quadrilaterals $Q$. In this case $m_\sim$ is defined as the
    convex hull of the union of $\si_d$-images of all these legal
    critical quadrilaterals.
\end{enumerate}

\end{dfn}

Let us discuss the minor sets from \Cref{d:minor-set}(2). Suppose that
$C_\sim$ is a periodic critical quadratic Fatou gap. The easiest case
is when there are no finite gaps attached to $C_\sim$ at the refixed
edge $M$ of $C_\sim$. In that case the corresponding legal critical
quadrilateral is the convex hull of $M$ and its sibling edge $M'$ of
$C_\sim$ so that $m_\sim=\si_d(M)$. Another simple case is when a fixed
return gap $G$ is attached to $C_\sim$ at its refixed edge $M$. In that
case a unique legally modified gap %$U\supset C_\sim$
is obtained by erasing $M$ and its grand orbit from $C_\sim$ (thus, $M$
and its pullbacks on the boundary of $C_\sim$ are replaced by
concatenations of the remaining edges of $G$ or appropriate preimages
of $G$). The minor set in this case is the $\si_d$-image of $G$.

A more involved is the case when there exists a finite gap $G$ of
rotational type attached to $C_\sim$ at $M$. Then there are several
images of $M$ that are edges of $G$. Denote the two of them closest to
$M$ in $\bd(G)$ by $L$ and $R$. Then the minor set $m_\sim$ coincides
with the convex hull of the image of the segment of the boundary of $G$
containing $M$ and stretching from $L$ to $R$ (not including either $L$
or $R$). Observe that, unlike in the original paper by Thurston
\cite{thu85} or in \cite{bopt15a}, the minor set $m_\sim$ \emph{is not}
a gap or a leaf of the corresponding invariant geodesic lamination
$\lam_\sim$.

Our aim is to show that, as in the quadratic case with Thurston's
quadratic minor lamination $\qml$, the family of minor sets of
invariant geodesic laminations from $\mathbb L(\mathcal Y)$ can be
viewed as the family of classes of equivalence of a laminational
(non-invariant!) equivalence relation $\sim_{\mathcal Y}$ such that the
quotient space $\uc/\sim_{\mathcal Y}$ of $\uc$ can be viewed as a
parameter model of $\mathbb L(\mathcal Y)$. Since we deal here with
minor sets and the critical Fatou gaps involved are all quadratic, we
will denote the geodesic (non-invariant!) lamination associated with
$\sim_{\mathcal Y}$ by $\qml_{\mathcal Y}$. If we create the
corresponding model in the plane, then we will have to ``pinch'' the
unit disk, which would yield the associated quotient space of not only
the unit circle but of the whole unit disk. This gives the ``pinched
disk'' model, which will be denoted by $\Mc(\mathcal Y)$. The boundary
of $\Mc(\mathcal Y)$ coincides with $\uc/\sim_{\mathcal Y}$.

We will also interpret bounded interior components of $\Mc(\mathcal Y)$
from the standpoint of dynamics. The description below is given without
proofs.

A bounded connected component of the interior of $\Mc(\mathcal Y)$ can
be of two types: \emph{quadratic} type and \emph{Siegel capture} type.
Components of quadratic type are similar to hyperbolic domains of the
combinatorial Mandelbrot set $\Mc_2$. Let $\Uc$ be a component of
quadratic type. It can be associated with an invariant geodesic
lamination $\lam_\sim$ with a periodic critical Fatou gap
$C_\sim\subset A(\mathcal Y)$ of period $n$ such that
$\si_d^n|_{C_\sim}$ is two-to-one. The association between $\Uc$ and
$\lam_\sim$ is uniquely defined by the properties described below.
Consider the legal modification $U$ of $C_\sim$. There is a continuous
monotone map $\psi:\bd(U)\to\uc$ that collapses all edges of $U$ and
semi-conjugates the restriction of $\si_d^n$ to $\bd(U)$ with $\si_2$.
Let $Q$ be a critical quadrilateral in $U$ such that $\psi(U)$ is a
critical quadrilateral (possibly degenerate), whose $\si_2$-image lies
in a minor set representing a boundary point of the combinatorial main
cardioid. Then $\si_d(Q)$ lies in a minor set corresponding to a
boundary point of $\Uc$. Conversely, any minor set corresponding to a
boundary point of $\Uc$ includes $\si_d(Q)$ for some critical
quadrilateral $Q\subset U$ such that $\si_2(\psi(U))$ lies in a minor
set representing a boundary point of the combinatorial main cardioid.
The lamination $\lam_\sim$ itself can be viewed as representing a
topological polynomial with an attracting periodic point inside the
Fatou component corresponding to $C_\sim$. Alternatively, we can think
of the corresponding topological polynomial as a topological polynomial
with a parabolic periodic point.

Components of Siegel capture type are very different from those that
appear in $\Mc_2$. Each such component is associated with an invariant
geodesic lamination of Siegel capture type (such invariant geodesic
laminations are excluded from $\mathbb L(\mathcal Y)$). Let $\Uc$ be a
component of Siegel capture type and $\lam_\sim$ the associated
lamination. The association between $\Uc$ and $\lam_\sim$ is uniquely
defined by the properties described below. There is an infinite
critical set $C_\sim$ of $\lam_\sim$ of Siegel capture type. Insert a
finite critical set (a leaf or a quadrilateral) into $C_\sim$. In the
new geodesic lamination instead of one gap $C_\sim$ we will have two
adjacent gaps separated by a common finite critical set (e.g., a common
edge). All such invariant geodesic laminations obtained by inserting
various critical sets in $C_\sim$ give rise to finite minor sets, which
form the boundary of $\si_d(C_\sim)$. These are precisely minor sets
corresponding to points in the boundary of $\Uc$. Thus, the boundary of
$\Uc$ can be naturally identified with $\si_d(C_\sim)$.

Recall that, given an invariant geodesic lamination $\lam$ with finite
critical sets $C_1$, $\ldots$, $C_r$, we call a full quadratically
critical portrait \emph{legal} for $\lam$ if it is geolaminational and
its critical quadrilaterals are contained in the critical sets of
$\lam$. This holds automatically if the critical quadrilaterals have
pairwise disjoint interiors and share opposite sibling edges with
critical sets of $\lam$.

\begin{lem}\label{l:link1}
Let $\lam_{\sim_1}$ and $\lam_{\sim_2}$ be two invariant geodesic
laminations from $\mathbb L(\mathcal Y)$. Suppose that their legal
modifications $\lam^{leg}_{\sim_1}$ and $\lam^{leg}_{\sim_2}$ have
legal quadratically critical portraits $\mathcal T_1=(Q_1, \mathcal Y)$
and $\mathcal T_2=(Q_2, \mathcal Y)$ respectively such that $Q_1$ and
$Q_2$ are strongly linked. Then $\sim_1=\sim_2$.
\end{lem}

Observe that if $Q_1$ and $Q_2$ share a diagonal then we can consider
the common diagonal as a (degenerate) critical quadrilateral that is
strongly linked with itself so that the conclusions of the lemma hold
in this case too.

\begin{proof}
Suppose first that at least one diagonal of a quadrilateral $Q_1$ or
$Q_2$ (say, a diagonal of $Q_1$) has non-periodic endpoints. Then by
construction $\sim_1$ has only finite critical sets. By
\Cref{l:per-sie-lam} then $\approx_{\sim_1}^{pS}=\sim_1$. On the other
hand, by \Cref{t:admi-link} the perfect-Siegel parts of $\lam_{\sim_1}$
and $\lam_{\sim_2}$ are equal. This implies that $\lam_{\sim_2}$ has
only finite critical sets, and $\sim_1=\sim_2$.

We may now assume that both $Q_1$ and $Q_2$ have periodic edges. By
\Cref{t:finite}, quadratically critical portraits $\mathcal T_1$ and
$\mathcal T_2$ give rise to invariant geodesic pullback laminations
$\lam(\mathcal T_1)$ and $\lam(\mathcal T_2)$. Let us show that they
coincide and $\approx_{\mathcal T_1}=\approx_{\mathcal T_2}$. Indeed,
by definition, they are quadratically critical and linked (because
$Q_1$ and $Q_2$ are linked). Consider a pullback $\oy$ of a leaf
$\oy_i$ with $2\le i\le d-1$ that belongs to $\lam(\mathcal T_1)$ and
prove that it is not linked with any edge of $Q_2$. Indeed, suppose
otherwise. Then $\si_d(\oy)$ and $\si_d(Q_2)$ are linked as well;
observe that $\si_d(Q_2)$ is a periodic leaf of $\lam(\mathcal T_2)$.
Recall that spikes of sets from $\mathcal T_2$ are sets $\oy_i$ with
$2\le i\le d-1$ and two diagonals of $Q_2$. By the assumptions, there
are no chains of spikes such that one endpoint of a chain is periodic
and the other one is not. Hence no two images $\si_d^q(\oy)$ and
$\si_d^q(Q_2)$ can have endpoint that coincide with distinct endpoints
of a chain of spikes. By \Cref{l:accorder} then $\si_d^q(\oy)$ and
$\si_d^q(Q_2)$ are linked for every $q$, a contradiction because for
some $q$ we have $\si_d^q(\oy)=\oy_i$, and $\si_d^q(Q_2)$ cannot be
linked with $\oy_i$.

Thus, pullbacks of the leaves $\oy_i$ with $2\le i\le d-1$ that belong
to $\lam(\mathcal T_1)$ are not linked with an edge of $Q_2$. By
definition this implies that they actually belong to $\lam(\mathcal
T_2)$. Similarly, the pullbacks of leaves $\oy_i$ with $2\le i\le d-1$
that belong to $\lam(\mathcal T_2)$ also belong to $\lam(\mathcal
T_1)$. Therefore, by definition, it follows that $\lam(\mathcal
T_2)=\lam(\mathcal T_1)$. This implies that there exists a periodic
critical quadratic Fatou gap $U$ of period, say, $m$, that is common
for both laminations and contains both $Q_1$ and $Q_2$. This allows us
to use the standard semiconjugacy $\psi$ of $\si_d^m:U\to U$ and
$\si_2:\uc\to \uc$. Then $\psi(Q_1)$ is either a diameter in $\cdisk$,
or a $\si_2$-critical quadrilateral with one edge being a major of a
periodic critical Fatou gap of $\si_2$. A similar observation applies
to $\psi(Q_2)$. The fact that $Q_1$ and $Q_2$ are strongly linked
implies then that the corresponding two $\si_2$-invariant geodesic
laminations coincide, and therefore $\sim_1=\sim_2$ as desired.
\end{proof}

We are ready to prove \Cref{t:para-lam}.

\begin{thm}[Parameter laminational equivalence
$\sim_{\mathcal Y}$]\label{t:para-lam} Minor sets of invariant geodesic
laminations from $\mathbb L(\mathcal Y)$ are classes of equivalence of
a non-invariant laminational equivalence relation $\sim_{\mathcal Y}$.
The corresponding ``pinched'' disk model $\Mc(\mathcal Y)$ contains
infinitely many pairwise disjoint copies of the combinatorial quadratic
Mandelbrot set $\ol\disk/\qml$. Components of the interior of
$\Mc(\mathcal Y)$ are either hyperbolic domains inside copies of
$\ol\disk/\qml$ or parameter components of Siegel capture type.
\end{thm}

\begin{proof}
Suppose that two geodesic laminations $\lam_\sim, \lam_\approx$ belong
to $\mathbb L(\mathcal Y)$ and have non-disjoint minor sets $m_\sim$
and $m_\approx$. Consider cases. Recall that by $A(\mathcal Y)=A$ we
denote the unique component of $\cdisk\sm \mathcal Y^+$ such that
$\si_d|_{\bd(A)}$ is two-to-one except for its critical boundary edges.

First, it may happen that there exists a common vertex $x$ of $m_\sim$
and $m_\approx$. Then by \Cref{l:crit-set} the entire set
$\si_d^{-1}(x)\cap \ol{A}$ is contained in $C_\sim\cap C_\approx$. We
can choose one critical chord $\oc\subset \si_d^{-1}(x)\cap \ol{A}$
that intersects $A$ and add $\oc$ to $\mathcal Y$ to form a new
augmented quadratically critical portrait $\mathcal Y'$. It follows
that the invariant geodesic laminations $\lam_\sim$ and $\lam_\approx$
are essentially equal. By \Cref{l:link1} this implies that
$\sim=\approx$ (the remark made right after the statement of
\Cref{l:link1} and before its proof shows that \Cref{l:link1} applies
in the case when $\lam_\sim$ and $\lam_\approx$ are essentially equal).

Second, it may happen that $m_\sim$ and $m_\approx$ do not have a
common vertex. Then there must exist an edge $\ell_\sim$ of $m_\sim$
and an edge $\ell_\approx$ of $m_\approx$ that cross. Moreover, those
edges can be chosen from legal modifications of $\lam_\sim$ and
$\lam_\approx$ as images of their critical quadrilaterals. In other
words, either edge pulls back to the corresponding collapsing
quadrilateral ($Q_\sim\subset C_\sim$ and, respectively,
$Q_\approx\subset C_\approx$) so that $Q_\sim$ and $Q_\approx$ are
strongly linked. Then, by \Cref{l:link1}, we have $\sim=\approx$. We
conclude that in any case minor sets of distinct invariant geodesic
laminations from $\mathbb L(\mathcal Y)$ are pairwise disjoint.

Let us show that the family of minor sets of invariant geodesic
laminations from $\mathbb L(\mathcal Y)$ is upper semicontinuous.
Consider a sequence of minor sets $m_1$, $m_2$, $\ldots$ of invariant
geodesic laminations $\lam_1$, $\lam_2$, $\ldots$ generated by
invariant laminational equivalence relations $\sim_1$, $\sim_2$,
$\ldots$. We may assume that the minor sets $m_i$ converge in the
Hausdorff sense, all these equivalence relations and invariant geodesic
laminations are distinct, and, by the above, all the minor sets $m_1$,
$m_2$, $\ldots$ are pairwise disjoint.

Then the limit of the sets $m_i$ is either a point or a leaf $X$. We
need to show that in fact $X$ is a subset of the minor set of an
invariant geodesic lamination from $\mathbb L(\mathcal Y)$ generated by
the appropriate laminational equivalence relation, say, $\simeq$. To
this end, we refine (if necessary) the sequence of geodesic laminations
$\lam_1$, $\lam_2$, $\ldots$ so that (by \cite{bmov13}) invariant
geodesic laminations $\lam_i$ will converge to some invariant geodesic
lamination $\lam$ in the Hausdorff sense. We have to find the desired
laminational equivalence relation $\simeq$ using the existence of
$\lam$ and its properties as a tool.

Indeed, pull back $X$ to the component $A(\mathcal Y)$ of $\cdisk\sm
\mathcal Y^+$ on which the map $\si_d$ is two-to-one. This will yield a
(generalized) critical quadrilateral, say, $Q$, such that the
quadratically critical portrait $\mathcal T=(Q, \mathcal Y)$ is
geolaminational (because $\mathcal T$ can be viewed as a quadratically
critical portrait of $\lam$). Hence, by \Cref{t:finite}, there exists
an invariant geodesic lamination $\lam_\simeq$ from $\mathbb L(\mathcal
Y)$ such that $Q$ is a subset of the corresponding critical set of a
legal modification of $\lam_\simeq$. By definition, this implies that
$X$  is contained in the corresponding minor set $m_\simeq$ as desired.

The remaining claims concerning copies of the quadratic Mandelbrot set
are rather standard and easily follow from the existence of invariant
geodesic laminations $\lam_\sim$ in $\mathbb L(\mathcal Y)$ such that
$\lam_\sim$ has a periodic critical Fatou gap of degree greater than
one (the latter in turn follows from \Cref{t:finite}). Finally, the
existence of component of $\Mc(\mathcal Y)$ associated with invariant
geodesic laminations of Siegel capture follows from the analysis given
right before \Cref{l:link1}. \end{proof}

%\appendix
%    Include appendix "chapters" here.
%\include{}

\backmatter

%\include{lam-mod-d-biblio}
%    Bibliography styles amsplain or harvard are also acceptable.
%\bibliographystyle{amsalpha}
%\bibliography{fxpt-mem-biblio}
%    See note above about multiple indexes.

\bibliographystyle{amsalpha}

\printindex

\end{document}